%% file: main_bbl.tex
\documentclass{article}
\usepackage[utf8]{inputenc}

\input{defs.sty}

\title{Prestacks of Tate type}
\author{Aron Heleodoro}
\date{\today}

\begin{document}

\maketitle

\setcounter{tocdepth}{2}
\tableofcontents

\section{Introduction}

\subsection{What is done in this paper?}

This paper initiates the study of an extension of derived algebraic geometry, where we consider Tate affine schemes as the basic building blocks. A \emph{Tate affine scheme} is a filtered diagram of affine schemes where the connecting morphisms are closed embeddings of finite presentation. 

\subsubsection{Tate schemes and Tate affine schemes}

There are a number of reasons why one might want to develop such a theory and we discuss some of them in \S \ref{subsec:why-are-we} below. For now we just mention that in many applications of derived algebraic geometry to the geometric Langlands program one has to deal with ind-schemes, for instance loop groups and the affine Grassmannian. Our first observation is that many of the ind-schemes that arise naturally are examples of Tate schemes.

\begin{thmannounce}[see \S \ref{subsec:examples-Tate-affine}]
The following are all examples of Tate schemes:
\begin{enumerate}[(i)]
    \item given any affine scheme $S$ of finite type, the loop prestack $S((t))$ is a Tate affine scheme;
    \item for any reductive group $G$ the affine Grassmannian $\Gr_G$ is a Tate scheme;
    \item any affine formal scheme whose underlying topological ring is adic and Noetherian gives rise to a Tate affine scheme.
\end{enumerate}
\end{thmannounce}

\subsubsection{Tate-coherent sheaves}

For many applications in representation theory it is necessary to have a theory of sheaves on ind-schemes. Such a theory is normally obtained by considering the colimit of the categories of sheaves on the schemes presenting the ind-scheme. This leads to the development of the theory of ind-coherent sheaves as done by Gaitsgory--Rozenblyum in \cites{GR-I,GR-II}.

Motivated by the desire to find a natural category for the ``dualizing sheaf'' of a Tate scheme to live on, we extend the formalism of ind-coherent sheaves to a theory of Tate-coherent sheaves, at least for Tate schemes satisfying a finiteness condition. The category of Tate-coherent sheaves is defined as the subcategory
\[
\TateCoh(S) \subset \Pro(\IndCoh(S)),
\]
where $\sF$ is an object of $\TateCoh(S)$ if it is an extension of an object of $\IndCoh(S)$ by an object of $\Pro(\Coh(S))$, that is $\sF$ is a Tate-object on the category $\Coh(S)$.

Section \ref{sec:sheaves} of the paper deals with extending the theory of ind-coherent sheaves (as formulated in \cites{GR-I,GR-II}) to Tate-coherent sheaves. This is done in two steps.

In the first step we develop a theory of Tate-coherent sheaves on schemes almost of finite type.

Firstly, we extend the formalism of ind-coherent sheaves with its arbitrary $!$-pullback, $*$-pushforward and base change for proper morphism to what we call \emph{pro-ind-coherent sheaves}, i.e.\ Pro-objects in the category of ind-coherent sheaves. This is done by Pro-extending the pushforward functor and defining the pullback as the appropriate left or right adjoint when the morphism is an open embedding or proper map. 

Secondly, when we consider the subcategory of Tate-objects we notice that the restriction of the $!$-pullback functor preserves Tate-coherent objects. So now we start with this functor and define the pushforward for Tate-coherent sheaves as the appropriate adjoints, analogously to the procedure for pro-ind-coherent sheaves. We use the formalism of the category of correspondences and check that the appropriate base change results still hold to obtain that the whole formalism is well-defined. The following summarizes these results

\begin{thmannounce}[Theorem \ref{thm:TateCoh-correspondence-functor-Sch}]
\label{thmann:TateCoh-on-schemes}
There exists a formalism of Tate-coherent sheaf for schemes almost of finite type, that is for any morphism $f:X \ra Y$ one has a $!$-pullback functor $f^!$ and a $*$-pushforward functor $f^{\rm Tate}_*$. These satisfy proper base change and are characterized by the properties:
\begin{enumerate}[(i)]
    \item for $g$ an open embedding one has an adjoint pair $(g^!,g^{\rm Tate}_*)$;
    \item for $f$ a proper morphism one has an adjoint pair $(f^{\rm Tate}_*,f^!)$.
\end{enumerate}
Moreover, one has a duality functor
\[
\bD^{\rm Tate}_X: \TateCoh(X) \ra \TateCoh(X)^{\rm op}
\]
that recovers the usual Serre duality when restricted to the subcategory $\Coh(X)$.
\end{thmannounce}

In the second step, we follow the ideas of Gaitsgory--Rozenblyum to extend the formalism of Theorem \ref{thmann:TateCoh-on-schemes} to any Tate scheme almost of finite type. This is done by defining a functor $\TateCoh^!$ that encodes the formalism of $!$-pullback for Tate schemes by right Kan extension from the functor for schemes, and defining the functor $\TateCoh$ that encodes the formalism of $*$-pushfoward by left Kan extension from Tate-coherent sheaves on schemes. Then we check that the natural base change isomorphisms are still satisfied, and general results from \cite{GR-I} about extension of $2$-functors in this context gives us a formalism on the category of correspondences.

\begin{thmannounce}[Theorem \ref{thm:TateCoh-correspondence-functor-Sch-Tate}]
\label{thmann:TateCoh-on-Tate-schemes}
There is an extension of the formalism of Theorem \ref{thmann:TateCoh-on-schemes} to Tate schemes locally almost of finite type defined as
\[
\TateCoh(S) := \lim_{I^{\rm op}}\TateCoh(S_i) \simeq \colim_{I}\TateCoh(S_i)
\]
where $S \simeq \colim_I S_i$ is an presentation of $S$.

Moreover, for any morphism $f:S \ra T$ between Tate schemes locally almost of finite type one has functors
\[
f^!: \TateCoh(T) \ra \TateCoh(S), \;\;\; \mbox{and} \;\;\; f^{\rm Tate}_*: \TateCoh(S) \ra \TateCoh(T).
\]
In the case where $f:S \ra T$ is an ind-proper morphism, one has adjunction $(f^{\rm Tate}_*,f^!)$ and for $g:S \ra T$ an open embedding one has $(g^!,g^{\rm Tate}_*)$; and these functors satisfy base change with respect to ind-proper morphisms.

\end{thmannounce}

We remark that the formalism of Theorem \ref{thmann:TateCoh-on-Tate-schemes} is not the most general formalism that one can develop, without much extra work one can develop a formalism of Tate-coherent sheaves for ind-inf-schemes, but we restrict ourselves to the case of Tate schemes for uniformity of presentation.

More importantly, one can ask for a version of the formalism where we don't need the assumption that our Tate schemes are locally almost of finite type. That should be possible, but we need to use a version of ind-coherent sheaves as considered in \cite{Raskin-homological}*{\S 6}. We plan to pursue that in a future version of this work.

\subsubsection{Prestacks of Tate type}

Now we introduce the main object of study in this paper, a \emph{prestack of Tate type} is a functor
\[
\sX: (\SchaffTate)^{\rm op} \ra \Spc,
\]
where $\SchaffTate$ is the category of Tate affine schemes, and $\Spc$ is the category of spaces. 

It turns out that one can develop a lot of the basic theory of derived algebraic geometry for such objects. Many of the notions that one can define for a usual prestack, i.e.\ a functor from $(\Schaff)^{\rm op}$ to $\Spc$, can be straightforwardly generalized to prestacks of Tate type. For instance, we have notions of coconnectivity, convergence, locally almost of finite type, and truncatedness. 

If our theory is to be of any use, any prestack should produce an example of a prestack of Tate type. This is actually the case and given any prestack $\sX_0$ its right Kan extension via the embedding $(\Schaff)^{\rm op} \hra (\SchaffTate)^{\rm op}$ produces a prestack of Tate type. In particular, we can define the notion of a \emph{prestack of Tate type $\sX$ being locally a prestack}, that is, the canonical map
\[
\sX \ra \RKE_{(\Schaff)^{\rm op} \hra (\SchaffTate)^{\rm op}}\left(\left.
\sX\right|_{(\Schaff)^{\rm op}}\right)
\]
is an isomorphism. The following result collects how the usual properties of a prestack interact with these properties when we see it as a prestack of Tate type. None of them are surprising, but they guarantee that the theory works well.

\begin{thmannounce}
Given $\sX_0$ a prestack, consider the associated prestack of Tate type
\[
\sX := \RKE_{(\Schaff)^{\rm op} \hra (\SchaffTate)^{\rm op}}(\sX_0).
\]
Then we have
\begin{enumerate}[(i)]
    \item if $\sX$ is locally almost of finite type, then $\sX_0$ is locally almost of finite type;
    \item if $\sX$ is $n$-coconnective, then $\sX_0$ is $n$-coconnective;
    \item $\sX$ is convergent if and only if $\sX_0$ is convergent;
    \item $\sX$ is $k$-truncated if and only if $\sX_0$ is $k$-truncated.
\end{enumerate}
\end{thmannounce}

The category of Tate affine schemes $\SchaffTate$ admits Grothendieck topologies analogous to the topologies for usual affine schemes. For instance, given a map $f:S \ra T$ between Tate affine scheme, this morphism is \'etale if one of the two following equivalent conditions is satisfied:
\begin{itemize}
    \item $f$ is \'etale as a morphism of prestacks;
    \item one can find common presentations $S \simeq \colim_I S_i$ and $T \simeq \colim_I T_i$ and morphisms
    \[
    \begin{tikzcd}
    S_i \ar[r,"f_i"] \ar[d] & T_i \ar[d] \\
    S \ar[r,"f"] & T
    \end{tikzcd}
    \]
    such that $f_i$ is \'etale.
\end{itemize}

One has similar descriptions for the Zariski and flat topologies. These topologies allow us to define the notion of stacks of Tate type, which we simply refer to as Tate stacks. 

As with the usual theory of stacks, we can introduce a hierarchy of geometric conditions that we call \emph{Tate $k$-Artin stacks}. We define these objects by taking Tate $0$-Artin stacks to be Tate affine schemes, and Tate $k$-Artin stacks to be Tate stacks whose diagonal morphism is Tate $(k-1)$-Artin representable and admits a smooth surjective cover from a Tate $(k-1)$-Artin stack. The following result shows that these notions are compatible with the usual notions from derived algebraic geometry.

\begin{thmannounce}
Let $\sX_0$ be a prestack and consider $\sX:= \RKE_{(\Schaff)^{\rm op} \hra (\SchaffTate)^{\rm op}}(\sX_0)$, then
\begin{enumerate}[(i)]
    \item $\sX$ satisfies flat (resp.\ \'etale, Zariski) descent with respect to Tate affine schemes if and only if $\sX_0$ satisfies flat (resp.\ \'etale, Zariski) descent with respect to affine schemes;
    \item if $\sX_0$ is a $k$-Artin stack, then $\sX$ is a Tate $k$-Artin stack.
\end{enumerate}
\end{thmannounce}

One of the aspects of the theory of prestacks of Tate type that differs from usual derived algebraic geometry is the notion of Tate schemes. The point here is that there are two notions that one could consider for Tate schemes: (1) ind-schemes whose closed embeddings are finitely presented and (2) Tate stacks that admit a Zariski atlas. We define the notion of a \emph{Zariski Tate stack} to be a prestack of Tate type $\sX$ that satisfy the following:
\begin{itemize}
    \item $\sX$ is an \'etale sheaf,
    \item the morphism $\sX \ra \sX \times \sX$ is representable be Tate affine scheme, and
    \item $\sX$ admits a Zariski atlas, i.e.\ there exists a surjective morphism $S \ra \sX$, where $S$ is a disjoint union of Tate affine schemes $S_a$, and $S_a \ra \sX$ is an open embedding.
\end{itemize}

Let $\SchTate$ denote the category of Tate schemes, $\StkZarTateqc$ the category of Zariski Tate stacks that admit a cover by finitely many Tate affine schemes, and for $k \geq 0$ let $\ArtinStkTaten{k}$ be the category of Tate $k$-Artin stacks. The following summarizes the relation between these geometrical objects.

\begin{thmannounce}
All the following functors are fully faithful
\[
\begin{tikzcd}
& & \SchTate \\
\SchaffTate \simeq \ArtinStkTaten{0} \ar[r,hook] & \StkZarTateqc \ar[ru,hook] \ar[rd,hook] & \\
& & \ArtinStkTaten{1}.
\end{tikzcd}
\]
\end{thmannounce}

At this moment we couldn't determine if the category $\SchTate$ embeds into the category $\ArtinStkTaten{1}$, as is the case for usual schemes, or if the notions of Zariski Tate stacks and Tate schemes are equivalent. We hope to return to these questions in the future.

Analogously to \cite{GR-II}*{Chapter 3} one can automatically extend the formalism of Theorem \ref{thmann:TateCoh-on-Tate-schemes} to the following $2$-category $\mbox{Corr}(\PStkTatelaft)^{\rm aff-aft-ind-proper}_{\rm aff-aft;all}$
\begin{itemize}
    \item objects are prestacks of Tate type locally almost of finite type;
    \item $1$-morphisms are correspondences, whose vertical morphisms are Tate affine schematic and almost of finite type, and whose horizontal morphisms are arbitrary;
    \item $2$-morphisms are Tate affine schematic, almost of finite type, and ind-proper.
\end{itemize}

\begin{thmannounce}[Theorem \ref{thm:TateCoh-correspondence-functor-PStk-Tate}]
There is a unique functor
\[
\TateCoh_{\mbox{Corr}(\PStkTatelaft)^{\rm aff-aft-ind-proper}_{\rm aff-aft;all}}: \mbox{Corr}(\PStkTatelaft)^{\rm aff-aft-ind-proper}_{\rm aff-aft;all} \ra \DGct
\]
extending the formalism of Theorem \ref{thmann:TateCoh-on-Tate-schemes}.
\end{thmannounce}

\subsubsection{Dimension and determinantal theories}

One application of the theory of Tate-coherent sheaves is that it gives a natural definition of a canonical $\Gm$-gerbe on any Tate scheme $S$. For $p:S \ra \Spec(k)$ the structure morphism of a Tate scheme locally almost of finite type we define the dualizing Tate-coherent sheaf as
\[
p^!(k) \in \TateCoh(S),
\]
where we consider $k$ as an object of $\Tate(k) \simeq \TateCoh(\Spec(k))$. By \cite{determinant-map}*{Theorem A} we have a morphism
\[
\sD: \sTate(S) \ra \sBPic(S),
\]
which can be seen as an index map. Thus the natural $\Gm$-gerbe on $S$ is defined as
\[
\omega^{(1)}_S = \sD(p^!(k)).
\]

In the case that $S$ is formally smooth and has a presentation
\[
S \simeq \colim_I S_i,
\]
where each $S_i$ is smooth, we have the following

\begin{thmannounce}[\ref{thm:trivialization-is-determinantal-theory}]
\label{thmann:determinantal-theory}
The data of a determinantal theory\footnote{We recall that a \emph{determinantal theory} (see \cite{Previdi}) on $S \simeq \colim_I S_i$ is the data of:
\begin{itemize}
    \item a line bundle $\sL_i \in \sPic(S_i)$ for every $i \in I$ and 
    \item an isomorphism
    \[
    \tau_{i,j}: f^*_{j,i}(\sL_{i})\otimes (\det(T^*(S_j/S_i)))^{-1} \overset{\simeq}{\ra}\sL_j
    \]
    for every $j \ra i$, where $f_{j,i}:S_j \hra S_i$ is the canonical inclusion.
\end{itemize}
} on $S$ is equivalent to a trivialization of the $\Gm$-gerbe
\[
\sD^{(1)}(p^!(k)) \in \sB\sPic(S).
\]
\end{thmannounce}

Similarly to Theorem \ref{thmann:determinantal-theory}, we obtain the notion of a dimension theory for a Tate scheme $S$, if instead of $\sD$ we consider the composite morphism $\sTate \ra \sB\sPic^{\rm gr} \ra \sB \bZ$.

Though the trivialization of the $\Gm$-gerbe $\sD(p^!(k))$ carries interesting information to perform the renormalization of certain sheaf categories and corresponding cohomologies see \cite{Raskin-homological} and \cite{Raskin-D-modules} for discussion, there are advantages to considering the $\Gm$-gerbe itself. For instance, in the case that we have a Tate object $V$ over a commutative ring $R$, the $\Gm$-gerbe $\sD^{(1)}(V)$ can be described as an appropriate category of modules for the Clifford algebra $\Cl(V \oplus V^{\vee})$, without the choice of a trivialization (see \cite{BBE}*{\S } for an account of these ideas). We hope to establish a similar result in families that would describe the $\Gm$-gerbe $\sD(p^!(k))$ without any choice of trivialization.

\subsubsection{Deformation theory}

\textbf{Disclamer:} At the moment section \ref{sec:deformation-theory} is written in a restricted set up and is more complicated than it has to be. A future version of this preprint will rework it.

One of the reasons that we want to study prestacks of Tate type is that we would like to have the ability to consider square-zero extension of a point on a prestack $x:S \ra \sX$ by a Tate sheaf on $S$. At the moment we developed a deformation theory, where we consider split square-zero extension of $S$ by an object of $\sF \in \TateCoh(S)$. 

The first result in the study of deformation theory for prestacks of Tate type is that for a Tate affine scheme $S$ locally almost of finite type and $\sF \in \TateCoh(S)$, there exists a well-behaved notion of split square-zero extension.

\begin{thmannounce}[Proposition \ref{prop:RealSplitSqZ-Tate-affine-schemes}]
\label{thmann:split-square-zero-extension-for-Tate-affine}
There exists a functor
\[
\RealSplitSqZ: \TateCoh(S)^{\leq 0} \ra (\SchaffTate)_{/S}.
\]
Moreover it is functorial with respect to push-forward via $f:S \ra T$ a morphism of Tate affine schemes.
\end{thmannounce}

Using Theorem \ref{thmann:split-square-zero-extension-for-Tate-affine} we define what it means for a prestack of Tate type $\sX$ locally almost of finite type to admit a pro-cotangent space at a point $x:S \ra \sX$, where $S$ is a Tate affine scheme locally almost of finite type. 

One of the main results is that in the case where $\sX$ is locally a prestack and convergent, the pro-cotangent space that we consider coincides with the pro-cotangent space defined in Gaitsgory--Rozenblyum at points given by affine schemes.

\begin{thmannounce}
\label{thmann:cotangent-space-at-a-point}
For $\sX$ a convergent prestack of Tate type, which is locally given by the prestack $\sX_0$. For $x:S_0 \ra \sX$ a point, where $S_{0}$ is an affine scheme, then we have an equivalence
\[
T^*_{x}\sX \simeq T^*_{x}\sX_0.
\]
\end{thmannounce}

We also develop the theory of pro-cotangent complexes for prestacks of Tate type locally almost of finite type. However, because of our initial choice to consider square-zero extension by objects of Tate-coherent sheaves, we need to restrict to convergent prestacks of Tate type. This restriction will be lifted when we rewrite Section 6. Similarly to Theorem \ref{thmann:cotangent-space-at-a-point}, we have a result that says that the existence of (pro-)cotangent complexes for a prestack guarantees the existence of its (pro-)cotangent complex when seen as a prestack of Tate type.

\begin{thmannounce}
For $\sX_0$ a prestack and $\sX:= \RKE_{(\Schaff)^{\rm op} \hra (\SchaffTate)^{\rm op}}(\sX_0)$ if $T^*\sX_0$ exists then $T^*\sX$ exists.
\end{thmannounce}

Finally, we also study some of the usual questions that one can address with deformation theory. In particular, we prove the following result that allows one to use the cotangent complex of prestacks of Tate type to detect if a morphism that induces an isomorphism at the classical level is actually an isomorphism of prestacks of Tate type.

\begin{thmannounce}
Let $\sX_1$ and $\sX_2$ be prestacks of Tate type that admit deformation theory, and assume that we have a morphism $f:\sX_1 \ra \sX_2$ such that
\[
\classical{f}: \classical{\sX_1} \ra \classical{\sX_2}
\]
is an isomorphism of the underlying classical prestacks of Tate type. Moreover, suppose that the canonical map
\[
f^{*}T^*\sX_2 \ra T^*\sX_1
\]
is an isomorphism. Then the map $f$ is an isomorphism of prestacks of Tate type.
\end{thmannounce}

\subsection{Why are we pursuing this?}
\label{subsec:why-are-we}

\subsubsection{Central extensions of higher loop Lie algebras}

Our initial motivation for embarking on this project was to answer the following question. In \cite{determinant-map} a higher determinant map was constructed
\begin{equation}
\label{eq:higher-determinant-map}
\sD^{(n)}: \sTate^{(n)} \ra \sB^{(n)}\sPic,    
\end{equation}
where $\sTate^{(n)}$ denotes the prestack of $n$-Tate objects, i.e.\ it sends an affine scheme $S$ to the category
\[
\Tate^n(\Perf(S));
\]
and $\sB^{(n)}\sPic$ is the prestack obtained by delooping the stack of line bundles $n$ times, i.e.\ inductively this is given by
\[
\sB^{(n)}\sPic(S) := \left| \cdots \begin{tikzcd}
 \sB^{(n-1)}\sPic(S) \times \sB^{(n-1)}\sPic(S) \ar[r,shift left = 1.5ex] \ar[r] \ar[r,shift right=1.5ex] & \sB^{(n-1)}\sPic(S) \ar[r,shift left] \ar[r,shift right] & \ast
\end{tikzcd} \right|.
\]

By considering the map induced by (\ref{eq:higher-determinant-map}) on automorphisms of a point, one obtains a central extension
\begin{equation}
\label{eq:central-extension-of-automorphism-of-n-Tate-object}
    1 \ra \B^{n-2}\Gm \ra G^{\wedge}_{n} \ra \GL(k((t_1))\ldots((t_n))) \ra 1.
\end{equation}
The central extension (\ref{eq:central-extension-of-automorphism-of-n-Tate-object}) can then be used to define central extensions of higher loop groups, e.g.\ $\GL_n((t_1))\ldots((t_n))$. However, we are interested in understanding the central extensions of the associated Lie algebras. We would like to determine the object $\fg^{\wedge}_{n}$ corresponding to the Lie algebra of $G^{\wedge}_{n}$, which has the structure of a central extension
\begin{equation}
    \label{eq:central-extension-Lie-algebra-n-Tate}
    0 \ra k[n-2] \ra \fg^{\wedge}_{n} \ra \End(k((t_1))\ldots((t_n))) \ra 0.
\end{equation}

We observe that the object $\End(k((t_1))\ldots((t_n)))$ will naturally have the structure of a $2n$-Tate object, and we would want to keep track of that structure. It is expected that the central extension (\ref{eq:central-extension-Lie-algebra-n-Tate}) is described by a higher analogue of the trace morphism for endomorphisms of a perfect complex (cf.\ \cite{BGW-ideals}). 

Now we notice, that the computation of \cite{determinant-map}*{Proposition 3.13} shows that the trace morphism for perfect complexes is recovered from the morphism
\[
\sD^{(0)}: \sPerf \ra \sPic
\]
as follows. For a point $Spec(k) \overset{V}{\ra} \sPerf$ the trace of an endomorphism $\varphi \in \End(V)$ is given by the morphism between cotangent complexes induced by $\sD^{(0)}$
\[
\sD^{(0)}_{V,*}: \End(V) \simeq T^*_{V}\sPerf \ra T^{*}_{\sD^{(0)}(V)}\sPic \simeq k.
\]

Thus, following the same strategy we hope to understand the higher trace morphisms characterizing the central extension (\ref{eq:central-extension-Lie-algebra-n-Tate}) by considering the morphism induced by (\ref{eq:higher-determinant-map}) on cotangent complexes.

The problem here is that the cotangent complex of a point $\Spec(k) \overset{V_n}{\ra} \sTate^{(n)}$ will naturally be an object of $\Pro(\QCoh(\Spec(k))^-)$. However, we expect that
\[
T^*_{V_n}\sTate^{(n)} \simeq \End(V)
\]
belongs to $\Tate^{2n}(\QCoh(\Spec(k)))$. Thus, we get

\begin{goal}
\label{goal:Tate-pro-cotangent-complex}
Refine the construction of the pro-cotangent complex at a point $S \ra \sX$ of a prestack to obtain an object in $\Tate^m(\QCoh(S))$.
\end{goal}

It turns out that a lot of theory needs to be set up in order to achieve the goal above. To explain why, we recall how the definition of the pro-cotangent complex of a prestack works. Given a point $x:S \ra \sX$ and a sheaf $\sF \in \QCoh(S)^{\leq 0}$ we consider 
\[
\Maps_{S/}(S_{\sF},\sX)
\]
the space of lifts of $x$ to the split square-zero extension $S_{\sF} := \Spec_S(\sO_{S}\oplus \sF)$. Then we obtain the pro-cotangent complex as the object that pro-co-represents \footnote{Strictly speaking the above only sees the connective part of $T^*_{x}\sX$ (see \S \ref{subsec:pro-cotangent-spaces} for the correct definition), but this informal account is enough to motivate our introduction.} the assignment $\sF \mapsto \Maps_{S/}(S_{\sF},\sX)$, i.e.
\[
\Hom_{\Pro(\QCoh(S)^{\leq 0})}(T^*_x\sX,\sF) \simeq \Maps_{S/}(S_{\sF},\sX).
\]

For a start, let's focus on the case $n=1$ of our Goal. If we want the object $T^*_{x}\sX$ to have a structure of $\Tate^2(\QCoh(S))$, or even $\Tate(\QCoh(S))$, we need to probe it with objects $\sF \in \Tate(\QCoh(S)^{\leq 0})$. This quickly introduces many complications:
\begin{enumerate}[(i)]
    \item one needs to perform the Tate construction on a prestable $\infty$-category to obtain $\Tate(\QCoh(S)^{\leq 0})$;
    \item given an object $\sF \in \Tate(\QCoh(S)^{\leq 0})$ and an affine scheme $S$, we need to define an object $S_{\sF}$ which performs the split square-zero of $S$;
    \item we need to understand what it means to lift a map from $S_{\sF} \ra \sX$ to the square-zero extension $S_{\sF}$;
    \item we need to have appropriate categories of sheaves on each $S$ where $T^*_{x}\sX$ will live.
\end{enumerate}

This text tries to address some of these points.

\subsubsection{Where else would such a theory be useful?}

\paragraph{Shifted symplectic structure on infinite-dimensional objects}

In recent years it has become clear that it is very important to pursue ideas of symplectic geometry in the context of derived geometry, namely for Artin derived stacks. However one of the crucial aspects of the theory is that to formulate what it means for a prestack $\sX$ to have a shifted symplectic structure, one needs the technical condition that the cotangent complex $T^*\sX$ is a perfect object, i.e.\ $T^*\sX \in \Perf(\sX)$. Nevertheless there are situations where one would expect a prestack $\sX$ whose cotangent complex is not perfect to admit a shifted symplectic structure, see for example \cite{Elliott} and \cite{Hilburn}. 

We suggest that our theory of cotangent complexes with a Tate structure is adequate to formulate this notion in big generality. Indeed, one of the reasons that one required $T^*_{x}\sX$ to be perfect is because an $n$-shifted symplectic structure produces an isomorphism
\[
T^*_{x}\sX \overset{\simeq}{\ra} T_{x}\sX[n].
\]
That is, the tangent complex of $\sX$ is $T_{x}\sX$, which is the dual object to $T^*_{x}\sX$ in the category $\Perf(\sX)$. In \cite{GR-II} the tangent complex for a prestack locally almost of finite type $\sX$ is considered as an object in $\IndCoh(\sX)$, thus it is not immediately clear what it would mean for it to be isomorphic to $T^*\sX$, which lives in $\Pro(\QCoh(\sX)^-)$. The advantage of formulating the cotangent complex $T^*_{x}\sX$ as an object of $\Tate(\QCoh(\sX)^-)$ instead of $\Pro(\QCoh(\sX)^-)$ is that if we require that $T^*_{x}\sX$ belongs to $\Tate(\Perf(\sX))$, this category is self-dual and we can also define a tangent complex in the category $\Tate(\Perf(\sX))$ as the dual of $T^*_{x}\sX$.

We are working on using the cotangent complex defined in this paper to formulate the notion of a shifted symplectic structure for a certain class of prestacks of Tate type.

\paragraph{The elusive $2$-affine Grassmannian}

Another of our motivations to pursue the study of such infinite-dimensional objects is to be able to define a higher version of the affine Grassmannian. The $2$-affine Grassmannian associated to a $2$-Tate object $k((t_1))((t_2))$ is the geometric object whose functor of points sends a commutative ring $R$ to
\[
\Gr_{k((t_1))((t_2))}(R) = \{ L \subset k((t_1))((t_2)) \; | \; L \mbox{ is a lattice}\}.
\]

There have been numerous attempts in literature to consider slices of $\Gr_{k((t_1))((t_2))}$. We expect that understanding $2$-Tate schemes (and higher) would allow us to get a better understanding of this geometric object.

\subsubsection{What connections do we expect to other theories in the literature?}

\paragraph{Placid $\infty$-stacks}

In the recent impressive preprint \cite{BKY} the authors considered a theory similar to what we develop in this paper. Their building blocks are prestacks $\sX$ that can be written as
\[
\sX \sim \lim_{J^{\rm op}}\colim_I S_{j,i},
\]
where $I$ and $J$ are filtered diagrams, $S_{i,j}$ are affine schemes of finite type, and the connecting morphisms
\[
S_{j',i} \ra S_{j,i}
\]
are finitely presented and smooth. Then they define $n$-placid stacks inductively as the prestacks $\sX$ that 
\begin{itemize}
    \item have a $(n-1)$-placid stack diagonal, and
    \item admit a smooth surjective morphism
    \[
    \sY \ra \sX,
    \]
    where $\sY$ is $(n-1)$-placid.
\end{itemize}

In \cite{BKY} placid stacks are used to prove versions of the decomposition theorem for semi-small maps and applied to the study of affine Springer theory.

Our Tate $n$-Artin stacks are geometric objects somewhat dual to $n$-placid stacks. In \S \ref{subsubsec:digression-pro-ind-presentation} we have a statement about how certain $0$-placid stacks give rise to Tate affine schemes. However, we would like to investigate the precise relation between the objects defined by Bouthier-Kazhdan-Yakhovsky and the ones we study in this article.

\paragraph{Tate affine schemes}

We notice that the theory of Tate affine schemes already has many interesting natural questions that we didn't get to address in this preprint. In \S \ref{subsubsec:commutative-algebra-in-Tate-objects} we make sense of commutative algebra objects in the category $\Tate(k)$; one question we would like to understand is if we can produce a functor from the category of commutative algebra objects in Tate objects to the category of Tate affine schemes.

Throughout this paper we took $k$ as a field of characteristic zero: the main reason is that we needed this assumption to use some results from \cites{GR-I,GR-II}. One generalization that one can pursue is to consider Tate affine schemes and Tate prestacks over a more general base. This generalization would certainly be desirable for certain applications of the theory.

A second, perhaps more interesting generalization, would be the following. If we think of Tate affine schemes as the theory of commutative algebras in the category $\Tate(k):= \Tate(\Perf(k))$, one could consider a base category more general than $\Perf(k)$. One natural generalization here would be the category of locally compact abelian groups, or a modern perspective on it as the category of condensed abelian groups. We notice that a general theory of how to lift algebraic geometry in a symmetric monoidal category $\sC$ to the category $\Tate(\sC)$, would apply to the formalism of condensed mathematics and allow one to consider very infinite-dimensional objects in the geometrical theories based on these objects.

\subsection{Outline of the paper}

\subsubsection{Description of sections}

Section \ref{sec:Tate-schemes} defines the main objects on which the theory of this article is built, namely Tate affine schemes. We also introduce Tate schemes, define natural conditions on these, and check that certain objects are examples of Tate schemes.

Section \ref{sec:sheaves} is one of the main building blocks for the theory we develop. We proceed roughly in two steps: the first is to construct Tate-coherent sheaves on schemes, the second is to bootstrap this formalism to Tate schemes. Each step naturally splits into two: first we consider the formalism of Pro-Ind-coherent sheaves, then we specialize to Tate-coherent sheaves inside of those.

In section \ref{sec:prestacks-Tate} we introduce prestacks of Tate type, which are simply functors from the opposite category of Tate affine schemes to spaces. We also define natural conditions on them and study their relation to usual conditions on prestacks.

In section \ref{sec:Tate-stacks} we study the geometric conditions that can be imposed on prestacks of Tate type: \'etale descent and admitting a smooth or Zariski cover.

In section \ref{sec:deformation-theory} we study the deformation theory of prestacks of Tate type. This section puts together a lot of the technical work from the previous sections. At the moment it is written in a somewhat restricted framework because of the initial conventions. We plan to rewrite this section in larger generality.

In section \ref{sec:application-gerbe} we apply the theory of Tate-coherent sheaves to define a dualizing gerbe for any Tate scheme, and when these Tate schemes satisfy certain technical conditions, we identify trivializations of the dualizing gerbe.

\subsubsection{What is not done in this paper?}

This article suggests many interesting open questions, some of them we plan to pursue in the future.

In section \ref{sec:Tate-schemes} it would be worthwhile to better understand the interaction between formal affine schemes and Tate affine schemes. The theory of ind-schemes is useful to study formal schemes (see \cite{Emerton}), so we expect that the class of Tate schemes can play an important role in studying formal schemes. In particular, we are interested in what type of adic geometry can be obtained from Tate affine schemes, that is, what is recovered by the localization of the category of Tate (affine) schemes with respect to an appropriate notion of admissible blow-ups.

In section \ref{sec:sheaves} one could envision a formalism of ind-coherent or Tate-coherent sheaves on ind-Tate-schemes, i.e.\ filtered colimits of Tate schemes taken in the category of prestacks of Tate type; and further extend this sheaf formalism for prestacks of Tate type with morphisms which are represented by such objects. Another interesting extension of the formalism would be to consider ind-coherent sheaves on Tate schemes which are \emph{not} locally almost of finite type. We hope to pursue this following the ideas in \cite{Raskin-homological}.

In section \ref{sec:Tate-stacks} it would be interesting to understand under which conditions Tate schemes admit any sort of atlas as a Tate stack. This is a slightly subtle problem to tackle because the presentation as a colimits of schemes doesn't allow one to use atlases of individual schemes in the colimit diagram to produce an atlas of the corresponding Tate scheme. We plan to return to this question later, once we develop a bit more of the theory of prestacks of Tate type and a form of representability result, that we hope would allow us to construct these atlases in a more abstract way.

Section \ref{sec:deformation-theory} stops short of addressing many questions. For example, what is the relation between deformation theory and the condition of a prestack of Tate type being locally a prestack; or what is the general definition of pro-cotangent complexes for prestacks of Tate type which are not convergent; and what is the relation between the pro-cotanget complex of a prestack of Tate type and the usual pro-cotangent complex of its underlying prestack. We will return to these questions in the near future.

\subsubsection{Conventions}

We will refer to $\infty$-categories (similarly any operation, property, or functors between them) simply as categories. In this paper by affine schemes we mean derived affine schemes, similarly schemes mean (derived) schemes, stacks are (derived) stacks, etc. Throughout this paper we work over a base field $k$ of characteristic $0$.

\subsubsection{Acknowledgements}

It is my pleasure to thank W.~Balderrama, C.~Dodd, C.~Elliott, and I.~Mirkovi\'c for the stimulating discussions; and A.~Alibek for comments on an earlier draft. Moreover, I would like to especially thank Nick Rozenblyum for his constant support throughout this project. Finally, this work owes a great intellectual debt: to the work of V.~Drinfeld on the notion of Tate objects and a vision of how to incorporate Tate objects into Algebraic Geometry, secondly to the work of O.~Braunling, M.~Groechenig, and J.~Wolfson on the modern formulation of Tate objects, and finally to the work of D.~Gaitsgory and N.~Rozenblyum on formalizing ind-coherent sheaves in great generality.

\section{Tate schemes}
\label{sec:Tate-schemes}

In this section we develop the theory of Tate schemes and Tate affine schemes. Tate affine schemes will be the basic building blocks for the theory of prestacks of Tate type; they should be thought of as a special class of commutative $k$-algebras endowed with a topology. In a certain sense, a good part of this section specializes the theory of indschemes as developed in \cite{DG-indschemes} and \cite{GR-II}*{Chapter 2}. However, some things are not formal; for example, we need to check that the subcategory of Tate (affine) schemes is well-behaved with respect to certain operations such as fiber products and certain push-outs.

We describe the contents of this section. In \S \ref{subsec:pro-schemes-of-ft} we review the basics of derived algebraic geometry that are used throughout this work. In \S \ref{subsec:Tate-schemes} we define Tate schemes and Tate affine schemes. In Section \ref{subsec:pro-commutative-algebras} we embed the category of Tate affine schemes into a certain category of commutative algebra objects. In section \ref{subsec:properties-Tate-schemes} we study some properties of Tate schemes. In \S \ref{subsec:other-definitions} we relate some of our notions to other objects considered in the literature. In \S \ref{subsec:examples-Tate-affine} we define the loop functor and give some examples of Tate schemes.

\subsection{Pro schemes of finite type}
\label{subsec:pro-schemes-of-ft}

In this section we set some of our notation and recall standard results from derived schemes over a ground field. Our main point is to convey the idea that schemes can be thought as Pro-objects in the category of schemes of finite type. Thus, motivating us to take the approach of defining Tate schemes as an appropriate subcategory of Ind-Pro-objects in schemes of finite type.

Let $\Vect$ denote the category of vector spaces over $k$ endowed with its usual t-structure $(\Vect^{\leq 0},\Vect^{\geq 0})$, we will be interested in $\CAlg(\Vect^{\leq 0})$, the category of commutative algebras all of whose cohomology groups vanish in positive degrees. We define the category of \emph{affine schemes} as 
\[
\Schaff := \left(\CAlg(\Vect^{\leq 0})\right)^{\rm op}.
\]

\subsubsection{Coconnective affine schemes}

An affine scheme $S$ is said to be \emph{$n$-coconnective} if $S = \Spec(A)$ such that
\[
\H^{-i}(A) = 0, \;\;\; \mbox{for all} \;\;\; i > n.
\]

We denote by ${^{\leq n}\Schaff}$ the subcategory of $\Schaff$ spanned by $n$-coconnective affine schemes. Moreover, one defines the category of \emph{eventually coconnective} affine schemes as
\[
\Schaffconv = \bigcup_{n \geq 0}{^{\leq n}\Schaff}.
\]

\begin{defn}
\label{defn:affine-schemes-of-finite-type}
For $S$ an eventually coconnective affine scheme we say that $S$ is \emph{of finite type} if $S = \Spec(A)$ with\footnote{Notice that only finitely many $\H^{-i}(A)$ are non zero by hypothesis.}
\begin{enumerate}[(i)]
    \item $\H^0(A)$ is a finitely generated $k$-algebra;
    \item and for every $i \geq 1$ 
    \[
    \H^{-i}(A) \;\; \mbox{finitely generated} \;\H^0(A)\;\mbox{-module.}
    \]
\end{enumerate}
\end{defn}

We let $\Schaffconv_{\rm ft}$ denote the subcategory of $\Schaffconv$ generated by the eventually coconnective affine schemes of finite type.

\begin{prop}
\label{prop:finite-type-is-compact}
Given $S \in \Schaffn$ the following are equivalent:
\begin{enumerate}[1)]
    \item $S$ is of finite type;
    \item $S$ is a cocompact object in $\Schaffn$, i.e.
    \[
    \Hom_{\Schaffconv}(-,S)
    \]
    commutes with cofiltered limits.
\end{enumerate}
\end{prop}

\begin{proof}
This is a particular case of \cite{HA}*{Proposition 7.2.4.27 (2)} when we restrict to $n$-coconnective affine schemes. Also notice that since $k$ and $\H^0(A)$ are Noetherian the condition of finite presentation of \emph{loc. cit.} is equivalent to our requirement of finite generation in Defintion \ref{defn:affine-schemes-of-finite-type}.
\end{proof}

\begin{cor}
\label{cor:pro-schemes-n-coconnective-ft-are-all}
There is an equivalence of categories
\[
\Pro(\Schaffn_{\rm ft}) \simeq \Schaffn
\]
\end{cor}

\begin{proof}
This follows from \cite{HTT}*{Proposition 5.3.5.11}.
\end{proof}

\begin{rem}
We notice that in the case of $n=0$, Corollary \ref{cor:pro-schemes-n-coconnective-ft-are-all} proves that the category of classical affine schemes is equivalent to the category of Pro-objects in classical affine schemes of finite type. This was observed in \cite{Kapranov-Vasserot} and is proved in \cite{EGAIV}*{Corollary 8.13.2}.
\end{rem}

\subsubsection{Almost of finite type}

For a general affine scheme $S$, i.e.\ not necessarilly $n$-coconnective for any $n$, one has to be a bit more careful to formulate its finiteness conditions. First we recall that one has a right adjoint to the natural inclusion
\begin{align*}
    \Schaff & \ra \Schaffn \\
    S & \mapsto {^{\leq n}S}
\end{align*}

\begin{defn}
\label{defn:affine-scheme-aft}
Given an affine scheme $S$, one says that $S$ is \emph{almost of finite type}, if for every $n \geq 0$ the scheme ${^{\leq n}S} \in \Schaffn$ is of finite type. We denote the category of affine schemes almost of finite type by $\Schaff_{\rm aft}$.
\end{defn}

\begin{rem}
\label{rem:aft-is-not-compact}
It is important to remark that the condition of being almost of finite type is \emph{not} equivalent to requiring that $S$ is a cocompact object in the category $\Schaff$. In fact, if $S$ is a cocompact object in $\Schaff$ then it is almost of finite type, but the converse doesn't hold in general. Indeed, if $S = \Spec(A)$ and $A$ is compact, then by \cite{HA}*{Proposition 7.2.4.27} $A$ is the finite colimit of a finitely generated free $k$-algebra and by \cite{HA}*{Proposition 7.2.4.31} the affine scheme associated to any one of those is almost of finite type. To see that the converse is false consider $k$ as a $k[\epsilon]/(\epsilon^2)$-algebra, again by \cite{HA}*{Proposition 7.2.4.31} one has that $\Spec(k)$ is almost of finite type over $k[\epsilon]/(\epsilon^2)$, however a simple calculation shows that $\Spec(k)$ is not cocompact.
\end{rem}

\begin{rem}
The categorical concept that coincides with the notion of $S$ being almost of finite type is that of \emph{almost cocompact object}\footnote{This should be compared with almost perfect complexes, which are the better category to consider in derived geometry, rather than the subcategory of quasi-coherent sheaves with coherent cohomology.}, i.e.\ ${^{\leq n}S}$ is a cocompact object in $\Schaffn$ for every $n$.
\end{rem}

\subsubsection{Quasi-compact schemes}

We take the point of view of \cite{GR-I}*{Chapter 2, Section 3} and we define the category of schemes $\Sch$ as follows.

\begin{defn}
\label{defn:schemes}
A scheme $Z$ is a prestack, i.e.\ a functor $\Schaffop \ra \Spc$ satisfying:
\begin{enumerate}[a)]
    \item $Z$ satisfies \'etale descent;
    \item the diagonal map $Z \ra Z\times Z$ is affine schematic and a closed embedding\footnote{As the convention in \cite{GR-II} for us schemes are what would normally be called quasi-separated schemes in the literature.};
    \item there exists a collection of maps $\{f_i:S_i \ra Z\}_I$ where $S_i$ are affine schemes such that
        \begin{itemize}
            \item each $f_i$ is an open embedding;
            \item for every $T \in (\Schaff)_{/Z}$ the induced maps $\{\classical{S_i\times_{Z}T} \ra \classical{T}\}$ cover $\classical{T}$.
        \end{itemize}
\end{enumerate}
A collection $\{f_i:S_i \ra Z\}_I$ satisfying condition c) above is called a \emph{Zariski cover of $Z$}. 
\end{defn}

One can generalize the properties we define for affine schemes for schemes by requiring the similar condition for a cover. We refer the reader to \cite{GR-I}*{Chapter 2, Section 3.3 and Section 3.5} for a discussion of why this works.

\begin{defn}
\label{defn:properties-of-schemes}
Let $Z \in \Sch$ we say that
\begin{enumerate}[1)]
    \item $Z$ is \emph{$n$-coconnective} if it admits a Zariski cover by affine schemes which are $n$-coconnective;
    \item $Z$ is \emph{locally almost of finite type} if it admits a Zariski cover with affine schemes from $\Schaff_{\rm aft}$;
    \item $Z$ is \emph{quasi-compact} if it admits a Zariski cover by a finite number of affine schemes;
    \item $Z$ is \emph{almost of finite type} if it is locally almost of finite type and quasi-compact.
\end{enumerate}
We denote the category of $n$-coconnective schemes by $\Schn$, the category of schemes locally almost of finite type by $\Schlaft$, the category of quasi-compact schemes by $\Schqc$ and the category of schemes almost of finite type by $\Schaft$.
\end{defn}

\begin{defn}
\label{defn:properties-of-n-coconnective-schemes}
Let $Z \in \Schn$ we say that 
\begin{enumerate}[1)]
    \item $Z$ is \emph{locally finite} if it admits a Zariski cover with affine schemes from $\Schaffnft$;
    \item $Z$ is of \emph{finite type} if it is locally finite and quasi-compact as a scheme.
\end{enumerate}
We denote the category of locally finite $n$-coconnective schemes by $\Schnlft$ and its of $n$-coconnective schemes of finite type by $\Schnft$.
\end{defn}

\paragraph{Cofiltered limits of $n$-coconnective schemes of finite type}

Let $\{Z_i\}_I$ denote a cofiltered diagram of $n$-coconnective schemes of finite type, such that for any $i \ra j$ the morphism
\[
Z_i \ra Z_j
\]
is affine. Then, the limit
\[
\lim_I Z_i
\]
exists in $\Sch$ and moreover belongs to the category $\Schn$. Indeed, let $f_j:Z \ra Z_j$ denote the projection morphism from the desired object $Z$, we notice that for any affine scheme $S_j \ra Z_j$ we can compute
\[
f^{-1}(S_j) \simeq \Spec(\colim_{i \in I^{\rm op}_{\leq j}}A_i),
\]
where ${A_i}$ is the filted diagram of commutative algebras corresponding to $f^{-1}_{i,j}(S_j)$ where $f_{i,j}:S_i \ra S_j$. Since each $A_i \in {^{\geq -n, \leq 0}\CAlg(\Vect)}$ their filtered colimit is also $n$-coconnective. This follows from
\[
\H^{-k}(A_i) = \pi_{k}(\DK(A_i)),
\]
where $\DK: \Vect^{\leq 0} \ra \Spc$ is the Dold-Kan functor and the facts that $\DK$ commutes with filtered colimits and that filtered colimits of $n$-truncated objects are $n$-truncated.

Let $\Pro^{\rm aff}(\Schnft)$ denote the subcategory of Pro-object $\Pro(\Schnft)$ where the connective morphisms are affine. By the discussion above we have a realization functor
\begin{equation}
    \label{eq:realization-n-coconnective-schemes-of-ft}
    \Pro^{\rm aff}(\Schnft) \simeq \Schn.    
\end{equation}

\begin{prop}
\label{prop:n-coconnective-schemes-are-limits-of-ft-n-coconnective}
Fix $n$ and let $\Schnqc = \Schn \cap \Schqc$. The realization functor (\ref{eq:realization-n-coconnective-schemes-of-ft}) factors through $\Schnqc$ and the resulting map is an equivalence of categories
\[
\Pro^{\rm aff}(\Schnft) \simeq \Schnqc.
\]
\end{prop}

\begin{proof}
We prove the result by induction on $n$. The case $n = 0$ is \cite{TT}*{Theorem C.9} (see also \cite{Raskin-D-modules}*{Theorem 3.4.1 (1)}).

Assume that the result holds for some $n$. Consider $X \in {^{\leq (n+1)}\mbox{Sch}}$ and let $X_n = {^{\leq n}X}$. By assumption we have
\[
X_n \simeq \lim_{I}X_{n,i}
\]
for $X_{n,i} \in \Schnft$ with affine connecting morphisms. By \cite{GR-II}*{Chapter 1, Proposition 5.4.2} for $n \geq 0$ the data of $X$ with an isomorphim ${^{\leq n}X} \simeq X_n$ is equivalent to the data of an object $\sF \in \QCoh(X_{n})^{\heartsuit}[n+1]$ and a morphism
\[
T^*X_n \ra \sF.
\]

By \cite{TT}*{C.4} there exists $i \in I$ and $\sF_i \in \QCoh(X_{n,i})^{\heartsuit}[n+1]$ such that $\sF \simeq f^{*}_{n,i}(\sF_i)$, where $f_{n,i}: X_n \ra X_{n,i}$. Moreover, the composite
\[
f^{*}_{n,i}T^*X_{n,i} \ra T^*X_n \ra f^{*}_{n,i}(\sF_i)
\]
determines a map $T^*X_{n,i} \ra \sF_i$ and we let
\[
X'_{n,i} := X_{n,i}\underset{(X_{n,i})_{\sF_{i}}}{\sqcup}X_{n,i}
\]
denote the correpsonding square-zero extension (see \cite{GR-II}*{Chapter 1, \S 5.1} for the notation). Thus, we claim that
\[
X' := \lim_{I'_{\leq i}}X'_{n,i'}
\]
where $I'_{\leq i} = \{i' \in I \; | \; i' \leq i\}$ is a square-zero extension of $X_n$ corresponding to $\sF$. In other words,
\[
X' \simeq X.
\]
\end{proof}

\subsection{A subclass of Ind-schemes}
\label{subsec:Tate-schemes}

In this section we define the main mathematical objects of this article, namely Tate affine schemes. These will constitute the building blocks of the theory that we are developing. We will also consider the non-affine version, that is Tate schemes. These objects are widely considered in the literature and form a subclass of the so-called ind-schemes.

\subsubsection{Tate affine schemes}

\paragraph{Affine ind-schemes}

\begin{defn}
\label{defn:ind-affine-schemes}
Let $\Ind(\Schaff)$ denote the subcategory of Ind-objects in affine schemes which can be written as a diagram $S \simeq \colim_I S_i$, such that for every $i \ra j$ in $I$ 
\[
f_{i,j}: S_i \ra S_j
\]
is a closed embedding, i.e.\ the underlying map of classical schemes is a closed embedding. We will refer to those as \emph{affine ind-schemes}.
\end{defn}

\begin{rem}
Notice that in Definition \ref{defn:ind-affine-schemes} the use of $\Ind(\Schaff)$ does not only mean that we considered Ind-objects in the category of affine schemes, we also imposed the condition of closed embedding in the connecting maps. We decided to omit this from the notation for clarity and to agree with other usages in the literature.
\end{rem}

Recall that given $S \in \Schaff$ one defines the subcategory
\[
\Coh(S) \subset \QCoh(S)
\]
as $\sF \in \QCoh(S)$ such that 
\begin{itemize}
    \item $\H^i(\sF) = 0$ for $|i| >> 0$;
    \item for all $i$, $\left.\H^i(\sF)\right|_{\classical{X}}$ is a coherent sheaf.
\end{itemize}

\paragraph{Tate affine schemes}
\label{par:Tate-affine-schemes}

\begin{defn}
\label{defn:Tate-affine-schemes}
An Ind-object in affine schemes is a \emph{Tate affine scheme} if it can be represented as
\[
S = \lim_{I} S_i
\]
where $I$ is a filtered diagram such that
\begin{enumerate}[a)]
    \item for every $i \ra j$ in $I$ 
    \[
    f_{i,j}: S_i \ra S_j
    \]
    is a closed embedding, i.e.\ the underlying map of classical schemes $\classical{S_i} \ra \classical{S_j}$ is a closed embedding;
    \item for every $i < j$ in $I$ the fiber 
    \[
    \Fib(\sO_{S_j} \ra \sO_{S_i})
    \]
    belongs to $\Coh(S_j)$.
\end{enumerate}

We let $\SchaffTate$ denote the subcategory of $\Ind(\Schaff)$ generated by the Tate affine schemes.
\end{defn}

\subsubsection{Tate schemes}

There are two natural ways to try to define the objects one would call Tate schemes. The first is as a subcategory of the category of indschemes, the second is to use the concept of prestacks of Tate type (see Section \ref{sec:prestacks-Tate}) and impose a condition similar to Definition \ref{defn:schemes}. We take the first option as our definition since many aspects of it are already developed in \cites{DG-indschemes,GR-II}. We will introduce the later objects, which we will call Zariski Tate stacks, in section \ref{subsec:Zariski-Tate-stacks} and compare them with the former.

\paragraph{ind-Schemes}

Recall that Gaitsgory and Rozenblyum (see \cite{GR-II}*{Chapter 2, \S 1.1}) defined indschemes as the full subcategory $\indSch$ of prestacks consisting of $\sX \in \PStk$ such that
\begin{enumerate}[(a)]
    \item $\sX$ is convergent;
    \item there exists a presentation
    \[
    \sX \simeq \colim_{I}Z_i
    \]
    where each $Z_i \in \Schqc$ and for every $i \ra j$ in $I$ the corresponding map $f_{i,j}:Z_i \ra Z_j$ is a closed embedding.
\end{enumerate}

\paragraph{Tate schemes}

\begin{defn}
\label{defn:Tate-schemes}
An indscheme $\sX \in \indSch$ is said to be a \emph{Tate scheme} if there exists a presentation
\[
\sX \simeq \colim_{I}Z_i
\]
such that for every $i \ra j$ in $I$ one has
\begin{enumerate}[(a)]
    \item $f_{i,j}: Z_i \ra Z_j$ is a closed embedding;
    \item the fiber
    \[
        \Fib(\sO_{Z_j} \ra \sO_{Z_i})
    \]
    belongs to $\Coh(Z_j)$.
\end{enumerate}

We will denote the category of Tate schemes by $\SchTate$.
\end{defn}

\begin{rem}
In the case that $k$ is an algebraically closed field, and restricting to classical schemes, i.e.\ $0$-coconnective\footnote{See \ref{subsubsec:coconnective-Tate-schemes} below for the notion of $0$-coconnective Tate scheme.}, Drinfeld in \cite{Drinfeld} already considered this class of schemes under the name of reasonable indschemes.
\end{rem}

\subsubsection{Morphisms between Tate schemes}

In this section we introduce the generalizations of some of the usual properties of morphisms of schemes. Most of them generalize in a standard way, the main exception is probably the notion of a proper map.

\begin{defn}
\label{defn:morphisms-of-Tate-schemes}
For $f:X \ra Y$ a map between Tate schemes, we say that $f$ is
\begin{enumerate}[(1)]
    \item \emph{flat} (resp.\ \emph{smooth, \'etale, open embedding, Zariski, surjective}) if for every $S \in (\Schaff)_{/Y}$, $S\underset{Y}{\times}X$ is an affine scheme and the corresponding map
    \[
    S\underset{Y}{\times}X \ra S
    \]
    is flat (resp.\ smooth, \'etale, open embedding, Zariski, surjective)\footnote{This is by definition the same as that $f$ is such a map when we consider $X$ and $Y$ as prestacks.}.
    \item a \emph{closed embedding} if the corresponding map of prestacks is a closed embedding\footnote{Recall a map $f: \sX_1 \ra \sX_2$ of prestacks is a \emph{closed} \emph{embedding} if for all $S_0 \in (\clSchaff)_{-/\classical{\sX_2}}$ the map
    \[
    \classical{\sX_1}\times_{\classical{\sX_2}}S_0 \ra S_0
    \]
    is a closed (resp. open) embedding.}.;
    \item a \emph{nil-isomorphism} if the corresponding map of prestacks is a nil-isomorphism\footnote{Recall a map $f: \sX_1 \ra \sX_2$ of prestacks is a \emph{nil-isomorphism} if for all $S \in (\Schaffred)_{-/\sX_2}$ the map induced between reduced prestacks
    \[
    \red{\left(\sX_1\times_{\sX_2}S\right)} \ra S
    \]
    is an isomorphism.}.
\end{enumerate}
\end{defn}

\begin{defn}
\label{defn:ind-proper-map}
Let $f:X \ra Y$ be a map between Tate schemes we say that $f$ is \emph{ind-proper} if given presentations $X \simeq \colim_{I}X_i$ and $Y \simeq \colim_{J}Y_j$ whenever one has a diagram as follows
    \[
    \begin{tikzcd}
    X_i \ar[r,"f_{ij}"] \ar[d] & Y_j \ar[d] \\
    X \ar[r,"f"] & Y
    \end{tikzcd}
    \]
    the map $f_{i,j}$ is proper.
\end{defn}

\begin{rem}
\label{rem:ind-proper-is-proper-for-the-presentation}
Definition \ref{defn:ind-proper-map} is equivalent to requiring that for any $S \in \Sch$ and a closed embedding $g:S \ra X$ the composite $f\circ g: S \ra Y$ is a proper map\footnote{Recall this means a \emph{proper} map of prestacks, i.e.\ $f:\sX \ra \sY$ is proper if for any $T \in \Schaff_{/\sY}$ the induced map $\sX\times_{\sY}T \ra T$ is proper.}.
\end{rem}

\begin{rem}
\label{rem:proper-is-ind-proper}
The condition of Definition \ref{defn:ind-proper-map} is not equivalent to asking that $f:X \ra Y$ is proper. Indeed, if $f$ is proper then $f$ is ind-proper, however the map 
\[
\mbox{Spf}(k[[t]]) \ra \Spec(k[t])
\] GR
is an ind-proper map, actually an ind-closed embedding, but is not proper.
\end{rem}

\subsubsection{Fiber products}

The following will be needed to define the $2$-category of correspondences of Tate affine schemes.

\begin{lem}
\label{lem:Tate-affine-schemes-has-fiber-products}
The category $\SchTate$ has fiber products.
\end{lem}

\begin{proof}
Suppose that 
\begin{equation}
    \label{eq:fiber-product-diagram}
    \begin{tikzcd}
    & X \ar[d,"f"] \\
    Z \ar[r,"g"] & Y
    \end{tikzcd}
\end{equation}
is a diagram of indschemes. Firstly, we notice that the fiber product as a prestack 
\[
Z\underset{Y}{\times}X
\]
exists and is convergent, since the truncation functor commutes with limits.

Secondly, by \cite{Hennion-Tate}*{Proposition 1.2} one can pick common presentations $Z \simeq \colim_{I}Z_i$, $X \simeq \colim_I X_i$ and $Y \simeq \colim_{I}Y_i$ and maps
\[
\begin{tikzcd}
& X_i \ar[d,"f_i"] \\
Z_i \ar[r,"g_i"] & Y_i
\end{tikzcd}
\]
whose colimit recovers (\ref{eq:fiber-product-diagram}). Since fiber products commute with filtered colimits one has an equivalence
\[
\colim_{I}Z_i\underset{Y_i}{\times}X_i \overset{\simeq}{\ra} Z\underset{Y}{\times}X.
\]
To simplify notation let
\[
W_i = Z_i\underset{Y_i}{\times}X_i
\]
for $i \in I$. Given any $i \ra j$ in $I$ we need to check that the map $W_i \ra W_j$ is a closed embedding and satisfy condition (b) from Definition \ref{defn:Tate-schemes}. The map $h_{i,j}: W_i \ra W_j$ is a closed embedding if for every open affine $U \ra W_j$ the map
\[
U\underset{W_j}{\times}W_i \ra U
\]
is closed. However, we notice that $U\underset{W_j}{\times}W_i$ can be obtained as the pullback
\begin{equation}
    \label{eq:affine-pullback-for-a-map-between-pullbacks}
    \begin{tikzcd}
    U\underset{W_j}{\times}W_i \ar[r] \ar[d] & X_i\underset{X_j}{\times}U \ar[d] \\
    Z_i \underset{Z_j}{\times}U \ar[r] & Y_i\underset{Y_j}{\times}U
    \end{tikzcd}
\end{equation}
which is a closed embedding since the maps
\[
Z_i \underset{Z_j}{\times}U \ra U, \;\; X_i\underset{X_j}{\times}U \ra U, \;\; \mbox{and} \;\; Y_i\underset{Y_j}{\times}U \ra U
\]
are closed embeddings.

Finally, we claim that for every $i \ra j$ in $I$ the fiber
\begin{equation}
\label{eq:coherent-fiber-for-pullback}
    \Fib(\sO_{W_j} \ra \sO_{W_i}) \in \Coh(W_j).
\end{equation}
Indeed, it is enough to check that for any affine open $U \subset W_j$
\[
\Fib(\sO_{W_j}(U) \ra \sO_{W_i}(U))
\]
is coherent over $U$. Since the category of quasi-coherent sheaves on $U$ is stable, finite limits and colimits commute, hence by (\ref{eq:affine-pullback-for-a-map-between-pullbacks}) one has a pushout diagram
\begin{equation}
\label{eq:fiber-as-a-pushout}
    \begin{tikzcd}
    \Fib(\sO_{Y_j}(U) \ra \sO_{Y_i}(U)) \ar[r] \ar[d] & \Fib(\sO_{Z_j}(U) \ra \sO_{Z_i}(U)) \ar[d] \\
    \Fib(\sO_{X_j}(U) \ra \sO_{X_i}(U)) \ar[r] & \Fib(\sO_{W_j}(U) \ra \sO_{W_i}(U)).
    \end{tikzcd}
\end{equation}
However, by definition the terms 
\[
\Fib(\sO_{Y_j}(U) \ra \sO_{Y_i}(U)), \;\; \Fib(\sO_{X_j}(U) \ra \sO_{X_i}(U)) \;\; \mbox{and} \;\; \Fib(\sO_{Z_j}(U) \ra \sO_{Z_i}(U))
\]
on (\ref{eq:fiber-as-a-pushout}) are coherent over $U$, hence their pushout is also coherent. This finishes the proof.
\end{proof}

\subsection{Pro-objects in commutative algebras}
\label{subsec:pro-commutative-algebras}

\subsubsection{Embedding of Tate affine schemes into Pro-commutative algebras}
\label{subsubsec:embedding-Tate-affine-into-Pro-commutative-algebras}

\paragraph{}

Recall that from our conventions one has that
\begin{equation}
\label{eq:defining-equivalence-affine-schemes}
\Schaffop \simeq \CAlg(\Vect^{\leq 0}_k).    
\end{equation}

\paragraph{}

Let $\Pro(\CAlg(\Vect^{\leq 0}_k))$ denote the category of Pro-objects in the category of commutative dg-algebras concentrated in degrees $\leq 0$. The category $(\SchaffTate)^{\rm op}$ naturally includes into the category $\Pro((\Schaff)^{\rm op})$, indeed one considers
\[
\SchaffTate \hra \Ind'(\Schaff)
\]
and passes to the opposite categories. Here $\Ind'(\Schaff)$ denotes the actual category of Ind-objects in affine schemes, and not what we denote by $\Ind(\Schaff)$ above. We also recall that
\[
(\Ind'(\Schaff))^{\rm op} \simeq \Pro((\Schaff)^{\rm op}).
\]

Thus, passing to Pro-objects in the equivalence (\ref{eq:defining-equivalence-affine-schemes}), one obtains a natural inclusion functor
\begin{equation}
    \label{eq:inclusion-Tate-affine-into-Pro-CAlg}
    \imath^{\rm geom}:\SchaffTateop \ra \Pro(\CAlg(\Vect^{\leq 0}_k)).
\end{equation}

Informally the functor $\imath^{\rm geom}$ is given by the assignment
\begin{align*}
    \imath^{\rm geom}: \;\; \SchaffTateop & \ra \Pro(\CAlg(\Vect^{\leq 0}_k)) \\
 S \simeq \colim_I S_i & \mapsto \lim_{I^{\rm op}}\Gamma(S_i,\sO_{S_i}).
\end{align*}

\paragraph{The essential image of $\imath^{\rm geom}$}
\label{par:definition-of-Geom-Alg}
One easily sees that the functor $\imath^{\rm geom}$ is fully faithful, and its essential image consists of Pro-object $\{A_{j}\}_{J} \in \Pro(\CAlg(\Vect^{\leq 0}_k))$ such that: 
\begin{enumerate}[a)]
    \item for each $i \ra j$ in $I$ the map $A_i \ra A_j$ induces a surjection 
    \[
    \H^0(A_i) \ra \H^0(A_j);
    \]
    \item and for each $i \ra j$ 
    \[
    \Fib(A_i \ra A_j) \in \Mod^{\rm fp}(A_i),
    \]
    where $\Mod^{\rm fp}(A_i)$ is the subcategory of finitely presented $A_i$-modules \footnote{As defined in \cite{HA}*{Definition 7.2.4.26}.}.
\end{enumerate}

We will denote by $\GeomAlg$ the full subcategory of $\Pro(\CAlg(\Vect^{\leq 0}_k))$ spanned by the objects in the essential image of $\imath^{\rm geom}$.

\begin{rem}
We refer the reader to the appendix \ref{sec:recollections-on-Pro-and-Tate} for a discussion of different categories of commutative algebras that one could have considered as the opposite of the category of Tate affine schemes and the relation between them.
\end{rem}

\subsubsection{Coconnective pro-commutative algebras}

We recall from \cite{GR-I}*{Chapter 2, \S 1.2.2} that one has fully faithful functor
\[
\CAlg(\Vect^{\geq -n,\leq 0}) \ra \CAlg(\Vect^{\leq 0})
\]
that admits a left adjoint $\tau^{\geq -n}$, such that the following diagram commutes
\[
\begin{tikzcd}
\CAlg(\Vect^{\leq 0}) \ar[d,"\obliv_{\CAlg}"] \ar[r,"\tau^{\geq -n}"] & \CAlg(\Vect^{\geq -n,\leq 0}) \ar[d,"\obliv_{\CAlg}"] \\
\Vect^{\leq 0} \ar[r,"\tau^{\geq -n}"] & \Vect^{\geq -n,\leq 0}
\end{tikzcd}
\]

By general properties\footnote{Indeed, any object in the essential image of 
\[
\CAlg(\Vect^{\geq -n,\leq 0}) \ra \Pro(\CAlg(\Vect^{\leq 0}))
\]
is cocompact, so by \cite{HTT}*{Proposition 5.3.5.11} (\ref{eq:inclusion-Pro-CAlg-truncated-into-all-Pro-CAlg}) is fully faithful. By \cite{HTT}*{Proposition 5.3.5.15} (\ref{eq:inclusion-Pro-CAlg-truncated-into-all-Pro-CAlg}) has a left adjoint if and only if the inclusion
\[
\CAlg(\Vect^{\geq -n,\leq 0}) \ra \CAlg(\Vect^{\leq 0})
\]
preserves finite limits, which follows from the fact that $\Vect^{\geq -n,\leq 0} \hra \Vect^{\leq 0}$ preserves finite limits.} of the Pro construction one obtains a functor
\begin{equation}
\label{eq:inclusion-Pro-CAlg-truncated-into-all-Pro-CAlg}
\Pro(\CAlg(\Vect^{\geq -n,\leq 0})) \ra \Pro(\CAlg(\Vect^{\leq 0})),
\end{equation}
which is fully faithful and has a left adjoint
\begin{equation}
\label{eq:truncation-Pro-CAlg}
\tau^{\geq -n}: \Pro(\CAlg(\Vect^{\leq 0})) \ra \Pro(\CAlg(\Vect^{\geq -n,\leq 0})).
\end{equation}

Moreover, the natural diagram
\begin{equation}
\label{eq:n-truncation-compatible-with-pro-calg-objects}
\begin{tikzcd}
\CAlg(\Vect^{\leq 0}) \ar[r,"\tau^{\geq -n}"] \ar[d] & \CAlg(\Vect^{\geq -n,\leq 0}) \ar[r] \ar[d] & \CAlg(\Vect^{\leq 0}) \ar[d] \\
\Pro(\CAlg(\Vect^{\leq 0})) \ar[r,"\tau^{\geq -n}"] & \Pro(\CAlg(\Vect^{\geq -n,\leq 0})) \ar[r] & \Pro(\CAlg(\Vect^{\leq 0}))
\end{tikzcd}    
\end{equation}
where the vertical arrows are the canonical inclusions commutes.

\subsubsection{Convergence of Tate affine schemes}
\label{subsubsec:convergence-Tate-affine-schemes}

\begin{defn}
\label{defn:Tate-affine'-scheme}
Following Definition \ref{defn:Tate-schemes} one could define a \emph{Tate affine' scheme} as a prestack $\sX$ which is convergent and can be represented as
\[
\sX \simeq \colim_I S_i
\]
where $S_i \in \Schaff$ and satisfy the conditions (a) and (b) above.
\end{defn}

\begin{lem}
\label{lem:Tate-affine-schemes-are-convergent}
Let $S$ be a Tate affine scheme, then considered as a prestack $S$ is convergent.
\end{lem}

\begin{proof}
Let $\imath^{\rm geom}(S) = \{A_i\}_{I^{\rm op}} \in \Pro(\CAlg(\Vect^{\leq 0}))$ be the pro-object in commutative algebras corresponding to $S$. We notice that by (\ref{eq:n-truncation-compatible-with-pro-calg-objects}) one has the equivalences
\begin{align*}
    \imath^{\rm geom}(S) & \simeq \lim_{I^{\rm op}}A_i \\
     & \simeq \lim_{I^{\rm op}}\lim_{n \geq 0} \tau^{\geq -n}(A_i) \\
     & \simeq \lim_{n \geq 0} \lim_{I^{\rm op}} \tau^{\geq -n}(A_i) \\
     & \simeq \lim_{n \geq 0}\tau^{\geq n}(\imath^{\rm geom}(S)),
\end{align*}
where the isomorphism from the first to the second line is the statement that any affine scheme is convergent (\cite{GR-I}*{Chapter 2, \S 1.4.3}).
\end{proof}

\begin{cor}
\label{cor:Tate-affine-and-Tate-affine'-agree}
For $\sX$ a prestack the following are equivalent:
\begin{enumerate}[(i)]
    \item there exists $S \in \SchaffTate$ such that
    \[
    \sX \simeq \Hom_{\Schaff}(-,S);
    \]
    \item $\sX$ satisfies Definition \ref{defn:Tate-affine'-scheme}.
\end{enumerate}
\end{cor}

Henceforth we will use the term Tate affine scheme to refer to any of the two equivalent conditions of Corollary \ref{cor:Tate-affine-and-Tate-affine'-agree}. The following is a consequence of the equivalence of both notions.

\begin{cor}
The natural map $\Schaff \hra \Sch$ gives rise to a canonical inclusion
\begin{equation}
    \label{eq:inclusion-Tate-affine-schemes-into-Tate-schemes}
    \SchaffTate \hra \SchTate.
\end{equation}
\end{cor}

\subsection{Properties of Tate schemes}
\label{subsec:properties-Tate-schemes}

In this section we develop the analogue of certain properties of derived affine schemes and derived schemes to the cases of Tate affine schemes and Tate schemes. A reference for a treatment of these properties for affine schemes and derived is \cite{GR-I}*{Chapter 2}. Most of these conditions are straightforward special cases of the corresponding condition for indschemes (see \cite{GR-II}*{Chapter 2, \S 1}). We record them here for our convenience.

\subsubsection{Coconnectivity conditions}
\label{subsubsec:coconnective-Tate-schemes}

We define the subcategory of \emph{$n$-coconnective Tate affine schemes} $\SchaffTaten$ as the subcategory of $\Ind(\Schaff)$ whose objects $T$ can be represented as
\[
T \simeq \colim_{I}T_i
\]
where each $T_i \in {^{\leq n}\Schaff}$.

\paragraph{$n$-coconnective Pro-commutative algebras}

Here is another definition one could have considered. We say that $S \in \SchaffTate$ is an \emph{$n$-coconnective Pro-commutative algebra} if $\imath^{\rm geom}(S)$ belongs to the essential image of (\ref{eq:inclusion-Pro-CAlg-truncated-into-all-Pro-CAlg}).

The following is proved in the same way as Lemma \ref{lem:Tate-affine-schemes-are-convergent}.

\begin{lem}
Suppose that $A \in \Pro(\CAlg(\Vect^{\leq 0}))$ belongs to the essential image of $\imath^{\rm geom}$, then $\tau^{\geq -n}(A)$ belongs to the essential image of $\left.\imath^{\rm geom}\right|_{\SchaffTaten}$.

In particular, a Tate affine scheme $S$ is $n$-coconnective if and only if $\imath^{\rm geom}(S)$ is an $n$-coconnective Pro-commutative algebra.
\end{lem}

\paragraph{Eventually coconnective Tate affine schemes}

We say that $T \in \SchaffTate$ is \emph{eventually coconnective} if $T \in \SchaffTaten$ for some $n$. We will denote by $\SchaffTateconv$ the category of eventually coconnective Tate affine schemes.

\paragraph{}

We recall that $n$-coconnective indscheme is an $n$-coconnective prestack
\[
Z \in \PStkn := \Fun((\Schaffn)^{\rm op},\Spc)
\]
that can be represented as
\[
Z \simeq \colim_{I}Z_i
\]
where each $Z_i \in \Schnqc$ and for every $i \ra j$ in $I$ the map $Z_i \ra Z_j$ is a closed embedding.

We denote the category of $n$-coconnective indschemes by $\indSchn$.

\paragraph{$n$-coconnective Tate schemes}

The category of $n$-coconnective Tate schemes $\SchTaten$ is defined as the full subcategory of $\indSchn$ of indschemes admitting a presentation that satisfy condition (b) from Definition \ref{defn:Tate-schemes}.

\begin{lem}
\label{lem:n-coconnective-Tate-affine-schemes-has-fiber-products}
The category $\SchTaten$ has fiber products.
\end{lem}

\begin{proof}
We notice that in the proof of Lemma \ref{lem:Tate-affine-schemes-has-fiber-products} we can consider levelwise fiber products of $n$-coconnective affine schemes, those have fiber products given by the truncation of the pullback in the category of affine schemes. Thus, considering the filtered colimit of those we obtain the fiber product in the category of $n$-coconnective Tate affine schemes.
\end{proof}

We will say that a Tate scheme $Z$ is \emph{eventually coconnective} if $Z \in \SchTaten$ for some $n$, and we will denote by $\SchTateconv$ the category of such Tate schemes.

\subsubsection{Finiteness conditions}
\label{subsubsec:finiteness-conditions}

\paragraph{Tate affine schemes of finite type}

\begin{defn}
\label{defn:n-coconnective-Tate-affine-schemes-finite-type}
For $S \in \SchaffTaten$ for some $n$, i.e.\ $S$ is an eventually coconnective Tate affine scheme. We say that $S$ is of \emph{finite type}, if there exists a presentation
\[
S \simeq \colim_I S_i
\]
where each $S_i \in \Schaffnft$, i.e.\ each $S_i$ is an $n$-coconnective affine scheme of finite type.
\end{defn}

We will denote by $\SchaffTatenft$ the category of $n$-coconnective Tate affine schemes of finite type.

\paragraph{Prestacks locally almost of finite type}

Following \cite{GR-I}*{Chapter 2, \S 1.7} we recall the definition.
\begin{defn}
\label{defn:prestacks-laft}
A prestack $\sX$ is said to be \emph{locally almost of finite type} if
\begin{enumerate}[(i)]
    \item $\sX$ is convergent;
    \item for every $n \geq 0$ one has
    \[
    \LKE_{\Schaffnftop \hra \Schaffnop}(\left.{^{\leq n}\sX}\right|_{\Schaffnftop}) \overset{\simeq}{\ra} {^{\leq n}\sX},
    \]
    i.e.\ ${^{\leq n}\sX}$ is a \emph{prestack locally of finite type}.
\end{enumerate}
\end{defn}

Let $\PStklaft$ denote the category of prestacks locally almost of finite type. For indschemes, we will define
\[
\indSchlaft := \indSch \cap \PStklaft
\]
the subcategory of indschemes locally almost of finite type.

\paragraph{Tate affine scheme almost of finite type}

When restricted to Tate affine schemes the conditions Proposition \ref{prop:prestack-laft-is-colimit-of-aft-schemes} are compatible with those of Definition \ref{defn:n-coconnective-Tate-affine-schemes-finite-type}.

\begin{lem}
\label{lem:defn-Tate-affine-scheme-aft}
Given an object $S \in \SchaffTate$ the following are equivalent:
\begin{enumerate}[(i)]
    \item $S$ is locally almost of finite type as a prestack;
    \item for each $n \geq 0$, ${^{\leq n}S}$ is of finite type;
    \item there exists a presentation
    \[
    S \simeq \colim_I S_i
    \]
    where each $S_i \in \Schaffaft$.
\end{enumerate}
\end{lem}

\begin{proof}
This is an immediate consequence of the fact that left Kan extensions and colimits commute and the definition of $\Schaffaft$.
\end{proof}

 A Tate affine scheme $S$ satisfying the conditions of Lemma \ref{lem:defn-Tate-affine-scheme-aft} will be called \emph{almost of finite type}. We will denote by $\SchaffTateaft$ the category of Tate affine schemes almost of finite type.

\paragraph{Tate schemes locally almost of finite type}

A Tate scheme $Z$ is said to be \emph{locally almost of finite type} if it is locally almost of finite type as a prestack. By \cite{GR-II}*{Chapter 2, \S 1.7} one has the following

\begin{prop}
\label{prop:prestack-laft-is-colimit-of-aft-schemes}
For $Z$ a Tate scheme the following are equivalent
\begin{enumerate}[(i)]
    \item $Z$ is a prestack locally almost of finite type;
    \item there exists a presentation
    \[
    Z \simeq \colim_I Z_i
    \]
    where each $Z_i \in \Schaft$.
\end{enumerate}
\end{prop}

We will denote by $\SchTatelaft$ the category of Tate schemes locally almost of finite type.

\paragraph{Stability of finiteness condition}

\begin{lem}
\label{lem:n-coconnective-Tate-affine-schemes-of-ft-has-fiber-products}
The category $\SchaffTatenft$ has fiber products.
\end{lem}

\begin{proof}
We argue in a similar way as in the proof of Lemma \ref{lem:n-coconnective-Tate-affine-schemes-has-fiber-products} and use the fact that the subcategory $\Schaffnft \subset \Schaffn$ is stable under fiber products. 
\end{proof}

\begin{rem}
\label{rem:Tate-schemes-aft-stable-under-fiber-products}
The subcategory $\SchTatelaft \subset \SchTate$ is stable under fiber products. Indeed, it is enough to check that the subcategory of prestacks locally almost of finite type is stable under fiber products. That follows from the fact that conditions (i) and (ii) in Definition \ref{defn:prestacks-laft} are stable under fiber products. Similarly the subcategory $\SchaffTateaft \subset \SchaffTate$ is stable under fiber products.
\end{rem}

\subsection{Relation to other definitions}
\label{subsec:other-definitions}

In this section using the notion of a pseudo-tensor category we formulate the notion of commutative algebra objects in the category of Tate objects. We also compare the notion of Tate affine schemes to reasonable ind-schemes (as defined by \cite{Drinfeld}) and locally compact indschemes (as defined by \cite{Kapranov-Vasserot}).

\subsubsection{Digression: commutative algebra objects in pseudo-tensor categories}
\label{subsubsec:CAlg-objects-in-pseudo-tensor-categories}

We let $\Vect_k$ denote the stable $\infty$-category of complex of $k$-vector spaces. By definition this category is equivalent to $\Ind(\Vect^{\rm fd}_k)$ and it is endowed with a symmetric monoidal structure which commutes with colimit on each variable.

Consider $\ProVect_k$ the category of Pro-objects in $\Vect_k$, this is endowed with a tensor product
\[
\otimes: \ProVect \times \ProVect \ra \ProVect
\]
that extends the tensor product of $\Vect$ and is characterized by the fact that $(-)\otimes(-)$ commutes with limits on both variables.

The following is an $\infty$-categorical generalization of the notion of a pseudo-tensor structure as considered in \cite{BD}*{\S 1.1}.

\begin{defn}
For $\sC$ an $\infty$-category a \emph{pseudo-tensor structure} is the data of symmetric monoidal strucutre $p_{\sD}:\sD^{\otimes} \ra \Finp$, a co-Cartesian fibration $p_{\sC}: \sC^{\otimes} \ra \Finp$ and a map $f: \sC^{\otimes} \ra \sD^{\otimes}$ such that
\begin{enumerate}[(i)]
    \item 
    \[
    \sC^{\ast} \simeq \sD^{\ast} (\simeq \ast)
    \]
    \item for all $n \geq 1$ the functor
    \[
    f^{<n>}: \sC^{<n>} \ra \sD^{<n>}
    \]
    is fully faithful.
\end{enumerate}
\end{defn}

\begin{rem}
Notice this differs from the definition of a symmetric monoidal category by dropping the requirement that for any $<n>$ one has an equivalence
\[
\sC^{<n>} \overset{\simeq}{\ra} \prod_{i \in <n>^{\circ}}\sC^{\{\ast,i\}}.
\]
\end{rem}

Given $\sC$ a pseudo-tensor category a \emph{commutative algebra object in $\sC$} is the data of a map $a: \Finp \ra \sC^{\otimes}$ of co-Cartesian fibrations over $\Finp$ which takes inert morphisms in $\Finp$ to inert morphisms in $\sC^{\otimes}$\footnote{Those are morphisms $f:X \ra X'$ in $\sC^{\otimes}$ such that (i) $f$ is $p_{\sC}$-co-Cartesian and $p_{\sC}(f)$ is an inert morphism of $\Finp$.}. We denote the category of such by $\CAlg(\sC)$ and one clearly has a map
\[
\CAlg(\sC) \ra \CAlg(\sD)
\]
which is fully faithful.

\begin{rem}
Differently than the situation for symmetric monoidal categories, it is not the case that the data of a commutative algebra object in $\sC$ is equivalent to that of a commutative monoid ($\Finp$-monoid) in $\sC$, as considered in \cite{HA}*{Definition 2.4.2.1}. Indeed, the notion of a commutative monoid a priori doesn't even make sense, since we don't have the identification between $\sC^{<n>}$ and $\prod^n_{i=1}\sC^{<1>}$.
\end{rem}

\paragraph{Construction of pseudo-tensor structure}

\begin{lem}
\label{lem:construction-of-pseudo-tensor-structure}
Let $p_{\sD}: \sD^{\otimes} \ra \Finp$ be a symmetric monoidal structure on $\sD$ consider $\sC \subset \sD$ a full subcategory, then there exists a co-Cartesian fibration and a $\sC^{\otimes} \ra \Finp$ and a map $f: \sC^{\otimes} \ra \sD^{\otimes}$ realizing $\sC$ as a pseudo-tensor category.
\end{lem}

\begin{proof}
This is motivated by the construction in the beginning of Section 2.2.1 from \cite{HA}. Let $\sC^{\otimes}$ denote the full subcategory of $\sC^{\otimes}$ spanned by those objects $X \in \sD^{\otimes}$ having the form $Y_1\otimes\cdots \otimes Y_n$ for $\{Y_i \in \sC\}_{i\in <n>}$.

We claim that $\sC^{\otimes} \ra \Finp$ is a co-Cartesian fibration. This is evident since a full subcategory $\sC^{\otimes} \ra \sD^{\otimes}$ is a trivial Kan fibration, thus also a co-Cartesian fibration and co-Cartesian fibrations are stable under composition.
\end{proof}

\subsubsection{Tate affine schemes as commutative algebras in $\ProVect^{\leq 0}$}

In this section we embed the category of Tate affine schemes into the category of commutative algebra objects in $\Pro(\Vect^{\leq 0})$.

By definition the canonical inclusion
\[
\Vect^{\leq 0} \ra \Pro(\Vect^{\leq 0})
\]
is a symmetric monoidal functor. Thus, since a commutative algebra object in $\Vect^{\leq 0}$ is a map of $\infty$-operads $\Finp \ra (\Vect^{\leq 0})^{\otimes}$ over $\Finp$, composition with $\imath^{\otimes}$ gives a map of $\infty$-operads over $\Finp$. So we obtain the map
\[
\CAlg(\Vect^{\leq 0}) \ra \CAlg(\Pro(\Vect^{\leq 0})).
\]
Since the category $\CAlg(\Pro(\Vect^{\leq 0}))$ admits cofitered limits\footnote{Recall those are computed in the underlying category $\Pro(\Vect^{\leq 0})$.} one obtains the functor

\begin{equation}
    \label{eq:pro-calg-to-calg-pro}
    \varphi: \Pro(\CAlg(\Vect^{\leq 0})) \ra \CAlg(\Pro(\Vect^{\leq 0}))
\end{equation}

\begin{lem}
\label{lem:Pro-CAlg-into-CAlg-Pro-is-fully-faithful}
The functor (\ref{eq:pro-calg-to-calg-pro}) is fully faithful.
\end{lem}

\begin{proof}
According to \cite{HTT}*{Proposition 5.3.5.11} $\varphi$ is fully faithful if and only if its restriction
\[
\left.\varphi\right|_{\CAlg\left(\Vect^{\leq 0}\right)} : \CAlg(\sC) \ra \CAlg(\Pro(\Vect^{\leq 0}))
\]
is fully faithful and its essential image consists of cocompact objects.

To check that $\left.\varphi\right|_{\CAlg(\sC)}$ is fully faithful we notice that this map is encoded by composition with $\imath^{\otimes}_{\Vect^{\leq 0}}$. By a similar reasoning as \cite{HA}*{Remark 2.1.3.8} (see Proposition 2.4.4.2 and Corollary 2.4.4.4 of \cite{HTT} for details.) one proves that $\imath^{\otimes}_{\Vect^{\leq 0}}$ is fully faithful if and only if the functor $\imath_{\Vect^{\leq 0}}$ is fully faithful and this last statement is evident.

A similar argument works to check that $\left.\varphi\right|_{\CAlg(\Vect^{\leq 0})}$ sends objects of $\CAlg(\sC)$ to cocompact objects of $\CAlg(\Pro(\Vect^{\leq 0}))$.
\end{proof}

Notice that we obtain a functor
\[
\imath^{\rm geom'}:\SchaffTateop \ra \CAlg(\Pro(\Vect^{\leq 0}))
\]
given as the composite
\[
\SchaffTateop \overset{\imath^{\rm geom}}{\ra} \Pro(\CAlg(\Vect^{\leq 0})) \ra \CAlg(\Pro(\Vect^{\leq 0}))
\]
where the functor $\imath^{\rm geom}$ is defined in \S \ref{subsubsec:embedding-Tate-affine-into-Pro-commutative-algebras}. The following is a consequence of Lemma \ref{lem:Pro-CAlg-into-CAlg-Pro-is-fully-faithful} and the fact that $\imath^{\rm geom}$ is fully faithful.

\begin{cor}
There is a fully faithful functor
\begin{equation}
    \label{eq:embedding-Tate-affine-schemes-into-CAlg-Pro}
    \imath^{\rm geom'}: \SchaffTateop \ra \CAlg(\Pro(\Vect^{\leq 0})).
\end{equation}

We denote by $\GeomAlg'$ the subcategory of $\CAlg(\Pro(\Vect^{\leq 0}))$ obtained as the essential image of $\imath^{\rm geom'}$.
\end{cor}

\subsubsection{Commutative algebra objects in $\Tate(k)^{\leq 0}$}
\label{subsubsec:commutative-algebra-in-Tate-objects}

We recall that the subcategory $\Tate_k$ of $ \Pro(\Vect)$ inherits a t-structure (see Corollary \ref{cor:t-structure-on-Tate-objects} for details).

Let $\Pro(\Vect^{\leq 0})^{\otimes} \ra \Finp$ denote the co-Cartesian fibration encoding the symmetric monoidal structure of $\Pro(\Vect^{\leq 0})^{\otimes}$. By Lemma \ref{lem:construction-of-pseudo-tensor-structure} the subcategory $\Tate^{\leq 0}_k \subseteq \Pro(\Vect^{\leq 0})$ defines a co-Cartesian fibration $(\Tate^{\leq 0}_k)^{\otimes} \ra \Finp$ and a map
\[
(\Tate^{\leq 0}_k)^{\otimes} \ra \Pro(\Vect^{\leq 0})^{\otimes}
\]
realizing a pseudo-tensor structure on $\Tate^{\leq 0}_k$.

\begin{cor}
\label{cor:CAlg-Tate-embeds-into-CAlg-Pro}
The above discussion gives an embedding of categories
\begin{equation}
    \label{eq:CAlg-Tate-embeds-into-CAlg-Pro}
    \CAlg(\Tate^{\leq 0}_k) \ra \CAlg(\Pro(\Vect^{\leq 0})).
\end{equation}
We will denote by $\CAlg^!$ the essential image of (\ref{eq:CAlg-Tate-embeds-into-CAlg-Pro}).
\end{cor}

\begin{rem}
At the moment, we don't have a comparison between the subcategories
\[
\begin{tikzcd}
\CAlg^! \ar[rd,hook] & \\
& \CAlg(\Pro(\Vect^{\leq 0})) \\
\GeomAlg^! \ar[ru,hook] &
\end{tikzcd}
\]
where $\CAlg^!$ is the essential image of commutative algebra objects in $\Tate^{\leq 0}_k$ and $\GeomAlg^!$ is the essential image of Tate affine schemes.
\end{rem}

\subsubsection{Reasonable indschemes}

In \cite{Raskin-homological}*{\S 6.8} (cf.\ \cite{Drinfeld}*{\S 6.3.3} for the same notion for classical ind-schemes) the notion of reasonable indschemes is introduced. A prestack $\sX$ is said to be a \emph{reasonable indscheme} if $\sX$ is convergent and it admits a presentation
\[
\sX \simeq \colim_I S_i
\]
where $I$ is a cofiltered diagram and
\begin{enumerate}[(a)]
    \item for every $i\in I$ the $S_i$ is an eventually coconnective quasi-compact scheme;
    \item for every $i \ra j$ the morphism
    \[
    f_{i,j}: S_i \hra S_j
    \]
    is almost finitely presented, which in the case that $S_i$ is eventually coconnective is equivalent to $(f_{i,j})_*(\sO_{S_i}) \in \Coh(S_j)$.
\end{enumerate}

Let's denote by $\indSch_{\rm reas}$ the subcategory of reasonable ind-schemes. We have the following

\begin{prop}
There is an embedding
\[
\indSch_{\rm reas} \hra \SchTateconv
\]
of the category of reasonable ind-schemes into the category of eventually coconnective Tate schemes.
\end{prop}

\begin{proof}
Notice that for $Z$ an eventually coconnective Tate scheme we have a presentation
\[
Z \simeq \colim_I Z_i,
\]
where each $Z_i \in \Schconv$. Let $f_{i,j}:Z_i \hra Z_j$ we have a fiber sequence
\[
\sI_{i,j} \ra \sO_{Z_{j}} \ra (f_{i,j})_*\sO_{Z_i}.
\]
Since $Z_i$ is eventually coconnective we have $\sO_{Z_j} \in \Coh(Z_j)$ so one has
\[
\sI_{i,j} \in \Coh(Z_j) \;\; \Leftrightarrow \;\; (f_{i,j})_*\sO_{Z_i} \Coh(Z_j).
\]
\end{proof}

\subsubsection{Digression: Pro-Ind presentation}
\label{subsubsec:digression-pro-ind-presentation}

This subsection is inspired by the discussion in \cite{Kapranov-Vasserot}*{\S 4.4}. In the theory of Tate objects in stable $\infty$-categories one has a flexibility in writing a Tate object, either as
\begin{itemize}
    \item an Ind-Pro-object, or
    \item a Pro-Ind object.
\end{itemize}

From our initial choice of considering Ind-Pro-objects on schemes of finite type, we seem to have broken this symmetry. The following statement tries to amend that.

Let $Z$ be an $n$-coconnective Tate scheme, by Proposition \ref{prop:n-coconnective-schemes-are-limits-of-ft-n-coconnective} for each $i$ we can write
\[
Z_i \simeq \lim_{J^{\rm op}_i}Z^j_{i}
\]
where $J_i$ is a filtered diagram, each $Z^j_i$ is an $n$-coconnective scheme of finite and the connecting morphisms are affine morphisms.

Let's suppose that for our Tate scheme we can find a diagram
\[
\widetilde{Z}: I \times J^{\rm op} \ra \Schnft
\]
where $I$ and $J$ are filtered that satisfies:
\begin{enumerate}[(a)]
    \item for every $i \ra i'$ the morphism
    \[
    \widetilde{Z}_{i,j} \hra \widetilde{Z}_{i',j}
    \]
    is a closed embedding;
    \item for every $j' \ra j$ the morphism
    \[
    \widetilde{Z}_{i,j'} \ra \widetilde{Z}_{i,j}
    \]
    is affine;
    \item for each $i\in I$ we have an isomorphism
    \[
    Z_i \simeq \lim_{J^{\rm op}}\widetilde{Z}_{i,j};
    \]
    \item for each $i \ra i'$ and $j' \ra j$ the diagram
    \[
    \begin{tikzcd}
    \widetilde{Z}_{i,j'} \ar[r] \ar[d] & \widetilde{Z}_{i',j'}\ar[d] \\
    \widetilde{Z}_{i,j} \ar[r] & \widetilde{Z}_{i',j}
    \end{tikzcd}
    \]
    is Cartesian.
\end{enumerate}
 
Then we have
\begin{prop}
\label{prop:Pro-Ind-presentation-of-Tate-scheme}
For $Z$ an $n$-coconnective Tate scheme satisfying the condition (a-d) above one has an isomorphism
\[
Z \simeq \lim_{J^{\rm op}} \colim_I\widetilde{Z}_{i,j}.
\]
\end{prop}

\begin{proof}
Since both the category $\Ind(\Pro(\Schnft))$ and $\Pro(\Ind(\Schnft))$ embed fully faithfully into the category of prestacks it is enouch to check that for any $S$ an affine scheme we have a morphism
\[
\colim_{I} \lim_{J^{\rm op}}\Maps(S,\widetilde{Z}_{i,j}) \ra \lim_{J^{\rm op}} \colim_I\Maps(S,\widetilde{Z}_{i,j}).
\]

This is a standard argument on the commutation of limits and colimits.

\end{proof}

In particular, if we denote the subcategory of $n$-coconnective Tate schemes that satisfy the condition (a-d) above by $\SchTateProInd$, then we have a fully faithful embedding
\[
\SchTateProInd \hra \Pro(\Ind(\Schnft)).
\]

\begin{rem}
Unfortunately, at the time we are not able to determine if the conditions (a-d) follow automatically from one of the previous properties of Tate schemes that we already introduced.
\end{rem}

\begin{rem}
In \cite{Kapranov-Vasserot}*{\S 4.4} the authors refer to an indscheme satisfying conditions (a-d) above as a \emph{locally compact indscheme}. We also didn't investigate if such indschemes are automatically Tate schemes.
\end{rem}

\subsection{Examples}
\label{subsec:examples-Tate-affine}

In this section we produce examples of Tate schemes and Tate affine schemes. The first subsection treats the loop functors, which produce a Tate affine scheme from an affine scheme of finite type. This example is extremely important in the study of geometric representation theory. In \S \ref{subsubsec:affine-grassmannian} we treat the example of the affine Grassmannian, which is another object of great interest. Finally in \S \ref{subsubsec:affine-formal-schemes} we have a quick discussion about how to get examples of Tate affine schemes from affine formal schemes, i.e.\ topological rings with an adic topology.

\subsubsection{Loop functors}

In this section we follow \cite{DG-indschemes}*{\S 9.2} to define the formal loops functor on prestacks, which when restricted to affine schemes of finite type produces examples of Tate affine schemes.

\paragraph{Definition of loop functor}

We follow \cite{DG-indschemes}*{\S 9.2}, 
\begin{defn}
\label{defn:loop-functors-on-prestacks}
Let $Z \in \PStk$ we consider the following prestacks
\begin{enumerate}[1.]
    \item \emph{finite loops on $Z$}: $Z[t]/t^k$, whose $S = \Spec(A)$-points are
    \[
    (Z[t]/t^k)(S) := \Maps(\Spec(A[t]/t^k),Z);
    \]
    \item \emph{positive loops on $Z$}: $Z[[t]]$, whose $S = \Spec(A)$-points are
    \[
    Z[[t]](S) := \Maps(\Spec(A[[t]]),Z);
    \]
    \item \emph{loops on $Z$}: $Z((t))$, whose $S = \Spec(A)$-points are
    \[
    Z((t))(S) := \Maps(\Spec(A((t))),Z).
    \]
\end{enumerate}
\end{defn}

Notice that by definition one has
\[
Z[[t]] \simeq \lim_{k} (Z[t]/t^k).
\]

\paragraph{Loops on affine schemes of finite type are Tate affine schemes}

\begin{lem}
\label{lem:formal-loops-of-affine-of-finite-type-are-Tate-affine}
Let $Z \in \Schaffnft$, then $Z((t)) \in \SchaffTaten$. 
\end{lem}

\begin{proof}
Firstly, we notice that any $Z \in \Schaffnft$ can be written as 
\[
Z \simeq Z_1 \underset{\bA^n}{\times} Z_2 \underset{\bA^n}{\times} \cdots \underset{\bA^n}{\times}Z_r
\]
where each $Z_i$ is given by a pullback
\[
\begin{tikzcd}
Z_i \ar[r] \ar[d] & \bA^n \ar[d,"f_i"] \\
0 \ar[r] & \bA^n.
\end{tikzcd}
\]
Secondly, remark that for any $\alpha:I \ra \Schaff$ a finite diagram one has an isomorphism
\[
(\lim_I \alpha_i)((t)) \simeq \lim_I \alpha_i((t)).
\]
Since the category of Tate affine schemes has fiber products, it is enough to check the claim for $Z = \bA^1$, since $\bA^n \simeq (\bA^1)^{\times n}$.

Let $S = \Spec(R)$ by definition one has
\begin{align*}
\bA^1((t))(S) & = \Maps_{\Schaff}(\Spec(R((t)),\Spec(k[x]))   \\
& \simeq \Hom(k[x],R((t))) \\
& \simeq R((t)) \\
& \simeq \colim_{i \geq 0} \lim_{j \geq 1} R\left<t^{-i},\ldots,t^{j-1}\right> \\
& \simeq \colim_{i \geq 0} \lim_{j \geq 1} \Hom(k[t],R^{\times (i+j)}) \\
& \simeq \colim_{i \geq 0} \lim_{j \geq 1} \Hom(k[t_{-i},\ldots,t_{j-1}],R).
\end{align*}

We now introduce some notation to express the answer. Let
\[
\bA^{\infty} := \prod_{n \geq 0}\bA^1 \simeq \Spec(k[t_0,t_1,\ldots])
\]
denote the infinite-dimensional affine space over $k$. Notice that the connecting maps in the presentation above are given by
\begin{align*}
    \imath: \bA^{\infty} & \ra \bA^{\infty} \\
    (x_0,x_1,\ldots) & \mapsto (0,x_0,x_2,\ldots),
\end{align*}
which are given by finitely presented ideals of defintion, namely $(t_0)$ inside $k[t_0,t_1,\ldots]$. Thus, one has
\[
\bA^{1}((t)) \simeq \colim_{i \geq 0}\left(\bA^{\infty} \overset{\imath}{\hra}\bA^{\infty} \hra \cdots\right)
\]
which proves that $\bA^{1}((t))$ is a Tate affine scheme.
\end{proof}

\paragraph{Generalized loop functors}

We introduce some notation, let $\bD = \Spec(k[[t]])$ and $\bD^{\times} = \Spec(k((t)))$, consider the categories
\[
\PStk_{/\bD} := \Fun\left((\Schaff_{/\bD})^{\rm op},\Spc\right) \;\;\; \mbox{and} \;\;\; \PStk_{/\bD^{\times}} := \Fun\left((\Schaff_{/\bD^{\times}})^{\rm op},\Spc\right)
\]
of prestacks over $\bD$ and over $\bD^{\times}$. One has the following generalization of Definition \ref{defn:loop-functors-on-prestacks}.

\begin{defn}
\label{defn:loop-functors-on-prestacks-over-bD}
Let $\sY_0 \in \PStk_{/\bD}$ one defines
\begin{enumerate}[1)]
    \item for some $n \geq 0$ one defines the \emph{$n$th loop space of $\sY_0$} to be the prestack $\L^n\sY_0$ whose $S = \Spec(A)$-points are
    \[
    \L^n\sY_0(S) := \sY_0(\Spec(A[t]/(t^n)));
    \]
    \item the \emph{positive loop space of $\sY_0$} to be the prestack $\L^+\sY_0$ whose $S = \Spec(A)$-points are
    \[
    \L^+\sY_0(S) := \sY_0(\Spec(A[[t]])).
    \]
\end{enumerate}
\end{defn}

\begin{defn}
\label{defn:loop-functor-on-prestacks-over-bD-times}
Let $\sX_0 \in \PStk_{/\bD^{\times}}$ one defines the \emph{loop space on $\sX_0$} to be the prestack $\L\sX_0$ whose $S = \Spec(A)$-points are
    \[
    \L\sY_0(S) := \sY_0(\Spec(A((t)))).
    \]
\end{defn}

\begin{lem}
\label{lem:generalized-loops-on-affine-ft-are-Tate-affine}
Suppose that $Y_0 \in \Schaff_{/\bD}$ or that $X_0 \in \Schaff_{/\bD^{\times}}$ then
\[
\L^+Y_0 \in \Schaff \;\;\; \mbox{and} \;\;\; \L X_0 \in \Ind(\Schaff).
\]
Moreover, if $X_0$ is of finite type then $\L X_0$ is a Tate affine scheme.
\end{lem}

\begin{proof}
The first statement is clear, the second follows from the proof\footnote{Notice that for any $k$-algebra $A$ we denote by $A((t))$ the push-out of $A \la k \ra k((t))$.} of Lemma \ref{lem:formal-loops-of-affine-of-finite-type-are-Tate-affine}.
\end{proof}

\begin{rem}
Lemma \ref{lem:generalized-loops-on-affine-ft-are-Tate-affine} was already observed in \cite{Drinfeld}*{Example 6.3.4}.
\end{rem}

\subsubsection{Affine Grassmannian}
\label{subsubsec:affine-grassmannian}

An important example of a Tate scheme is the affine Grassmannian. In this section we check that in the case of a reductive group the affine Grassmannian is a Tate scheme.

First we need a discussion of how the condition (b) in the definition of Tate schemes interacts with the notion of being $0$-coconnective.

\begin{lem}
\label{lem:condition-b-on-classical-schemes}
Let $\imath_0:X_0 \hra Y_0$ be a closed immersion of quasi-compact Noetherian classical schemes, then if $\ker(\sO_{Y_0} \ra \imath_{*}(\sO_{X_0}))$ is a coherent sheaf on $Y_0$, then the induced morphism
\[
\imath: X \hra Y
\]
between $X := \LKE_{\clSchaffop \hra \Schaffop}(X_0)$ and $Y := \LKE_{\clSchaffop \hra \Schaffop}(X_0)$ satisfies
\[
\Fib(\sO_{Y} \ra \imath_{*}(\sO_X)) \in \Coh(Y).
\]
\end{lem}

\begin{proof}
First we notice that $\Fib(\sO_{Y} \ra \imath_{*}(\sO_X)) \in \Coh(Y)$ if and only if the morphism $\imath:X \hra Y$ is locally almost finitely presented (see Definition 4.2.0.1 in \cite{SAG}). By definition it is enough to consider an affine scheme $\Spec(A) \ra Y$ and consider the induced map
\begin{equation}
    \label{eq:affine-pullback-of-closed-immersion}
    \Spec(B) := \Spec(A) \underset{Y}{\times}X \ra \Spec(A).
\end{equation}

By \cite{HA}*{Theorem 7.4.3.18} we have that (\ref{eq:affine-pullback-of-closed-immersion}) is almost of finite presentation if and only if 
\begin{itemize}
    \item the cotangent complex $L_{B/A}$ is almost perfect, and
    \item $\Spec(\H^0(B)) \ra \Spec(\H^0(A))$ is of finite presentation.
\end{itemize}

Now we notice that the condition on $\ker(\sO_{Y_0} \ra \imath_{*}(\sO_{X_0}))$ is equivalent to $\imath_{0}$ being a morphism of finite presentation by \cite{stacks-project}*{Tag 01TV}. Finally, the computation of the cotangent complex of a closed inclusion gives that
\[
L_{B/A} \simeq \frac{\sI_0}{\sI^2_0},
\]
where $\sI_0 = \ker(\sO_{Y_0} \ra \imath_{*}(\sO_{X_0}))$. Since we assumed our schemes to be Noetherian we have that $L_{B/A}$ is coherent and in particular almost perfect.
\end{proof}

\begin{prop}
Let $G$ denote an reductive affine algebraic group, then the affine Grassmannian $\Gr_G$ is a Tate scheme.
\end{prop}

\begin{proof}
By \cite{DG-indschemes}*{Theorem 9.3.4} the indscheme $\Gr_G$ is $0$-coconnective and locally almost of finite type. Thus we have a presentation
\[
\Gr_G \simeq \colim_I \Gr_{G,i}
\]
where each $\Gr_i$ is almost of finite type and 
\[
f_{i,j}: \Gr_i \hra \Gr_j
\]
are closed embedding. Now to check that the maps $f_{i,j}$ satisfy condition (b) from the definition of Tate schemes, Lemma \ref{lem:condition-b-on-classical-schemes} implies that it is enough to check that
\[
\classical{f_{i,j}}: \classical{\Gr_i} \hra \classical{\Gr_j}
\]
are finitely presented.

For concreteness we will focus on the case where $G = \GL_n$ is the linear group for some $n$. In this case it is well-known (see \cite{Xinwen}*{Theorem 1.1.3}) that
\[
\classical{\Gr_{\GL_n}} \simeq \colim_{i \geq 1}\Gr_{i},
\]
where $\Gr_i$ is the classical prestack whose $R$-points, for $R$ a classical $k$-algebra, are
\[
\Gr_i(R) \simeq \{ R[[t]]-\mbox{quotients of the module }\frac{t^{-i}R[[t]]}{t^i R[[t]]}\mbox{ which are projective as }R\mbox{-modules}\}.
\]

We claim that
\begin{equation}
    \label{eq:inclusion-N-to-N+1}
    \Gr_i \hra \Gr_i
\end{equation}
has a finitely presented ideal. Indeed, \cite{Xinwen}*{Lemma 1.1.6} gives that is a closed subscheme of the Grassmannian of $\Gr(2ni)$ given by the subset of projective $R$-submodules which are stable under a nilpotent morphism $\Phi$. We notice that the inclusion (\ref{eq:inclusion-N-to-N+1}) is obtained by pulling back the closed inclusion
\begin{equation}
    \label{eq:inclusion-usual-Grassmannians}
    \Gr(2ni) \hra \Gr(2n(i+1))
\end{equation}
via $\Gr_{i+1} \hra \Gr(2n(i+1))$. The claim now follows from the fact that the inclusion (\ref{eq:inclusion-usual-Grassmannians}) has finitely generated ideal.

For a general reductive group one needs to check that the embedding of 
\[
\Gr_{\leq \mu} \hra \Gr_{\leq \lambda}
\]
for $\lambda > \mu$ is finitely presented, see \cite{Xinwen}*{\S 2.1} for the notation. We leave the details to the reader.
\end{proof}

\subsubsection{Affine Formal schemes}
\label{subsubsec:affine-formal-schemes}

There are at least two point of views on affine formal schemes: (1) as first introduced by Grothendieck in \cite{EGAI}*{I. \S 10} one considers an appropriate class of topological commutative rings and (2) one considers a certain subclass of indschemes. In this section we quickly review the interaction between these two points of view and how a natural class of the (1) type is recovered via the Tate condition on (2).

\paragraph{Affine formal schemes \`a la EGA}

We refer the reader to \cite{stacks-project}*{Tag 07E8} for the following definitions, see also \cite{Fujiwara-Kato}*{Chapter 0, \S 7}.

For a topological ring $A$ with the topology defined by a filtration $\{F^{\lambda}\}_{\Lambda}$ is called \textit{admissible} if $A$ admits an ideal of definition and $A$ is Hausdorff complete, i.e.\
\[
A \simeq \lim_{\Lambda}\frac{A}{F^{\lambda}},
\]
where the right-hand side is equipped with the projective limit topology.

Moreover, we say that $A$ is an \textit{adic ring} if $A$ is admissible and $I$-adic for some ideal $I \subseteq A$.

\begin{construction}
Given an admissible ring $A$, one defines a topologically locally ringed space $\Spf(A)$\footnote{See \url{https://stacks.math.columbia.edu/tag/0AHY} for a discussion of the category of locally topologically ringed spaces.}, as
\begin{itemize}
    \item the underlying topological space of $\Spf(A)$ is the set of open prime ideals of $A$ endowed with the Zariski topology on those ideals, i.e.\ this is isomorphic to $\Spec(A/I)$ as a topological space;
    \item the sheaf of topological rings is\footnote{Here each term of the inverse limit is endowed with the pseudo-discrete topology, a sheaf of discrete rings $\sO$ becomes a sheaf of topological rings by posing that for any quasi-compact $U$, $\sO(U)$ has the discrete topology, and for a general $V = \cup_{I}U_i$ where each $U_i$ is quasi-compact one endows $\sO(U)$ with the induced topology from $\prod_{I}\sO(U_i)$.}
    \[
    \sO_{\Spf(A)} := \lim_{I \in \sI_{A}}\left.\widetilde{\frac{A}{I}}\right|_{\Spf(A)},
    \]
    where $\sI_{A}$ is the set of ideals of definition of $A$ and $\widetilde{\frac{A}{I}}$ is the sheaf of discrete rings on $\Spec(A)$ associated to $\frac{A}{I}$. 
\end{itemize}
In particular, if $A$ is an adic ring, the set $\{I^k\}_{k \geq 1}$ is cofinal in $\sI_{A}$ and one has an equivalence
    \[
    \sO_{\Spf(A)} \simeq \lim_{k \geq 1}\left.\widetilde{\frac{A}{I^k}}\right|_{\Spf(A)}.
    \]
\end{construction}

\begin{defn}
\label{defn:affine-formal-schemes-EGA}
A morphism $f: \Spf(A) \ra \Spf(B)$ is determined by a continuous homomorphism of rings $\varphi: B \ra A$. We will denote by $\SchaffFormalEGA$ the category of \emph{affine formal schemes \`a la EGA}, defined as the opposite of the category of admissible topological rings.
\end{defn}

Notice that given any affine scheme $S = \Spec(A)$ one has a affine formal scheme \`a la EGA by considering $A$ with the discrete topology\footnote{I.e.\ here $F^{\bullet} = \{0,A\}$ forms a family inducing the discrete topology on $A$ and making $A$ an admissible topological ring.}. It is also clear that any homomorphism of discrete rings produces a continuous homomorphism of the associated discrete topological rings. Thus, one has a fully faithful functor
\[
\imath^{\rm EGA}: \clSchaff \hra \SchaffFormalEGA.
\]

\paragraph{Formal algebraic spaces viewpoint}

In this section we follow \cite{stacks-project}*{Tag 0AI7} and approach the theory of formal schemes using diagrams of usual schemes (see also \cite{Emerton}). Let the category of affine formal schemes $\SchaffFormal$ be defined as the subcategory of $\clindSchaff$ consisting of those classical ind-affine-schemes $S$ such that
\[
\red{S} \in \Schaff,
\]
where $\red{S}$ denotes the reduced scheme underlying $S$.

Notice that given $X \in \SchaffFormalEGA$ we define a prestack $\h_X$ by
\[
\h_X(S) := \Hom_{\SchaffFormalEGA}(\imath^{\rm EGA}(S),X).
\]
By \cite{stacks-project}*{Tag 0AI1} this defines a fully faithful embedding
\[
\h_{(-)}: \SchaffFormalEGA \ra \PStk.
\]

\begin{lem}
Let $X = \Spf(A)$, where $\{F^{\lambda}\}_{\Lambda}$ is a family exhibiting $A$ as an admissible topological ring. One has an equivalence
\[
\h_X \simeq \colim_{\Lambda^{\rm op}}\Spec(A/F^{\lambda}).
\]
In particular, the functor $\h_{(-)}$ factors as
\[
\imath^{\rm EGA}: \SchaffFormalEGA \ra \SchaffFormal.
\]
\end{lem}

One can also characterize what is the essential image of $\h_{(-)}$ in the category $\SchaffFormal$, the following is a particular case of \cite{stacks-project}*{Tag 0AIC}.

\begin{lem}
For an object $X$ of $\SchaffFormal$ the following are equivalent:
\begin{enumerate}[(i)]
    \item there exists $Y \in \SchaffFormalEGA$ such that
    \[
    \imath^{\rm EGA}(Y) \simeq X;
    \]
    \item there exists an affine scheme $X_0$ and an ind-closed map $f:X \ra X_0$.
\end{enumerate}
\end{lem}

\paragraph{Affine formal schemes corresponding to classical Tate affine schemes}

In this section we try to understand which subcategory of $\SchaffFormalEGA$ corresponds to classical Tate affine schemes in $\SchaffFormal$. That is, let
\[
\clSchaffTateformal := \clSchaffTate \cap \SchaffFormal
\]
denote the subcategory of Tate affine schemes $S$ whose underlying reduced prestack is equivalent to a classical scheme.

We want to understand the category $\SchaffFormalEGATate$ defined by the pullback
\[
\begin{tikzcd}
\SchaffFormalEGATate \ar[r] \ar[d] & \clSchaffTateformal \ar[d] \\
\SchaffFormalEGA \ar[r] & \SchaffFormal.
\end{tikzcd}
\]

The next result gives a necessary condition on an admissible topological ring to produce a Tate affine scheme whose underlying reduced prestack is a classical scheme.

\begin{prop}[see \cite{Fujiwara-Kato}*{Chapter 0, Proposition 7.2.11}]
\label{prop:admissible-ring-associated-to-Tate-affine-scheme}
For a Tate affine scheme $X \simeq \colim_I X_i$ with an ind-closed morphism $X \ra X_0$, the associated topological ring $A := \lim_{I^{\rm op}}\Gamma(X_i,\left.\sO_{X}\right|_{X_i})$ is adic.
\end{prop}

For the moment we couldn't describe the category $\SchaffFormalEGATate$ more concretely. However we notice that in the case where we consider the subcategory
\[
\clSchaffTateft \cap \clSchaffTateformal
\]
where $\clSchaffTateft$ denotes the subcategory of $0$-coconnective Tate affine schemes of finite type then this category contains a well-known subcategory of affine formal schemes \`a la EGA.

Recall that an ideal of defition $I \subset A$ is said to be of \emph{finite type} if it is a finitely generated ideal of $A$.

\begin{prop}
Let $\SchaffFormalEGANoeth$ denote the subcategory of $\SchaffFormalEGA$ generated by formal affine schemes whose underlying topological ring is adic and Noetherian. Then one has a fully faithful embedding
\[
\SchaffFormalEGANoeth \simeq \clSchaffTateft \cap \clSchaffTateformal.
\]
\end{prop}

\begin{proof}
Let $A$ be a Noetherian adic ring with $I \subset A$ an ideal of definition of finite type, then we claim that
\[
\h_{\Spf(A)} \simeq \colim_{n \geq 1}\h_{\Spec(A/I^n)}
\]
is a Tate affine scheme. Indeed, we need to check that for every for $\ell > k$ the quotient $\frac{I^{k}}{I^{\ell}}$ is finitely presented as an $A/I^{\ell}$-module. Notice that $I$ finitely presented implies that $I^k$ is finitely presented, thus the quotient
\[
\frac{I^{k}}{I^{\ell}}
\]
is finitely presented as an $A$-module. Thus, the restriction to an $A/I^{\ell}$-module is also finitely presented. Moreover, the canonical map $\Spf(A) \ra \Spec(A/I)$ where we consider $A/I$ with the discrete topological gives an isomorphism between $\red{\h_{\Spf(A)}}$ and $\Spec(A/I)$. Thus, the canonical inclusion functor $\imath: \SchaffFormalEGANoeth \hra \SchaffFormal$ factors as
\[
\SchaffFormalEGANoeth \hra \clSchaffTateft \cap \clSchaffTateformal.
\]
\end{proof}

\section{Sheaves on Tate schemes}
\label{sec:sheaves}

In this section we develop the formalism of Tate coherent sheaves on Tate schemes locally almost of finite type. In \S \ref{subsec:indcoh-review}, we review the formalism of Ind-coherent sheaves restricted to Tate schemes, this follows directly from the formalism in \cites{GR-I,GR-II} where this formalism is constructed over ind-(inf-)schemes. In \S \ref{subsec:Pro-ind-coherent-on-schemes}, we construct the formalism of Pro-Ind-coherent sheaves on schemes. The discussion is analogous to the construction of Ind-coherent sheaves on schemes as performed in \cite{GR-I}*{Chapter 4}. Namely, we first construct an arbitrary $\star$-pushforward functor via Pro-extension, and then construct an arbitrary $!$-pullback functor via a factorization formalism, where $!$ is defined to be the left (resp.\ right) adjoint of $\star$ for open (resp.\ proper) morphisms. In \S \ref{subsec:Tate-coherent-sheaves}, we define the formalism of Tate-coherent sheaves on schemes by using pro-ind-coherent sheaves. The distinction from \S \ref{subsec:Pro-ind-coherent-on-schemes} is that we start with arbitrary pullback !-functors, inherited from the Pro-Ind-coherent sheaves and define an arbitrary $\star$-pushforward by passing to adjoints.

In \S \ref{subsec:technical-presentability}, we discuss a technical condition on our category of Pro-Ind-coherent sheaves needed to apply the formalism of extensions from \cite{GR-I}*{Chapter 8}. The main point is that considering arbitrary Pro-objects on a presentable category is not presentable. We get around this difficulty by being careful regarding the cardinality of the diagrams presenting the Pro-objects. In \S \ref{subsec:Pro-ind-coherent-on-Tate-schemes}, we extend the formalism of Pro-Ind-coherent sheaves from schemes to Tate schemes. This follows from the formalism of extensions for the categories of correspondences. One could have considered the formalism for ind-inf-schemes with the same arguments, but we restrict to Tate schemes for uniformity of presentation. Pro-ind-coherent sheaves have already been considered in \cite{Drinfeld}*{\S 6.3.2}. In \S \ref{subsec:Tate-coherent-sheaves-on-Tate-schemes}, we finally define the formalism of Tate-coherent sheaves on Tate schemes. This section is similar to \S \ref{subsec:Pro-ind-coherent-on-Tate-schemes} and relies on the results of \S \ref{subsec:Tate-coherent-sheaves}. 

\paragraph{What do we want?}

This section is devoted to the construction of an assignment
\[
S \in \SchTateaft \rightsquigarrow \TateCoh(S) \in \DGc
\]
and to understand what functorialities this construction has.

\paragraph{Tate objects}

We recall the following concept. For $S$ a Tate scheme locally almost of finite type one defines the category of \emph{Tate coherent sheaves}\footnote{We will denote
\[
\ProIndCoh(S) = \Pro(\IndCoh(S))
\]
for clarity.} 
\[
\TateCoh(S) \subset \ProIndCoh(S)
\]
as the full subcategory of Pro-objects in $\IndCoh(S)$ which can be represented as
\[
\sF \simeq \lim_{I} \sF_i
\]
with $\sF_i \in \IndCoh(S)$, where for each $i \ra j$ in $I$ one has
\[
\Fib(\sF_i \ra \sF_j) \in \Coh(S).
\]

We notice that the category $\TateCoh(S)$ is equivalent to the smallest stable subcategory of $\Pro(\IndCoh(S))$ generated by the essential images of $\IndCoh(S)$, $\Pro(\Coh(S))$.

\paragraph{Why one needs Pro-ind-sheaves?}

A naive guess would be that we can extend the functors (\ref{eq:IndCoh-!-affine-Tate}) and (\ref{eq:IndCoh-*-affine-Tate}) to Tate objects directly, however this does not quite work.

Let $f:S \ra T$ be a map of schemes almost of finite type, one of the difficulties of following exactly the path of \cite{GR-I}*{Chapter 4, \S 2} is that one can not expect to extend
\[
f^{\rm Ind}_*:\IndCoh(S) \ra \IndCoh(T) \hra \TateCoh(T)
\]
to a functor out of $\TateCoh(S)$. Essentially, the problem is that the extension 
\[
f^{\rm ProInd}_*: \ProIndCoh(S) \ra \ProIndCoh(T)
\]
obtained by Pro extending the map $f^{\rm Ind}_*:\IndCoh(S) \ra \ProIndCoh(T)$ does not preserve Tate objects in $S$ for an arbitrary map $f$\footnote{The reason is the same as why arbitrary push-forwards do not preserve coherent objects.}.

This suggests that our first goal should be to define a functor
\[
\ProIndCoh^!: (\Schaft)^{\rm op} \ra \DGc
\]
which encodes the $!$-pullback formalism and also a functor
\[
\ProIndCoh: \Schaft \ra \DGc
\]
which will encode the $*$-pushforward formalism. 

\paragraph{Tate-coherent sheaves inside pro-ind-coherent sheaves}

Once the formalism of pro-ind-coherent sheaves is in place, we observe that the $!$-pullback from pro-ind-coherent sheaves preserves Tate objects. However, for an arbitrary map $f$ the pushforward $f^{\rm ProInd}_*$ does \emph{not} preserve the Tate condition. To circumvent this problem we proceed in an analogy with how the arbitrary pullback for ind-coherent sheaves is constructed.

Namely, given $f:S \ra T$ an arbitrary map of schemes almost of finite type, one has a factorization
\begin{equation}
    \label{eq:factorization-f-in-introduction-of-subsection}
    S \overset{g}{\ra} Y \overset{h}{\ra} T
\end{equation}
where $g$ is open and $h$ is proper. For $h$ the restriction of $h^{\rm ProInd}_*$ to Tate-objects sends them to Tate-objects and one has an adjunction $()$. For $g$ we \emph{define} $g^{\rm Tate}_*$ as the \emph{right} adjoint to $g^!$. Thus, one has
\[
f^{\rm Tate}_* := h^{\rm Tate}_* \circ g^{\rm Tate}_*
\]

As usual one needs to use base change of Tate-coherent sheaves to check that this assignment is well-defined, i.e.\ does not depend on the factorization (\ref{eq:factorization-f-in-introduction-of-subsection}) and has the correct coherence data.

\subsection{Review of ind-coherent sheaves}
\label{subsec:indcoh-review}

In this section we briefly review the theory of ind-coherent sheaves as developed by Gaitsgory--Rozenblyum in \cites{GR-I,GR-II}. We specialize their theory to the special case of Tate schemes almost of finite type. Using Ind-coherent sheaves we develop the formalism of Pro-Ind-coherent sheaves and Tate-coherent sheaves on Tate schemes locally almost of finite type on sections \ref{subsec:Pro-ind-coherent-on-schemes} and \ref{subsec:Tate-coherent-sheaves-on-Tate-schemes}.

A word on the locally almost of finite type condition. A more general theory of Ind-coherent sheaves exist on the whole category of Tate schemes, which has a natural $*$-pullback, see \cite{Raskin-homological}*{\S 6}. We could take that theory as a starting point to develop an analogous formalism of Pro-Ind-coherent and Tate-coherent sheaves. These theories should be related to the ones in this section in the special case of Tate schemes locally almost of finite type. We plan to return to this point in the future.

\subsubsection{Ind-coherent sheaves on schemes almost of finite type}
\label{subsubsec:ind-coherent-sheaves-on-schemes-almost-of-finite-type}

Let's recall the formalism of $\IndCoh$ for schemes almost of finite type as developed in \cite{GR-I}*{Chapter 4 and Chapter 5}. 

For a scheme almost of finite type\footnote{One can define $\Coh(S)$ for a larger class of schemes, following \cite{Gaitsgory-IndCoh}*{Section 2.2}, however the theory of coherent sheaves behave better when restricted to the schemes almost of finite type.} $S$ we define
\[
\Coh(S) \subset \QCoh(S)
\]
as the subcategory of bounded complexes with coherent cohomology. 

One has a functor
\[
\IndCoh^!: (\Schaft)^{\rm op} \ra \DGc
\]
which assigns to any morphism $f:S \ra T$ a pullback map
\[
f^!: \IndCoh(T) \ra \IndCoh(S).
\]

The $!$-pullback functor has a covariant counter-part $f^{\rm Ind}_{*}$, this is determined as the unique functor that makes the following diagram commutes
\[
\begin{tikzcd}
\IndCoh(S) \ar[r,"\Psi_S"] \ar[d,"f^{\rm Ind}_{*}"'] & \QCoh(S) \ar[d,"f_{*}"] \\
\IndCoh(T) \ar[r,"\Psi_T"] & \QCoh(T) \\
\end{tikzcd}
\]
where $\Psi_{S}: \IndCoh(S) \ra \QCoh(S)$ is the canonical functor given by Ind extending the inclusion $\Coh(S) \hra \QCoh(S)$.

One recalls the following properties which essentially define $f^!$ for all morphisms.

\begin{prop}[\cite{GR-I}*{Chapter 5, Theorem 3.4.3}]
\label{prop:IndCoh-adjunctions-for-schemes}
Let $f:S \ra T$ be a map of schemes almost finite type:
\begin{enumerate}[(1)]
    \item when $f$ is an open immersion $f^!$ is the \emph{left} adjoint of $f^{\rm Ind}_*$;
    \item when $f$ is proper $f^!$ is the \emph{right} adjoint of $f^{\rm Ind}_*$.
\end{enumerate}
\end{prop}

\begin{rem}
To actually construct $f^!$ for all morphisms using the above two requirements one needs to go through the formalism of $\IndCoh$ for correspondences. This is done in \cite{GR-II}*{Chapter 5} whereto we refer the reader for details. We will perform this formal construction when extending the formalism of ind-coherent to pro-ind-coherent and Tate-coherent sheaves in section \ref{subsec:Pro-ind-coherent-on-schemes} below.
\end{rem}

\subsubsection{Ind-coherent sheaves on Tate schemes locally almost of finite type}
\label{subsubsec:indcoh-for-Tate-aft}

The results of \cite{GR-II}*{Chapter 3} extend the formaslism of $\IndCoh$ from the previous subsection to prestacks of locally almost of finite type and ind-inf-schematic morphisms between those (see \emph{loc.\ cit.\ } for a precise formulation). Here we will summarize their results restricted to the case of Tate schemes almost of finite type.

\paragraph{}

The pullback functoriality is defined for all prestacks locally of finite type.

\begin{defn}[$!$-pullback]
\label{defn:!-pullback-prestacks-laft}
Following \cite{GR-I}*{Chapter 5, \S 3} we define the functor
\[
\IndCoh^!: \PStk_{\rm laft} \ra \DGc
\]
by the \textit{right Kan extension} of $\IndCoh^!$ on affine schemes almost of finite type. 
\end{defn}

\paragraph{}

The direct image is a bit more subtle and only exists for ind-inf-schematic between prestacks locally almost of finite type. We will however restrict to the particular of ind-schemes, as those are sufficient for our purposes, and any map between those is ind-schematic and in particular ind-inf-schematic. 

\begin{defn}[$*$-pushforward]
\label{defn:*-pushforwad-ind-schemes}
Following \cite{GR-II}*{Chapter 3, \S 1} we define the functor\footnote{Here $\Ind(\Sch)$ is the subcategory of prestacks $\sX$ that are: (i) convergent and (ii) can be represented as $\sX \simeq \colim_{I}X_i$ where $X_i$ is a quasi-compact scheme.}
\[
\IndCoh: \indSchlaft \ra \DGc
\]
as the \emph{left Kan extension} of $\IndCoh$ for affine schemes.
\end{defn}

\paragraph{Restriction to Tate schemes}

We notice that $\SchTatelaft$ is a subcategory of both $\PStk_{\rm laft}$ and of $\indSchlaft$ hence we can consider the restrictions
\begin{equation}
\label{eq:IndCoh-!-affine-Tate}
\IndCoh^!: (\SchTatelaft)^{\rm op} \ra \DGc
\end{equation}
and
\begin{equation}
\label{eq:IndCoh-*-affine-Tate}
\IndCoh: \SchTatelaft \ra \DGc.    
\end{equation}

\paragraph{}

The following Theorem summarizes all the properties of the functors (\ref{eq:IndCoh-!-affine-Tate}) and (\ref{eq:IndCoh-*-affine-Tate}). We emphasize that to correctly formalize them one needs to consider an extension of the above functors to a functor from the $2$-category of correspondences (see \S \ref{par:functor-out-of-2-category-of-correspondences} below for a discussion of this, or \cite{GR-II}*{Chapter 3, \S 5} for more details.)

\begin{thm}
\label{thm:package-IndCoh-affine-ind-schemes}
The functors (\ref{eq:IndCoh-!-affine-Tate}) and (\ref{eq:IndCoh-*-affine-Tate}) have the following properties:
\begin{enumerate}[1)]
    \item (proper adjunction) suppose that $f:S \ra T$ is an ind-proper map (see Definition \ref{defn:ind-proper-map}) in $\SchTatelaft$ then one has an adjunction
    \[
    (f^{\rm Ind}_*,f^!);
    \]
    \item (open embedding) suppose that $g: S \ra T$ is an open embedding in $\SchTatelaft$ then one has an adjunction
    \[
    (f^!,f^{\rm Ind}_*);
    \]
    \item (base change) suppose that
    \[
    \begin{tikzcd}
    S' \ar[r,"g_S"] \ar[d,"f'"'] & S \ar[d,"f"] \\
    T' \ar[r,"g_T"] & T
    \end{tikzcd}
    \]
    is a Cartesian diagram where $f$ is ind-proper, then the canonical map
    \begin{equation}
        \label{eq:base-change-map}
        (f')^{\rm Ind}_*\circ g^!_S \ra g^!_T\circ f^{\rm Ind}_*
    \end{equation}
    arising from using adjunction 1) to $g^!_S\circ f^! \simeq (f')^{!}\circ g^!_T$ is an isomorphism. 
    
    Similarly, if we suppose that $g_T$ is ind-proper, then we can use $f^{\rm Ind}_* \circ (g^{\rm Ind}_S)_* \simeq (g^{\rm Ind}_T)_* \circ (f')^{\rm Ind}_*$ to obtain the map (\ref{eq:base-change-map}) which is also an equivalence.
\end{enumerate}
\end{thm}

\begin{rem}
The above theorem is a restatement of \cite{GR-I}*{Chapter 5, Theorem 3.4.3}, \cite{GR-II}*{Chapter 3, \S 2.1 and 2.2} and \cite{GR-II}*{Chapter 5, Theorem 5.4.3 and \S 5.5} for the particular case of Tate schemes.
\end{rem}

\paragraph{Ind-coherent sheaves on a filtered diagram}

For a Tate scheme one can describe the category of Ind-coherent sheaves more concretely.

\begin{prop}
\label{prop:limit-and-colimit-presentation-of-IndCoh-on-Tate-schemes}
Let $T \in \SchTatelaft$ and consider a presentation $T \simeq \colim_I T_i$ where each $T_i \in \Schaft$, i.e.\ $T_i$ is a scheme almost of finite type, one has an equivalence
\[
\IndCoh(T) = \lim_{I^{\rm op}}\IndCoh(T_i) \simeq \colim_I\IndCoh(T_i),
\]
where the limit is taken with respect to the maps $f^!_{i,j}:\IndCoh(T_j) \ra \IndCoh(T_i)$ and the colimit with respect to $(f^{\rm Ind}_{i,j})_*:\IndCoh(T_i) \ra \IndCoh(T_j)$.
\end{prop}

\begin{proof}
The first equivalence is the definition of $\IndCoh$ on prestacks locally almost of finite type restricted to the case of Tate schemes locally almost of finite type. The second equivalence follows from Lemma \ref{lem:limits-to-colimits-of-categories} and Proposition \ref{prop:IndCoh-adjunctions-for-schemes} (2).
\end{proof}

\subsection{Pro-ind-coherent sheaves}
\label{subsec:Pro-ind-coherent-on-schemes}

In this section we extend the formalism of ind-coherent sheaves to that of pro-ind-coherent sheaves. We follow closely the path of Gaitsgory--Rozenblyum \cites{GR-I,GR-II} and the summary is that essentially everything works as it is supposed to.

\subsubsection{Definition and arbitrary push-forward}

\paragraph{}

Given a scheme $S$ almost of finite type, we will denote by
\[
\ProIndCoh(S) := \Pro(\IndCoh(S))
\]
the category of Pro-objects in $\IndCoh(S)$, i.e.\ the opposite of the category of left exact functors 
\[
\IndCoh(S) \ra \Spc.
\]

\paragraph{}

Given $f:S \ra T$ an arbitrary morphism between schemes almost of finite type, one defines the $*$-pushforward
\begin{equation}
\label{eq:pro-ind-pushforward}
    f^{\rm ProInd}_*: \ProIndCoh(S) \ra \ProIndCoh(T)
\end{equation}
as the Pro-extension of the functor
\[
f^{\rm Ind}_*: \IndCoh(S) \ra \IndCoh(T).
\]

\paragraph{}

Notice that since the functor $f^{\rm Ind}_*$ is continuous, so in particular left exact, then $f^{\rm ProInd}_*$ is also continuous. We denote the resulting functor by
\begin{equation}
    \label{eq:ProIndCoh-functor-from-schemes}
    \ProIndCoh: \Schaft \ra \DGc.
\end{equation}

\paragraph{t-structure on pro-ind-coherent sheaves}

For any affine scheme $S$ the category $\ProIndCoh(S)$ has a t-structure determined by
\[
\ProIndCoh(S)^{\geq 0} := \Pro(\IndCoh(S)^{\geq 0})
\]
and $\ProIndCoh(S)^{\leq 0}$ is defined as the left orthogonal to $\ProIndCoh(S)^{\geq 1} := \ProIndCoh(S)^{\geq 0}[-1]$. One can check that this defines a t-structure as in the proof of \cite{AGH}*{Proposition 2.13} (see also \cite{SAG}*{Lemma C.2.4.3}).

\subsubsection{Open morphisms}

\paragraph{$*$-pullback for eventually coconnective morphisms}
\label{par:*-pullback-for-eventually-connective}

Recall that a map $f:S \ra T$ is said to be \emph{eventually coconnective} if one of the following equivalent conditions is satisfied:
\begin{enumerate}[(i)]
    \item $f^*:\QCoh(T) \ra \QCoh(S)$ sends $\Coh(S)$ to $\Coh(T)$;
    \item $f^*$ has finite Tor amplitude, i.e.\ $f^*$ is left t-exact up to a finite cohomological shift\footnote{Recall this means that $f^*$ sends $\QCoh(T)^{\geq 0}$ to $\QCoh(S)^{\geq -n} = (\QCoh(S)^{\geq 0})[n]$, for some $n \in \bZ_{\geq 0}$.}.
\end{enumerate}

In this case one has a functor
\[
f^{\rm Ind, *}: \IndCoh(T) \ra \IndCoh(S)
\]
obtained by Ind extending $\left.f^*\right|_{\Coh(T)}: \Coh(T) \ra \Coh(S)$. 

For $f:S \ra T$ eventually connective we define
\[
f^{\rm ProInd, *}: \ProIndCoh(T) \ra \ProIndCoh(S)
\]
as the Pro extension of $f^{\rm Ind, *}$.

The following is a general consequence of the interaction between taking Pro-objects and passing to left adjoints, see Proposition \ref{prop:left-adjoint-of-Pro-extended-functors}.

\begin{lem}
\label{lem:left-adjoint-is-Pro-extension}
For $f$ eventually connective, the functor $f^{\rm ProInd, *}$ is left adjoint to $f^{\rm ProInd}_*$.
\end{lem}

\paragraph{Base change for open morphisms}
\label{par:base-change-open}

Consider
\[
\begin{tikzcd}
    S' \ar[r,"g_S"] \ar[d,"f'"'] & S \ar[d,"f"] \\
    T' \ar[r,"g_T"] & T
\end{tikzcd}
\]
a Cartesian diagram in $\Schaft$ where $f$ is eventually coconnective, by adjunction the natural isomorphism
\[
g^{\rm ProInd}_{T,*} \circ (f')^{\rm ProInd}_{*} \ra f^{\rm ProInd}_{*} \circ g^{\rm ProInd}_{S,*}
\]
gives a natural transformation
\begin{equation}
\label{eq:base-change-open-morphism}
f^{\rm ProInd, *} \circ g^{\rm ProInd}_{T,*} \ra g^{\rm ProInd}_{S,*} \circ (f')^{\rm ProInd, *}.
\end{equation}

\begin{prop}
\label{prop:base-change-eventually-coconnective}
The natural transformation (\ref{eq:base-change-open-morphism}) is an isomorphism.
\end{prop}

\begin{proof}
By Lemma \ref{lem:left-adjoint-is-Pro-extension} the map (\ref{eq:base-change-open-morphism}) is equivalent to 
\[
\Pro(f^{\rm Ind, *}) \circ \Pro(g^{\rm Ind}_{T,*}) \simeq f^{\rm ProInd, *} \circ g^{\rm ProInd}_{T,*} \ra g^{\rm ProInd}_{S,*} \circ (f')^{\rm ProInd, *} \simeq \Pro(g^{\rm Ind}_{S,*}) \circ \Pro((f')^{\rm Ind, *}),
\]
where we denote by $\Pro(F)$ the Pro-extension a functor $F$. By the functoriality of Pro-extensions, one can rewrite the above as
\[
\Pro(f^{\rm Ind, *} \circ g^{\rm Ind}_{T,*}) \ra \Pro(g^{\rm Ind}_{S,*} \circ (f')^{\rm Ind, *})
\]
which is an isomorphism since the map
\[
f^{\rm Ind, *} \circ g^{\rm Ind}_{T,*} \ra g^{\rm Ind}_{S,*} \circ (f')^{\rm Ind, *}
\]
is an isomorphism by \cite{GR-I}*{Chapter 4, Proposition 3.2.2}.
\end{proof}

\paragraph{}

\begin{prop}
\label{prop:open-immersion-fully-faithful-pushforward}
Let $\jmath:S \ra T$ be an open embedding, the functor
\[
\jmath^{\rm ProInd}_*: \ProIndCoh(S) \ra \ProIndCoh(T)
\]
is fully faithful.
\end{prop}

\begin{proof}
It is enough to check that the adjunction morphism
\[
\jmath^{\rm ProInd,*}\circ \jmath^{\rm ProInd}_{*} \ra \id_{\ProIndCoh(S)}
\]
is an isomorphim. Let $\sG \simeq \lim_{I}\sG_I \in \ProIndCoh(S)$ one has
\begin{align*}
    \jmath^{\rm ProInd, *}\circ \jmath^{\rm ProInd}_{*}(\lim_I \sG_i) & \simeq \jmath^{\rm ProInd, *}(\lim_I \jmath^{\rm Ind_*}(\sG_i)) \\
    & \simeq \lim_I \jmath^{\rm Pro, *}\circ\jmath^{\rm Pro}_*(sG_i) \\
    & \simeq \lim_I \sG_i.
\end{align*}
Here the isomorphism of the first line follows from the fact that $\jmath^{\rm ProInd}_*$ preserves limits and cocompact objects, the isomorphism in the second line is the definition of $\jmath^{\rm ProInd, *}$ and in the final line we used \cite{GR-I}*{Chapter 4, Proposition 4.1.2}.
\end{proof}

\subsubsection{Proper morphisms}

\paragraph{!-pullback for proper maps}
\label{par:!-pullback-for-proper}

Let $f:S \ra T$ be a proper map between schemes almost of finite type, we define the !-pullback functor
\[
f^!: \ProIndCoh(T) \ra \ProIndCoh(S)
\]
as the right adjoint of $f^{\rm ProInd}_*$. 

\begin{rem}
\label{rem:right-adjoint-is-Pro-extension}
For $f$ a proper map, by Proposition \ref{prop:left-adjoint-of-Pro-extended-functors} the functor $f^!$ is equivalent to the Pro-extension of the functor
\[
f^!: \IndCoh(T) \ra \IndCoh(S).
\]
\end{rem}

\begin{warning}
It is \emph{not} the case that for an arbitrary morphism $f:S \ra T$ the functor
\[
f^!: \ProIndCoh(T) \ra \ProIndCoh(S),
\]
defined in \S \ref{par:!-pullback-arbitrary-maps-ProIndCoh-schemes} below, is the Pro extension of $f^!$ for Ind-coherent sheaves.
\end{warning}

\paragraph{Base change for proper morphisms}
\label{par:base-change-proper}

Consider
\[
\begin{tikzcd}
    S' \ar[r,"g_S"] \ar[d,"f'"'] & S \ar[d,"f"] \\
    T' \ar[r,"g_T"] & T
\end{tikzcd}
\]
a Cartesian diagram in $\Schaft$ where $f$ is proper, by adjunction the natural isomorphism
\[
f^{\rm ProInd}_{*} \circ g^{\rm ProInd}_{S,*} \overset{\simeq}{\ra} g^{\rm ProInd}_{T,*} \circ (f')^{\rm ProInd}_{*} 
\]
gives a natural transformation
\begin{equation}
\label{eq:base-change-proper-morphism}
g^{\rm ProInd}_{S,*} \circ (f')^{!} \ra f^! \circ g^{\rm ProInd}_{T,*}
\end{equation}

\begin{prop}
\label{prop:base-change-proper}
The natural transformation (\ref{eq:base-change-proper-morphism}) is an isomorphism.
\end{prop}

\begin{proof}
One needs to check that for all $\sF \in \ProIndCoh(T')$ the map
\[
g^{\rm ProInd}_{S,*} \circ (f')^{!}(\sF) \ra f^! \circ g^{\rm ProInd}_{T,*}(\sF)
\]
is an isomorphism. Notice that all the functors in (\ref{eq:base-change-proper-morphism}) commute with limits. Indeed, by Proposition \ref{prop:left-adjoint-of-Pro-extended-functors}, $g^{\rm ProInd}_{T,*} \simeq \Pro(g^{\rm Ind}_{T,*})$ and similarly for $g^{\rm ProInd}_{S,*}$. Thus, it is enough to check the statement for objects $\sF \in \IndCoh(T')$ and we are reduced to the statement \cite{GR-I}*{Chapter 4, Proposition 5.2.2}.
\end{proof}

\paragraph{Pullback compatibility}

Consider the Cartesian diagram in $\Schaft$
\[
\begin{tikzcd}
    S' \ar[r,"g_S"] \ar[d,"f'"'] & S \ar[d,"f"] \\
    T' \ar[r,"g_T"] & T
\end{tikzcd}
\]
where $f$ (hence $f'$) is proper and $g_{T}$ (hence $g_S$) is eventually coconnective. Then Proposition \ref{prop:base-change-eventually-coconnective} gives a base change isomorphism
\[
(f')^{\rm ProInd}_* \circ g^{\rm ProInd, *}_S \simeq g^{\rm ProInd, *}_T \circ f^{\rm ProInd}_*.
\]
We apply the $(f^{\rm ProInd}_*,f^!)$-adjunction to the above to obtain
\begin{equation}
    \label{eq:compatibility-of-pullbacks-ProInd}
    g^{\rm ProInd, *}\circ f^! \ra (f')^{!}\circ g^{\rm ProInd, *}.
\end{equation}

\begin{prop}
\label{prop:pullback-compatibility}
When $g_T$, and hence $g_S$, are open embeddings the map (\ref{eq:compatibility-of-pullbacks-ProInd}) is an isomorphism.
\end{prop}

\begin{proof}
We notice that by Lemma \ref{lem:left-adjoint-is-Pro-extension} and Remark \ref{rem:right-adjoint-is-Pro-extension} one has
\[
g^{\rm ProInd, *}_S \simeq \Pro(g^{\rm Ind, *}_S) \;\;\; \mbox{and} \;\;\; f^! \simeq \Pro(f^!),
\]
and similarly for $g^{\rm ProInd, *}_T$ and $(f')^!$. Thus, we are reduced to the claim for the corresponding ind-coherent sheaves functors, which is \cite{GR-I}*{Chapter 4, Proposition 5.3.4}.
\end{proof}

\begin{rem}
\label{rem:both-ways-give-same-compatibility-map}
As in \cite{GR-I}*{Remark 5.3.2} we notice that one can also obtain the map (\ref{eq:compatibility-of-pullbacks-ProInd}) by first considering the isomorphism given by Proposition \ref{prop:base-change-proper} and then using the $(g^{\rm ProInd, *}_S,g^{\rm ProInd}_{S,*})$-adjunction. 
\end{rem}

\subsubsection{Zariski Descent}

In this section we prove that $\ProIndCoh$ can be glued together from a Zariski cover.

\paragraph{}

Let $f:U \ra X$ be a Zariski cover, i.e.\ $U$ is the union of affine schemes whose union is all of $X$. Consider $U^{\bullet}$ the \c{C}ech nerve of $f$, the functors $\ProIndCoh$ with $*$-pullback define a cosimplicial category
\[
\ProIndCoh(U^{\bullet}),
\]
which is augmented by $\ProIndCoh(X)$. We have

\begin{prop}
The functor
\[
\ProIndCoh(X) \ra \Tot(\ProIndCoh(U^{\bullet}))
\]
is an equivalence of categories.
\end{prop}

\begin{proof}
By induction it is enough to prove the result for a cover $X = U_1\cap U_2$, where $U_{12} = U_1\cap U_2$. Let
\[
\jmath_1: U_1 \hra X,\; \jmath_2: U_2 \hra X, \; \jmath_{12}: U_{12} \hra X, \; \jmath_{12,1}: U_{12} \hra U_1, \; \mbox{and} \; \jmath_{12,2}: U_{12} \hra U_2.
\]
We need to check that the functor
\begin{equation}
    \label{eq:from-X-to-descent-data}
    \ProIndCoh(X) \ra \ProIndCoh(U_1)\underset{\ProIndCoh(U_{12})}{\times}\ProIndCoh(U_2)
\end{equation}
that sends $\sF \in \ProIndCoh(X)$ to
\[
\{(\jmath^{\rm ProInd, *}_1(\sF), \jmath^{\rm ProInd, *}_2(\sF), \jmath^{\rm ProInd, *}_{12,1} \circ \jmath^{\rm ProInd, *}_{1}(\sF)\simeq \jmath^{\rm ProInd, *}_{12,2} \circ \jmath^{\rm ProInd, *}_{2}(\sF))\}
\]
is an equivalence.

Consider the right adjoint to (\ref{eq:from-X-to-descent-data})
\begin{equation}
\label{eq:right-adjoint-to-descent-data}
    \ProIndCoh(U_1)\underset{\ProIndCoh(U_{12})}{\times}\ProIndCoh(U_2) \ra \ProIndCoh(X)    
\end{equation}
given by sending $(\sF_1,\sF_2,\sF_{12},\jmath^{\rm ProInd,*}_{12,1}(\sF_1)\simeq \sF_{12} \simeq \jmath^{\rm ProInd, *}_{12,2}(\sF_2))$ to
\[
\Fib\left(((\jmath_{1})^{\rm ProInd}_{*}(\sF_1)\oplus (\jmath_{2})^{\rm ProInd}_{*}(\sF_2)) \ra (\jmath_{12})^{\rm ProInd}_{*}(\sF_{12})\right)
\]
where the maps $\jmath_{2})^{\rm ProInd}_{*}(\sF_2) \ra (\jmath_{12})^{\rm ProInd}_{*}(\sF_{12})$ comes from applying $(\jmath_{i})^{\rm ProInd}_*$ to the adjunction $\id_{\ProIndCoh(U_i)} \ra (\jmath_{12,i})^{\rm ProInd}_{*}\circ (\jmath_{12,i})^{\rm ProInd, *}$ for $i = 1,2$.

Consider the composition of (\ref{eq:right-adjoint-to-descent-data}) with (\ref{eq:from-X-to-descent-data})
\[
\ProIndCoh(U_1)\underset{\ProIndCoh(U_{12})}{\times}\ProIndCoh(U_2) \ra \ProIndCoh(X) \ra \ProIndCoh(U_1)\underset{\ProIndCoh(U_{12})}{\times}\ProIndCoh(U_2)
\]
we claim this is an equivalence. Indeed, we need to check that
\[
\jmath^{\rm ProInd, *}_i\left(\Fib\left(((\jmath_{1})^{\rm ProInd}_{*}(\sF_1)\oplus (\jmath_{2})^{\rm ProInd}_{*}(\sF_2)) \ra (\jmath_{12})^{\rm ProInd}_{*}(\sF_{12})\right)\right) \simeq \sF_i,
\]
is an isomorphism for $i=1,2$ and that the corresponding glueing morphisms agree. 

We do the case $i=1$ and leave the others to the reader. First we notice that Proposition \ref{prop:base-change-eventually-coconnective} guarantes base change with respect to
\begin{equation}
    \label{eq:base-change-proof-of-Zariski-descent}
    \begin{tikzcd}
    U_{12} \ar[r,"\jmath_{12,2}"] \ar[d,"\jmath_{12,1}"] & U_2 \ar[d,"\jmath_{2}"] \\
    U_1 \ar[r,"\jmath_{1}"] & X.
    \end{tikzcd}
\end{equation}
Thus, base change with respect to (\ref{eq:base-change-proof-of-Zariski-descent}) and Proposition \ref{prop:open-immersion-fully-faithful-pushforward} gives that
\[
\jmath^{\rm ProInd, *}_1\circ (\jmath_1)^{\rm ProInd}_*(\sF_1) \oplus \jmath^{\rm ProInd, *}_1\circ\jmath^{\rm ProInd, *}_2(\sF_2) \simeq \sF_1 \oplus (\jmath_{12,1})^{\rm ProInd}_{*}(\sF_{12}),
\]
and also that
\[
\jmath^{\rm ProInd, *}_1\circ(\jmath_{12})^{\rm ProInd}_{*}(\sF_{12}) \simeq (\jmath_{12,1})^{\rm ProInd}_{*}(\sF_{12}).
\]

We now check that the composition of (\ref{eq:from-X-to-descent-data}) with (\ref{eq:right-adjoint-to-descent-data})
\[
\ProIndCoh(X) \ra \ProIndCoh(U_1)\underset{\ProIndCoh(U_{12})}{\times}\ProIndCoh(U_2) \ra \ProIndCoh(X)
\]
is an isomorphism. It is enough to check that for any $\sF \simeq \lim_J \sF_j$ and object of $\ProIndCoh(X)$ with $\sF_j \in \IndCoh(X)$ the canonical map
\[
\sF \ra \Fib\left(((\jmath_{1})^{\rm ProInd}_{*}(\jmath^{\rm ProInd,*}_1(\sF))\oplus (\jmath_{2})^{\rm ProInd}_{*}(\jmath^{\rm ProInd,*}_2(\sF))) \ra (\jmath_{12})^{\rm ProInd}_{*}(\jmath^{\rm ProInd,*}_{12}(\sF))\right)
\]
is an isomorphism. Since all the functors commute with limits by construction one has that it is enough to check that for each $j \in J$ the map
\[
\sF_j \ra \Fib\left(((\jmath_{1})^{\rm Ind}_{*}(\jmath^{\rm Ind,*}_1(\sF_j))\oplus (\jmath_{2})^{\rm Ind}_{*}(\jmath^{\rm Ind,*}_2(\sF_j))) \ra (\jmath_{12})^{\rm Ind}_{*}(\jmath^{\rm Ind,*}_{12}(\sF_j))\right)
\]
is an isomorphism, and this follows from \cite{GR-I}*{Chapter 4, Proposition 4.2.2}.
\end{proof}

\subsubsection{Pro-ind-coherent sheaves on the category of correspondences}

To obtain a general formalism of $!$-pullback for pro-ind-coherent sheaves we proceed as in \cite{GR-I}*{Chapter 5, \S 2}.

\paragraph{Derived Nagata Theorem}

Suppose that $f:S \ra T$ is an arbitrary morphism of schemes almost of finite type, the following is one of the essential ingredients in the definition of $!$-pullback functor for all morphisms.

\begin{prop}[\cite{GR-I}*{Chapter 5, Proposition 2.1.6}]
\label{prop:factorization-category-is-contractible}
Given any map $f:S \ra T$ in $\Schaft$, the category $\Factor(f)$ of factorizations of $f$ as
\[
X \overset{j}{\ra} Z \overset{g}{\ra} Y 
\]
where $j$ is an open embedding and $g$ is a proper morphism, is contractible.
\end{prop}

\paragraph{Heuristics of arbitrary $!$-pullbacks}

This result from the previous paragraph allows us to explain the heuristics of the definition of $f^!$ for an arbitrary morphism $f:S \ra T$. Given an arbitrary $f:S \ra T$ consider a factorization
\begin{equation}
    \label{eq:factorization-f}
    S \overset{j}{\ra} S' \overset{g}{\ra} T    
\end{equation}
where $j$ is an open embedding and $g$ is proper. We notice that $j$ is eventually connective. Indeed, $j$ is open if $j$ is flat and the underlying map of classical schemes $\classical{j}:\classical{S} \ra \classical{S'}$ is an open embedding of classical schemes. But we notice that for any flat morphism $j:X \ra Z$ the map $j^*:\IndCoh(S') \ra \IndCoh(S)$ is t-exact. Thus, we define $f^!:\ProIndCoh(T) \ra \ProIndCoh(S)$ as follows
\[
f^! := j^{\rm ProInd,*}\circ g^{!},
\]
for a given factorization $f = g \circ j$. 

The hard work is in checking that $f^!$ does not depend on the factorization (\ref{eq:factorization-f}), i.e.\ that different choices of factorization give isomorphic pullback functors and are compatible with its functoriality. Moreover, one expect this functor to have base change properties that generalize Proposition \ref{prop:base-change-eventually-coconnective} and Proposition \ref{prop:base-change-proper}.

\paragraph{The $2$-category of correspondences}
\label{par:the-2-category-of-correspondences-on-schemes-aft}

We follow the idea of Gaitsgory--Rozenblyum to formalize the arbitrary $!$-pullback and its base change property as a functor out of the $2$-category of correspondences.

Let $vert, horiz$ and $adm \subseteq vert\cap horiz$ be three classes of $1$-morphisms in $\Schaffaft$, those need to satisfy some natural conditions, we refer the reader to \cite{GR-I}*{Chapter 7, \S 1.1} for details. We define the $2$-category $\Corr(\Schaffaft)^{adm}_{vert; horiz}$ as follows:
\begin{itemize}
    \item objects of $\Corr(\Schaft)^{adm}_{vert; horiz}$ are the same as objects of $\Schaft$;
    \item for two objects $S$ and $T$ of $\Corr(\Schaft)^{adm}_{vert; horiz}$ we define $\Hom_{\Corr(\Schaft)^{adm}_{vert; horiz}}(S,T)$ to be the category of diagrams
    \[
    \begin{tikzcd}
    U \ar[r,"f"] \ar[d,"g"] & T \\
    S & 
    \end{tikzcd}
    \]
    where $f \in horiz$ and $g \in vert$.
    \item for a pair of correspondences $(U,f,g)$ and $(U',f',g')$ between two objects $S$ and $T$ a $2$-morphism is a diagram
    \[
    \begin{tikzcd}
    U' \ar[rrrd,"f'"] \ar[dddr,"g'"'] \ar[dr,"h"'] & & & \\
    & U \ar[rr,"f"'] \ar[dd,"g"] & & T \\
    & & & & \\
    & S & & 
    \end{tikzcd}
    \]
    where $h \in adm$.
\end{itemize}

Compositions of $1$-morphisms are given by taking fiber products and composing the vertical and horizontal maps, composition of $2$-morphisms is just the usual composition.

\paragraph{}

We also recall that by \cite{GR-I}*{Chapter 1, \S 8.3} there exists an $(\infty,2)$-category enhancement of the $(\infty,1)$-category of DG-categories, we will denote this $2$-category by $\DGct$.

\paragraph{}
\label{par:functor-out-of-2-category-of-correspondences}

\begin{thm}
\label{thm:ProIndCoh-from-correspondences}
There exists an extension of (\ref{eq:ProIndCoh-functor-from-schemes}) to a functor of $2$-categories
\begin{equation}
    \label{eq:ProIndCoh-functor-from-correspondences-proper-all-all}
    \ProIndCoh_{\Corraftpaa}: \Corraftpaa \ra \DGct.
\end{equation}
\end{thm}

\begin{proof}
The proof follows exactly the strategy of \cite{GR-I}*{Chapter 5, \S 2}.

\textit{Step 1:} the restriction of (\ref{eq:ProIndCoh-functor-from-schemes}) to the class of open morphisms and viewed as functor
\[
\ProIndCoh: \Schaft \ra (\DGct)^{\rm 2-op}
\]
satisfies the left Beck--Chevalley condition with respect to open morphisms\footnote{See Definition \ref{defn:left-Beck-Chevalley} for what this means, and \cite{GR-I}*{Chapter 10, \S 2.1.6} for the meaning of 2-op.}. Thus, by Theorem \ref{thm:left-and-right-Beck-Chevalley-extension-result} (a) one can extend $\ProIndCoh$ to a functor
\begin{equation}
\label{eq:ProIndCoh-functor-from-correspondences-open-all-open}
\ProIndCoh_{\Corraftoao}: \Corraftoao \ra (\DGct)^{\rm 2-op}.    
\end{equation}

\textit{Step 2:} Let $\ProIndCoh_{\Corraftao}$ denote the restriction of $\ProIndCoh_{\Corraftoao}$ to the $1$-subcategory $\Corraftao$, where the only $2$-morphisms allowed are isomorphisms. Then \cite{GR-I}*{Chapter 7, Theorem 4.1.3} says that the data of (\ref{eq:ProIndCoh-functor-from-correspondences-open-all-open}) is equivalent to the data of 
\begin{equation}
    \label{eq:ProIndCoh-functor-from-correspondences-all-open}
    \ProIndCoh_{\Corraftao}: \Corraftao \ra \DGc.    
\end{equation}

\textit{Step 3:} one applies \cite{GR-I}*{Chapter 7, Theorem 5.2.4} to the functor (\ref{eq:ProIndCoh-functor-from-correspondences-all-open}) to obtain the functor (\ref{eq:ProIndCoh-functor-from-correspondences-proper-all-all}). We notice that the conditions necessary to apply \cite{GR-I}*{Chapter 7, Theorem 5.2.4} are satisfied\footnote{We use the numbering of \cite{GR-I}*{Chapter 7} to refer to the conditions.}:
\begin{itemize}
    \item \textit{(Conditions 5.1.1 and 5.1.2}) these are exactly as in the ind-coherent case, see \cite{GR-I}*{Chapter 7, \S 5.1.3};
    \item (\textit{Condition 5.1.4}) is the factorization condition, i.e.\ Proposition \ref{prop:factorization-category-is-contractible};
    \item (\textit{Condition 5.2.2}) requires that $\ProIndCoh$ satisfies a pullback compatibility between open morphisms and proper maps, this is Proposition \ref{prop:pullback-compatibility}.
\end{itemize}
\end{proof}

\paragraph{$!$-pullback for arbitrary maps}
\label{par:!-pullback-arbitrary-maps-ProIndCoh-schemes}

Consider the inclusion $(\Schaft)^{\rm op} \ra \Corraftpaa$ given by
\[
(S \overset{f}{\ra} T) \in (\Schaft)^{\rm op} \mapsto \begin{tikzcd}
S \ar[r,"\simeq"] \ar[d,"f"] & S \\
T & 
\end{tikzcd} \in \Corraftpaa
\]

Let
\begin{equation}
    \label{eq:proindcoh-shriek-on-schemes}
    \ProIndCoh^!:=  \left.\ProIndCoh_{\Corraftpaa}\right|_{(\Schaft)^{\rm op}}: (\Schaft)^{\rm op} \ra \DGc
\end{equation}
denote the restriction of the functor constructed in Theorem \ref{thm:ProIndCoh-from-correspondences}. One has the following
\begin{itemize}
    \item the restriction of $\ProIndCoh^!$ to $(\Schaft)_{\rm proper}^{\rm op}$\footnote{The subscript means that we only consider proper maps between schemes almost of finite type.} corresponds to the functor defined in \ref{par:!-pullback-for-proper};
    \item the restriction of $\ProIndCoh^!$ to $(\Schaft)_{\rm event-coconn}^{\rm op}$\footnote{The subscript means that we only consider eventually-coconnective maps between schemes almost of finite type.} corresponds to the functor defined in \ref{par:*-pullback-for-eventually-connective}.
\end{itemize}

The main consequence of Theorem \ref{thm:ProIndCoh-from-correspondences} is

\begin{cor}
\label{cor:arbitrary-base-change-ProIndCoh}
For an arbitrary Cartesian diagram in $\Schaft$
\[
\begin{tikzcd}
    S' \ar[r,"g_S"] \ar[d,"f'"'] & S \ar[d,"f"] \\
    T' \ar[r,"g_T"] & T
\end{tikzcd}
\]
one has a canonical isomorphism:
\begin{equation}
    \label{eq:base-change-ProInd-general}
    g^!_T \circ f^{\rm ProInd}_* \simeq (f')^{\rm ProInd}_* \circ g^!_S.
\end{equation}
Moreover, in the case that:
\begin{enumerate}[(a)]
    \item $g_T$ is proper, then the morphism $\la$ in (\ref{eq:base-change-ProInd-general}) arises from the $(g^{\rm ProInd}_*,g^!)$-adjunction from the isomorphism
    \[
    f^{\rm ProInd}_* \circ g^{\rm ProInd}_{S,*} \simeq g^{\rm ProInd}_{T,*} \circ (f')^{\rm ProInd}_{*};
    \]
    \item $f$ is proper, then the morphism $\la$ in (\ref{eq:base-change-ProInd-general}) arises from the $(f^{\rm ProInd}_*,f^!)$-adjunction from the isomorphism
    \[
    (f')^!\circ g^!_T \simeq g^!_S \circ f^!;
    \]
    \item $g_T$ is open, then the morphism $\ra$ in (\ref{eq:base-change-ProInd-general}) arises from the $(g^!,g^{\rm ProInd}_*)$-adjuntion from the isomorphism
    \[
    f^{\rm ProInd}_* \circ g^{\rm ProInd}_{S,*} \simeq g^{\rm ProInd}_{T,*} \circ (f')^{\rm ProInd}_{*};
    \]
    \item $f$ is open, then the morphism $\ra$ in (\ref{eq:base-change-ProInd-general}) arises from the $(f^!,f^{\rm ProInd}_*)$-adjunction from the isomorphism
    \[
    (f')^!\circ g^!_T \simeq g^!_S \circ f^!.
    \]
\end{enumerate}
\end{cor}

\subsubsection{Symmetric monoidal structure and duality on pro-ind-coherent sheaves}

\paragraph{Tensor structure on pro-ind-sheaves}

Let $S_1$ and $S_2$ be two schemes almost of finite type, there is a naturally defined functor
\begin{equation}
    \label{eq:box-tensor-ProIndCoh}
    \boxtimes: \ProIndCoh(S_1)\times \ProIndCoh(S_2) \ra \ProIndCoh(S_1\times S_2).
\end{equation}
Indeed, by \cite{GR-I}*{Chapter 4, Lemma 6.3.2} one has a functor
\[
\boxtimes: \IndCoh(S_1) \times \IndCoh(S_2) \ra \IndCoh(S_1\times S_2)
\]
which Pro-extends to
\[
\boxtimes: \Pro(\IndCoh(S_1) \times \IndCoh(S_2)) \ra \ProIndCoh(S_1\times S_2),
\]
then we notice that
\[
\ProIndCoh(S_1)\times \ProIndCoh(S_2) \simeq \Pro(\IndCoh(S_1) \times \IndCoh(S_2))
\]
since the Pro-object construction commutes with limits.

The following is a direct consequence of \cite{GR-I}*{Chapter 4, Proposition 6.3.4.}.

\begin{prop}
\label{prop:box-tensor-is-fully-faithful}
The functor (\ref{eq:box-tensor-ProIndCoh}) is:
\begin{enumerate}[(i)]
    \item fully faithful;
    \item if the ground field $k$ is perfect, then it is also an equivalence.
\end{enumerate}
\end{prop}

Recall that by \cite{GR-I}*{Chapter 4, Proposition 6.3.7} the functor
\[
\IndCoh: \Schaffaft \ra \DGc
\]
has a unique symmetric monoidal structure. Thus, one obtains
\begin{prop}
\label{prop:ProIndCoh-functor-is-symmetric-monoidal}
The functor
\[
\ProIndCoh: \Schaffaft \ra \DGc
\]
has a unique symmetric monidal structure extending that of $\IndCoh$.
\end{prop}

As in \cite{GR-I}*{Chapter 5, \S 4.1}, one has the following consequence of Proposition \ref{prop:ProIndCoh-functor-is-symmetric-monoidal} for the functor out of the $2$-category of correspondences
\begin{thm}
\label{thm:ProIndCoh-for-correspondences-is-symmetric-monoidal}
The functor
\[
\ProIndCoh_{\Corraaftpaa}: \Corraaftpaa \ra \DGct
\]
carries a symmetric monoidal structure that extends that of $\ProIndCoh$.
\end{thm}

\paragraph{Duality}

As in \cite{GR-I}*{Chapter 5, \S 4.2} we notice that category $\Corraftaa$ has an anti-involution $\omega$ given by sending objects to themselves and $1$-morphisms as
\[
\begin{tikzcd}
U \ar[r,"f"] \ar[d,"g"] & S \\
T & 
\end{tikzcd}
\;\;\; \overset{\omega}{\longmapsto} \;\;\;
\begin{tikzcd}
U \ar[r,"g"] \ar[d,"f"] & T \\
S & 
\end{tikzcd}
\]
Moreover, by \cite{GR-I}*{Chapter 9, Proposition 2.3.4} all objects of $\Corraftaa$ are dualizable, and $\omega$ is the dualization functor. Thus, general considerations about symmetric monoidal categories\footnote{See \cite{GR-I}*{Chapter 5, \S 4.2.1}} give

\begin{thm}
\label{thm:duality-for-ProIndCoh-sheaves}
The restriction of (\ref{eq:ProIndCoh-functor-from-correspondences-proper-all-all}) to the category $\Corraftaa$ fits into the commutative diagram
\[
\begin{tikzcd}
\left(\Corraftaa\right)^{\rm op} \ar[ddd,"\omega"] \ar[rrrrr,"\left(\ProIndCoh_{\Corraftaa}\right)^{\rm op}"] & & & & & \left(\DGcd\right)^{\rm op} \ar[ddd,"\rm dualization"] \\
& & & & & \\
& & & & & \\
\Corraftaa \ar[rrrrr,"\ProIndCoh_{\Corraftaa}"] & & & & & \DGcd
\end{tikzcd}
\]
\end{thm}

Concretely, Theorem \ref{thm:duality-for-ProIndCoh-sheaves} is saying that given $S \in \Schaffaft$ one has an equivalence
\begin{equation}
    \label{eq:dualization-ProIndCoh}
    \bD^{\rm ProInd}_{S}: \ProIndCoh(S) \simeq \ProIndCoh(S)^{\vee},
\end{equation}
and for any morphism $f:S \ra T$ one has an equivalence
\[
(f^!)^{\vee} \simeq f^{\rm ProInd}_{*},
\]
where we abuse notation and denote $(f^!)^{\vee}$ as the composite
\[
\ProIndCoh(S) \overset{(\bD^{\rm ProInd}_{S})^{-1}}{\ra} \ProIndCoh(S)^{\vee} \overset{(f^!)^{\vee}}{\ra} \ProIndCoh(T)^{\vee} \overset{\bD^{\rm ProInd}_{T}}{\ra} \ProIndCoh(T).
\]

\subsection{Tate-coherent sheaves}
\label{subsec:Tate-coherent-sheaves}

In this section we consider the subcategory of Tate-coherent sheaves inside of pro-ind-coherent sheaves and specialize the functors constructed in Section \ref{subsec:Pro-ind-coherent-on-schemes} to Tate objects. 

As we discussed before the $!$-pullback functor restricts well, however the pushforward forward doesn't. We amend the situation by constructing an arbitrary pushforward functor of Tate objects by gluing the usual pushforward for proper maps and an 'exceptional', i.e.\ right adjoint, pushforward for open morphisms.

\subsubsection{Pullback of Tate-coherent sheaves}
\label{subsubsec:pullback-TateCoh-sheaves}

In this subsection we analyze the behaviour of the functor (\ref{eq:proindcoh-shriek-on-schemes}) when restricted to the subcategories of Tate-coherent sheaves.

\paragraph{Proper pushforward preserves Tate objects}
\label{par:proper-!-pullback-preserves-TateCoh-schemes}

We start with the following observation.

\begin{lem}
\label{lem:proper-pushforward-preserve-Tate-Coh}
Suppose that $f:S \ra T$ is a proper map between schemes almost of finite type, then the functor $\left.f^{\rm ProInd}_*\right|_{\TateCoh(S)}$ factors as follows
\[
\begin{tikzcd}
\TateCoh(S) \ar[r,dashed,"f^{\rm Tate}_*"] \ar[rd,"\left.f^{\rm ProInd}_*\right|_{\TateCoh(S)}"'] & \TateCoh(T) \ar[d,hook] \\
& \ProIndCoh(T)
\end{tikzcd}
\]
\end{lem}

\begin{proof}
By the definition of $\TateCoh(S)$ it is enough to check that the restriction of $f^{\rm ProInd}_*$ to $\Coh(S)$ factors through $\Coh(T)$. This is exactly \cite{GR-I}*{Chapter 4, Lemma 5.1.4}.
\end{proof}

\paragraph{}
\label{par:correct-adjunction-TateCoh-proper-maps}

Notice that Lemma \ref{lem:proper-pushforward-preserve-Tate-Coh} also implies that the restriction $\left.f^!\right|_{\TateCoh(T)}$ sends $\TateCoh(T)$ to the subcategory $\TateCoh(S)$. We denote the associated functor by
\[
f^!: \TateCoh(T) \ra \TateCoh(S)
\]
and we notice that $(f^{\rm Tate}_*,f^!)$ forms an adjoint pair for $f$ a proper morphism.

\paragraph{Coconnective pullbacks preserve Tate objects}
\label{par:coconnective-pullback-Tatecoh-affine-schemes}

\begin{lem}
\label{lem:coconnective-pullback-preserve-TateCoh}
Let $f:S \ra T$ be an eventually coconnective morphism, then the functor $\left.f^{\rm ProInd, *}\right|_{\TateCoh(T)}$ sends the category $\TateCoh(T)$ to the subcategory $\TateCoh(S) \subset \ProIndCoh(S)$.

Moreover, the induced functor which we denote by
\[
f^{\rm Tate, *}: \TateCoh(T) \ra \TateCoh(S)
\]
preserves colimits.
\end{lem}

\begin{proof}
By definition of eventually coconnective morphisms (see \S \ref{par:*-pullback-for-eventually-connective}) the restriction $\left.f^{\rm ProInd}\right|_{\Coh(T)}$ sends $\Coh(T)$ to $\Coh(S)$, thus it preserves Tate-objects.

The second claim of the Lemma follows from the fact that $f^{\rm ProInd, *}$ is a left adjoint (see Lemma \ref{lem:left-adjoint-is-Pro-extension}).
\end{proof}

\paragraph{}
\label{par:correct-adjunction-TateCoh-coconnective-maps}

Lemma \ref{lem:coconnective-pullback-preserve-TateCoh} implies the functor $f^{\rm Tate, *}: \TateCoh(T) \ra \TateCoh(S)$ has a right adjoint $f^{\rm Tate}_*$. Thus, when $f$ is a coconnective morphisms one has an adjunction $(f^{\rm Tate, *},f^{\rm Tate}_*)$.

\paragraph{}

As a consequence of \S \ref{par:!-pullback-arbitrary-maps-ProIndCoh-schemes} and Lemmas \ref{lem:proper-pushforward-preserve-Tate-Coh} and \ref{lem:coconnective-pullback-preserve-TateCoh} the functor (\ref{eq:proindcoh-shriek-on-schemes}) restricts to a functor
\begin{equation}
    \label{eq:TateCoh-shriek-on-schemes}
    \TateCoh^!: (\Schaft)^{\rm op} \ra \DGc.
\end{equation}

\subsubsection{Tate-coherent sheaves on the category of correspondences}
\label{subsubsec:TateCoh-on-correspondences-of-schemes}

In goal of this section is to extend the functor (\ref{eq:TateCoh-shriek-on-schemes}) to a functor
\[
\TateCoh_{\Corraftoaa}: \Corraftoaa \ra \DGct.
\]

\paragraph{}

Here is the heuristics of how this functor is constructed.

Suppose that $f:S \ra T$ is an arbitrary map between schemes almost of finite type. By Proposition \ref{prop:factorization-category-is-contractible} one has a factorization of $f$ as
\[
S \overset{j}{\ra} U \overset{g}{\ra} T
\]
where $j$ is an open morphism and $g$ is proper. As a consequence of \S \ref{subsubsec:pullback-TateCoh-sheaves}, the $!$-pullback functors for pro-ind-coherent sheaves restricted to the category of Tate-coherent sheaves factors as follows
\[
\begin{tikzcd}
\ProIndCoh(T) \ar[rr,"g^{!}"] & & \ProIndCoh(U) \ar[rr,"j^{!}"] & & \ProIndCoh(S) \\
\TateCoh(T) \ar[u,hook] \ar[rr,"\left.g^!\right|_{\TateCoh(T)}"',dashed] & & \TateCoh(U) \ar[u,hook] \ar[rr,"\left.j^!\right|_{\TateCoh(U)}"',dashed] & & \TateCoh(S) \ar[u,hook]
\end{tikzcd}
\]

Thus, one defines the functor $f^{\rm Tate}_*: \TateCoh(S) \ra \TateCoh(T)$ as follows
\begin{equation}
    \label{eq:defn-informal-Tate-pushforward}
    f^{\rm Tate}_{*} := g^{\rm Tate}_* \circ j^{\rm Tate}_*,
\end{equation}
where $g^{\rm Tate}_* := \left.g^{\rm ProInd}_*\right|_{\TateCoh(U)}$ and $j^{\rm Tate}_*$ is the \emph{right adjoint} to $j^!$, which exists by Lemma \ref{lem:coconnective-pullback-preserve-TateCoh}. 

The rest of this section is devoted to check that this assignment is well-defined, i.e.\ does not depend on the factorization of $f$. The technical way to achieve this is by defining a functor from the category of correspondences, we do this following the strategy of \cite{GR-I}*{Chapter 5, \S 2} adapted to the case of pushforwards\footnote{In \emph{loc.\ cit.} Gaitsgory and Rozenblyum use this idea to define the $!$-pullback for arbitrary morphisms.}.

\paragraph{Proper base change for Tate-coherent sheaves}

Consider
\begin{equation}
    \label{eq:base-change-proper-TateCoh}
    \begin{tikzcd}
    S' \ar[r,"g_S"] \ar[d,"f'"'] & S \ar[d,"f"] \\
    T' \ar[r,"g_T"] & T
    \end{tikzcd}
\end{equation}
a Cartesian diagram in $\Schaft$ where $f$ is proper, by adjunction the natural isomorphism
\[
g^!_S \circ f^! \overset{\simeq}{\ra} (f')^! \circ g^!_T
\]
gives a $2$-isomorphism
\begin{equation}
    \label{eq:base-change-2-morphism-proper-maps-TateCoh-schemes}
    (f')^{\rm Tate}_{*} \circ g^!_S \ra g^!_T \circ f^{\rm Tate}_{*}
\end{equation}
of functors from $\TateCoh(S)$ to $\TateCoh(T')$.

\begin{lem}
\label{lem:right-Beck-Chevalley-for-proper-morphisms-TateCoh-shriek}
The map (\ref{eq:base-change-2-morphism-proper-maps-TateCoh-schemes}) is an isomorphism.
\end{lem}

\begin{proof}
This follow from Corollary \ref{cor:arbitrary-base-change-ProIndCoh}(b) and \S \ref{par:correct-adjunction-TateCoh-proper-maps}.
\end{proof}

\paragraph{}

Lemma \ref{lem:right-Beck-Chevalley-for-proper-morphisms-TateCoh-shriek} implies that the functor 
\[
\TateCoh^!: (\Schaft)^{\rm op} \ra \DGc
\]
satisfies the right Beck--Chevalley condition with respect to proper morphisms. Thus, Theorem \ref{thm:left-and-right-Beck-Chevalley-extension-result}(b) implies

\begin{prop}
One has a unique extension
\[
\TateCoh^!_{\Corraftppa}: \Corraftppa \ra \DGct
\]
whose restriction 
\[
\TateCoh_{\rm proper} := \left.\TateCoh^!_{\Corraftppa}\right|_{(\Schaff)_{\rm proper}}
\]
is obtained from $\TateCoh^!$ by passing to left adjoints.
\end{prop}

\paragraph{Open base change for Tate-coherent sheaves}
\label{par:base-change-open-morphism-TateCoh}

For $g:S \ra T$ an open morphism, by Lemma \ref{lem:coconnective-pullback-preserve-TateCoh} the functor $g^!: \TateCoh(T) \ra \TateCoh(S)$ admits a right adjoint, which we denote by
\[
g^{\rm Tate}_*: \TateCoh(S) \ra \TateCoh(T).
\]

Let
\begin{equation}
    \label{eq:pullback-diagram-open-base-change-TateCoh}
    \begin{tikzcd}
    S' \ar[r,"f'"] \ar[d,"g_S"'] & T' \ar[d,"g_T"] \\
    S \ar[r,"f"'] & T
    \end{tikzcd}
\end{equation}
be a Cartesian diagram in $\Schaft$ where $g_T$ is open, consider the 2-morphism
\begin{equation}
    \label{eq:base-change-2-morphism-proper-and-open-TateCoh}
    f^! \circ g^{\rm Tate}_{T, *} \ra g^{\rm Tate}_{S,*}\circ (f')^!
\end{equation}
obtained from the $(g^!_S,g^{\rm Tate}_{S,*})$ and $(g^!_T,g^{\rm Tate}_{T,*})$ adjunctions from the isomorphism
\[
(f')^! \circ g^!_T  \overset{\simeq}{\ra} g^!_S \circ f^!. 
\]

\begin{prop}
\label{prop:open-morphism-right-Beck-Chevalley-TateCoh}
The map (\ref{eq:base-change-2-morphism-proper-and-open-TateCoh}) is an isomorphism.
\end{prop}

\begin{proof}
This follows from Corollary \ref{cor:arbitrary-base-change-ProIndCoh}(d) applied to the diagram (\ref{eq:pullback-diagram-open-base-change-TateCoh}) and Lemma \ref{lem:coconnective-pullback-preserve-TateCoh}.
\end{proof}

\paragraph{Open morphisms and Ind-coherent sheaves}

The following result gives a compatibility between the restriction of $f^{\rm Tate}_*$ to ind-coherent sheaves and $f^{\rm Ind}_*$.

\begin{lem}
\label{lem:compatibility-TateCoh-pushforward-with-IndCoh-pushforward-for-open-morphisms}
Let $f:S \ra T$ be an eventually coconnective morphism between schemes almost of finite type. The commutative diagram
\[
\begin{tikzcd}
\TateCoh(S) & \TateCoh(T) \ar[l,"f^!"] \\
\IndCoh(S) \ar[u,"\imath_{S}"'] & \IndCoh(T) \ar[l,"f^!"'] \ar[u,"\imath_S"]
\end{tikzcd}
\]
gives rise to a $2$-morphism
\begin{equation}
\label{eq:comparison-between-right-adjoint-and-inclusions}
\imath_{T} \circ f^{\rm Ind}_* \ra f^{\rm Tate}_*\circ \imath_{S}.   
\end{equation}
If $f$ is an open embedding, then (\ref{eq:comparison-between-right-adjoint-and-inclusions}) is an isomorphism.
\end{lem}

\begin{proof}
The $2$-morphism (\ref{eq:comparison-between-right-adjoint-and-inclusions}) is obtained by applying the $(f^!,f^{\rm Tate}_*)$-adjunction to
\begin{equation}
\label{eq:inclusion-of-counit-proof-of-adjoints-and-inclusions}
\imath_{S} \circ f^!\circ f^{\rm Ind}_* \ra \imath_{S}.    
\end{equation}
To check that (\ref{eq:inclusion-of-counit-proof-of-adjoints-and-inclusions}) is an isomorphism, we notice that $\imath_S$ is fully faithful and that the counit 
\[
f^!\circ f^{\rm Ind}_* \ra \id_{\IndCoh(S)}
\]
is an isomorphism, since $f$ is an open embedding (see \cite{GR-I}*{Chapter 4, Proposition 4.1.2}).
\end{proof}

\paragraph{Compatibility of Tate-coherent pushforwards}

Consider the Cartesian diagram
\[
\begin{tikzcd}
S' \ar[r,"f'"] \ar[d,"g_S"'] & T' \ar[d,"g_T"] \\
S \ar[r,"f"'] & T
\end{tikzcd}
\]
where $g_T$ is open and $f$ is proper. Consider the isomorphism
\[
g^{\rm Tate}_{S,*}\circ (f')^! \overset{\simeq}{\ra} f^!\circ g^{\rm Tate}_{T,*}
\]
given by the base change with respect to the $(g^!,g^{\rm Tate}_*)$-adjunction (see \ref{par:base-change-open-morphism-TateCoh}). By applying the $(f^{\rm Tate}_*,f^!)$-adjunction one obtains a map
\begin{equation}
    \label{eq:2-morphism-compatibiliy-of-pushforward}
    f^{\rm Tate}_*\circ g^{\rm Tate}_{S,*} \ra g^{\rm Tate}_{T,*}\circ (f')^{\rm Tate}_*.
\end{equation}

\begin{prop}
\label{prop:pushforward-TateCoh-compatible}
The morphism (\ref{eq:2-morphism-compatibiliy-of-pushforward}) is an isomorphism.
\end{prop}

Before giving the proof of Proposition \ref{prop:pushforward-TateCoh-compatible} we need the following result.

\begin{prop}
\label{prop:pushforward-open-embedding-of-TateCoh-is-fully-faithful}
For $g: S \ra T$ an open embedding the functor
\[
g^{\rm Tate}_*: \TateCoh(S) \ra \TateCoh(T)
\]
is fully faithful.
\end{prop}

\begin{proof}
Let $\sG \in \TateCoh(S)$, then $\sG \simeq \lim_{I}\sG_i$ where $I$ is cofiltered and $\sG_i \in \IndCoh(S)$, since the functor $g^{\rm Tate}_*$ commutes with limits, it is enough to check that the restriction $\left.g^{\rm Tate}_*\right|_{\IndCoh(S)}$ is fully faithful. Note, that since the restriction of $g^!: \TateCoh(T) \ra \TateCoh(S)$ to $\IndCoh(T)$ is equivalent to $g^!$ from \S \ref{subsubsec:ind-coherent-sheaves-on-schemes-almost-of-finite-type}, by Proposition \ref{prop:IndCoh-adjunctions-for-schemes} and the uniqueness of adjoints one has an equivalence
\[
\left.g^{\rm Tate}_*\right|_{\IndCoh(S)} \simeq g^{\rm Ind}_*.
\]
Thus, the result follows from \cite{GR-I}*{Chapter 4, Proposition 4.1.2}.
\end{proof}

\begin{proof}[Proof of Proposition \ref{prop:pushforward-TateCoh-compatible}]

Given $g_{T}:T' \ra T$ an open embedding, one obtains a fiber sequence of stable $\infty$-categories
\begin{equation}
    \label{eq:exact-sequence-of-localization-for-TateCoh}
    \TateCoh(T')_{Z_{T'}} \ra \TateCoh(T) \overset{g^!_{T}}{\ra} \TateCoh(T')
\end{equation}
where $Z_{T'} := T \backslash T'$ and $\TateCoh(T')_{Z_{T'}}$ denotes the subcategory of $\TateCoh(T)$ of Tate-coherent sheaves supported on $Z_{T'}$. Notice that the functor
\[
(g^!_T, \imath^{\rm Tate}_{Z_{T'},*}\circ \imath^!_{Z_{T'}}): \TateCoh(T) \ra \TateCoh(T') \times \TateCoh(T')_{Z_{T'}}
\]
is conservative. Thus, for any $\sG \in \TateCoh(S')$ it is enouch to check that 
\begin{equation}
    \label{eq:2-morphism-pushforward-on-open}
    g^!_T \circ f^{\rm Tate}_*\circ g^{\rm Tate}_{S,*}(\sG) \ra g^{!}_T\circ g^{\rm Tate}_{T,*}\circ (f')^{\rm Tate}_*(\sG)    
\end{equation}
and
\begin{equation}
    \label{eq:2-morphism-pushforward-on-closed}
    \imath^{\rm Tate}_{Z_{T'}} \circ \imath^!_{Z_{T'}} \circ f^{\rm Tate}_*\circ g^{\rm Tate}_{S,*}(\sG) \ra  \imath^{\rm Tate}_{Z_{T'}} \circ \imath^!_{Z_{T'}} \circ g^{\rm Tate}_{T,*}\circ (f')^{\rm Tate}_*(\sG)   
\end{equation}
are isomorphisms.

We start with (\ref{eq:2-morphism-pushforward-on-open}), By the base change isomorphism (\ref{eq:base-change-2-morphism-proper-maps-TateCoh-schemes}) the map (\ref{eq:2-morphism-pushforward-on-open}) is equivalent to
\begin{equation}
\label{eq:2-morphism-pushforward-on-open-base-changed}
(f')^{\rm Tate}_* \circ g^!_S \circ g^{\rm Tate}_{S,*}(\sG) \ra g^{!}_T\circ g^{\rm Tate}_{T,*}\circ (f')^{\rm Tate}_*(\sG).       
\end{equation}
Now, Proposition \ref{prop:pushforward-open-embedding-of-TateCoh-is-fully-faithful} implies that the canonical map
\[
g^! \circ g^{\rm Tate} \ra \id
\]
is an isomorphism for any open embedding $g$, thus (\ref{eq:2-morphism-pushforward-on-open-base-changed}) is an isomorphism.

For (\ref{eq:2-morphism-pushforward-on-closed}), since the functor $\imath^{\rm Tate}_{Z_{T'},*}$ is conservative, one notices that it is enough to check that
\begin{equation}
\label{eq:2-morphism-pushforward-on-closed-simplified}
    \imath^!_{Z_{T'}} \circ f^{\rm Tate}_*\circ g^{\rm Tate}_{S,*}(\sG) \ra \imath^!_{Z_{T'}} \circ g^{\rm Tate}_{T,*}\circ (f')^{\rm Tate}_*(\sG)
\end{equation}
is an isomorphism. The righthand side of (\ref{eq:2-morphism-pushforward-on-closed-simplified}) vanishes by the exactness of (\ref{eq:exact-sequence-of-localization-for-TateCoh}). Whereas the lefthand side of (\ref{eq:2-morphism-pushforward-on-closed-simplified}) is equivalent to
\[
f^{\rm Tate}_{Z,*} \circ \imath^!_{Z_{S'}} \circ g^{\rm Tate}_{S,*}(\sG),
\]
where the morphisms $f_{Z}$ and $\imath_{Z_{S'}}$ are determined by the Cartesian diagram
\[
\begin{tikzcd}
Z_{S'} \ar[r,"\imath_{Z_{S'}}"] \ar[d,"f_{Z}"] & S \ar[d,"f"] \\
Z_{T'} \ar[r,"\imath_{Z_{T'}}"] & T
\end{tikzcd}
\]

This again vanishes, by a similar exact sequence as (\ref{eq:exact-sequence-of-localization-for-TateCoh}) but for $S$ where we notice that $Z_{S'}$ is the complement of $S'$ in $S$.

This finishes the proof of the proposition.
\end{proof}

\paragraph{Tate-coherent sheaves on the correspondence category}

Consider the subcategory
\[
\Corraftpa \hra \Corraftppa
\]
and let
\[
\TateCoh^!_{\Corraftpa} := \left.\TateCoh^!_{\Corraftppa}\right|_{\Corraftpa}.
\]

\begin{thm}
\label{thm:TateCoh-correspondence-functor-Sch}
There exists a unique extension of the functor $\TateCoh^!_{\Corraftpa}$ to a functor
\begin{equation}
    \label{eq:TateCoh-correspondences-schemes-aft-open-all-all}
    \TateCoh_{\Corraftoaa}: \Corraftoaa \ra \DGct.
\end{equation}
\end{thm}

\begin{proof}
We will deduce this from applying Theorem \ref{thm:extension-of-correspondence-via-factorization} to the category $\Corr(\Schaft)$ with the following classes of morphisms: $horiz = vert = \mbox{ all}$, $adm = \mbox{ open}$ and $co-adm = \mbox{ proper}$.

The conditions on the category (see \ref{subsubsec:conditions-on-category}) are satisfied exactly by the same reasons as in \cite{GR-I}*{\S 2.1.5}. 

To check the conditions on the functor (see \ref{subsubsec:conditions-on-functor}), let
\[
\TateCoh^! := \left.\TateCoh^!_{\Corraftpa}\right|_{(\Schaft)^{\rm op}},
\]
notice this is canonically isomorphic to $\TateCoh^!$ as denoted in (\ref{eq:TateCoh-shriek-on-schemes}) thus we keep the same notation.

First, we claim that $\TateCoh^!$ satisfies the right Beck--Chevalley condition with respect to open morphisms, this is the content of Proposition \ref{prop:open-morphism-right-Beck-Chevalley-TateCoh}.

Finally, the condition from paragraph \ref{par:factorization-compatibility} is Proposition \ref{prop:pushforward-TateCoh-compatible}. 

Thus, Theorem \ref{thm:extension-of-correspondence-via-factorization} implies the existence of the functor (\ref{eq:TateCoh-correspondences-schemes-aft-open-all-all}).
\end{proof}

\paragraph{Proper adjunction}

In this paragraph we explain that one can obtain a Tate-coherent formalism from (\ref{eq:TateCoh-correspondences-schemes-aft-open-all-all}) for the $2$-category of correspondences of schemes almost of finite type whose admissible morphisms are proper maps.

Let
\[
\TateCoh_{\Corraftao} := \left.\TateCoh_{\Corraftoaa}\right|_{\Corraftao}.
\]

Consider furthermore the restriction
\[
\TateCoh_{\star} := \left.\TateCoh_{\Corraftao}\right|_{(\Schaft)_{vert}}.
\]

By Proposition \ref{prop:base-change-eventually-coconnective} the functor $\TateCoh_{\star}$ satisfies the left Beck--Chevalley with respect to proper maps. Thus, we can apply \cite{GR-I}*{Chapter 7, Theorem 5.2.4}\footnote{This is a somewhat dual result to \ref{thm:extension-of-correspondence-via-factorization} where one considers co-admissible morphisms inside horizotal morphisms instead of vertical morphisms and left instead of right Beck--Chevalley condition.} to obtain the classes of maps
\[
horiz = vert = all, \;\; adm = proper, \;\; \mbox{and} \;\; co-adm = open.
\]

\begin{cor}
\label{cor:TateCoh-correspondences-schemes-aft-proper-all-all}
There exists a unique functor
\begin{equation}
\label{eq:TateCoh-correspondences-schemes-aft-proper-all-all}
    \TateCoh_{\Corraftpaa}: \Corraftpaa \ra \DGct    
\end{equation}
with an identification 
\[
\left.\TateCoh_{\Corraftpaa}\right|_{(\Schaft)_{vert}} \simeq \TateCoh_{\star}.
\]
\end{cor}

\subsubsection{Duality and t-structure on Tate-coherent sheaves}

\paragraph{Dualizability}

We recall the following general result about Tate objects (see \cite{Hennion-Tate}*{Lemma 2.10}).

\begin{lem}
\label{lem:duality-of-Tate-categories}
For $\sC$ a stable presentable $\infty$-category with an anti-equivalence
\[
\bD_{\sC}: \sC \overset{\simeq}{\ra} \sC^{\rm op}
\]
there exists a canonical anti-equivalence
\begin{equation}
\label{eq:duality-Tate-of-sC}
\bD_{\Tate(\sC)}: \Tate(\sC) \overset{\simeq}{\ra} \Tate(\sC)^{\rm op}.    
\end{equation}
\end{lem}

\begin{cor}
For $S$ a scheme almost of finite type the restriction of (\ref{eq:dualization-ProIndCoh}) to the category $\TateCoh(S)$ factors through $\Tate(Coh(S))^{\rm op}$, we denote the resulting functor
\begin{equation}
    \label{eq:duality-functor-TateCoh}
    \bD^{\rm Tate}_{S}: \Tate(\Coh(S)) \ra \Tate(\Coh(S))^{\rm op}.
\end{equation}
Moreover, the functor (\ref{eq:duality-functor-TateCoh}) is obtained by extending the usual Serre duality functor
\[
\bD^{\rm Serre}_{S}: \Coh(S) \overset{\simeq}{\ra} \Coh(S)^{\rm op}
\]
via Lemma \ref{lem:duality-of-Tate-categories}.
\end{cor}

\begin{proof}
We notice that by Corollary \ref{cor:duality-and-Tate-construction} one has that the duality functor $(-)^{\vee}$ sends the subcategory $\Tate(\Coh(S)) \subset \ProIndCoh(S)$ to
\[
\Tate(\Coh(S)^{\vee})^{\rm op}.
\]
However, by \cite{GR-I}*{Chapter 5, \S 4.2.10} the restriction of the duality functor to $\Coh(S)$ sends it to $\Coh(S)^{\rm op}$ via Serre duality. This gives that $(-)^{\vee}: \TateCoh(S)$ maps to
\[
\Tate(\Coh(S))^{\rm op}.
\]
The check that this agrees with the abstract definition of Lemma \ref{lem:duality-of-Tate-categories} is given by tracing the definitions (see also \cite{BGHW}*{Proposition 3.1 and Proposition 3.5} for the analogous statement for ordinary exact categories).
\end{proof}

\paragraph{t-structure on $\TateCoh$}
\label{par:t-structure-on-TateCoh}

For any scheme $S$ almost of finite type, we claim that the subcategory
\[
\TateCoh(S)^{\leq 0} := \TateCoh(S) \cap \ProIndCoh(S)^{\leq 0}
\]
and the right orthogonal to $\TateCoh(S)^{\leq -1} := \TateCoh(S)^{\leq 0}[1]$ define a t-structure on $\TateCoh(S)$.

It is enough to check that the truncation functor $\tau^{\leq 0}:\ProIndCoh(S) \ra \ProIndCoh(S)^{\leq 0}$ sends the subcategory $\TateCoh(S)$ to $\TateCoh(S)^{\leq 0}$. Indeed, given a diagram $\{\sF_i\}_I$ representing an object of $\TateCoh(S)$, for any $i \ra j$ one has
\[
\Fib\left(\tau^{\leq 0}\sF_i \ra \tau^{\leq 0}(\sF_j)\right) \in \Coh(S)^{\leq 0}
\]
since $\tau^{\leq 0}$ preserves limits and the inclusion of $\Coh(S) \subset \IndCoh(S)$ is t-exact.

\paragraph{t-exactness of pushforward}

The following result studies the relation between the pushforward of Tate-coherent sheaves and its t-structure.

\begin{prop}
\label{prop:pushforward-TateCoh-schemes-is-left-t-exact}
Let $f:S \ra T$ be a morphim in $\Schaft$, then $f^{\rm Tate}_*$ is left t-exact. Moreover, if $f$ is an affine morphism, then $f^{\rm Tate}_*$ is t-exact.
\end{prop}

\begin{proof}
We notice that we can analyse separately the cases where $f$ is proper or $f$ is open.

Suppose that $f$ is a proper morphism. Then the following diagram commutes
\[
\begin{tikzcd}
\ProIndCoh(S) \ar[r,"f^{\rm ProInd}_*"] & \ProIndCoh(T) \\
\TateCoh(S) \ar[r,"f^{\rm Tate}_*"] \ar[u,"\imath_{S}"] & \TateCoh(T) \ar[u,"\imath_{T}"]
\end{tikzcd}
\]
Let $\sF \in \TateCoh(S)^{\geq 0}$, since the t-structure of $\TateCoh(T)$ is determined by that of $\ProIndCoh(T)$ it is enough to check that 
\[
\imath_{T}(f^{\rm Tate}_*(\sF)) \simeq f^{\rm ProInd}_*(\imath_{S}(\sF)) \in \ProIndCoh(T)^{\geq 0}.
\]
Notice that $\imath_S(\sF) \in \ProIndCoh(S)$ and for a presentation $\imath_S(\sF) \simeq \lim_{A}\sF_{\alpha}$ for $\sF_{\alpha} \in \IndCoh(S)^{\geq 0}$ (which exists by Lemma \ref{lem:t-structure-on-Pro-objects} (a)) one has
\[
f^{\rm ProInd}_*(\imath_{S}(\sF)) \simeq \lim_{A}f^{\rm Ind}_*(\sF_{\alpha}).
\]
Finally, since $f^{\rm Ind}_*$ is left t-exact (\cite{GR-I}*{Chapter 4, Proposition 2.1.2}) this proves the case of $f$ a proper morphism.

In the case that $f$ is an open morphism, since $f^{\rm Tate}_*$ is a right adjoint it commutes with limits, whence by applying Lemma \ref{lem:t-structure-on-Pro-objects} (ii) to the composite functor
\[
\TateCoh(S) \overset{f^{\rm Tate}_*}{\ra} \TateCoh(T) \overset{\imath_S}{\ra} \ProIndCoh(T)
\]
it is enough to check that $\left.f^{\rm Tate}_*\right|_{\IndCoh(S)}$ is left t-exact. However by Lemma \ref{lem:compatibility-TateCoh-pushforward-with-IndCoh-pushforward-for-open-morphisms}, one has
\[
\left.f^{\rm Tate}_*\right|_{\IndCoh(S)} \simeq f^{\rm Ind}_*
\]
which is left t-exact by \cite{GR-I}*{Chapter 4, Proposition 2.1.2}.

The case of $f$ representable follows a similar argument, but with the input that $f^{\rm Ind}_*$ is t-exact whenever $f$ is an affine morphism. Indeed, for an affine morphism $f:S \ra T$ the following diagram, which defines $f^{\rm Ind}_*$,
\[
\begin{tikzcd}
\IndCoh(S) \ar[r,"\Psi_S"] \ar[d,"f^{\rm Ind}_*"] & \QCoh(S) \ar[d,"f_{*}"] \\
\IndCoh(T) \ar[r,"\Psi_T"] & \QCoh(T)
\end{tikzcd}
\]
imply that $f^{\rm Ind}_*$ is t-exact, because $\Psi_S$, $\Psi_T$ are t-exact and $f_*$ is also t-exact, whenever $f$ is an affine morphism.
\end{proof}

\paragraph{The subcategory of admissible Tate-coherent sheaves}
\label{par:defn-of-admissible-Tate-Coh}

For the construction of split square-zero extensions in section \ref{subsubsec:split-square-zero-extensions} one needs to consider the following subcategory $\TateCoha(S)^{\leq 0} \subset \TateCoh(S)^{\leq 0}$ consisting of diagram $\{\sF_i\}_{I}$ of cofiltered of objects in $\IndCoh(S)^{\leq 0}$ such that for any $i \ra j$ the induced map
\[
\H^0(\sF_i) \ra \H^0(\sF_j)
\]
is surjective.

\subsection{Technical digression: Pro construction and presentability}
\label{subsec:technical-presentability}

The goal of this subsection is to deal with the problem that the usual passage from a presentable $\infty$-category $\sC$ to its category of Pro-objects $\Pro(\sC)$ does not produce a presentable category. One could get by without $\Pro(\sC)$ being presentable for certain manipulations, however to extend the functor
\[
\ProIndCoh: \Schaffaft \ra \DGc
\]
to a functor from $\SchaffTateaft$ and use the formalism of \cite{GR-I}*{Chapter 8} to produce the correct pullback functor and base change isomorphims, one needs to have presentable categories.

In this section we present two different ways to get around this problem. The solution of \S \ref{subsubsec:cardinality-bounds} is the best one, however we present the solution of \S \ref{subsubsec:opposite-category} first for the reader that doesn't want to be bothered by cardinality considerations.

\subsubsection{The opposite category}
\label{subsubsec:opposite-category}

The duality for $\IndCoh$ as a functor from the category of correspondences (see \cite{GR-I}*{Chapter 5, Theorem 4.2.5}, or restrict the statement of Theorem \ref{thm:duality-for-ProIndCoh-sheaves} to ind-coherent sheaves) gives the following equivalence
\[
\IndCoh(S)^{\vee} \simeq \IndCoh(S)
\]
for every $S \in \Schaffaft$. One can pass to to compact objects in this equivalence to obtain
\begin{equation}
    \label{eq:Serre-duality-Coh}
    \Coh(S)^{\rm op} \simeq \Coh(S),
\end{equation}
which is the usual Serre duality (Section 4.2.10 from \cite{GR-I} for more details). By passing to Pro-objects on (\ref{eq:Serre-duality-Coh}) one obtains
\begin{equation}
\label{eq:opposite-of-IndCoh-dual-is-ProCoh}
(\IndCoh(S)^{\vee})^{\rm op} \simeq \Pro(\Coh(S)).    
\end{equation}

\paragraph{}

This discussion motivates the following definition
\begin{defn}
\label{defn:opposite-category-to-ProIndCoh}
Let $S \in \Schaft$ one defines
\[
\ProIndCoh^{\rm op}(S) := \Ind(\IndCoh(S)^{\vee}).
\]
\end{defn}

\begin{lem}
For each $S \in \Schaft$ the category $\ProIndCoh^{\rm op}(S)$ is presentable, moreover one has an equivalence
\[
(\ProIndCoh^{\rm op}(S))^{\rm op} \simeq \ProIndCoh(S).
\]
\end{lem}

\begin{proof}
We notice that $\IndCoh(S)^{\vee} \simeq \IndCoh(S)$ and the later satisfies
\[
(\IndCoh(S))^{\rm c} \simeq \Coh(S).
\]
Thus by \cite{HTT}*{Lemma 5.4.2.4} one has that $\Ind(\IndCoh(S)^{\vee})$ is accessible. Finally, since $\Ind(\IndCoh(S)^{\vee})$ admits small colimits, it is also presentable.

The second statement follows from (\ref{eq:opposite-of-IndCoh-dual-is-ProCoh}).
\end{proof}

\begin{construction}
\label{cons:ProIndCoh-opposite-on-Ind-affine}
Given $\sX \in \indSchlaft$, we define
\[
\ProIndCoh^{\rm op}_{\indSchlaft}(\sX) := \LKE_{\Schaffaft \hra \indSchlaft}(\ProIndCoh^{\rm op})(\sX),
\]
and the category of Pro-Ind-coherent sheaves as 
\[
\ProIndCoh_{\indSchlaft}(\sX) := \left(\ProIndCoh^{\rm op}_{\indSchlaft}(\sX)\right)^{\rm op}.
\]
\end{construction}

\paragraph{}

Now that we have this construction flushed out, we make the following notational convention that resolves the issue of non-presentability of pro-ind-coherent sheaves in section \ref{subsec:Tate-coherent-sheaves-on-Tate-schemes}.

\begin{notation}
\label{not:take-opposite}
In the following, we will omit the subscript $\indSchlaft$ in the functors $\ProIndCoh^{\rm op}_{\indSchlaft}$ and $\ProIndCoh_{\indSchlaft}$ for readability. Moreover, in the constructions with $\ProIndCoh$ functors that follow we will technicality be performing these construction for $\ProIndCoh^{\rm op}$ and passing to opposite categories at the end of the statement, thus circumventing the problem with the non-presentability of $\ProIndCoh(S)$.
\end{notation}

\subsubsection{Cardinality bounds}
\label{subsubsec:cardinality-bounds}

A possibly more elegant way to deal with the presentability problem is to be more careful with the set-theoretical subtleties of our construction. 

\paragraph{}

Following the treatment of \cite{BGW-Tate} and \cite{thesis}, we consider $\kappa$ a regular cardinal and fix $\lambda > \kappa$ a strongly inaccessible cardinal and fix a universe $\sU(\lambda)$. We will refer to an $\infty$-category $\sC$ as \emph{essentially $\lambda$-small}, if it satisfies
\begin{itemize}
    \item the set of isomorphism classes of objects in $\h\sC$ has cardinality $<\lambda$;
    \item for every morphism $f:X \ra Y$ in $\sC$ and ever $i \geq 0$, the homotopy set
    \[
    \pi_i(\Hom^{\rm R}(X,Y),f)
    \]
    has cardinality $< \lambda$.
\end{itemize}

\paragraph{}

The following is an extension of the results in \cite{thesis}*{\S 2.2.1}.

\begin{prop}
Given an essentially small category $\sC$ one has:
\begin{enumerate}[(1)]
    \item the category $\Ind_{\kappa}(\sC)$ obtained from $\sC$ by formally adjoining all $\kappa$-small filtered colimits is essentially small and $\lambda$-presentable\footnote{See Definition \ref{defn:l-presentable-category} for what we mean by this precisely.};
    \item the category $\Pro_{\kappa}(\sC)$ obtained from $\sC$ by formally adjoining all $\kappa$-small cofiltered limits is essentially small and $\lambda$-presentable.
\end{enumerate}
\end{prop}

\begin{proof}
(1) is standard (cf. \cite{HTT}*{Theorem 5.5.1.1 (4)}), for (2) see Lemma \ref{lem:Pro-kappa-of-essentially-lambda-small-is-essentially-small} and Corollary \ref{cor:Pro-k-sC-is-l-presentable} in the Appendix.
\end{proof}

\begin{notation}
\label{not:fix-cardinality}
In the discussion that follows $\DGc$ will denote the stable $\infty$-category of $\Vect$-modules which are $\lambda$-presentable. All the Pro and Ind constructions considered will tacitly be with respect to a certain uncountable regular cardinal $\kappa < \lambda$, however we will omit $\kappa$ from the notation for readability.
\end{notation}

\subsection{Pro-ind-coherent sheaves on Tate schemes}
\label{subsec:Pro-ind-coherent-on-Tate-schemes}

The goal of this section is to construct a functor
\[
\ProIndCoh_{\CorrTatelaftipaa}: \CorrTatelaftipaa \ra \DGct.
\]

We follow the strategy of \cite{GR-II}*{Chapter 3, \S 5} and use \cite{GR-I}*{Chapter 8, Theorem 1.1.9} to extend the functor
\[
\ProIndCoh_{\Corraftpaa}: \Corraftpaa \ra \DGct,
\]
from the $2$-category of correspondences of schemes almost of finite type to the $2$-category of Tate schemes locally almost of finite type. In the following paragraphs we will introduce the necessary notation and define certain functors that we need to be able to apply the extension theorem.

\subsubsection{Definition and colimit presentation}

In this section we extend the category of Pro-ind-coherent sheaves to any Tate scheme. Consider the functor
\[
\ProIndCoh^!_{\SchTatelaft} := \RKE_{(\Schaft)^{\rm op} \hra (\SchTatelaft)^{\rm op}}(\ProIndCoh^!) : (\SchTatelaft)^{\rm op} \ra \DGc.
\]

\paragraph{}

\begin{lem}
\label{lem:ProIndCoh-on-Tate-schemes-has-left-adjoint-for-closed-inclusion}
Let $\imath:T_0 \ra S$ be a closed embedding of $S \in \Schaft$ into a Tate scheme locally almost of finite type $S$. Then the canonical functor
\[
\imath^!: \ProIndCoh(S) \ra \ProIndCoh(T_0)
\]
has a left adjoint $\imath^{\rm ProInd}_*$.
\end{lem}

\begin{proof}
Let $S \simeq \colim_I S_i$ be a prestantion of $S$, by definition one has
\[
\ProIndCoh(S) \simeq \lim_{I^{\rm op}}\ProIndCoh(S_i).
\]
We now apply Lemma \ref{lem:adjoints-of-limits-of-categories} to the situation where $\sD \simeq \lim_{I^{\rm op}}\ProIndCoh(T_0)$ is the limit of the constant diagram.
\end{proof}

\paragraph{}

The following is an analogue of Proposition \ref{prop:limit-and-colimit-presentation-of-IndCoh-on-Tate-schemes} for pro-ind-coherent sheaves.

\begin{prop}
\label{prop:limit-and-colimit-presentation-of-ProIndCoh-on-Tate-schemes}
Let $S \in \SchTatelaft$ and consider a presentation $T \simeq \colim_I S_i$ where each $S_i \in \Schaft$, i.e.\ a scheme almost of finite type, one has an equivalence
\[
\ProIndCoh^!_{\SchTatelaft}(T) = \lim_{I^{\rm op}}\ProIndCoh^!(T_i) \simeq \colim_I\ProIndCoh^!(T_i),
\]
where the limit is taken with respect to the maps $f^!_{i,j}:\ProIndCoh(T_j) \ra \ProIndCoh(T_i)$ and the colimit with respect to $(f^{\rm Ind}_{i,j})_*:\ProIndCoh(T_i) \ra \ProIndCoh(T_j)$.
\end{prop}

\begin{proof}
The proof is the same as Proposition \ref{prop:limit-and-colimit-presentation-of-IndCoh-on-Tate-schemes}, we again just notice the importance of \S \ref{subsec:technical-presentability} in defining $\ProIndCoh$ so that the resulting categories are presentable.
\end{proof}

\subsubsection{Pro-ind-coherent sheaves for correspondences of Tate schemes}
\label{subsubsec:ProIndCoh-on-correspondences-of-Tate-schemes}

\paragraph{The functors already defined}

We recall that one has
\[
\ProIndCoh \simeq \left.\ProIndCoh_{\Corraftpaa}\right|_{\left(\Schaft\right)_{vert}} \]
and
\[
\ProIndCoh^! \simeq \left.\ProIndCoh_{\Corraftpaa}\right|_{\left(\Schaft\right)^{\rm op}_{horiz}}.
\]
We also denote by
\[
\ProIndCoh_{\rm proper} := \left.\ProIndCoh\right|_{\left(\Schaft\right)_{\rm proper}} \;\;\; \mbox{and} \;\;\;
\ProIndCoh^!_{\rm proper} := \left.\ProIndCoh^!\right|_{\left(\Schaft\right)^{\rm op}_{\rm proper}}
\]
their restriction to the admissible maps in $\Corraftpaa$.

\paragraph{}

Let $(\SchTatelaft)_{\rm ind-proper}$ denote the subcategory of $\SchTatelaft$ where one considers only ind-proper morphisms between Tate schemes. We will denote by
\[
\ProIndCoh^!_{(\SchTatelaft)_{\rm ind-proper}} := \left.\ProIndCoh^!_{\SchTatelaft}\right|_{(\SchTatelaft)_{\rm ind-proper}}.
\]

The following is a formal result
\begin{prop}[\cite{GR-II}*{Chapter 3, Corollary 1.3.5}]
The naturally defined natural transformation
\[
\ProIndCoh^!_{(\SchTatelaft)_{\rm ind-proper}} \ra \RKE_{((\Schaft)_{\rm proper})^{\rm op} \hra ((\SchTatelaft)_{\rm ind-proper})^{\rm op}}\left(\ProIndCoh^!_{(\Schaft)_{\rm proper}}\right)
\]
is an isomorphism.
\end{prop}

\paragraph{Correspondence category of Tate schemes}
\label{par:correspondence-category-Tate-schemes}

Recall that the category $\SchTateaft$ has all fiber products by Lemma \ref{lem:Tate-affine-schemes-has-fiber-products} and Remark \ref{rem:Tate-schemes-aft-stable-under-fiber-products}. Thus, we can define the $2$-category of correspondences
\[
\CorrTatelaftipaa
\]
where we take all morphisms for vertical and horizontal maps, and ind-proper maps (recall Definition \ref{defn:ind-proper-map}) for admissible maps. 

\paragraph{}

Proposition \ref{prop:prestack-laft-is-colimit-of-aft-schemes} gives a canonical inclusion functor 
\[
\imath: \Schaft \ra \SchTatelaft,
\]
by sending a scheme almost of finite type to the constant Ind-diagram. This inclusion preserves the classes of vertical, horizontal and admissible maps. Indeed, by Remark \ref{rem:proper-is-ind-proper} any proper map between schemes is ind-proper. We also notice that $\imath$ preserves Cartesian squares. Thus, we obtain a functor
\[
\Corr(\imath): \Corraftpaa \ra \CorrTatelaftipaa.
\]

\paragraph{}

Notice that, by Proposition \ref{prop:prestack-laft-is-colimit-of-aft-schemes}, for every $S \in \SchTatelaft$ there exists $T \in \Schaft$ and a map
\[
f:T \ra S
\]
such that $f$ is admissible, i.e.\ proper.

\paragraph{Conditions on the category $\DGct$}
\label{par:conditions-on-the-2-category-of-DGCats}

We claim that the $2$-category $\DGct$ satisfies all the conditions from \cite{GR-I}*{Chapter 8, \S 1.1.5}. Indeed, most of the items are standard, we simply highlight that (1) is satisfied, even though we are considering $\infty$-categories of Pro-objects by the discussion of \S \ref{subsec:technical-presentability} and \cite{HTT}*{Proposition 5.5.3.8}.

\paragraph{Extending the functors to Tate affine schemes}

Consider the functors
\[
\ProIndCoh_{\SchTatelaft} := \LKE_{\Schaft \hra \SchTatelaft}\left(\ProIndCoh\right),
\]
and
\[
\ProIndCoh^!_{\SchTatelaft} := \RKE_{(\Schaft)^{\rm op} \hra (\SchTatelaft)^{\rm op}}\left(\ProIndCoh^!\right).
\]

We also let
\[
\ProIndCoh_{(\SchTatelaft)_{\rm ind-proper}} := \LKE_{(\Schaft)_{\rm proper} \hra (\SchTatelaft)_{\rm ind-proper}}\left(\ProIndCoh_{\rm proper}\right)
\]
and
\[
\ProIndCoh^!_{(\SchTatelaft)_{\rm ind-proper}} := \RKE_{(\Schaft)_{\rm proper}^{\rm op} \hra (\SchTatelaft)_{\rm ind-proper}^{\rm op}}\left(\ProIndCoh^!_{\rm proper}\right).
\]

\paragraph{}

The following is a technical condition we need to apply the extension theorem.

\begin{lem}
\label{lem:technical-condition-extension-ProIndCoh-colimit-of-proper}
For every morphism $S' \ra S$ in $\SchTateaft$ the map
\begin{equation}
    \label{eq:technical-condition-extension-Tate-affine}
    \colim_{T \in (\Schaft)_{\rm proper \; in \; S}}\ProIndCoh_{(\SchTateaft)_{\rm proper}}(S'\times_{S}T) \ra \ProIndCoh_{(\SchTateaft)_{\rm proper}}(S')
\end{equation}
is an equivalence.
\end{lem}

\begin{proof}
By applying Proposition \ref{prop:limit-and-colimit-presentation-of-ProIndCoh-on-Tate-schemes} to the functor $\ProIndCoh_{\rm proper}$ we are reduced to check that
\[
\ProIndCoh_{(\SchTateaft)_{\rm proper}}(S') \ra \lim_{T \in (\Schaft)_{\rm proper \; in \; S}} \ProIndCoh_{(\SchTateaft)_{\rm proper}}(S'\times_{S}T)
\]
is an isomorphism. Since the construction of Pro-object commutes with limits, this follows from \cite{GR-II}*{Chapter 3, Corollary 2.1.9}.
\end{proof}

\paragraph{Conditions for extension}
\label{par:conditions-for-extension}

To apply \cite{GR-I}*{Chapter 8, Theorem 1.1.9} we need to check the conditions:

\begin{enumerate}[(1)]
    \item The functor $\ProIndCoh_{\SchTatelaft}$ satisfies the left Beck--Chevalley condition with respect to proper morphisms.
    \item The canonical map
    \[
    \ProIndCoh_{(\SchTatelaft)_{\rm ind-proper}} \ra \left.\ProIndCoh_{\SchTateaft}\right|_{(\SchTatelaft)_{\rm ind-proper}}
    \]
    is an equivalence.
    \item The functor $\ProIndCoh^!_{\SchTatelaft}$ satisfies the right Beck--Chevalley condition with respect to proper morphisms.
    \item The canonical map
    \[
    \left.\ProIndCoh^!_{\SchTatelaft}\right|_{(\SchTatelaft)_{\rm ind-proper}} \ra \ProIndCoh^!_{(\SchTatelaft)_{\rm ind-proper}}
    \]
    is an equivalence.
\end{enumerate}

\paragraph{}

\begin{thm}
There exists a uniquely defined functor
\[
\ProIndCoh_{\CorrTatelaftipaa}: \CorrTatelaftipaa \ra \DGct
\]
equipped with identifications
\[
\ProIndCoh_{\Corraftpaa} \simeq \ProIndCoh_{\CorrTatelaftipaa} \circ \Corr(\imath),
\]
such that
\[
\LKE_{(\Schaft)_{\rm proper}\hra (\SchTateaft)_{\rm ind-proper}}(\ProIndCoh_{\rm proper}) \simeq \left.\ProIndCoh_{\CorrTatelaftipaa}\right|_{(\SchTateaft)_{\rm ind-proper}}.
\]
Moreover, the functor $\ProIndCoh_{\CorrTatelaftipaa}$ has the following properties:
\begin{itemize}
    \item the natural transformation
    \[
    \LKE_{F_{vert}}(\ProIndCoh) \ra \left.\ProIndCoh_{\CorrTatelaftipaa}\right|_{(\SchTatelaft)_{vert}}
    \]
    is an isomorphism;
    \item the natural transformation
    \[
    \left.\ProIndCoh_{\CorrTatelaftipaa}\right|_{(\SchTatelaft)^{\rm op}_{horiz}} \ra \RKE_{F^{\rm op}_{horiz}}(\ProIndCoh^!)
    \]
    is also an isomorphism.
\end{itemize}
\end{thm}

\begin{proof}
This is a consequence of \cite{GR-I}*{Chapter 8, Theorem 1.1.9}. We just check the necessary conditions spelled out before the statement of the Theorem. 
\begin{itemize}
    \item Condition (1) is Proposition \ref{prop:base-change-ind-proper-ProIndCoh};
    \item Condition (2) is Proposition \ref{prop:ProIndCoh-from-proper-agrees-with-restriction};
    \item Condition (3) is Proposition \ref{prop:base-change-ind-proper-ProIndCoh^!};
    \item Condition (4) is Proposition \ref{prop:ProIndCoh^!-from-proper-agrees-with-restriction}.
\end{itemize}

Finally, as in \cite{GR-II}*{Chapter 3, \S 5.1} there is a technical condition (*), which in our case is given by Lemma \ref{lem:technical-condition-extension-ProIndCoh-colimit-of-proper}.
\end{proof}

\subsubsection{t-structure on $\ProIndCoh$ on Tate schemes}

This section defines a t-structure on pro-ind-coherent sheaves on Tate schemes.

\paragraph{}

Let $S$ be a Tate scheme almost of finite type, in this section we extend the t-structure on $\IndCoh(S)$ (see \cite{GR-II}*{Chapter 3, \S 1.2}) to a t-structure on $\ProIndCoh(S)$.

\begin{prop}
\label{prop:t-structure-ProIndCoh-on-Tate-affine-schemes}
Let $S \in \SchaffTateaft$ then one defines a subcategory
\[
\ProIndCoh(S)^{\geq 0} := \left\{\sF \in \ProIndCoh(S) \;| \; \imath^!(\sF) \in \ProIndCoh(S_0), \;\; \mbox{for any closed embedding}\;\;\imath:S_0 \hra S\right\}.
\]
The subcategory $\ProIndCoh(S)^{\geq 0}$ defines a t-structure on $\ProIndCoh(S)$.
\end{prop}

\begin{proof}
The proof is completely analogous to the analysis in \cite{GR-II}*{Chapter 3, \S 1.2}, so we omit it.
\end{proof}

\paragraph{}

\begin{lem}
\label{lem:t-structure-on-ProIndCoh-Tate-schemes-is-Pro-extension}
For $S$ a Tate scheme locally almost of finite type, the connective part of the t-structure on $\ProIndCoh(S)$ can be described equivalently as
\[
\ProIndCoh(S)^{\geq 0} \simeq \Pro(\IndCoh(S)^{\geq 0}).
\]
\end{lem}

\begin{proof}
Suppose that $\sF \in \Pro(\IndCoh(S)^{\geq 0})$ that implies that one can write
\[
\sF \simeq \lim_J \sF_j
\]
for a cofiltered set $J$ and $\sF_j \in \IndCoh(S)^{\geq 0}$. Now, recall that by definition $\sF_j \in \IndCoh(S)^{\geq 0}$ if for any closed subscheme $\imath:S_0 \hra S$ one has
\[
\imath^!(\sF_j) \in \IndCoh(S_0)^{\geq 0}.
\]

Since $\imath^!$ commutes with limits one has an equivalence
\[
\imath^!(\sF) \simeq \lim_J \imath^{!}(\sF_j)
\]
which proves that $\sF$ belongs to $\ProIndCoh(S)^{\geq 0}$ as defined in Proposition \ref{prop:t-structure-ProIndCoh-on-Tate-affine-schemes}.

The converse inclusion is proved in exactly the same way.
\end{proof}

\paragraph{}

We notice that as in \cite{GR-II}*{Chapter 3, Lemma 1.2.5} for any closed embedding
\[
\imath:S_0 \ra S
\]
from $S_0 \in \Schaff$ into $S \in \SchaffTateaft$ the functor
\[
\imath^!: \ProIndCoh(S) \ra \ProIndCoh(S_0)
\]
is t-exact. In particular, one has the following analogue of Proposition \ref{prop:limit-and-colimit-presentation-of-ProIndCoh-on-Tate-schemes}

\begin{cor}
\label{cor:limit-and-colimit-presentation-of-connective-ProIndCoh-on-Tate-affine-schemes}
Let $S \in \SchaffTateaft$ and consider a presentation $T \simeq \colim_I S_i$ where each $S_i \in \Schaffaft$, i.e.\ an affine scheme locally of finite type, one has an equivalence
\[
\ProIndCoh(T)^{\geq 0} = \lim_{I^{\rm op}}\ProIndCoh(T_i)^{\geq 0} \simeq \colim_I\ProIndCoh(T_i),
\]
where the limit is taken with respect to the maps $f^!_{i,j}$ and the colimit with respect to $(f^{\rm Ind}_{i,j})_*$.
\end{cor}

\subsection{Tate-coherent sheaves on Tate schemes}
\label{subsec:Tate-coherent-sheaves-on-Tate-schemes}

The goal of this section is to extend the functor (\ref{eq:TateCoh-correspondences-schemes-aft-proper-all-all}) to a functor
\[
\TateCoh_{\CorrTatelaftipaa}: \CorrTatelaftipaa \ra \DGct.
\]

\subsubsection{Definition and ind-proper maps}
\label{subsubsec:defn-and-ind-proper-maps}

\paragraph{Coherent sheaves on Tate schemes}

We recall the following characterization of coherent sheaves on Tate schemes.

\begin{lem}[\cite{GR-II}*{Chapter 3, Corollary 1.1.8}]
Let $S$ be a Tate scheme locally almost of finite type and $\sF \in \IndCoh(S)$, the following are equivalent
\begin{enumerate}[(i)]
    \item $\sF$ is a compact object of $\IndCoh(S)$;
    \item there exists a closed embedding $\imath:S_0 \hra S$ with $S_0$ a scheme almost of finite type, such that
    \[
    \sF \simeq \imath_*(\sF_0)
    \]
    for some $\sF_0 \in \Coh(S_0)$.
\end{enumerate}
\end{lem}

\paragraph{Tate condition on Tate schemes}

Let $S$ be a Tate scheme locally almost of finite type, one can define the subcategory $\TateCoh(S)$ of the category $\ProIndCoh(S)$ as those pro-diagrams of sheaves $\{\sF_i\}_I$ with $\sF_i \in \IndCoh(S)$ such that for every morphism $i \ra j$ in $I$ the sheaf
\[
\Fib(\sF_i \ra \sF_j) \in \Coh(S),
\]
i.e.\ there exists a closed subscheme $\imath:S_0 \hra S$ such that
\[
\Fib(\sF_i \ra \sF_j) \simeq \imath^{\rm Ind}_{*}(\sG_0)
\]
for some $\sG_0 \in \Coh(S_0)$.

\begin{lem}
\label{lem:TateCoh-on-Tate-schemes-is-pushforward-from-closed-schemes}
For $S \in \SchTatelaft$, let $\sF \in \ProIndCoh(S)$ and consider $S \simeq \colim_I S_i$ a presentation of $S$, with $S_i \in \Schaft$ and denote the canonical inclusions by $\imath_i: S_i \hra S$. The the following are equivalent:
\begin{enumerate}[(i)]
    \item there exists $i \in I$ such that
    \[
    \sF \simeq \imath^{\rm ProInd}_{i,*}(\sG_i)
    \]
    for some $\sG_i \in \TateCoh(S_i)$;
    \item $\sF \in \TateCoh(S)$.
\end{enumerate}
\end{lem}

\begin{proof}
Assume (i), let $\sG_i \simeq \lim_{K_i}\sG_{i,k}$ for a cofiltered diagram $\{\sG_{i,k}\}$ in $\IndCoh(S_i)$. Notice one has
\[
\imath^{\rm ProInd}_{i,*}(\sG_i) \simeq \lim_{K_i}\imath^{\rm Ind}_*(\sG_{i,k}).
\]
Moreover, since $\IndCoh(S_i)$ is a stable $\infty$-category for any $k \ra k'$ in $K_i$ one has
\[
\Fib(\sG_{i,k} \ra \sG_{i,k'}) \simeq \Cofib(\sG_{i,k} \ra \sG_{i,k'})[1].
\]

Thus, one has
\[
\imath^{\rm Ind}_*\Fib(\sG_{i,k} \ra \sG_{i,k'}) \simeq \Fib(\imath^{\rm Ind}_*(\sG_{i,k}) \ra \imath^{\rm Ind}_*(\sG_{i,k'})),
\]
that is $\sF \in \TateCoh(S)$.

Assume (ii), and let $\sF \simeq \lim_J\sF_j$ where $\{\sF_{j}\}_{J}$ is a cofiltered diagram in $\IndCoh(S)$. By Proposition \ref{prop:limit-and-colimit-presentation-of-ProIndCoh-on-Tate-schemes} one has
\[
\sF \simeq \imath^{\rm ProInd}_{i,*}(\sG_i)
\]
for some $\sG_i \in \ProIndCoh(S_i)$. Again the same calculation as in the prove that (i) implies (ii) gives that $\sG_i \in \TateCoh(S_i)$.
\end{proof}

\paragraph{}

The following is a formal consequence of Lemma \ref{lem:proper-pushforward-preserve-Tate-Coh} and Definition \ref{defn:ind-proper-map}.

\begin{prop}
\label{prop:ind-proper-preserve-TateCoh-on-Tate-schemes}
For any ind-proper map $f:S \ra T$ Tate schemes almost of finite type, the functor $f^{\rm ProInd}_*: \ProIndCoh(S) \ra \ProIndCoh(T)$ factors as
\[
f^{\rm Tate}_*: \TateCoh(T) \ra \TateCoh(S)
\]
and it has a right adjoint $f^!$.
\end{prop}

Thus, by considering the subcategory of Tate-coherent schemes one obtains a functor
\[
\TateCoh^!_{\SchTatelaft}: (\SchTatelaft)^{\rm op} \ra \DGc.
\]

\paragraph{}

The following is a consequence of Lemma \ref{lem:TateCoh-on-Tate-schemes-is-pushforward-from-closed-schemes}.

\begin{cor}
The canonical map
\[
\TateCoh^!_{\SchTatelaft} \ra \RKE_{(\Schaft)^{\rm op}\hra (\SchTatelaft)^{\rm op}}\TateCoh^!
\]
is an isomorphism. Moreover, given a presentation $T \simeq \colim_I S_i$ where each $S_i \in \Schaft$, one has equivalences
\[
\TateCoh^!_{\SchTatelaft}(T) = \lim_{I^{\rm op}}\TateCoh^!(T_i) \simeq \colim_I\TateCoh^!(T_i).
\]
\end{cor}

\subsubsection{Tate-coherent sheaves for correspondences of Tate schemes}

The goal of this section is to extend the functor 
\[
\TateCoh_{\Corraftpaa}: \Corraftpaa \ra \DGct    
\]
to a functor
\[
\TateCoh_{\CorrTatelaftipaa}: \CorrTatelaftipaa \ra \DGct.
\]

Notice that whereas the functor $\TateCoh_{\Corraftpaa}$ is determined by the functor (\ref{eq:TateCoh-correspondences-schemes-aft-open-all-all}), we can't apply extension procedure to the later functor, essentially because the category $\TateCoh(S)$ on a Tate scheme locally almost of finite type $S$ is described in terms of the corresponding categories $\TateCoh(S_i)$ where $S_i \hra S$ are closed embeddings where $S_i$ is a scheme almost of finite type, whereas the functor (\ref{eq:TateCoh-correspondences-schemes-aft-open-all-all}) has explicit base change with respect to open embeddings.

\paragraph{}

We proceed exactly as in \S \ref{subsubsec:ProIndCoh-on-correspondences-of-Tate-schemes}. The following functors are already defined
\[
\TateCoh \simeq \left.\TateCoh_{\Corraftpaa}\right|_{\left(\Schaft\right)_{vert}} \]
and
\[
\TateCoh^! \simeq \left.\TateCoh_{\Corraftpaa}\right|_{\left(\Schaft\right)^{\rm op}_{horiz}}.
\]
We also denote by
\[
\TateCoh_{\rm proper} := \left.\TateCoh\right|_{\left(\Schaft\right)_{\rm proper}} \;\;\; \mbox{and} \;\;\;
\TateCoh^!_{\rm proper} := \left.\TateCoh^!\right|_{\left(\Schaft\right)^{\rm op}_{\rm proper}}
\]
their restriction to the admissible maps in $\Corraftpaa$.

\paragraph{Extending the functors to Tate affine schemes}
\label{par:extension-of-TateCoh-functors}

Consider the functors
\[
\TateCoh_{\SchTatelaft} := \LKE_{\Schaft \hra \SchTatelaft}\left(\TateCoh\right),
\]
and
\[
\TateCoh^!_{\SchTatelaft} := \RKE_{(\Schaft)^{\rm op} \hra (\SchTatelaft)^{\rm op}}\left(\TateCoh^!\right).
\]

We also let
\[
\TateCoh_{(\SchTatelaft)_{\rm ind-proper}} := \LKE_{(\Schaft)_{\rm proper} \hra (\SchTatelaft)_{\rm ind-proper}}\left(\TateCoh_{\rm proper}\right)
\]
and
\[
\TateCoh^!_{(\SchTatelaft)_{\rm ind-proper}} := \RKE_{(\Schaft)_{\rm proper}^{\rm op} \hra (\SchTatelaft)_{\rm ind-proper}^{\rm op}}\left(\TateCoh^!_{\rm proper}\right).
\]

\paragraph{}

\begin{thm}
\label{thm:TateCoh-correspondence-functor-Sch-Tate}
There exists a uniquely defined functor
\begin{equation}
\label{eq:TateCoh-correspondences-Tate-schemes-aft-ind-proper}
    \TateCoh_{\CorrTatelaftipaa}: \CorrTatelaftipaa \ra \DGct    
\end{equation}
equipped with identifications
\[
\TateCoh_{\Corraftpaa} \simeq \TateCoh_{\CorrTatelaftipaa} \circ \Corr(\imath),
\]
such that
\[
\LKE_{(\Schaft)_{\rm proper}\hra (\SchTatelaft)_{\rm ind-proper}}(\TateCoh_{\rm proper}) \simeq \left.\TateCoh_{\CorrTatelaftipaa}\right|_{(\SchTatelaft)_{\rm ind-proper}}.
\]
Moreover, the functor $\ProIndCoh_{\CorrTatelaftipaa}$ has the following properties:
\begin{itemize}
    \item the natural transformation
    \[
    \LKE_{F_{vert}}(\TateCoh) \ra \left.\TateCoh_{\CorrTatelaftipaa}\right|_{(\SchTatelaft)_{vert}}
    \]
    is an isomorphism;
    \item the natural transformation
    \[
    \left.\TateCoh_{\CorrTatelaftipaa}\right|_{(\SchTatelaft)^{\rm op}_{horiz}} \ra \RKE_{F^{\rm op}_{horiz}}(\TateCoh^!)
    \]
    is also an isomorphism.
\end{itemize}
\end{thm}

\begin{proof}
We will apply \cite{GR-I}*{Chapter 8, Theorem 1.1.9} again. The source and target category are the same as in section \ref{subsubsec:ProIndCoh-on-correspondences-of-Tate-schemes}, so by \S \ref{par:correspondence-category-Tate-schemes} and \ref{par:conditions-on-the-2-category-of-DGCats} they satisfy the necessary conditions to apply the theorem. Finally we need to check conditions (1-4) as listed in \S \ref{par:conditions-for-extension}, but for the Tate-coherent functors in place of the pro-ind-coherent functors.

We just check the necessary conditions spelled out before the statement of the Theorem. 
\begin{itemize}
    \item Condition (1) is Corollary \ref{cor:base-change-ind-proper-TateCoh};
    \item Condition (3) is Corollary \ref{cor:base-change-ind-proper-TateCoh^!};
    \item Conditions (2) and (4) are Corollary \ref{cor:TateCoh-from-proper-agrees-with-restriction}.
\end{itemize}

Finally, the technical condition (*) of \cite{GR-I}*{Chapter 8, \S 1.1.7}, in our case follows from Lemma \ref{lem:technical-condition-extension-ProIndCoh-colimit-of-proper}.
\end{proof}

\subsubsection{t-structure on Tate-coherent sheaves on Tate schemes}

This section defines a t-structure on Tate-coherent sheaves on Tate schemes and also considers a subcategory of its connective objects that will be useful for the study of square-zero extensions of Tate affine schemes.

\begin{prop}
\label{prop:t-structure-TateCoh-on-Tate-schemes}
Let $S \in \SchaffTateaft$ then one defines a subcategory
\[
\TateCoh(S)^{\geq 0} := \left\{\sF \in \TateCoh(S) \;| \; \imath^!(\sF) \in \TateCoh(S_0)^{\geq 0}, \;\; \mbox{for any closed embedding}\;\;\imath:S_0 \hra S\right\}.
\]
The subcategory $\TateCoh(S)^{\geq 0}$ defines a t-structure on $\TateCoh(S)$.
\end{prop}

\begin{proof}
This is a consequence of Proposition \ref{prop:t-structure-ProIndCoh-on-Tate-affine-schemes} and \S \ref{par:t-structure-on-TateCoh}.
\end{proof}

\begin{lem}
\label{lem:t-structure-on-TateCoh-of-Tate-scheme}
For $S \simeq \colim_I S_i$, where $S_i \in \Schaft$ and $\imath_i:S_i \hra S$ one has the following:
\begin{enumerate}[(a)]
    \item An object $\sF \in \TateCoh(S)^{\geq 0}$ if and only if for every $i \in I$ one has $\imath^!_i(\sF) \in \TateCoh(S_i)^{\geq 0}$;
    \item the category $\TateCoh(S)^{\leq 0}$ is generated under cofiltered limits by the essential images of the functors
    \[
    (\imath_i)^{\rm Tate}_* : \IndCoh(S_i)^{\leq 0} \ra \TateCoh(S).
    \]
\end{enumerate}
\end{lem}

\begin{proof}
The point (a) is a consequence of the fact that any closed embedding $S_0 \hra S$ factors through $S_i \hra S$ for some $i \in I$.

For (b) consider $\sF \simeq \lim_{A}\sF_{\alpha}$ an object of $\TateCoh(S)^{\leq 0}$, where $\sF_{\alpha} \in \IndCoh(S)$. By definition of t-structure for any $\sG \in \TateCoh(S)^{\geq 1}$ one has
\[
\Hom_{\TateCoh(S)}(\sF,\sG) \simeq 0.
\]
In particular, for any $i \in I$ one has
\[
\Hom_{\TateCoh(S_i)}(\imath^{!}_i(\sF),\imath^!_i(\sG)) \simeq \Hom_{\TateCoh(S)}((\imath_i)^{\rm Tate}_*\circ \imath^!_i(\sF),\sG) \simeq 0.
\]

Since $\imath^!_i(\sG) \in \TateCoh(S_i)^{\geq 1}$, one has that $\imath^{!}_i(\sF) \in \TateCoh(S_i)^{\leq 0}$, moreover since
\[
(\imath_i)^{\rm Tate}_*\circ \imath^!_i(\sF) \simeq \lim_{A}(\imath_i)^{\rm Ind}_*\circ \imath^!_i(\sF_{\alpha})
\]
one has that $(\imath_i)^{\rm Tate}_*\circ \imath^!_i(\sF)$ is the cofiltered limit of objects of the form we claim. Finally, since
\[
\TateCoh(S)^{\leq 0} \simeq \lim_{I^{\rm op}}\TateCoh(S_i)^{\leq 0},
\]
this implies that for any $\sF \in \TateCoh(S)$ one has
\[
\sF \simeq \lim_{I^{\rm op}}(\imath_i)^{\rm Tate}_*\circ \imath^!_i(\sF).
\]
\end{proof}

\begin{cor}
\label{cor:pushforward-scheme-into-Tate-scheme-closed-is-t-exact}
For $\imath:S_0 \ra S$ a closed embedding from a scheme $S_0$ into a Tate scheme $S$ one has the functor
\[
\imath^{\rm Tate}_*: \TateCoh(S_0) \ra \TateCoh(S)
\]
is t-exact.
\end{cor}

\begin{proof}
Notice that $\imath^{\rm Tate}_*$ is automatically right t-exact. So we only need to check that for $\sF \in \TateCoh(S_0)^{\geq 0}$ and any $\imath':S'_0 \hra S$ closed embedding, one has
\[
(\imath')^!\circ \imath^{\rm Tate}_* (\sF) \in \TateCoh(S'_0)^{\geq 0}.
\]

Let $S \simeq \colim_I S_i$ be a presentation of $S$. By applying \cite{GR-II}*{Lemma 1.1.10} to the case of Tate-coherent sheaves one obtains that
\[
(\imath')^!\circ \imath^{\rm Tate}_* (\sF) \simeq \lim_{J}(\imath'_j)^! \circ (\imath_j)^{\rm Tate}_*(\sF)
\]
where $J$ is a cofinal diagram in $I$ such that for each $j \in I$ one has maps
\[
\imath_j: S_0 \hra S_j, \;\;\; \mbox{and} \;\;\; \imath'_j:S'_0 \hra S_j.
\]

Now we notice that $(\imath_j)^{\rm Tate}_*$ is t-exact, since it is a closed embedding, and that $(\imath'_j)^!$ is left t-exact, again since $\imath'_j$ is a closed embedding\footnote{Indeed, a closed embedding is an affine morphism so the t-exactness of the push-forward follows from Proposition \ref{prop:pushforward-TateCoh-schemes-is-left-t-exact}, whereas the left t-exactness of the $!$-pullback is a formal consequence of the right t-exactness of the pushforward.}.
\end{proof}

The following is a generalization of \cite{GR-II}*{Chapter 3, Lemma 1.4.9.}.

\begin{prop}
\label{prop:pushforward-TateCoh-between-Tate-schemes-is-left-t-exact}
Let $f:S \ra T$ be a morphism between Tate schemes almost of finite type, then $f^{\rm Tate}_*$ is left t-exact. Moreover, if $f$ is representable by a Tate affine scheme, then $f^{\rm Tate}_*$ is t-exact.
\end{prop}

\begin{proof}
Again by how we defined $f^{\rm Tate}_*$ the proof has two cases.

When $f$ is proper, it is enough to check that for $\sF \in \TateCoh(S)^{\geq 0}$ the sheaf
\[
f^{\rm ProInd}_*(\sF)
\]
belongs to $\ProIndCoh(T)^{\geq 0}$. By definition 
\[
f^{\rm ProInd}_*(\sF) \simeq \lim_{A}f^{\rm Ind}(\sF_{\alpha})
\]
where $\{\sF_{\alpha}\}_{A}$ is a cofiltered diagram in $\IndCoh(S)^{\geq 0}$. Thus, the result follows from \cite{GR-II}*{Chapter 3, Lemma 1.4.9.}, that is that $f^{\rm Ind}$ on Tate schemes is left t-exact.

In the case when $f$ is an open embedding, $f^{\rm Tate}_*$ is a right adjoint, thus by Lemma \ref{lem:t-structure-on-TateCoh-of-Tate-scheme} (b) we can suppose that $\sF = \imath^{\rm Tate}_*(\sF_0)$ where $\sF_0 \in \TateCoh(S_i)^{\geq 0}$ for some closed embedding $\imath: S_0 \hra S$, where $S_0$ is a scheme almost of finite type.

However, the composite map
\[
f \circ \imath: S_0 \ra T
\]
factors through some $T_i \hra T$, where $T \simeq \colim_I T_i$ is a presentation of $T$. Let $f_0:S_0 \ra T_0$ be the induced map one, and $\jmath_i:T_i \hra T$ the canonical inclusion has
\[
(f \circ \imath)^{\rm Tate}_* \simeq (\jmath_i)^{\rm Tate}_* \circ (f_0)^{\rm Tate}_* \circ \imath^{\rm Tate}_*.
\]
By Corollary \ref{cor:pushforward-scheme-into-Tate-scheme-closed-is-t-exact} the functors $\imath^{\rm Tate}_*$ and $(\jmath_i)^{\rm Tate}_*$ are t-exact, and by Proposition \ref{prop:pushforward-TateCoh-schemes-is-left-t-exact} the functor $(f_0)^{\rm Tate}_*$ is left t-exact.

The statement for $f$ an affine morphism follows exactly the same argument.
\end{proof}

\paragraph{Admissible Tate-coherent sheaves}
\label{par:defn-of-admissible-Tate-Coh-on-Tate-affine-schemes}

Similarly to \S \ref{par:defn-of-admissible-Tate-Coh} for a $S$ a Tate scheme locally almost of finite type one defines the category 
\[
\TateCoha(S)^{\leq 0} := \left\{\{\sF_i\}_{I} \in \TateCoh(S)^{\leq 0} \; | \; \mbox{ for any }i \ra j\mbox{ in }I, \; \H^0(\sF_i) \ra \H^0(\sF_j) \; \mbox{ is surjective. }\right\}.
\]

The following result will be of use in \ref{subsubsec:split-square-zero-extensions}.

\begin{cor}
\label{cor:colimits-and-limits-presentation-of-TateCoha}
Let $S \in \SchaffTate$ and consider a presentation $T \simeq \colim_I S_i$ where each $S_i \in \Schaffaft$, i.e.\ an affine scheme almost of finite type, one has an equivalence
\[
(\TateCoha(T)^{\leq 0}) = \lim_{I^{\rm op}}(\TateCoha(T_i)^{\leq 0}) \simeq \colim_I(\TateCoha(T_i)^{\leq 0}),
\]
where the limit is taken with respect to the maps $\tau^{\leq 0}\circ f^!_{i,j}$ and the colimit with respect to $(f^{\rm Tate}_{i,j})_*$.
\end{cor}

\begin{proof}
We notice that for $f_{i,j}:T_i \hra T_j$ a closed embedding, we have that $(f_{i,j})^{\rm Tate}_{*}$ is t-exact, since $(f_{i,j})^{\rm Pro}_*$ is t-exact. This implies that $f^!$ is left t-exact by using the $(f^{\rm Tate}_*,f^!)$-adjunction. Thus, one obtains that
\[
\TateCoh(T) \simeq \lim_{I^{\rm op}}(\TateCoh(T_i)^{\leq 0}) \simeq \colim_I(\TateCoh(T_i)^{\leq 0})
\]
where in the limit we consider the composite functors
\[
\tau^{\leq 0}\circ f^!_{i,j}: \TateCoh(T_j)^{\leq 0} \ra \TateCoh(T_i) \ra \TateCoh(T_{i})^{\leq 0}.
\]

Finally, it is enough to check the functors $\tau^{\leq 0}\circ f^!_{i,j}$ and $(f_{i,j})^{\rm Tate}_*$ preserves admissible objects. Indeed, the functors $(f_{i,j})^{\rm Tate}_*$ are t-exact, thus they preserve the admissibility condition and for $\tau^{\leq 0}\circ f^!_{i,j}$ it is clear, since we are taking a truncation.
\end{proof}

\section{Prestacks of Tate type}
\label{sec:prestacks-Tate}

In this section we introduce the main object of this article, prestacks of Tate type. The definition is a straightforward generalization of the concept of a prestack. In \S \ref{subsec:properties-prestacks-Tate}, we define the usual properties for prestacks of Tate type: convergence, coconnectivity, locally almost of finite type, and truncatedness. In \S \ref{subsec:locally-prestack-condition}, we define what it means for a prestack of Tate type to be locally a prestack: that is, its functor of points is obtained from a usual prestack via right Kan extension from affine schemes to Tate affine schemes. We then investigate how this condition relates to the conditions of convergence, $n$-coconnectivity, and finite type. To address the last two questions we need to study the commutation of right and left Kan extensions. In \S \ref{subsec:sheaves-on-prestacks-Tate}, we extend the formalism of Pro-Ind-coherent sheaves and Tate-coherent sheaves to prestacks of Tate type; this follows the arguments in Gaitsgory--Rozenblyum, so we keep the discussion to a minimum level of detail.

\subsection{The notion of prestacks of Tate type}

\paragraph{}

Recall that a prestack is a functor $\sX: \Schaffop \ra \Spc$. In \S \ref{par:Tate-affine-schemes} we defined the category $\SchaffTate$ of Tate affine schemes. The notion of prestacks of Tate type is obtained by substituting the category of affine schemes by that of Tate affine schemes.

\paragraph{}

\begin{defn}
\label{defn:prestack-of-Tate-type}
A \emph{prestack of Tate type} is a functor
\[
\sX: \SchaffTateop \ra \Spc.
\]
We denote by $\PStkTate$ the category of prestacks of Tate type.
\end{defn}

\paragraph{}

One has a fully faithful embedding 
\[
\SchaffTate \hra \PStkTate
\]
given by the Yoneda embedding, that is, it sends $S \in \SchaffTate$ to 
\[
\h_S := \Maps_{\SchaffTate}(-,S) \in \PStkTate.
\]

\paragraph{}

One has a natural morphism
\begin{equation}
\label{eq:prestacks-Tate-restriction-to-prestack}
\PStkTate \ra \PStk    
\end{equation}
given by sending a functor $\sX: (\SchaffTate)^{\rm op} \ra \Spc$ to its restriction $\left.\sX\right|_{(\Schaff)^{\rm op}}$. The functor (\ref{eq:prestacks-Tate-restriction-to-prestack}) has a right adjoint which will be study in section \ref{subsec:locally-prestack-condition} below.

\paragraph{}
\label{par:Tate-affine-schematic-morphism}

We say that a map $f: \sX_1 \ra \sX_2$ of prestacks of Tate type is \emph{Tate affine schematic} or \emph{representable} if for any $x_2:S \ra \sX_2$, where $S \in \SchaffTate$, one has
\[
\sX_1\underset{\sX_2}{\times}S \in \SchaffTate.
\]

\begin{defn}
\label{defn:flat-morphisms-prestacks-Tate}
Given $f:\sX_1 \ra \sX_2$ a Tate affine schematic morphism between prestacks of Tate type we will say that $f$ is \emph{flat} (resp.\ \emph{smooth, \'etale, open embedding, Zariski}) if for every $S \in (\SchaffTate)_{/\sX_2}$ the corresponding map of Tate affine schemes
\[
\sX_1\underset{\sX_2}{\times}S \ra S
\]
is flat (resp.\ smooth, \'etale, open embedding, Zariski).
\end{defn}

\subsubsection{Examples}

From Lemma \ref{lem:formal-loops-of-affine-of-finite-type-are-Tate-affine} we know that given $G$ an affine algebraic group of finite type we have that
\[
\L G := G((t))
\]
is a Tate affine scheme, and by the Yoneda embedding we can see $\L G$ as an object of $\PStkTate$. Now we let
\[
\B \L G := \left|\cdots \; \substack{\rightarrow\\[-1em] \rightarrow \\[-1em] \rightarrow\\[-1em] \rightarrow} \; \L G \times \L G \; \substack{\rightarrow\\[-1em] \rightarrow \\[-1em] \rightarrow} \; \L G \;  \substack{\rightarrow\\[-1em] \rightarrow} \; \ast \right|
\]
denote the geometric realization of the simplicial object constructed from $\L G$ by considering its multiplication and projection maps. We claim 

\begin{lem}
$\B \L G$ is a group object in the category of prestacks of Tate type.
\end{lem}

\begin{proof}
This is automatic. The group structure of $\B \L G$ is exhibited by the simplicial object \[
(\B \L G)_{\bullet}: \N \Delta^{\rm op} \ra \PStkTate
\]
given by
\[
(\B \L G)_{n} := \left|\cdots \; \substack{\rightarrow\\[-1em] \rightarrow \\[-1em] \rightarrow\\[-1em] \rightarrow} \; (\L G)^{\times n} \times (\L G)^{\times n} \; \substack{\rightarrow\\[-1em] \rightarrow \\[-1em] \rightarrow} \; (\L G)^{\times n} \;  \substack{\rightarrow\\[-1em] \rightarrow} \; \ast\right|
\]
which is isomorphic to $\B (\L G)^{\times n}$.
\end{proof}

\subsection{Properties of prestacks of Tate type}
\label{subsec:properties-prestacks-Tate}

\subsubsection{Coconnectivity}

For any natural number $n\geq 0$, we let $\PStkTaten$ denote the category of functors
\[
(\SchaffTaten)^{\rm op} \ra \Spc.
\]
One has a natural functor
\[
\PStkTate \ra \PStkTaten
\]
given by restriction via the inclusion $(\SchaffTaten)^{\rm op} \hra (\SchaffTaten)^{\rm op}$. So given any $\sX$ an object of $\PStkTate$ one has a natural functor
\begin{equation}
    \label{eq:defn-n-coconnective-prestack-Tate-type}
    \LKE_{\SchaffTatenop \hra \SchaffTateop}(\left.\sX\right|_{{^{\leq n}\SchaffTateop}}) \ra \sX.
\end{equation}

One says that $\sX$ is an \emph{$n$-coconnective prestack of Tate type} if the map (\ref{eq:defn-n-coconnective-prestack-Tate-type}) is an equivalence.

\paragraph{Concrete description of LKE}
\label{par:concrete-LKE-description}

For $\sX' \in \PStkTaten$ one can describe explicitly the left Kan extension in (\ref{eq:defn-n-coconnective-prestack-Tate-type}) by
\[
\LKE_{({^{\leq n}\SchaffTate)^{\rm op}} \hra (\SchaffTate)^{\rm op}}(\sX')(S) \simeq \colim_{S' \in \left(({^{\leq n}\SchaffTate})_{/S}\right)^{\rm op}}\sX'(S'),
\]
where any map $S'_1 \ra S'_2$ in ${^{\leq n}\SchaffTate}_{/S}$ gives a morphism
\[
\sX'(S'_2) \ra \sX'(S'_1)
\]
and one considers the colimit of those.

\begin{rem}
\label{rem:truncated-S-is-a-final-object}
Any map $S' \ra S$ with $S' \in {^{\leq n}\SchaffTate}$ factors through ${^{\leq n}S}$\footnote{Indeed, if $S = \Spec(A)$ one has ${^{\leq n}S} = \Spec(\tau^{\geq -n}(A))$ by definition, and one can easily see that if $f:\Spec(B) \ra \Spec(A)$ with $B \simeq \tau^{\geq -n}(B)$ then $f$ factors as
\[
\Spec(B) \ra \Spec(\tau^{\geq -n}(A)) \ra \Spec(A).
\]
}, thus one has
\[
\LKE_{({^{\leq n}\SchaffTate)^{\rm op}} \hra (\SchaffTate)^{\rm op}}(\sX')(S) \simeq \colim_{S' \in \left({^{\leq n}\SchaffTate})_{/{^{\leq n}S}}\right)^{\rm op}}\sX'(S'),
\]
in particular the left Kan extension of $\sX'$ to a prestack evaluated at $S$ only depends on ${^{\leq n}S}$.
\end{rem}

\paragraph{Classical and eventually coconnective prestacks of Tate type}

We will refer to a $0$-coconnective prestacks of Tate type as a \emph{classical prestack of Tate type}. 

For $\sX \in \PStkTate$ one says that $\sX$ is \emph{eventually coconnective} if $\sX$ is $n$-coconnective for some $n$.

\subsubsection{Convergence}
\label{subsubsec:convergence-prestacks-Tate}

For $\sX \in \PStkTate$ a prestack of Tate type one has a natural map
\begin{equation}
\label{eq:defn-convergent-prestack-Tate-type}
\sX \ra \RKE_{\SchaffTateconvop \hra \SchaffTateop}(\left.\sX\right|_{\SchaffTateconvop}).
\end{equation}

One says that $\sX$ is a \emph{convergent} prestack of Tate type if (\ref{eq:defn-convergent-prestack-Tate-type}) is an equivalence.

\paragraph{Concrete description of RKE}

Analogous to \S \ref{par:concrete-LKE-description} for $\sX' \in {^{\leq n}\PStkTate}$ the right Kan extension by $\SchaffTatenop \hra \SchaffTateop$ is explicitly given by
\[
\RKE_{\SchaffTatenop \hra \SchaffTateop}(\sX')(S) = \lim_{S' \in \left((\SchaffTaten)_{/S}\right)^{\rm op}}\sX'(S'),
\]
that is given a map $S'_1 \ra S'_2$ in $(\SchaffTaten)_{/S}$ one has maps
\[
\sX'(S'_2) \ra \sX'(S'_1)
\]
and one considers the limit of those maps.

We notice that Remark \ref{rem:truncated-S-is-a-final-object} is equivalent to $S'$ is a final object in $\left((\SchaffTaten)_{/S}\right)^{\rm op}$, in particular this gives the equivalence
\[
\RKE_{\SchaffTatenop \hra \SchaffTateop}(\sX')(S) \simeq \sX'({^{\leq n}S}).
\]

\paragraph{}

The following is analogue to \cite{GR-I}*{Chapter 2, Lemma 1.4.4}.

\begin{lem}
Let $S$ be a Tate affine scheme, then considered as a prestack of Tate type $S$ is convergent.
\end{lem}

\begin{proof}
Let $\h_S \in \PStkTate$ be the prestack of Tate type represented by $S$ and $T \in \SchaffTate$, we pick a presentation $T \simeq \colim_J T_j$, then one has
\begin{align*}
    \h_S(T) & \simeq \Maps_{\SchaffTate}(T,S) \\
    & \simeq \Maps_{\Ind(\Schaff)}(\colim_J T_j,S) \\
    & \simeq \lim_{J^{\rm op}}\Maps_{\Ind(\Schaff)}(T_j,S) \\
    & \simeq \lim_{J^{\rm op}} \lim_{T'_j \in (\Schaffconv_{/T_j})^{\rm op}} \Maps_{\Ind(\Schaff)}(T'_j,S) \\
    & \simeq \lim_{T' \in ((\SchaffTateconv)_{/T})^{\rm op}}\Maps_{\Ind(\Schaff)}(T',S) \\
    & \simeq \RKE_{\SchaffTateconvop \hra \SchaffTateop}(\left.\h_S\right|_{\SchaffTateconvop})(T),
\end{align*}
where the isomorphism between the third and forth line follows from Lemma \ref{lem:Tate-affine-schemes-are-convergent}.
\end{proof}

\begin{rem}
Notice, that after section \ref{subsec:locally-prestack-condition} the above Lemma is a formal consequence of the compatibility of convergence for prestacks and convergence of their right Kan extension to a prestack of Tate type.
\end{rem}

\paragraph{}

The following is analogue to \cite{GR-I}*{Chapter 2, Proposition 1.4.7}.

\begin{lem}
\label{lem:concrete-condition-convergent-prestack-Tate-type}
Let $\sX$ be a prestack of Tate type, then the map (\ref{eq:defn-convergent-prestack-Tate-type}) is an equivalence if and only if the map
\[
\sX(S) \ra \lim_{n \geq 0}\sX({^{\leq n}S})
\]
is an equivalence.
\end{lem}

\begin{proof}
We notice that the functor
\begin{align*}
    \bZ_{\geq 0} & \ra (\SchaffTateconv)_{/S} \\
    n & \mapsto \tau^{\leq n}(S) = {^{\leq n}S}
\end{align*}
is cofinal. 

Thus, one has
\begin{align*}
\RKE_{\SchaffTateconvop \hra \SchaffTateop}(\left.\sX\right|_{\SchaffTateconvop})(S) & \simeq \lim_{S'\left((\SchaffTateconv)_{/S}\right)^{\rm op}} \sX(S') \\
& \simeq \lim_{\left(\bZ_{\geq 0}\right)^{\rm op}}\sX(\tau^{\leq n}(S))
\end{align*}
where the first isomorphism is the concrete description of the right Kan extension and the second follows from the cofinality statement.
\end{proof}

\subsubsection{Finiteness conditions}

\paragraph{}

Recall the definition of the category $\SchaffTatenft$ from \S \ref{subsubsec:coconnective-Tate-schemes}. Given an object $\sX' \in \PStkTaten$ one says that $\sX'$ is of \emph{finite type} if the canonical map
\[
\LKE_{\SchaffTatenftop \hra \SchaffTatenop}(\left.\sX'\right|_{\SchaffTatenftop}) \ra \sX'
\]
is an equivalence.

\paragraph{}

More generally one can consider

\begin{defn}
\label{defn:laft-prestack-Tate}
For $\sX \in \PStkTate$ one says that $\sX$ is \emph{locally almost of finite type} if
\begin{enumerate}[1)]
    \item $\sX$ is convergent;
    \item for every $n$, ${^{\leq n}\sX}$ is of finite type.
\end{enumerate}
We will denote the category of prestacks of Tate type locally almost of finite type by $\PStkTatelaft$.
\end{defn}

\paragraph{}

Exactly the same proof as in \cite{GR-I}*{Chapter 2, Proposition 1.7.6} gives the following statement

\begin{prop}
\label{prop:prestacks-Tate-laft-as-functors}
One has an equivalence
\[
\PStkTatelaft \overset{\simeq}{\ra} \Fun(\SchaffTateconvftop,\Spc)
\]
given by restriction along the inclusion $\SchaffTateconvft \hra \SchaffTate$. 

The inverse functor is given by applying the left Kan extension along
\[
\SchaffTateconvft \hra \SchaffTateconv
\]
followed by the right Kan extension along
\[
\SchaffTateconv \hra \SchaffTate.
\]
\end{prop}

\subsubsection{Truncatedness}

\paragraph{}

For $k=0,1,\ldots$, we denote $\Spc_{\leq k} \subset \Spc$ the full subcategory of \emph{k-truncated} spaces, i.e.\ for $S'$ the connect component of any object of $\Spc_{\leq k}$ one has
\[
\pi_{i}(S') = 0
\]
for $i > k$. In particular, $\Spc_{\leq 0} = \Set$.

The embedding $\Spc_{\leq k} \hra \Spc$ has a left adjoint, which gives rise to a localization
\[
\Spc \ra \Spc_{\leq k} \ra \Spc
\]
whose composite we denote $\P_{\leq k}$.

\paragraph{}

For any space $S$, the assignment $k \ra \P_{\leq k}(S)$ is a functor
\[
(\bZ)^{\rm op} \ra \Spc
\]
called the \emph{Postinikov tower} of $S$. It is a fact that 
\[
S \ra \lim_{k \geq 0}\P_{\leq k}(S)
\]
is an isomorphism.

\paragraph{}

\begin{defn}
\label{defn:n-truncated-prestack-of-Tate-type}
For $\sX \in \PStkTaten$ one says that $\sX$ is $k$-truncated if the functor $\sX: \SchaffTatenop \ra \Spc$ factors through $\Spc_{\leq k}$.
\end{defn}

\paragraph{}

\begin{example}
Suppose that $\sX \in \PStkTaten$ is representable by a Tate affine scheme $T \in \SchaffTaten$, then $\sX$ is $n$-truncated. 

Indeed, by \cite{GR-II}*{Chapter 2, Lemma 1.2.3} for any $S_0 \in \Schaff$ one has
\[
\h_T(S_0) \simeq \Maps_{\SchaffTate}(S_0,T) \simeq \Maps_{\Ind(\Schaff)}(S_0,T)
\]
is $n$-truncated, since $T$ is an ind-scheme. Now, for $S \in \SchaffTate$ and a presentation $S \simeq \colim_I S_i$ one has
\[
\h_T(S) \simeq \Maps_{\SchaffTate}(S,T) \simeq \lim_{I^{\rm op}}\Maps_{\SchaffTate}(S_i,T)
\]
and the result follows from the fact that $\Spc_{\leq n}$ has cofiltered limits and those are equivalent to the corresponding limits in $\Spc$.
\end{example}

\subsection{Locally a prestack condition}
\label{subsec:locally-prestack-condition}

In this section we discuss prestacks of Tate type which are obtained from usual prestacks.

\subsubsection{Prestacks of Tate type from prestacks}

\paragraph{}

Let $\sY \in \PStk$ one can consider a the prestack of Tate type
\[
\RKE_{\Schaffop \hra \SchaffTateop}(\sY)
\]
obtained by \emph{right} Kan extension from the category of affine schemes to the category of Tate affine schemes. 

\begin{rem}
The reason to consider the right Kan extension is the fact that $\SchaffTate$ is a full subcategory of $\Ind(\Schaff)$, which is generated by filtered colimits from $\Schaff$. Thus, the opposite category $(\Ind(\Schaff))^{\rm op}$ is generated by adjoining cofiltered limits to $(\Schaff)^{\rm op}$. 

One should compare this with the notion of a locally almost of finite type prestack (see \S \ref{subsubsec:finiteness-conditions}) that uses the left Kan extension from ($n$-coconnective) affine schemes of finite type to ($n$-coconnective) affine schemes, because the later are generated by the former under cofiltered limits by Corollary \ref{cor:pro-schemes-n-coconnective-ft-are-all}.
\end{rem}

\paragraph{} 

Given a prestack of Tate type $\sX$ we will say that $\sX$ is \emph{locally a prestack} if the canonical functor
\[
\sX \ra \RKE_{\Schaffop \hra \SchaffTateop}(\left.\sX\right|_{\Schaffop})
\]
is an equivalence. We denote by $\PStkTatelp$ the full subcategory of $\PStkTate$ generated by the prestacks of Tate type which are locally prestacks.

\paragraph{} 

The following result is analogous to Lemma \ref{lem:concrete-condition-convergent-prestack-Tate-type}.

\begin{lem}
\label{lem:concrete-condition-locally-prestack}
A prestack of Tate type $\sX$ is locally a prestack if and only if for every Tate affine scheme $S$ and a presentation $\colim_I S_i$ of $S$ one has an equivalence
\[
\sX(S) \overset{\simeq}{\ra} \lim_{I^{\rm op}}\left.\sX\right|_{(\Schaff)^{\rm op}}(S_i).
\]
\end{lem}

\begin{proof}
For $S\in \SchaffTate$, let $S \simeq \colim_I S_i$ be a presentation of $S$, then consider the functor
\begin{align*}
    I & \ra \Schaff_{/S} \\
    i & \mapsto S_i \overset{g_i}{\ra} S,
\end{align*}
where $g_i$ is the canonical map from $S_i$ into the colimit. 

Notice that the map $I \ra \Schaff_{/S}$ is cofinal. Indeed, this is a consequence of \cite{DG-indschemes}*{Lemma 1.5.4 and Corollary 1.6.6}. 

Thus, we obtain
\begin{align*}
    \sX(S) & \ra \RKE_{\Schaffop \hra \SchaffTateop}(\left.\sX\right|_{\Schaffop}) \\
    & \simeq \lim_{T \in \left(\Schaff_{/S}\right)^{\rm op}}\left.\sX\right|_{\Schaffop}(T) \\
    & \simeq \lim_{I^{\rm op}}\left.\sX\right|_{\Schaffop}(S_i) \\
    & \simeq \lim_{I^{\rm op}}\sX(S_i),
\end{align*}
the first and last isomorphism are from the definitions and the middle is because the functor $I^{\rm op} \ra \left(\Schaff_{/S}\right)^{\rm op}$ is final.
\end{proof}

\paragraph{}

The following is an immediate consequence of Lemma \ref{lem:concrete-condition-convergent-prestack-Tate-type}.

\begin{lem}
\label{lem:Tate-affine-schemes-are-locally-prestacks}
For $S\in \SchaffTate$ the corresponding prestack of Tate type, i.e.
\[
\h_S = \Hom_{\SchaffTate}(-,S) \in \PStkTate
\]
is locally a prestack.
\end{lem}

\subsubsection{Compatibility of conditions}

In this section we check which properties from \S \ref{subsec:properties-prestacks-Tate} on a prestack of Tate type which is locally a prestack recover the usual properties on prestacks as defined in \cite{GR-I}*{Chapter 2, \S 1}.

\paragraph{Convergence}

\begin{lem}
\label{lem:Tate-prestack-convergence-compatibility}
Assume that $\sX \in \PStkTatelp$, i.e.\ $\sX \simeq \RKE_{\Schaffop \hra \SchaffTateop}(\sX_0)$ for some $\sX_0 \in \PStk$, then $\sX$ is a convergent prestack of Tate type if and only if $\sX_0$ is a convergent prestack.
\end{lem}

\begin{proof}
This is simply the fact that right Kan extensions are associative, see \cite{HTT}*{Proposition 4.3.2.8}.
\end{proof}

\paragraph{Truncatedness}

\begin{lem}
For a prestack of Tate type $\sX$ locally a prestack, then 
$\sX \in \PStkTaten$ is $k$-truncated if and only if $\left.\sX\right|_{\Schaffnop}$ if $k$-truncated.
\end{lem}

\begin{proof}
One notices that the right Kan extension functor is concretely computed by considering limits in $\Spc_{\leq k}$, however this category is closed under all limits inside $\Spc$ by \cite{HTT}*{Proposition 5.5.6.5}.
\end{proof}

\paragraph{Digression: swapping left and right Kan extensions}

We will consider the Cartesian diagram
\begin{equation}
\label{eq:inclusion-of-n-coconnective-affine-schemes-into-Tate-affine-schemes}
\begin{tikzcd}
\Schaffnop \ar[r,"\imath"] \ar[d,"\jmath^{\leq n}"] & \Schaffop \ar[d,"\jmath"] \\
\SchaffTatenop \ar[r,"\imath_{\rm Tate}"] & \SchaffTateop.
\end{tikzcd}    
\end{equation}

\begin{lem}
\label{lem:coconnective-affine-schemes-satisfy-RKE-base-change}
The diagram (\ref{eq:inclusion-of-n-coconnective-affine-schemes-into-Tate-affine-schemes}) satisfies the conditions of Lemma \ref{lem:initial-slice-categories-give-isomorphism-of-base-change}.
\end{lem}

\begin{proof}
Consider $T' \in \SchaffTatenop$ we claim that the functor
\[
\beta: \jmath^{\leq n}_{T'/} \ra \jmath_{\imath_{\rm Tate}(T')/}
\]
is initial. Indeed, notice that
\[
\jmath^{\leq n}_{T'/} \simeq \{\jmath^{\leq n}(T'_0) \ra T' \; | \; T'_0 \in \Schaffn\},
\]
and
\[
\jmath_{\imath_{\rm Tate}(T')/} \simeq \{\jmath(T_0) \ra \imath_{\rm Tate}(T') \; | \; T_0 \in \Schaff\}.
\]

So given $(\jmath(T_0) \ra \imath^{\rm Tate}(T'))$ and object of $\jmath_{\imath_{\rm Tate}(T')/}$ we need to check that the category $\beta_{/(\jmath(T_0) \ra \imath^{\rm Tate}(T'))}$ is non-empty and connected. This follows from the fact that the object
\[
\begin{tikzcd}
\imath_{\rm Tate}(\jmath^{\leq n}({^{\leq n}T_0})) \ar[r] \ar[d] & \imath_{\rm Tate}(T') \ar[d] \\
 \jmath(T_0) \ar[r] & \imath_{\rm Tate}(T')
\end{tikzcd}
\]
is final in the category $\beta_{/(\jmath(T_0) \ra \imath^{\rm Tate}(T'))}$.
\end{proof}

Let $\sX'_0 \in \PStkn$ denote a functor
\[
\sX'_0: \Schaffnop \ra \Spc,
\]
then by Lemma \ref{lem:coconnective-affine-schemes-satisfy-RKE-base-change} one obtains a morphism
\begin{equation}
\label{eq:comparison-LKE-RKE-to-RKE-LKE}
\alpha_{\sX'_0}: \LKE_{\imath_{\rm Tate}}\circ\RKE_{\jmath^{\leq n}}(\sX'_0) \ra \RKE_{\jmath}\circ\LKE_{\imath}(\sX'_0).    
\end{equation}

\paragraph{$n$-coconnective}

In this section we apply the digression about commuting left and right Kan extension to prove a compatibility result for the coconnectivity conditions for a prestack of Tate type and a usual prestack.

\begin{prop}
\label{prop:Tate-prestack-coconnectivity-compatibility}
Suppose that $\sX \in \PStkTate$ is an $n$-coconnective prestack of Tate type and locally a prestack, and let $\sX_0:=\left.\sX\right|_{\Schaffop}$ denote its underlying prestack. Then $\sX_0$ is an $n$-coconnective prestack.
\end{prop}

\begin{proof}
We notice that the natural transformation 
\begin{equation}
\label{eq:comparison-between-different-extensions-for-prestack-Tate}
\alpha_{{^{\leq n}\sX_0}}: \LKE_{\imath_{\rm Tate}}\circ\RKE_{\jmath^{\leq n}}({^{\leq n}\sX_0}) \ra \RKE_{\jmath}\circ\LKE_{\imath}({^{\leq n}\sX_0})    
\end{equation}
is obtained by composing the map
\begin{equation*}
{^{\leq n}\sX_0} \ra \left.\LKE_{\imath}({^{\leq n}\sX_0})\right|_{\Schaffnop},
\end{equation*}
which is an isomorphism since $\imath$ is fully faithful, with the map
\begin{equation*}
\LKE_{\imath_{\rm Tate}}({^{\leq n}\sX}) \ra \sX
\end{equation*}
which is also an isomorphism, since $\sX$ is an $n$-coconnective prestack of Tate type. Thus, one has that $\alpha_{{^{\leq n}\sX_0}}$ is an isomorphism.

Notice that Lemma \ref{lem:coconnective-affine-schemes-satisfy-RKE-base-change} gives the isomorphism
\[
{^{\leq n}\sX} \overset{\simeq}{\ra} \RKE_{\jmath^{\leq n}}({^{\leq n}\sX_0}).
\]
So, the left-hand side of (\ref{eq:comparison-between-different-extensions-for-prestack-Tate}) is isomorphic to $\sX$ and we obtain
\[
\sX \overset{\simeq}{\ra} \RKE_{\jmath}\circ\LKE_{\imath}({^{\leq n}\sX_0}).
\]
Since $\jmath$ is fully faithful one has an equivalence
\[
\sX_0 \simeq \left.\sX\right|_{\Schaffop} \overset{\simeq}{\ra} \LKE_{\imath}({^{\leq n}\sX_0}),
\]
that is $\sX_0$ is $n$-coconnective.
\end{proof}

\paragraph{Finiteness condition}

To compare the finiteness condiions, we notice that the we have a Cartesian diagram
\begin{equation}
\label{eq:inclusion-of-n-coconnective-affine-schemes-of-ft-into-n-coconnective-Tate-affine-schemes}
\begin{tikzcd}
\Schaffnftop \ar[r,"\imath_{\rm ft}"] \ar[d,"\jmath^{\leq n}_{\rm ft}"'] & \Schaffnop \ar[d,"\jmath^{\leq n}"] \\
\SchaffTatenftop \ar[r,"\imath^{\rm Tate}_{\rm ft}"'] & \SchaffTatenop 
\end{tikzcd}    
\end{equation}

\begin{lem}
\label{lem:coconnective-ft-affine-schemes-satisfy-RKE-base-change}
The diagram (\ref{eq:inclusion-of-n-coconnective-affine-schemes-of-ft-into-n-coconnective-Tate-affine-schemes}) satisfies the conditions of Lemma \ref{lem:initial-slice-categories-give-isomorphism-of-base-change}.
\end{lem}

\begin{proof}
For every $S' \in \SchaffTatenft$ we need to check that the functor
\[
\alpha: (\jmath^{\leq n}_{\rm ft})_{S'/} \ra \jmath^{\leq n}_{\imath^{\rm Tate}_{\rm ft}(S')/}
\]
is initial. First we notice that
\[
(\jmath^{\leq n}_{\rm ft})_{S'/} \simeq \{\jmath^{\leq n}_{\rm ft}(S'_0) \ra S' \;| \; S'_0 \in \Schaffnft\}
\]
and
\[
\jmath^{\leq n}_{\imath^{\rm Tate}_{\rm ft}(S')} \simeq \{\jmath^{\leq n}(S_0) \ra \imath^{\rm Tate}_{\rm ft}(S') \; | \; S_0 \in \Schaffn\}.
\]
Indeed, $\alpha$ is initial if for every object $(\jmath^{\leq n}(S_0) \ra \imath^{\rm Tate}_{\rm ft}(S'))$ from $\jmath^{\leq n}_{\imath^{\rm Tate}_{\rm ft}(S')}$ the slice category $\alpha_{/\jmath^{\leq n}_{\imath^{\rm Tate}_{\rm ft}(S')}}$ is non-empty and connected. Let $S' \simeq \colim_I S'_i$ be a presentation of $S'$, where $S'_i \in \Schaffnft$, then one has 
\[
\begin{tikzcd}
\imath^{\rm Tate}_{\rm ft}\circ \jmath^{\leq n}_{\rm ft}(S_0\underset{S'}{\times}S'_i) \ar[r] \ar[d] & \imath^{\rm Tate}_{\rm ft}\circ \jmath^{\leq n}_{\rm ft}(S'_i) \ar[r] & \imath^{\rm Tate}_{\rm ft}(S') \ar[d] \\
\jmath^{\leq n}(S_0) \ar[rr] & & \imath^{\rm Tate}_{\rm ft}(S')
\end{tikzcd}
\]
So the category $\alpha_{/\jmath^{\leq n}_{\imath^{\rm Tate}_{\rm ft}(S')}}$ is non-empty. Since by definition of Tate schemes, the category $I$ is filtered, one has that this category is also connected.
\end{proof}

Similarly to the previous paragraph, we have the following result
\begin{cor}
\label{cor:Tate-prestack-laft-compatibility}
Let $\sX \in \PStkTatelp$ and suppose that $\sX$ is locally almost finite type as a prestack of Tate type. Let $\sX_0 = \left.\sX\right|_{\Schaffop}$, then $\sX_0$ is also locally almost of finite type.
\end{cor}

\begin{proof}
The convergence condition for $\sX_0$ follows from Lemma \ref{lem:Tate-prestack-convergence-compatibility}, whereas the finiteness condition for ${{^\leq n}\sX_0}$ follows from the same argument as in the proof of Proposition \ref{prop:Tate-prestack-coconnectivity-compatibility}, but using Lemma \ref{lem:coconnective-ft-affine-schemes-satisfy-RKE-base-change} as input.
\end{proof}

\subsection{Sheaves on Prestacks of Tate type almost of finite type}
\label{subsec:sheaves-on-prestacks-Tate}

In this section we follow the strategy of \cite{GR-II}*{Chapter 3, \S 5.4} to extend the formalism of Sections \ref{subsec:Pro-ind-coherent-on-Tate-schemes} and \ref{subsec:Tate-coherent-sheaves-on-Tate-schemes} to prestacks of Tate type which are almost of finite type. As in \emph{loc.\ cit.\ } the pullback and pushforward functors don't \emph{both} exist for arbitrary morphisms. In the case of pro-ind-coherent sheaves we will have an arbitrary pullback and pushforward will exist only for a certain type of morphisms, namely Tate schematic\footnote{This is probably not the most general framework on can develop, as in Gaitsgory--Rozenblyum one could probably consider ind-inf-schematic morphisms, since we don't have an application in mind for those we don't pursue this generality here.} (defined below).

\subsubsection{Morphisms of prestacks of Tate type}

To develop the formalism of Pro-Ind-coherent sheaves of Tate-coherent sheaves for prestacks of Tate type, we need to consider certain special types of morphisms between those. Recall the notion of Tate affine schematic, see \S \ref{par:Tate-affine-schematic-morphism}. The following is a generalization to Tate schemes locally almost of finite type.

\begin{defn}
\label{defn:schematic-morphisms-prestack-Tate}
For a Tate affine schematic morphism $f:\sX \ra \sY$ between prestacks of Tate type, one says that $f$ is \emph{almost of finite type (aft)} if for every morphism $S \ra \sY$ where $S \in \SchaffTateaft$ the pullback
\begin{equation}
    \label{eq:base-change-map-defn-of-aff-aft-morphisms}
    S\underset{\sY}{\times}\sX
\end{equation}
is a Tate affine scheme almost of finite type.
Moreover, for $f:\sX \ra \sY$ a Tate affine schematic morphism almost of finite type, we say that $f$ is ind-proper if for every $S$ as above the map (\ref{eq:base-change-map-defn-of-aff-aft-morphisms}) is ind-proper.
\end{defn}

\subsubsection{Extension of Pro-Ind-coherent sheaves}
\label{subsubsec:extension-ProIndCoh-to-prestacks-Tate}

In this section we extend the formalism of pro-ind-coherent sheaves with pullbacks to all prestacks of Tate type. Recall that in \S \ref{subsec:Pro-ind-coherent-on-Tate-schemes} we defined a functor
\[
\ProIndCoh^!: (\SchTatelaft)^{\rm op} \ra \DGc.
\]

Let
\[
\ProIndCoh^!_{\SchaffTateaft} := \left.\ProIndCoh^!\right|_{(\SchaffTateaft)^{\rm op}}
\]
denote the $!$-pullback formalism on the category of Tate affine schemes almost of finite type.

One defines the functor
\[
\ProIndCoh^!_{\rm \PStkTatelaft}: \PStkTatelaft \ra \DGc
\]
as the right Kan extension of $\ProIndCoh^!:(\SchaffTateaft)^{\rm op} \ra \DGc$ via the Yoneda embedding\footnote{Recall that any Tate affine scheme is automatically convergent, in particular the Yoneda embedding into $\PStkTate$ factors through the category of laft prestacks of Tate type.}
\begin{equation}
    \label{eq:inclusion-Tate-affine-schemes-into-prestacks-Tate}
    \SchaffTateaft \hra \PStkTatelaft. 
\end{equation}

Given any morphisms $f:\sX_1 \ra \sX_2$ of prestacks of Tate type one denotes the $!$-pullback by
\[
f^{!}: \ProIndCoh(\sX_2) \ra \ProIndCoh(\sX_1).
\]

\paragraph{}

Let 
\[
\mbox{Corr}(\PStkTatelaft)^{\rm aff-aft-ind-proper}_{\rm aff-aft;all}
\]
denote the $2$-category of correspondences between prestack of Tate type which are locally almost of finite type, where
\[
\mbox{aff-aft  and   aff-aft-ind-proper}
\]
denote the classes of morphisms which are Tate affine schematic and almost of finite type and ind-proper and Tate affine schematic and almost of finite type, respectively.

\paragraph{}

We claim
\begin{thm}
\label{thm:ProIndCoh-correspondences-prestacks-Tate-laft-affine-morphisms}
There exists a uniquely defined functor
\[
\ProIndCoh_{\mbox{Corr}(\PStkTatelaft)^{\rm aff-aft-ind-proper}_{\rm aff-aft;all}}: \mbox{Corr}(\PStkTatelaft)^{\rm aff-aft-ind-proper}_{\rm aff-aft;all} \ra \DGct,
\]
equipped with isomorphisms
\[
\left.\ProIndCoh_{\mbox{Corr}(\PStkTatelaft)^{\rm aff-aft-ind-proper}_{\rm aff-aft;all}}\right|_{(\PStkTatelaft)^{\rm op}} \simeq \ProIndCoh^!_{\rm \PStkTatelaft}
\]
and
\[
\left.\ProIndCoh_{\mbox{Corr}(\PStkTatelaft)^{\rm aff-aft-ind-proper}_{\rm aff-aft;all}}\right|_{\CorrTateaaftipaa} \simeq \ProIndCoh_{\CorrTateaaftipaa}
\]
where the latter two isomorphisms are compatible.
\end{thm}

\begin{proof}
We notice that the restriction of $\ProIndCoh_{\CorrTatelaftipaa}$ to the category of correspondence of Tate affine schemes almost of finite type, whose admissible morphisms are ind-proper gives a functor
\begin{equation}
\label{eq:ProIndCoh-ind-proper-on-affine-Tate-aft}
\ProIndCoh_{\CorraTateaftipaa}: \CorraTateaftipaa \ra \DGc.    
\end{equation}

One wants to apply \cite{GR-I}*{Chapter 8, Theorem 6.1.5} to the functor (\ref{eq:ProIndCoh-ind-proper-on-affine-Tate-aft}).

For that we notice that the fully faithful functor (\ref{eq:inclusion-Tate-affine-schemes-into-prestacks-Tate}) preserves Cartesian diagrams, and sends any morphism in $\SchaffTateaft$ to a Tate affine schematic and aft morphism in $\PStkTatelaft$. Indeed, any morphism between Tate affine schemes almost of finite type is an aft morphism.

The conditions on the 2-category $\DGct$ are satisfied in the same way as in \cite{GR-I}*{Chapter 8, \S 6.1.3}.

Finally, the condition of Theorem 6.1.5 from \cite{GR-I}*{Chapter 8} asks that for any $S \in \SchaffTateaft$ the map of categories
\[
(\SchaffTateaft)_{/S} \ra (\PStkTatelaft)_{\rm aff-aft \; in \; S}
\]
is an equivalence. Indeed, this follows from noticing that if $\sX \in \PStkTatelaft$ and $f: \sX \ra S$ is Tate affine schematic and aft morphism to $S$ a Tate affine scheme almost of finite type, then $\sX \in \SchaffTateaft$.
\end{proof}

\subsubsection{Extension of Tate-coherent sheaves}
\label{subsubsec:extension-TateCoh-to-PStk-Tate}

We proceed exactly as in section \ref{subsubsec:extension-ProIndCoh-to-prestacks-Tate}. Let $\TateCoh^!_{\SchaffTateaft}$ denote the restriction of the $!$-pullback functor from \S \ref{subsec:Tate-coherent-sheaves-on-Tate-schemes} to the category of Tate affine schemes almost of finite type.

One defines the functor
\[
\TateCoh^!_{\rm \PStkTatelaft}: \PStkTatelaft \ra \DGc
\]
as the right Kan extension of $\TateCoh^!_{\SchaffTateaft}:(\SchaffTateaft)^{\rm op} \ra \DGc$ via the natural inclusion (\ref{eq:inclusion-Tate-affine-schemes-into-prestacks-Tate}).

Given a map of prestacks of Tate type $f: \sX_1 \ra \sX_2$ we will also denote by
\[
f^!: \TateCoh(\sX_2) \ra \TateCoh(\sX_1)
\]
the $!$-pullback functor.

\begin{thm}
\label{thm:TateCoh-correspondence-functor-PStk-Tate}
There exists a unique extension
\[
\TateCoh_{\mbox{Corr}(\PStkTatelaft)^{\rm aff-aft-ind-proper}_{\rm aff-aft;all}}: \mbox{Corr}(\PStkTatelaft)^{\rm aff-aft-ind-proper}_{\rm aff-aft;all} \ra \DGct
\]
equipped with isomorphisms
\[
\left.\TateCoh_{\mbox{Corr}(\PStkTatelaft)^{\rm aff-aft-ind-proper}_{\rm aff-aft;all}}\right|_{(\PStkTatelaft)^{\rm op}} \simeq \TateCoh^!_{\PStkTatelaft}
\]
and
\[
\left.\TateCoh_{\mbox{Corr}(\PStkTatelaft)^{\rm aff-aft-ind-proper}_{\rm aff-aft;all}}\right|_{\CorrTateaaftipaa} \simeq \TateCoh_{\CorrTateaaftipaa}.
\]
\end{thm}

\begin{proof}
We restrict the functor (\ref{eq:TateCoh-correspondences-Tate-schemes-aft-ind-proper}) to the category of correspondences of Tate affine schemes almost of finite type with ind-proper admissible morphisms:
\begin{equation}
\label{eq:TateCoh-ind-proper-on-affine-Tate-aft}
\TateCoh_{\CorraTateaftipaa}: \CorraTateaftipaa \ra \DGc.    
\end{equation}

By the same arguments as in the proof of \ref{thm:ProIndCoh-correspondences-prestacks-Tate-laft-affine-morphisms} we can apply \cite{GR-I}*{Chapter 8, Theorem 6.1.5} to the functor \ref{eq:TateCoh-ind-proper-on-affine-Tate-aft}.
\end{proof}

\section{Tate stacks}
\label{sec:Tate-stacks}

In this section we study geometric conditions that can be imposed on a prestack of Tate type. In \S \ref{subsec:descent-and-stacks}, we introduce Grothendieck topologies on Tate affine schemes, and after that we discuss the descent condition for prestacks of Tate type. We refer to prestacks of Tate type that satisfy descent with respect to the \'etale topology as Tate stacks. We also investigate how descent interacts with other natural conditions on Tate stacks. Possibly the most interesting result is that descent with respect to Tate affine schemes for a prestack of Tate type which is locally a prestack is equivalent to descent with respect to affine schemes for the usual Grothendieck topologies. In \S \ref{subsec:Tate-Artin-stacks}, we study the conditions for a Tate stack to be a Tate Artin stack. These are natural geometrically meaningful conditions to consider, and we study the relation between Tate Artin $k$-stacks and usual Artin $k$-stacks. In \S \ref{subsec:Zariski-Tate-stacks}, we introduce the notion of Tate stacks that admit a Zariski atlas, which we call Zariski Tate stacks. Zariski Tate stacks serve as an intermediate concept between Tate affine schemes and Tate Artin 1-stacks, and we prove that they can also be seen as Tate schemes, i.e.\ ind-schemes that satisfy the Tate condition. 

\subsection{Descent and Stacks of Tate type}
\label{subsec:descent-and-stacks}

In this section we introduce the notion of a stack of Tate type. In subsection \ref{subsubsec:topologies-Tate-affine} we first have a discussion of topologies on the category of Tate affine schemes. In subsection \ref{subsubsec:descent} we introduce the descent condition and in the remaining subsections we discuss how it interacts with different properties of a prestack of Tate type.

\subsubsection{Topologies on Tate affine schemes}
\label{subsubsec:topologies-Tate-affine}

In this section we specialize the concepts of a faithfully flat (resp.\ \'etale and Zariski) morphisms to Tate affine schemes and show that they can be characterized levelwise maps with the same property for their representations. We then use this more concrete description to check that the collections of those maps induce topologies on the category $\SchaffTate$.

For a map $f:S \ra T$ between Tate affine schemes, recall that by \ref{defn:morphisms-of-Tate-schemes} (1) we say that $f$ is faithfully flat (resp.\ \'etale, Zariski) if the corresponding map as prestacks is faithfully flat (resp.\ \'etale, Zariski).

\begin{prop}
\label{prop:description-of-flat-morphisms-Tate-affine}
For $f:S \ra T$ a map in $\SchaffTate$, the following are equivalent:
\begin{enumerate}[(i)]
    \item the map $f$ is faithfully flat (resp.\ \'etale, Zariski);
    \item there exists presentations $S \simeq \colim_I S_i$ and $T \simeq \colim_J T_j$, such that whenever there are maps $f_{i,j}$ fitting in a diagram
    \begin{equation}
        \label{eq:factorization-of-flat-morphism}
        \begin{tikzcd}
        S_i \ar[r,"f_{i,j}"] \ar[d] & T_j \ar[d] \\
        S \ar[r,"f"] & T
        \end{tikzcd}
    \end{equation}
    the maps $f_{i,j}$ are faithfully flat (resp.\ \'etale, Zariski).
\end{enumerate}
\end{prop}

\begin{proof}
Since all the different conditions have the same formal properties, namely 2-out-of-3 and stable under base change the same proof applies to all three cases. We phrase it for the faithfully flat case for concreteness.

First we prove that (i) implies (ii). Let $T \simeq \colim_J T_j$ be any presentation of $T$, then in the Cartesian diagram
\[
\begin{tikzcd}
    S\underset{T}{\times}T_j \ar[d] \ar[r,"f_{i,j}"] & T_j \ar[d] \\
    S \ar[r,"f"] & T
\end{tikzcd}
\]
the morphism $f_{i,j}$ is faithfully flat. Now we notice that for each $i \ra j$ the induced map
\[
S\underset{T}{\times}T_i \ra S\underset{T}{\times}T_j
\]
is closed and satisfies the Tate condition on the defining ideal. Moreover, since filtered colimits and finite products commute one has an isomorphism
\[
\colim_J S\underset{T}{\times}T_j \simeq S.
\]

Now assume that (ii) holds. Let $U \ra T$ where $U$ is an affine scheme, then for some $j \in J$ one has the factorization
\[
U \ra T_i \ra T.
\]

Thus, the pullback of $U$ to $S$ fits into the morphism
\begin{equation}
    \label{eq:factorization-of-pullback-flat}
    S_j\underset{T_i}{\times}U \ra S\underset{T}{\times}U \ra U
\end{equation}
by a certain presentation of $S$ with the property the composite of (\ref{eq:factorization-of-pullback-flat}) is faithfully flat. Thus, the map
\[
S_j\underset{T_i}{\times}U \ra S\underset{T}{\times}U
\]
is faithfully flat.
\end{proof}

\begin{rem}
\label{rem:description-of-flat-morphisms-Tate-affine-n-coconnective}
The statement of Proposition \ref{prop:description-of-flat-morphisms-Tate-affine} also holds in the category $\SchaffTaten$. Indeed, it is proved in a similar way, but considering pullbacks in the category of $n$-coconnective Tate affine schemes.
\end{rem}

\paragraph{}

\begin{prop}
Consider the category $\SchaffTate$:
\begin{itemize}
    \item \textit{(flat topology)} the collection of faithfully flat morphisms determines a Grothendieck topology;
    \item \textit{(\'etale topology)} the collection of \'etale morphism determines a Grothendieck topology;
    \item \textit{(Zariski topology)} the collection of Zariski morphisms determines a Grothendieck topology.
\end{itemize}
\end{prop}

\begin{proof}
We invoke \cite{SAG}*{Proposition A.3.2.1} which states that a collection $S$ of morphisms on the category $\SchaffTate$ determines a Grothendieck topology if it satisfies conditions (a-d) of \emph{loc.\ cit}.

Condition (a) asks that $S$ contains all equivalences and is stable under compositions. This is clear from both characterizations of Proposition \ref{prop:description-of-flat-morphisms-Tate-affine}.

Condition (b) requires that the class of morphisms $S$ is stable under pullbacks. This is clear from condition (i) in Proposition \ref{prop:description-of-flat-morphisms-Tate-affine}.

Conditions (c) and (d) require that the class of morphisms $S$ is stable under pullbacks and that pushouts in the category $\SchaffTate$ are universal, respectively. 

Assume that
\[
\begin{tikzcd}
S \ar[r] \ar[d,"f"] & S' \ar[d,"f'"] \\
T \ar[r] & T'
\end{tikzcd}
\]
is a pushout of Tate affine schemes and that $f$ is faithfully flat. By the proof of Lemma \ref{lem:push-out-affine-Tate-schemes} one see that for any $U \ra T'$ where $U \in \Schaff$ there exist a cofiltered diagram $I$ and presentations $S \simeq \colim_I S_i$, $S' \simeq \colim_I S'_i$ and $T \simeq \colim_I T_i$ such that one has a factorization
\[
U \ra S'_i\underset{S_i}{\sqcup}T_i
\]
for some $i \in I$. Thus the properties (c) and (d) follow from the same statements for the category $\Schaff$.
\end{proof}

\paragraph{}

\begin{rem}
\label{rem:concrete-description-and-covering-name}
Given a map $f:S' \ra S$ of Tate affine schemes one says that $f$ is a 
\begin{itemize}
    \item \emph{flat covering} if $f$ is faithfully flat;
    \item \emph{\'etale covering} if $f$ is \'etale;
    \item \emph{Zariski covering} if $f$ is Zariski, i.e.\ 
    \[
    S' \simeq \sqcup_I S'_i \ra S
    \]
    with each $S'_i \ra S$ and open embedding and $S' \ra S$ induces a surjective map on the underlying classical schemes. 
\end{itemize}

Given any covering $f: S' \ra S$, one defines its \v{C}ech covering as the groupoid object in $\SchaffTate$ given by
\begin{equation}
    \label{eq:Cech-nerve}
    (S'/S)^n = S'\underset{S}{\times} \cdots \underset{S}{\times}S'
\end{equation}
where one has $(n+1)$ copies of $S$ and the structure maps are given by the natural projections.
\end{rem}

\subsubsection{Descent}
\label{subsubsec:descent}

In this subsection we formulate the descent condition for prestacks of Tate type, this is exactly analogous to the case of usual prestacks.

\paragraph{}

For $\sX$ a prestack of Tate type, we say that $\sX$ satisfies flat (resp.\ \'etale, Zariski) descent if for any flat (resp.\ \'etale, Zariski) covering $f:S' \ra S$, the map
\[
\sX(S) \ra \lim_{[n] \in \Delta^{\rm op}}\sX((S'/S)^n)
\]
is an isomorphism.

\paragraph{}

In the rest of the discussion and the paper when talking about descent we will mean \'etale descent, unless otherwise specified\footnote{One could also do a discussion for one of the other topologies in a very similar way.}.

Let $\StkTate \subset \PStkTate$ denote the subcategory of prestacks that satisfy \'etale descent.

Given a morphism $f:\sX_1 \ra \sX_2$ of prestacks of Tate type, we say that $f$ is an \'etale equivalence if for every $\sY \in \StkTate$ one has an equivalence
\[
\Maps(\sX_2,\sY) \ra \Maps(\sX_1,\sY).
\]

\paragraph{}

The inclusion
\[
\StkTate \hra \PStkTate
\]
admits a left adjoint, so $\StkTate$ is a localization of $\PStkTate$. We denote the localization functor
\[
L: \PStkTate \ra \StkTate \ra \PStkTate.
\]

\begin{lem}
The functor $L$ is left exact, i.e.\ it commutes with finite limits.
\end{lem}

\begin{proof}
This is a consequence that the localization of $\PStkTate$ is a topological localization (see \cite{HTT}*{Corollary 6.2.1.6}), since it is achieved by considering a Grothendieck topology on a category of presheaves (see \cite{HTT}*{Proposition 6.2.2.7}).
\end{proof}

Given a morphism $f:\sX_1 \ra \sX_2$ of prestacks we say that $f$ is an \emph{\'etale surjection}\footnote{This is also refered to as an \emph{effective epimorphism} in the $\infty$-topos of \'etale sheaves on $\SchaffTate$.} if for every $S \in \SchaffTate$ and a point $x_2:S \ra \sX_2$, there exists an \'etale cover $\varphi:S' \ra S$ such that
\[
x_2\circ \varphi: S' \ra \sX_2
\]
belongs to the essential image of $\sX_{1}(S') \ra \sX_2(S')$.

The following follows from \cite{HTT}*{Corollary 6.2.3.5}.

\begin{lem}
\label{lem:descent-of-effective-epis}
For $f:\sX_1 \ra \sX_2$ an \'etale surjection, the induced map
\[
\left|(\sX_1/\sX_2)^{\bullet}\right| \ra \sX_2
\]
is an isomorphism, where $(\sX_1/\sX_2)^{\bullet}$ is the \v{C}ech nerve of $f$.
\end{lem}

As a consequence of Lemma \ref{lem:descent-of-effective-epis} we have

\begin{lem}
For $\sX \in \PStkTate$, the unit of the adjunction
\[
\sX \ra \L(\sX)
\]
is an \'etale surjection.
\end{lem}

\subsubsection{Descent condition and locally a prestack condition}
\label{subsubsec:descent-condition-and-lp}

The following is a completely formal observation that makes checking the descent condition of a prestack of Tate type easier in the case when it is locally a prestack.

\begin{lem}
\label{lem:descent-for-prestack-Tate-lp}
Let $\sX \in \PStkTatelp$, then the following are equivalent
\begin{enumerate}[(i)]
    \item $\sX$ satisfies descent;
    \item $\left.\sX\right|_{(\Schaff)^{\rm op}}$ satisfies descent.
\end{enumerate}
\end{lem}

\begin{proof}
The implication (i) $\Rightarrow$ (ii) is clear, since any \'etale morphism between usual affine schemes is an \'etale morphism between Tate affine schemes.

Let $\sX_0 = \left.\sX\right|_{(\Schaff)^{\rm op}}$, by assumption we have an equivalence
\[
\sX \ra \RKE_{(\Schaff)^{\rm op} \hra (\SchaffTate)^{\rm op}}(\sX_0)
\]
and suppose that $\sX_0$ satisfies descent with respect to \'etale morphism between affine schemes.

Given $S' \ra S$ an \'etale morphism between Tate affine schemes we obtain a canonical morphism
\begin{equation}
    \label{eq:descent-equation-lp-prestack-Tate}
    \sX(S) \ra \Tot(\sX((S'/S)^{\bullet})).
\end{equation}
The right-hand side of (\ref{eq:descent-equation-lp-prestack-Tate}) can be described as 
\[
\lim_{[n] \in \Delta^{\rm op}}\sX((S'/S)^n),
\]
where for each $n$ we have
\[
\sX((S'/S)^n) \simeq \lim_{T_0 \in (\Schaff_{/(S'/S)^n})^{\rm op}} \sX_0(T_0).
\]
We claim that for every $n$ the functor
\begin{align*}
    (\Schaff_{/S})^{\rm op} & \ra (\Schaff_{/(S'/S)^n})^{\rm op} \\
    (S_0 \ra T) & \mapsto S_0 \underset{S}{\times}(S'/S)^n
\end{align*}
is initial. Indeed, since the category $\Schaff_{/S}$ is filtered we only need to check that for every $T_0 \ra (S'/S)^n$ with $T_0 \in \Schaff$ there exists a factorization
\[
\begin{tikzcd}
T_0 \ar[r] \ar[d,"f"] & (S'/S)^n \ar[d] \\
S_0 \ar[r] & S
\end{tikzcd}
\]
with $f$ \'etale. This follows from Proposition \ref{prop:description-of-flat-morphisms-Tate-affine}. Thus we have
\[
\sX((S'/S)^n) \simeq \lim_{S_0 \in (\Schaff_{/S})^{\rm op}} \sX_0(S_0 \underset{S}{\times}(S'/S)^n) \simeq \lim_{S_0 \in (\Schaff_{/S})^{\rm op}} \sX_0((S'_0/S_0)^n),
\]
where $S'_0 = S_0 \underset{S}{\times}S'$. Since limits and totalizations commute we obtain that the right-hand side of (\ref{eq:descent-equation-lp-prestack-Tate}) becomes
\[
\lim_{S_0 \in (\Schaff_{/S})^{\rm op}}\Tot(\sX_0((S'_0/S_0)^n))
\]
which by the descent condition for $\sX_0$ is equivalent to
\[
\lim_{S_0 \in (\Schaff_{/S})^{\rm op}} \sX(S_0) \simeq \sX(S).
\]
\end{proof}

Let $\overline{\L}: \PStk \ra \Stk$ denote the sheafification functor for usual prestacks, the following is a consequence of Lemma \ref{lem:descent-for-prestack-Tate-lp}.

\begin{cor}
\label{cor:sheafification-and-RKE-to-Tate-affine}
For $\sX_0 \in \PStk$ one has an equivalence
\[
\L(\RKE_{\Schaffop \hra \SchaffTateop}(\sX_0)) \overset{\simeq}{\ra} \RKE_{\Schaffop \hra \SchaffTateop}(\overline{L}(\sX_0)).
\]
\end{cor}

\subsubsection{Descent for Tate affine schemes}

The following is an important fact that shows that the theory we are developping has the expected formal properties.

\begin{prop}
\label{prop:Tate-affine-schemes-are-Tate-stacks}
The Yoneda embedding
\[
\SchaffTate \hra \PStkTate
\]
belongs to $\StkTate$. 
\end{prop}

\begin{proof}
Notice that by Lemma \ref{lem:Tate-affine-schemes-are-locally-prestacks} and \ref{lem:descent-for-prestack-Tate-lp} it is enough to check that Tate affine schemes are usual stacks, i.e.\ satisfy descent when considered as element of $\PStk$. This follows from \cite{GR-II}*{Chapter 2, Proposition 1.2.2} since flat descent implies \'etale descent.
\end{proof}

One also has the following result that allows one to identify Tate stacks which are Tate affine schemes.

\begin{prop}
Given a pullback diagram in $\StkTate$ where all objects are prestacks of Tate type locally almost of finite type
\[
\begin{tikzcd}
S' \ar[r,"f"] \ar[d] & S \ar[d] \\
\sY' \ar[r] & \sY
\end{tikzcd}
\]
where $S,S'$ and $\sY'$ are Tate affine schemes and $f$ is an \'etale cover, one has that $\sY$ is a Tate affine scheme.
\end{prop}

The following corollary summarizes how properties of morphisms between prestacks of Tate type imply properties of the corresponding Tate stacks, the are proved as in \cite{GR-I}*{Chapter 2, \S 2.4}.

\begin{cor}
\begin{enumerate}[(a)]
    \item For $f:\sX \ra \sY$ a Tate affine schematic morphism in $\PStkTate$ the corresponding morphism $\L(f): \L(\sX) \ra \L(\sY)$ is also Tate affine schematic;
    \item for $f: \sX \ra \sY$ a flat (resp.\ smooth, \'etale, open embedding, Zariski) morphism in $\PStkTate$ the morphism $\L(f): \L(\sX) \ra \L(\sY)$ is also flat (resp.\ smooth, \'etale, open embedding, Zariski).
\end{enumerate}
\end{cor}

\subsubsection{$n$-coconnective stacks}

This subsection discusses how the \'etale descent condition interacts with restriction and left Kan extension with respect to the inclusion $\SchaffTaten \subset \SchaffTate$.

\paragraph{Sheafification of $n$-coconnective prestacks}

Let $\StkTaten$ denote the subcategory of $\PStkTaten$ consisting of objects that satisfy descent with respect to \'etale covers ${^{\leq n}S'} \ra {^{\leq n}S}$ in $\SchaffTaten$\footnote{A map $f:S_0 \ra T_0$ in $\SchaffTaten$ is said to be \'etale (resp. faithfully flat, Zariski) if it is so when seen as a map in $\SchaffTate$, via the canonical inclusion $\SchaffTaten \subset \SchaffTate$.}.

One obtains that $\StkTaten$ is a localization of $\PStkTaten$, we denote the left-exact localization functor by
\[
{^{\leq n}L}: \PStkTaten \ra \StkTaten \ra \PStkTaten.
\]

The sheafification functor ${^{\leq n}L}$ on truncated objects can be described more explicitly by using \cite{HTT}*{\S 6.5.3}.

Consider the endo-functor on $\PStkTaten$
\begin{equation}
    \label{eq:+-construction-on-prestacks-Tate}
    \sX \mapsto \sX^+,
\end{equation}
where its value on an object $S\in \SchaffTaten$ is given by
\[
\sX^+(S) := \colim_{S' \in ((\SchaffTaten)_{\rm etale \; in \; S})^{\rm op}}\lim_{\Delta^{\rm op}}\sX((S'/S)^{\bullet})
\]
where $(\SchaffTaten)_{\rm etale \; in \; S}$ denotes the category of \'etale covers of $S$.

In particular, if $\sX$ is $(k-2)$-truncated for $k \geq 2$, then one has
\[
{^{\leq n}L}(\sX) \simeq \sX^{+^{k}},
\]
where $\sX^{+^{k}}$ denotes the $k$th iteration of the functor (\ref{eq:+-construction-on-prestacks-Tate}).

Since the category of \'etale cover of $S$ is cofiltered\footnote{Since it admits finite limits.} the colimit in the definition of $\sX^+$ is filtered, thus one has

\begin{lem}
\label{lem:sheafification-preserves-k-connected-objects}
The functor ${^{\leq n}L}: \PStkTaten \ra \PStkTaten$ sends $k$-truncated objects to $k$-truncated objects.
\end{lem}

\paragraph{Truncation of stacks and right Kan extension from coconnective}

The following are tautological from the definitions

\begin{lem}
\begin{enumerate}[(a)]
    \item The restriction functor $\PStkTate \ra \PStkTaten$ sends the category $\StkTate$ to the category $\StkTaten$.
    \item The functor
    \[
    \LKE_{(\SchaffTaten)^{\rm op} \hra (\SchaffTate)^{\rm op}}:\PStkTaten \ra \PStkTate
    \]
    sends \'etale equivalences to \'etale equivalences.
\end{enumerate}
\end{lem}

Notice that the right Kan extension functor along $(\SchaffTaten)^{\rm op} \hra (\SchaffTate)^{\rm op}$
\[
\RKE_{(\SchaffTaten)^{\rm op} \hra (\SchaffTate)^{\rm op}}:\PStkTaten \ra \PStkTate
\]
sends $\StkTaten$ to $\StkTate$. Thus the restriction $\sX \mapsto {^{\leq n}\sX}$ sends \'etale equivalences to \'etale equivalences. Thus, we have
\begin{cor}
For $\sX \in \PStkTate$, one has
\[
{^{\leq n}\L}({^{\leq n}\sX}) \simeq {^{\leq n}\L(\sX)}.
\]
\end{cor}

Similarly, if we have a functor
\[
\sX': (\SchaffTateconv)^{\rm op} \ra \Spc,
\]
which we think of as a compatible family of $\sX'_n \in \PStkTaten$. Let
\[
\sX := \RKE_{(\SchaffTateconv)^{\rm op} \hra (\SchaffTate)^{\rm op}}.
\]

\begin{lem}
\label{lem:check-stack-condition-on-truncations}
If for every $n$ one has $\sX'_n \in \StkTaten$, then $\sX \in \Stk$.
\end{lem}
\begin{proof}
It follows from the description of the right Kan extension from $\SchaffTateconv$ as a limit over the truncations, see Lemma \ref{lem:concrete-condition-convergent-prestack-Tate-type}.
\end{proof}

Thus, we have
\begin{cor}
Suppose that $\sX \in \PStkTate$ belongs to $\StkTate$, then so does ${^{\rm conv}\sX}$.
\end{cor}

\paragraph{The $n$-coconnective condition for a stack}

We note that the functor
\[
\LKE_{(\SchaffTaten)^{\rm op} \hra (\SchaffTate)^{\rm op}}: \PStkTaten \ra \PStkTate
\]
does \emph{not} send the subcategory $\StkTaten$ to $\StkTate$. Instead, we need to consider the composite
\[
\StkTaten \hra \PStkTaten \overset{\LKE_{(\SchaffTaten)^{\rm op} \hra (\SchaffTate)^{\rm op}}}{\ra} \PStkTate \overset{L}{\ra} \StkTate
\]
to obtain the left adjoin to the restriction functor $\StkTate \ra \StkTaten$. We denote this composite by $\LLKE_{(\SchaffTaten)^{\rm op} \hra (\SchaffTate)^{\rm op}}$.

The left adjoint $\LLKE_{(\SchaffTaten)^{\rm op} \hra (\SchaffTate)^{\rm op}}$ is fully faithful, thus we identify $\StkTaten$ with a subcategory of $\PStkTaten$. We will denote by ${^{L}\tau^{\leq n}}: \StkTate \ra \StkTate$ the corresponding localization functor.

\paragraph{}

A stack of Tate type $\sX$ is said to be \emph{$n$-coconnective} if the canonical map
\[
{^{L}\tau^{\leq n}}(\sX) \ra \sX
\]
is an isomorphism. 

In particular, we will say that a stack of Tate type is \emph{classical} if it is $0$-coconnective.

We will also say that a stack of Tate type is \emph{eventually coconnective} if it is $n$-coconnective for some $n$.

\subsubsection{Descent and locally of finite type condition}

In this section we investigate how the condition of being a stack interacts with the locally of finite type condition.

\paragraph{}

Fix an non-negative integer $n$, we consider the \'etale topology on the category $\SchaffTatenft$ induced by \'etale covers. Let
\[
^{\leq n}L_{\rm ft}: \PStkTatenlft \ra \NearStknlft \ra \PStkTatenlft
\]
denote the localization, where $\NearStknlft$ denotes the subcategory of functor $\sX:\SchaffTatenft \ra \Spc$ that satisfy descent with respect to \'etale covers.

Similarly to Lemma \ref{lem:sheafification-preserves-k-connected-objects} one has

\begin{lem}
The functor $^{\leq n}L_{\rm ft}: \PStkTatenlft \ra \PStkTatenlft$ sends $k$-truncated objects to $k$-truncated objects.
\end{lem}

\paragraph{}

The restriction via $\SchaffTatenft \hra \SchaffTaten$ gives a functor
\[
\PStkTaten \ra \PStkTatenlft
\]
that sends $\StkTaten$ to $\NearStknlft$. By adjunction, the left Kan extension
\[
\LKE_{\SchaffTatenft \hra \SchaffTaten}: \PStkTatenlft \ra \PStkTaten
\]
sends \'etale equivalences to \'etale equivalences.

\begin{lem}
\label{lem:RKE-prestacks-Tate-n-lft-to-prestacks-Tate-n}
The functor of right Kan extension
\[
\RKE_{\SchaffTatenft \hra \SchaffTaten}: \PStkTatenlft \ra \PStkTaten
\]
sends $\NearStknlft$ to $\StkTaten$.
\end{lem}

\begin{proof}
The proof is exactly as in \cite{GR-I}*{Chapter 2, Lemma 2.7.4}.
\end{proof}

\paragraph{}

The following is a consequence of Lemma \ref{lem:RKE-prestacks-Tate-n-lft-to-prestacks-Tate-n}

\begin{cor}
\label{cor:restriction-to-finite-type-preserves-etale-equivalences}
\begin{enumerate}[(1)]
    \item The restriction functor $\PStkTaten \ra \PStkTatenlft$ sends \'etale equivalences to \'etale equivalences.
    \item For $\sX \in \PStkTaten$ one has:
    \[
    {^{\leq n}L}(\left.\sX)\right|_{\SchaffTatenft} \simeq {^{\leq n}L}_{\rm ft}(\left.\sX\right|_{\SchaffTatenft})
    \]
\end{enumerate}
\end{cor}

\paragraph{}

\begin{prop}
\label{prop:near-stack-n-LKE-to-stack-n}
Suppose that $\sX$ an object of $\PStkTatenlft$ is $k$-truncated for some $k$, and that $\sX \in \NearStknlft$, then
\[
\LKE_{\SchaffTatenft \hra \SchaffTaten}(\sX) \in \StkTaten.
\]
\end{prop}

Before giving the proof of Proposition \ref{prop:near-stack-n-LKE-to-stack-n}, we prove an auxiliary result. Consider $f:S_1 \ra S_2$ an \'etale morphism in $\SchaffTaten$, consider the category $f_{\rm ft}$ whose objects are commutative diagrams
\[
\begin{tikzcd}
S_{1} \ar[d,"f"'] \ar[r] & S'_1 \ar[d,"f'"] \\
S_2 \ar[r] & S'_2
\end{tikzcd}
\]
where $S'_1,S'_2 \in \SchaffTatenft$ and $f'$ is \'etale, and morphisms are the obvious triangles of commutative diagrams. We have forgetful functors
\begin{equation}
    \label{eq:forget-f-ft}
    \{S_2 \ra S'_2 \; | \; S'_2 \in \SchaffTatenft \} \la f_{\rm ft} \ra \{S_1 \ra S'_1 \; | \; S'_1 \in \SchaffTatenft \}.
\end{equation}

\begin{lem}
\label{lem:initial-functors-near-stacks}
Both forgetful functors in (\ref{eq:forget-f-ft}) are initial.
\end{lem}

\begin{proof}
We first prove that the functor
\[
f_{\rm ft} \ra \{S_1 \ra S'_1 \; | \; S'_1 \in \SchaffTatenft \}
\]
is initial. 

Notice that both categories are filtered, since they admit fiber products by Lemma \ref{lem:n-coconnective-Tate-affine-schemes-of-ft-has-fiber-products}. Hence it is enough to prove that for any $f:S_1 \ra S'_1$, there exists $S''_1 \ra S''_2$ an \'etale morphism in the category $\SchaffTatenft$ and morphisms
\begin{equation}
    \label{eq:initial-diagram-for-right-forgetful}
    \begin{tikzcd}
    S_{1} \ar[r] & S'_1 \\
    S_{1} \ar[u,"\id_{S_1}"'] \ar[d,"f"'] \ar[r] & S''_1 \ar[d,"f'"] \ar[u] \\
    S_2 \ar[r] & S''_2.
    \end{tikzcd}
\end{equation}

We prove this result by induction on $n$.

\textit{Case $n=0$.} we notice that by Remark \ref{rem:description-of-flat-morphisms-Tate-affine-n-coconnective} if we simultaneously straighten the morphisms $f:S_1 \ra S_2$ and $S_1 \ra S''_1$ one has presentations
\[
S_{1} \simeq \colim_I S_{1,i}, \;\;\; S_{2} \simeq \colim_I S_{2,i}, \;\;\; \mbox{and} \;\;\; S''_1 \simeq \colim_I S''_{1,i},
\]
where $S_{1,i},S_{2,i} \in \Schaffn$ and $S''_{1,i} \in \Schaffnft$. By the usual result in algebraic geometry, we can form the diagram
\begin{equation}
    \label{eq:initial-diagram-for-right-for-affines}
    \begin{tikzcd}
    S_{1,i} \ar[r] & S'_{1,i} \\
    S_{1,i} \ar[u,"\id_{S_{1,i}}"'] \ar[d,"f_i"'] \ar[r] & S''_{1,i} \ar[d,"f'_i"] \ar[u] \\
    S_{2,i} \ar[r] & S''_{2,i}.
    \end{tikzcd}
\end{equation}
for every $i \in I$. By considering the filtered colimit of (\ref{eq:initial-diagram-for-right-for-affines}) we obtain the diagram (\ref{eq:initial-diagram-for-right-forgetful}).

\textit{Inductive case.} we give the same argument as in \cite{GR-I}*{Chapter 2, \S Lemma 2.8.2}, we will need a result in describing square-zero extensions of Tate affine schemes.

We claim that the functor
\[
f_{\rm ft} \ra \{S_2 \ra S'_2 \; | \; S'_2 \in \SchaffTatenft \}
\]
is initial. 

Again this is done by induction. 

\textit{Case $n=0$.} We are given $S_2 \ra S'_2$ a morphism of classical Tate affine schemes, where $S'_2$ is of finite type, and $S_1 \ra S_2$ an \'etale morphism of classical Tate affine schemes. We claim that the composite 
\[
S_1 \ra S_2 \ra S'_2
\]
factors via $S'_1 \in \clSchaffTateft$. Indeed, by considering presentations of $S_1, S_2$ and $S'_2$ above this reduces to the classical case. Moreover, we can pick $S'_1$ such that the induced morphism $S'_1 \ra S'_2$ is \'etale.

\textit{Inductive step.} We notice that the diagram of categories
\[
\begin{tikzcd}
f_{\rm ft} \ar[r] \ar[d] & \classical{f_{\rm ft}} \ar[d] \\
\{S_2 \ra S'_2 \; | \; S'_2 \in \SchaffTatenft \} \ar[r] & \{\classical{S_2} \ra {^{0}S'_2} \; | \; {^{0}S'_2} \in \clSchaffTateft \}
\end{tikzcd}
\]
is a pullback square. Indeed, recall that $f': S'_1 \ra S'_2$ is an \'etale morphism if $f'$ is flat and 
\[
\classical{f'}: \classical{S'_1} \ra \classical{S'_2} 
\]
is \'etale. Thus, given ${^{0}f'}: {^{0}S'_{2}} \ra {^{0}S'_1}$ an \'etale morphism between classical Tate affine scheme and $S'_2 \in \SchaffTatenft$ such that
\[
\classical{S'_2} \simeq {^{0}S'_2}
\]
there exits a unique $f': S'_1 \ra S'_2$ such that
\[
\classical{f'} \simeq {^{0}f'}.
\]
Namely,
\[
S'_1 \simeq S'_2\underset{{^{0}S'_2}}{\times}{^{0}S'_1}
\]
and $f'$ is the pullback of ${^{0}f'}$, which is then \'etale.
\end{proof}

\begin{proof}[Proof of Proposition \ref{prop:near-stack-n-LKE-to-stack-n}]
Consider $\sX'$ an object of $\NearStknlft$ and let $\sX$ denote its left Kan extension to $\PStkTaten$, given $S_1 \ra S_2$ an \'etale morphism we need to check that the map
\begin{equation}
    \label{eq:descent-condition-near-stack}
    \sX(S_2) \ra \Tot(\sX((S_1/S_2)^{\bullet}))
\end{equation}
is an isomorphism.

By definition of left Kan extension for $S \in \SchaffTaten$ one has
\[
\sX(S) \simeq \colim_{S' \in ((\SchaffTatenft)_{S/})^{\rm op}}\sX'(S').
\]
Notice, that by simultaneously straightening diagrams for $S'$ and $S$, by \cite{GR-I}*{Chapter 2, Theorem 1.5.3 (b)} we have that the category $((\SchaffTatenft)_{S/})^{\rm op}$ is filtered. Thus, if $\sX'$ is $k$-truncated, then so is $\sX$.

So we can replace the totalization in (\ref{eq:descent-condition-near-stack}) by a limit in the category of $k$-truncated spaces, which we will denote by $\Tot^{\leq k}$. The left-hand side of (\ref{eq:descent-condition-near-stack}) becomes
\[
\colim_{S'_2 \in ((\SchaffTatenft)_{S_2/})^{\rm op}}\sX'(S'_2).
\]

By Lemma \ref{lem:initial-functors-near-stacks} the morphism $(f_{\rm ft})^{\rm op} \ra ((\SchaffTatenft)_{S_1/})^{\rm op}$ is cofinal, thus we rewrite the right-hand side of (\ref{eq:descent-condition-near-stack}) as
\[
\Tot^{\leq k}\left(\colim_{(f_{\rm ft})^{\rm op}}\sX'((S'_{1}/S'_2)^{\bullet})\right).
\]
The category $(f_{\rm ft})^{\rm op}$ is filtered, since it contains push-outs, and since $\Tot^{\leq k}$ is a finite limit we can commute the limit and colimit to obtain
\[
\colim_{(f_{\rm ft})^{\rm op}}\Tot^{\leq k}\left(\sX'((S'_{1}/S'_2)^{\bullet})\right).
\]

Since $\sX'$ satisfies descent with respect to \'etale morphism between Tate affine schemes of finite type, we have
\[
\colim_{(f_{\rm ft})^{\rm op}}\Tot^{\leq k}\left(\sX'((S'_{1}/S'_2)^{\bullet})\right) \simeq \colim_{(f_{\rm ft})^{\rm op}}\sX'(S'_2).
\]

Finally, by Lemma \ref{lem:initial-functors-near-stacks} the morphism $(f_{\rm ft})^{\rm op} \ra ((\SchaffTatenft)_{S_2/})^{\rm op}$ is cofinal, thus
\[
\colim_{(f_{\rm ft})^{\rm op}}\sX'(S'_2) \simeq \colim_{((\SchaffTatenft)_{S_2/})^{\rm op}}\sX'(S'_2).
\]

This finishes the proof of the proposition.
\end{proof}

\subsubsection{Tate Stacks locally almost of finite type}

Recall that the category $\PStkTatelaft$ can be identified with the category of functors
\[
\Fun((\SchaffTateconvft)^{\rm op},\Spc)
\]
by Proposition \ref{prop:prestacks-Tate-laft-as-functors}.

If we consider the \'etale topology on the category $\SchaffTateconvft$ we can define a localization of $\PStkTatelaft$ that we denote $\NearStklaft$. Let $\L_{\rm laft}$ dentote the localization functor
\[
\PStklaft \ra \NearStklaft \ra \PStklaft.
\]

We define the category of Tate stacks locally almost of finite type as
\[
\StkTatelaft := \StkTate \cap \PStkTatelaft.
\]

The restriction $\PStkTate \ra \PStkTatelaft$ sends the category $\StkTate$ to the category $\NearStklaft$. Moreover, as in Corollary \ref{cor:restriction-to-finite-type-preserves-etale-equivalences} we have

\begin{lem}
For $\sX \in \PStkTate$ one has
\[
\left.\L(\sX)\right|_{\SchaffTateconvft} \simeq \L_{\rm laft}(\left.\sX\right|_{\SchaffTateconvft}).
\]
\end{lem}

Also, as a consequence of Proposition \ref{prop:near-stack-n-LKE-to-stack-n} and Lemma \ref{lem:check-stack-condition-on-truncations} we have

\begin{cor}
For $\sX$ an object of $\NearStklaft$, thought as a prestack of Tate type via the inclusion
\[
\NearStklaft \subset \PStkTatelaft \subset \PStkTate
\]
if for each $n$ the restriction ${{\leq n}\sX}$ of $\sX$ to $\SchaffTatenft$ is $k_{n}$-truncated for some $k_n \in \bN$, then $\sX \in \NearStklaft$.
\end{cor}

\subsection{Tate Artin stacks}
\label{subsec:Tate-Artin-stacks}

In this section we introduce the notion of Artin $k$-stack of Tate type, that we will simply call Tate Artin $k$-stacks. The treatment follows closely \cite{GR-I}*{Chapter 2, \S 4} and the only novel feature is in understanding the relation between Tate Artin $k$-stacks and the condition of locally a prestack.

\subsubsection{Setting up Artin stacks of Tate type}
\label{subsubsec:setting-up-Artin}

For every $k \geq 0$ we will define a subcategory $\ArtinStkTaten{k}$ of $\StkTate$ that we will refer to as Tate $k$-Artin stacks, or $k$-Artin stacks of Tate type.

\paragraph{Tate $0$-Artin stacks}

For $k=0$, we say that an object $\sX \in \StkTate$ is a Tate $0$-Artin stack if
\[
\sX \simeq L(\sqcup_{I}S_i)
\]
where $S_i \in \SchaffTate$. We let
\[
\ArtinStkTaten{0}
\]
denote the category of Tate $0$-Artin stack.

\paragraph{Tate $(k-1)$-Artin stacks}

We define Tate $k$-Artin stacks inductively for $k \geq 1$. We will also define what it means for a morphism in $\PStkTate$ to be $k$-representable, and for a $k$-representable morphism what it means to be flat (resp.\, smooth, \'etale and surjective). In the case of $n=0$ these notions are a particular case of Definition \ref{defn:morphisms-of-Tate-schemes}.

Thus, our inductive definition assumes the following properties:
\begin{itemize}
    \item any Tate $(k-1)$-Artin stack is a Tate $k$-Artin stack;
    \item any morphism that is $(k-1)$-representable is $k$-representable;
    \item a $(k-1)$-representable morphism is flat (resp.\, smooth, \'etale and surjective) if and only if it is such when viewed as a $k$-representable morphism;
    \item the class of $k$-representable morphisms (resp. $k$-representable and flat/smooth/\'etale/surjective) is stable under compositions and base changes.
\end{itemize}

It will follow inductively from the definition that the class of Tate $k$-Artin stacks is closed under fiber products.

\paragraph{Inductive step}

Suppose that the notions of the paragraph above were defined for $k' < k$. We will say that $\sX \in \StkTate$ is a \emph{Tate $k$-Artin stack} if the following conditions hold:
\begin{enumerate}[(a)]
    \item the diagonal map $\sX \ra \sX\times \sX$ is $(k-1)$-representable;
    \item there exists $\sZ \in \ArtinStkTaten{(k-1)}$ and a map $f: \sZ \ra \sX$, which is $(k-1)$-representable, which is smooth and surjective.
\end{enumerate}
We will call a pair $\sZ \in \ArtinStkTaten{(k-1)}$ and $f:\sZ \ra \sX$ smooth and surjective an atlas of $\sX$. Notice that by induction one can also find an atlas where $\sZ \in \ArtinStkTaten{0}$.

For a morphism $f:\sX \ra \sY$ in $\PStkTate$ we say that $f$ is \emph{$k$-representable} if for every $S \ra \sY$ with $S \in \SchaffTate$ the fiber product $S \underset{\sY}{\times}\sX$ is a Tate $k$-Artin stack.

Given a morphism $f:\sX \ra S$ from a Tate $k$-Artin stack $\sX$ to a Tate affine scheme $S$, we will say that $f$ is flat (resp.\, smooth, \'etale and surjective) if for some atlas $\sZ \ra \sX$ the induced map $\sZ \ra S$ is flat (resp.\, smooth, \'etale and surjective).

Finally, for a $k$-representable morphism $f:\sX \ra \sY$ between prestacks of Tate type, we say that $f$ is flat (resp.\, smooth, \'etale and surjective) if for every $S \ra \sY$ with $S \in \SchaffTate$ the induced morphism $S \underset{\sY}{\times}\sX \ra S$ is flat (resp.\, smooth, \'etale and surjective). 

\paragraph{Tate Artin stacks}

We will say that $\sX$ is a \emph{Tate Artin stack}, if it is a Tate $k$-Artin stack for some $k \geq 0$.

\paragraph{Quasi-compactness and quasi-separatedness}

For a Tate $k$-Artin stack $\sX$ we will say that $\sX$ is \emph{quasi-compact} if there exists a smooth atlas $S \ra \sX$ where $S \in \SchaffTate$.

For a $k$-morphism $f:\sX_1 \ra \sX_2$ in $\PStkTate$, we say that $f$ is \emph{quasi-compact}, if its base change by a Tate affine scheme yields a quasi-compact Tate $k$-Artin stack.

For $0 \leq k' \leq k$, we define the notion of $k'$-quasi-separatedness of a Tate $k$-Artin stack or a $k$-representable morphism inductively on $k'$.

We say that a Tate $k$-Artin stack $\sX$ is \emph{$0$-quasi-separated} if the diagonal map $\sX \ra \sX\times \sX$ is quasi-compact, as a $(k-1)$-representable map. We say that a $k$-representable map is $0$-quasi-separated if its base change by a Tate affine scheme yields a $0$-quasi-separated Tate $k$-Artin stack.

For $k' > 0$, we say that a Tate $k$-Artin stack $\sX$ is $k'$-quasi-separated if the diagonal map $\sX \ra \sX \times \sX$ is $(k'-1)$-quasi-separated, as a $(k-1)$-representable map. We shall say that a $k$-representable map is $k'$-quasi-separated if its base change by a Tate affine scheme yields a $k'$-quasi-separated Tate $k$-Artin stack.

We shall say that a Tate $k$-Artin stack is quasi-separated if its $k'$-quasi-separated for all $k'$, $0 \leq k' \leq k$. We shall say that a $k$-representable map is quasi-separated if its base change by a Tate affine scheme yields a quasi-separated Tate $k$-Artin stack.

\subsubsection{Verification of the inductive hypothesis}

For any $k \geq 0$ the class of $k$-representable is clearly stable under base change. Moreover, we also have
\begin{lem}
\begin{enumerate}[(a)]
    \item For any $f:\sX \ra sY$ in $\PStkTate$ a $k$-representable morphism, the diagonal map
    \[
    \sX \ra \sX\underset{\sY}{\times}\sX
    \]
    is $(k-1)$-representable.
    \item Any morphism between Tate $k$-Artin stacks is $k$-representable.
\end{enumerate}
\end{lem}

\begin{proof}
For (a) we notice that for any $S \ra \sY$ one has an isomorphism
\[
\sX \underset{\sX\underset{\sY}{\times}\sX}{\times} S \simeq (\sX\underset{\sY}{\times}S)\underset{(\sX\underset{\sY}{\times}S)\times( \sX\underset{\sY}{\times}S)}{\times}S,
\]
thus the claim follows from $\sX\underset{\sY}{\times}S$ being a Tate $k$-Artin stack.

For (b) for instance we notice that
\[
\sX \underset{\sX\underset{\sY}{\times}\sX}{\times} S \simeq (\sX\underset{\sY}{\times}S)\underset{(\sX\underset{\sY}{\times}S)\times( \sX\underset{\sY}{\times}S)}{\times}S
\]
implies that the diagonal of $\sX\underset{\sY}{\times}S$ is a Tate $(k-1)$-Artin. An atlas for $\sX\underset{\sY}{\times}S$ can be produce by taking the pullback of an atlas of $\sX$.
\end{proof}

The following is proved exactly as \cite{GR-I}*{Chapter 2, Proposition 4.2.4}.

\begin{prop}
For $f: \sX \ra \sY$ a $k$-representable morphism, if $\sY$ is a Tate $k$-Artin stack, then $\sX$ is also a Tate $k$-Artin stack.
\end{prop}

We also have that the composition of $k$-representable and flat (resp.\, smooth, \'etale and surjective) morphisms is again $k$-representable and flat (resp.\, smooth, \'etale and surjective). Indeed, this is proved as \cite{GR-I}*{Chapter 2, Proposition 4.2.6}.

\subsubsection{Descent}

The following is a consequence of the definitions.

\begin{cor}
For $\sX$ a Tate $k$-Artin stack and $f: \sZ \ra \sX$ a smooth atlas, the natural map
\[
\L(|\sZ^{\bullet}/\sX|) \simeq \sX
\]
is an isomorphism. Moreover, for any $n$ the restriction ${^{\leq n}\sX} \in \PStkTaten$ is $(n+k)$-truncanted.
\end{cor}

Given $\sX^{\bullet}$ a groupoid-object in $\StkTate$ and let
\[
\sX := \L(|\sX^{\bullet}|).
\]

These results are proved exactly as in \cite{GR-I}*{Chapter 2, \S 4.3}.

\begin{cor}
\begin{enumerate}[(a)]
    \item If $\sX^1$ and $\sX^0$ are Tate $k$-Artin stacks, and the maps $\sX^1 \rightrightarrows \sX^0$ are smooth and the map $\sX^1 \ra \sX^0 \times \sX^0$ is $(k-1)$-representable, then $\sX$ is a $k$-Artin stack.
    \item If $\sX \in \StkTate$, and $f:\sZ \ra \sX$ is $(k-1)$-representable, smooth and surjective, where $\sZ$ is a Tate $k$-Artin stack, then $\sX$ is a Tate $k$-Artin stack.
\end{enumerate}
\end{cor}

\subsubsection{Usual properties of Tate Artin stacks}

\paragraph{$n$-coconnectivity}

By considering $\SchaffTaten$ as the initial objects in the inductive definition of \S \ref{subsubsec:setting-up-Artin} we obtain the category $\tArtinStkTaten{k}{n}$ of $n$-coconnective Tate $k$-Artin stacks.

The restriction functor under $\SchaffTaten \hra \SchaffTate$ sends the category $\ArtinStkTaten{k}$ to $\tArtinStkTaten{k}{n}$. We claim that the same proof as in \cite{GR-I}*{Chapter 2, Proposition 4.4.3} gives the result

\begin{cor}
The functor
\[
\LLKE_{\SchaffTatenop \hra \SchaffTateop}: \StkTaten \ra \StkTate
\]
sends the category $\tArtinStkTaten{k}{n}$ to the category $\ArtinStkTaten{k}$, and realizes and equivalence between $\tArtinStkTaten{k}{n}$ and Tate $k$-Artin stacks which are $n$-coconnective as a Tate stack\footnote{Recall this is different than requiring that they are $n$-coconnective as prestacks of Tate type}.
\end{cor}

\paragraph{Convergence}

The following result guarantees that Tate $k$-Artin stacks are convergent, which is important because this allows use to define their pro-cotangent complexes as we will see in \S \ref{subsec:pro-cotangent-complex}.

\begin{prop}
Any Tate $k$-Artin stack $\sX$ is convergent as a prestack of Tate type.
\end{prop}

\begin{proof}
The proof is exactly as in \cite{GR-I}*{Chapter 2, Proposition 4.4.9}.
\end{proof}

\paragraph{Locally almost of finite type condition}

Similarly to \cite{GR-I}*{Chapter 2, Proposition 4.5.2} one can prove that the set of conditions is equivalent. 

\begin{prop}
\label{prop:equivalent-condition-laft-Tate-Artin-stack}
Let $\sX$ be an object of $\ArtinStkTaten{k}$, the following conditions are equivalent:
\begin{enumerate}[(i)]
    \item $\sX \in \StkTatelaft$;
    \item $\sX$ admits an atlas $f:Z \ra \sX$ with $Z\in \ArtinStkTaten{0} \cap \StkTatelaft$;
    \item for a $k$-representable ppf morphism $Z \ra \sX$ with $Z \in \ArtinStkTaten{0}$, we have $Z \in \ArtinStkTaten{0} \cap \StkTatelaft$;
    \item for a $k$-representable ppf morphism $\sX' \ra sX$, we have $\sX' \in \StkTatelaft$
\end{enumerate}
\end{prop}

We refer to any Tate $k$-Artin stack $\sX$ satisfying the equivalent conditions of Proposition \ref{prop:equivalent-condition-laft-Tate-Artin-stack} as a \emph{Tate $k$-Artin stack locally almost of finite type}.

\subsubsection{Locally a prestack condition}

The following result investigates the relation between the locally a prestack condition and $k$-Artin condition.

\begin{prop}
For $\sX$ an object of $\PStkTatelp$, and let $\sX_0$ denote the restriction of $\sX$ to an object of $\PStk$. If $\sX_0$ for some $k$ $\sX_0$ is a $k$-Artin stack, then $\sX$ is a Tate $k$-Artin stack.
\end{prop}

\begin{proof}
We check the result by induction on $k$.

For $k=0$, we have 
\[
\sX_0 \simeq \overline{L}(\sqcup_{I}S_{0,i})
\]
for $S_{0,i} \in \Schaff$. Thus, one has
\[
\RKE_{\Schaffop \hra \SchaffTateop}(\sX_0) \simeq \L(\RKE_{\Schaffop \hra \SchaffTateop}(\sqcup_{I} S_{0,i}))
\]
by Corollary \ref{cor:sheafification-and-RKE-to-Tate-affine}. Notice that the limit computing the right Kan extension is over a connected category, so it commutes with coproducts and can be rewritten as
\[
\L(\RKE_{\Schaffop \hra \SchaffTateop}(\sqcup_{I} S_{0,i})) \simeq \L(\sqcup_{I}\RKE_{\Schaffop \hra \SchaffTateop}(S_{0,i})).
\]

Finally, we notice that tautologically 
\[
\RKE_{\Schaffop \hra \SchaffTateop}(S_{0,i}) \simeq \imath(S_{0,i})
\]
where $\imath: \Schaff \hra \SchaffTate$. This proves the base case.

Now suppose that we proved the result for Tate $(k-1)$-Artin stacks and also Tate $(k-1)$-representable morphisms. Given $\sX$ we need to check the conditions:

\textit{(a)} the morphism $\sX \ra \sX\times \sX$ is $(k-1)$-representable. Let $S \in \SchaffTate$ we notice that
\[
\sX \underset{\sX\times \sX}{\times}S \simeq \RKE_{\Schaffop \hra \SchaffTateop}(\sX_0\underset{\sX_\times \sX_0}{\times} \overline{S})
\]
where $\overline{S} := \left.S\right|_{\Schaffop}$, where we see $S$ as a prestack of Tate type. Since $\sX_0$ is assumed to be a $k$-Artin stack, we have that $\sX_0\underset{\sX_\times \sX_0}{\times} \overline{S}$ is a $(k-1)$-Artin stack and the inductive hypothesis gives that $\sX \underset{\sX\times \sX}{\times}S$ is a Tate $(k-1)$-Artin stack.

\textit{(b)} there exists an atlas $f: \sZ \ra \sX$. Consider $f_0: \sZ_0 \ra \sX_0$ an atlas of the $k$-Artin stack $\sX_0$, i.e.\ $f_0$ is $(k-1)$-representable, smooth, surjective and $\sZ_0$ is a $(k-1)$-Artin stack. We notice that the right Kan extension of $f_0$ gives
\[
f := \RKE_{\Schaffop \hra \SchaffTateop}(f_0): \sZ \ra \sX,
\]
where $\sZ = \RKE_{\Schaffop \hra \SchaffTateop}(\sZ_0)$ is a Tate $(k-1)$-Artin stack and $f$ is Tate $(k-1)$-representable by the inductive hypothesis. The condition that $f$ is smooth and surjective is clear.
\end{proof}

\subsection{Zariski Tate stacks}
\label{subsec:Zariski-Tate-stacks}

Whereas the approach to Tate Artin stacks of \S \ref{subsec:Tate-Artin-stacks} follows exactly the formal structure of \cite{GR-I}*{Chapter 2, \S 4}, the discussion of `Tate schemes' is slightly more subtle. The main point is that we already have a definition of Tate schemes, see \S \ref{subsec:Tate-schemes}. The idea of this section is to define a different category, that we call Zariski Tate stacks, as a subcategory of stacks of Tate type which admit a Zariski cover. The main result of this section is that one has a fully faithful embedding of Zariski Tate stacks into Tate schemes.

\subsubsection{Definition of Zariski Tate stacks}

Recall that a map of prestacks of Tate type $f:\sX \ra \sY$ is an affine open embedding (see \ref{defn:flat-morphisms-prestacks-Tate}) if for every $S \in (\SchaffTate)_{\sY}$ the pullback is an affine Tate scheme and the induced map
\[
\sX \underset{\sY}{\times}S \ra S
\]
is an open embedding.

\paragraph{}

We will say that $X \in \PStkTate$ is a \emph{Zariski Tate stack} if:
\begin{enumerate}[(a)]
    \item $X$ satisfies \'etale descent;
    \item the diagonal map $X \ra X\times X$ is affine schematic, and for every $T \in (\SchaffTate){/(X\times X)}$, the induced map
    \[
    \classical{T\underset{X\times X}{\times} X} \ra \classical{T}
    \]
    is a closed embedding;
    \item there exists a collection of Tate affine schemes $S_i$ and maps $f_i: S_i \ra X$ (called a \emph{Zariski atlas}), such that:
    \begin{itemize}
        \item each $f_i$ is an open embedding;
        \item for every $T \in (\SchaffTate)_{/X}$, the images of the maps
        \[
        \classical{T\underset{X}{\times}S_i} \ra \classical{T}
        \]
        cover $\classical{T}$.
    \end{itemize}
\end{enumerate}

We will denote by
\[
\StkZarTate \subset \StkTate
\]
the subcategory of almost Tate scheme.

We define a \emph{quasi-compact} Zariski Tate stack as one that admits a cover $\sqcup_I S_i$, where $I$ is finite.

It follows from Proposition \ref{prop:description-of-flat-morphisms-Tate-affine} that if $(S_i \ra X)_{I}$ is a Zariski atlas then the morphism
\[
\sqcup_I S_i \ra X
\]
is an \'etale surjection. Thus, Lemma \ref{lem:descent-of-effective-epis} implies that

\begin{lem}
For $X$ a Zariski Tate stack and $(S_i \ra X)_{I}$ a Zariski atlas one has an equivalence
\[
\L(\left|(\sqcup_I S_i/X)^{\bullet}\right|) \ra X.
\]
\end{lem}

\subsubsection{Comparison with Tate schemes}

The following result shows that quasi-compact Zariski Tate stacks can be embedded into the category of Tate schemes. Let $\StkZarTateqc$ denote the category of quasi-compact Zariski Tate stacks.

\begin{thm}
There is a fully faithful functor
\[
\jmath: \StkZarTateqc \ra \SchTate.
\]
\end{thm}

\begin{proof}
Let $X \in \StkZarTateqc$ and consider $\{S_i \ra X\}_I$ a Zariski cover of $X$. Since $I$ is finite, one can find a filtered diagram $K$ and $S_{i,k}$ such that 
\[
\bigsqcup_I S_i \simeq \colim_K \bigsqcup_{I}S_{i,k}.
\]

For each $k \in K$ one defines
\[
X_k := L(\left|(\sqcup_{I}S_{i,k})^{\bullet}/X\right|_{\rm PreStk}),
\]
which by \cite{GR-I}*{Chapter 2, Lemma 3.1.6} is a quasi-compact scheme. Thus, we let
\[
\jmath(X) := \colim_k X_k.
\]

Now, we need to check that $\jmath(X)$ is a Tate scheme. In the following we will argue that we can use the filtered diagram $\{X_{k}\}$ to produce a new diagram that exhibits $\jmath(X)$ as a Tate scheme.

Let $k \ra \ell$ be a morphism in $K$ consider
\[
f_{k,\ell}: X_k \ra X_{\ell}.
\]
By considering the induced map on any affine scheme mapping to $X_{\ell}$ we notice that $f_{k,\ell}$ is an embedding, thus by Nagata's theorem one has a factorization of $f_{k,\ell}$ as follows
\[
X_{k} \overset{\jmath_{k,\ell}}{\ra} \overline{X_k} \overset{\imath_{k,\ell}}{\ra} X_{\ell},
\]
where $\jmath_{k,\ell}$ is open and $\imath_{k,\ell}$ is closed. Now we notice that 
\[
\Fib\left(\sO_{X_{\ell}} \ra (f_{k,\ell})_{*}\sO_{X_k}\right) \in \Coh(X_{\ell}).
\]
Indeed, one has that $\Fib\left(\sO_{X_{\ell}} \ra (f_{k,\ell})_{*}\sO_{X_k}\right) \in \QCoh(X_{\ell})$ so it is enough to check that for any affine scheme $h:T \ra X_{\ell}$ one has
\[
h^{*}\Fib\left(\sO_{X_{\ell}} \ra (f_{k,\ell})_{*}\sO_{X_k}\right) \in \Coh(T).
\]
Equivalently, it is enough to check this for $T = S_{i,\ell}$ for any $i \in I$. In other words we are considering
\[
S_{i,\ell}\underset{X_{\ell}}{\times}X_{k} \ra X_k.
\]
By taking the intersection of the above with $S_{j,k}$ for $j \in I$ one has
\[
g_{ij,k}: S_{j,k}\cap (S_{i,\ell}\underset{X_{\ell}}{\times}X_{k}) \ra X_k,
\]
which is equivalent to the pullback of $S_{j,k}\cap S_{i,\ell} \ra S_{j,\ell}$ via $S_{j,k} \ra S_{j,\ell}$. Since the later map is a closed embedding satisfying the coherent ideal condition the same holds for $g_{ij,k}$ and hence one has that $h^{*}\Fib\left(\sO_{X_{\ell}} \ra (f_{k,\ell})_{*}\sO_{X_k}\right)$ is actually coherent.

Since 
\[
(f_{k,l})_* \simeq (\imath_{k,\ell})_* \circ (\jmath_{k,\ell})_{*}
\]
and the functor $(\imath_{k,\ell})_*$ preserves coherent objects since $\imath_{k,\ell}$ is a closed embedding, one obtains that $(\jmath_{k,\ell})_*\sO_{X_{k}}$ is coherent. 

Now we pick an arbitray $k_0 \in K$. We notice that $K_{\geq k_0} \ra K$ is cofinal in $K$, i.e.
\[
\colim_{K_{\geq k_0}}X_{k} \simeq \colim_{K}X_k.
\]

Now there are two cases: (1) $X_k$ stabilize for certain $k$ sufficiently large, in this case one has that $X \in \Sch$, so in particular $X \in \SchTate$; or (2) for any $k \in K$ there exists $\ell$ such that the map $f_{k,\ell}$ is not an isomorphism. Then starting with $k_0 \in K$ one defines an inductive diagram $K'_{\geq k_0}$ by $k'_0 = k_0$, then
\[
k'_{i+1} := \min\{\ell \geq k_{i} \; | \; \overline{X_{k_i}} \hra X_{k_{\ell}}\}.
\]

We also define a new functor $\overline{X}: K'_{\geq k_0} \ra \Sch$ by
\[
\overline{X}_{\ell} := \overline{X_{\ell}},
\]
where we notice that by definition one has
\begin{itemize}
    \item $\colim_{K'_{\leq k_0}}\overline{X}_{k'} \simeq \colim_{k}X_k$; 
    \item for every $k' \ra \ell'$ in $K'_{\geq k_0}$ the map
    \[
    \overline{X}_{k'} \ra \overline{X}_{\ell'}
    \]
    is a closed embedding with coherent ideal of definition.
\end{itemize}

Thus, the diagram $\{\overline{X}_{k'}\}_{K'_{\geq k_0}}$ exhibits $\jmath(X)$ as a Tate scheme.

The functor $\jmath$ is fully faithful, since it is enough to check it at the level of Tate stacks, where $\jmath(X)$ and $X$ agree by construction.
\end{proof}

\subsubsection{Comparison with Tate Artin stacks}

The following is a comparison between the notion of Zariski Tate stacks and Tate Artin stacks.

\begin{prop}
Given $Z \in \StkZarTate$, then $Z$ is a Tate $1$-Artin stack.
\end{prop}

\begin{proof}
By condition (i), $Z$ is a Tate stack. So we need to check that $Z \ra Z \times Z$ is $0$-representable, that is Tate affine schematic, which follows from condition (ii). And that we can find a smooth and surjective cover by a disjoint union of Tate affine schemes. This follows from considering a Zariski cover $\{U_i \ra Z\}$ of $Z$ and taking
\[
f: \sqcup_I U_i \ra Z,
\]
where $f$ is surjective by definition and smooth since it is the union of open embeddings which are smooth morphisms.
\end{proof}

\section{Deformation theory}
\label{sec:deformation-theory}

In this section we develop some of the deformation theory for prestacks of Tate type, closely following that of \cite{GR-II}*{Chapter 1}. The main differences occur when we need to study the condition of locally a prestack and further compare our pro-cotangent spaces and complexes with those of \emph{loc.\ cit}. The discussion of square-zero extension is also more subtle since we consider extension of Tate affine schemes by Tate-coherent sheaves.

\begin{warn}
At the moment the set up in this section is a bit awkward, since we considered square-zero extensions with respect to ind-coherent sheaves. This causes at least two problems: (i) the theory needs to consider affine schemes almost of finite type, and (ii) there a lot of technicalities in the construction of the pro-cotangent complex, starting in \S \ref{subsec:pro-cotangent-complex}.
\end{warn}

We plan to rework this section by considering a more general type of square-zero extension in a future version of this preprint. The section is included as it is to demonstrate that many aspects of the deformation theory for prestacks can be generalized to prestacks of Tate type.

In \S \ref{subsec:preparations}, we start with a discussion of push-outs of Tate affine schemes and Tate schemes, as well as a subtle extension of the split square-zero construction from quasi-coherent sheaves on affine schemes to Tate-coherent sheaves on Tate affine schemes. This is one of the crucial constructions that takes the theory off the ground. The extension is essentially done in two parts: in the first part we extend the category of modules from ind-coherent sheaves to Tate-coherent sheaves, and in the second we extend the split-square zero functor to Tate affine schemes. In \S \ref{subsec:pro-cotangent-spaces}, we define the pro-cotangent spaces for prestacks of Tate type, i.e.\ for every test Tate affine scheme $S$ we obtain an object in the category of Pro-objects on Tate-coherent sheaves on $S$. 

In \S \ref{subsec:properties-of-pro-cotangent-spaces}, we consider the usual conditions one can impose on pro-cotangent spaces, namely connectivity, corepresentability, and convergence. We also address the natural questions of compatibility between these conditions and natural conditions on the prestacks of Tate type. In \S \ref{subsec:deformation-theory-lp-and-Tate-conditions}, we study how the pro-cotangent space interacts with the locally a prestack condition. In this section we also compare our pro-cotangent spaces with the pro-cotangent space in the sense of Gaitsgory--Rozenblyum, and we use a generalization of Serre duality for Tate-coherent sheaves to define tangent spaces for the prestacks of Tate type whose underlying prestack admits an eventually connective pro-cotangent complex. 

In \S \ref{subsec:pro-cotangent-complex}, we finally define the pro-cotangent complexes. This definition is subtle: if we consider the $!$-pullback, we have a more general definition, but it becomes harder to compare our cotangent complexes with the ones usually considered in the literature. For this reason we restrict our attention to convergent prestacks of Tate type, for which we can define the cotangent complex using the *-Tate-pullback. We further discuss the usual conditions one can impose on the pro-cotangent complex. In \S \ref{subsec:square-zero-extensions}, we discuss general (i.e.\ non-split) square-zero extension. For a Tate scheme $X$ locally almost of finite type, the cotangent complex gives an object of the category of Tate-coherent sheaves on $X$, whose cohomology is concentrated in non-positive degrees. Thus, we can use morphisms from the cotangent complex to Tate-coherent sheaves concentrated in negative degrees to construct square-zero extensions of Tate schemes. We introduce a structure sheaf for Tate schemes and use it to establish a relation between square-zero extensions and the cotangent complex. 

In \S \ref{subsec:infinitesimal-cohesiveness-and-deformation-theory}, we define the natural notion of infinitesimal cohesiveness, which is a requirement to preserves push-outs of Tate affine schemes with respect to square-zero extensions. We also consider the important condition of admitting deformation theory, that is when a prestack of Tate type is infinitesimally cohesive and admits a pro-cotangent complex. In \S \ref{subsec:consequences-of-deformation-theory}, we prove a couple of interesting and useful results for prestacks of Tate type that admit a deformation theory, e.g.\ that the descent can be checked at the level of the underlying classical prestack of Tate type, and that isomorphisms are detected by isomorphisms of the underlying classical prestacks and an isomorphism between the cotangent complexes.

\begin{notation}
Throughout this section we will always consider our test schemes to be Tate affine schemes \textit{almost of finite type}. In particular, we denote by $\SchaffTate$ the category of Tate affine schemes almost of finite type in contrast with our previous notation.
\end{notation}

\subsection{Preparations}
\label{subsec:preparations}

\subsubsection{Push-outs of Tate affine schemes}

In this section we study the existence of push-outs in the category $\SchaffTate$.

\begin{lem}
\label{lem:push-out-affine-Tate-schemes}
Let $g_S: S \ra S'$ and $f: S \ra T$ be maps of Tate affine schemes, then there exists a Tate affine scheme $T'$ that fits into a pushout diagram
\begin{equation}
    \label{eq:pushout-diagram-for-Tate-affine-schemes}
    \begin{tikzcd}
    S \ar[r,"g_S"] \ar[d,"f"] & S' \ar[d,"f'"] \\
    T \ar[r,"g_T"] & T'
    \end{tikzcd}    
\end{equation}
in the category $\SchaffTate$.
\end{lem}

\begin{proof}
Consider $S = \Spec(A)$, $S' = \Spec(A')$ and $T = \Spec(B)$, where $A,A'$ and $B$ are objects of the category $\Pro(\CAlg(\Vect^{\leq 0}))$. It is clear that\footnote{We abused notation and denoted by $f:B \ra A$ the map induced by $g:\Spec(A) \ra \Spec(B)$, and similarly for $g_S$.}
\[
\begin{tikzcd}
\tau^{\leq 0}(A'\times_{A}B) \ar[r] \ar[d] & B \ar[d,"f"] \\
A' \ar[r,"g_S"] & A
\end{tikzcd}
\]
is a pullback diagram in the category $\Pro(\CAlg(\Vect^{\leq 0}))$. We claim that there exists $S' \in \SchaffTate$ such that $\imath^{\rm geom}(S') = \tau^{\leq 0}(A'\times_{A}B)$, i.e.\
\[
S' = \Spec(\tau^{\leq 0}(A'\times_{A}B)).
\]

By \cite{Hennion-Tate}*{Proposition 1.2} one can find a cofiltered diagram $I$ and maps $f_i: B_i \ra A_i$, whose limit recovers $f:B \ra A$. It is clear that
\[
\tau^{\leq 0}(A'\times_A B) \simeq \lim_I\tau^{\leq 0}(A'\times_{A_i}B_i).
\]
So we only need to check that the diagram $\{\tau^{\leq 0}(A'\times_{A_i}B_i)\}_I$ is an object of $\GeomAlg$. That is for any map $i \ra j$ in $I$ one needs to check that the induced map 
\[
h_{i,j}: \tau^{\leq 0}(A^2\times_{A_i}B_i) \ra \tau^{\leq 0}(A'\times_{A_j}B_j)
\]
satisfies conditions a) and b) from \S \ref{par:definition-of-Geom-Alg}. 

First we notice that we can factor $h_{i,j}$ as
\[
\tau^{\leq 0}(A'\times_{A_i}B_i) \overset{\alpha_{i,j}}{\ra} \tau^{\leq 0}(A'\times_{A_j}B_i) \overset{\beta_{i,j}}{\ra} \tau^{\leq 0}(A'\times_{A_j}B_j),
\]
where $\alpha_{i,j}$ is the pullback of $A_i \ra A_j$ via $A' \ra A_j$ and $\beta_{i,j}$ is the pullback of $B_i \ra B_j$ via $\tau^{\leq 0}(A'\times_{A_j}B_j) \ra B_j$. Finally, we notice that if one has a pullback diagram
\[
\begin{tikzcd}
\tau^{\leq 0}(C_2\times_{C_1}D_1) \ar[r,"\gamma'"] \ar[d] & D_1 \ar[d] \\
C_2 \ar[r,"\gamma"] & C_1
\end{tikzcd}
\]
where $\gamma$ induces a surjection on $\H^0$ and such that $\Fib(\gamma) \in \Mod^{\rm fp}(C_2)$, then the same holds for $\gamma'$.
\end{proof}

The proof of Lemma \ref{lem:push-out-affine-Tate-schemes} gives the following

\begin{cor}
\label{cor:push-out-affine-schemes-preserved-in-Tate-affine-schemes}
The inclusion functor
\[
\Schaff \hra \SchaffTate
\]
preserves push-outs.
\end{cor}

\subsubsection{Push-outs of Tate schemes}

In this section we study push-outs of Tate schemes in a certain particular case. 

Let
\begin{equation}
    \label{eq:push-out-Tate-schemes}
    \begin{tikzcd}
    X \ar[r,"f_1"] \ar[d,"f_2"] & X_1 \ar[d] \\
    X_2 \ar[r] & \widetilde{X}
    \end{tikzcd}
\end{equation}
be a diagram in $\SchTate$ where $f_1$ and $f_2$ are closed embeddings.

\begin{lem}
\label{lem:push-out-Tate-schemes}
The diagram (\ref{eq:push-out-Tate-schemes}) is a push-out diagram if and only if there exists a filtered diagram $I$ and presentations
\[
X \simeq \colim_I X_i, \;\; X_{1} \simeq \colim_I X_{1,i}, \;\; X_{2} \simeq \colim_I X_{2,i}, \; \mbox{and} \; \widetilde{X} \simeq \colim_I\widetilde{X}_i
\]
such that for each $i\in I$ the diagram
\[
\begin{tikzcd}
X_i \ar[r,"f_{1,i}"] \ar[d,"f_{2,i}"] & X_{1,i} \ar[d] \\
X_{2,i} \ar[r] & \widetilde{X}_i
\end{tikzcd}
\]
is a push-out diagram in $\Sch$.
\end{lem}

\begin{proof}
The proof is a straight-forward chasing of the definitions. Since $f_1$ and $f_2$ are closed embeddings we notice that we can find a diagram $I$ presenting $X$, $X_1$ and $X_2$ such that the morphisms
\[
f_{1,i}: X_i \ra X_{1,i}, \;\;\; \mbox{and} \;\;\; f_{2,i}: X_i \ra X_{2,i}
\]
are closed embeddings for each $i \in I$. Then, by \cite{GR-II}*{Chapter 1, Corollary 1.2.4} the push-out in the category $\Sch$
\[
\widetilde{X}_i := X_{1,i}\underset{X_i}{\sqcup}X_{2,i}
\]
exists. Thus, by the same proof as Lemma \ref{lem:push-out-affine-Tate-schemes} one has that 
\[
\colim_I \widetilde{X}_i \simeq X_{1}\underset{X}{\sqcup}X_2.
\]
\end{proof}

We notice that as a consequence of \cite{GR-II}*{Chapter 1, Lemma 1.2.2} one has

\begin{prop}
\label{prop:affine-push-out-and-scheme-push-out-agree}
Consider $S_1 \overset{f_1}{\la} S \overset{f_2}{\ra} S_2$ a diagram in $\SchaffTate$ where $f_1$ and $f_2$ are closed embeddings. Then one has an equivalence
\[
\imath(S_{1}\underset{S}{\sqcup}S_2) \simeq \imath(S_1)\underset{\imath(S)}{\sqcup}\imath(S_2)
\]
where $\imath: \SchaffTate \hra \SchTate$ is the inclusion of Tate affine schemes into Tate schemes\footnote{See \S \ref{subsubsec:convergence-Tate-affine-schemes}.}, and the push-out on each side is taken in the respective category.
\end{prop}

\paragraph{The case of nil-isomorphisms}

We can now consider a certain class of push-outs of Tate schemes that will be useful when we try to understand the cotangent spaces of Tate schemes.

\begin{lem}
\label{lem:push-out-of-nil-isomorphisms-for-Tate-schemes-exist}
For a diagram in $\SchTate$
\[
\begin{tikzcd}
X \ar[r,"f_1"] \ar[d,"f_2"] & X_1 \\
X_2 & 
\end{tikzcd}
\]
where $f_{1}$ is a closed nil-isomorphism and $f_2$ is an arbitrary morphism. The push-out in $\SchTate$
\[
\widetilde{X}:= X_2\underset{X}{\sqcup}X_1
\]
exists, and the map $\widetilde{f_1}: X_{2} \ra \widetilde{X}$ is a closed nil-isomorphism.
\end{lem}

\begin{proof}
This is a consequence of \cite{GR-II}*{Chapter 1, Corollary 1.3.5} and the proof of Lemma \ref{lem:push-out-Tate-schemes}.
\end{proof}

The following is a generalization of Proposition \ref{prop:affine-push-out-and-scheme-push-out-agree} and will be used to check that Tate schemes admit a cotangent complex.

\begin{cor}
\label{cor:affine-push-out-nil-isomorphism-agrees-with-scheme-push-out}
Consider $S_1 \overset{f_1}{\la} S \overset{f_2}{\ra} S_2$ a diagram in $\SchaffTate$ where $f_1$ is a closed nil-isomorphism and $f_2$ is a nil-isomorphism. Then one has an equivalence
\[
\imath(S_{1}\underset{S}{\sqcup}S_2) \simeq \imath(S_1)\underset{\imath(S)}{\sqcup}\imath(S_2)
\]
where $\imath: \SchaffTate \hra \SchTate$ is the inclusion of Tate affine schemes into Tate schemes, and the push-out on each side is taken in the respective category. Moreover, the result also holds if $f_2$ is a nil-isomorphism and $f_1$ is an arbitrary morphism.
\end{cor}

\subsubsection{Split square-zero extensions}
\label{subsubsec:split-square-zero-extensions}

In this section we extend the construction of \cite{GR-II}*{Chapter 1, \S 2.1} to the case of Tate-coherent sheaves on Tate affine schemes.

Recall that \cite{HA}*{Proposition 7.3.4.5}\footnote{Indeed, in the notation of \emph{loc.\ cit.\ } we take $\sO^{\otimes} = \Finp$, and $\sC := \Mod(R)^{\leq 0}$ and $\sC^{\otimes} \ra \sO^{\otimes}$ is the coCartesian fibration encoding the symmetric monoidal structure of $\sC$. So for $S = \Spec(R)$ the functor $\RealSplitSqZ$ corresponds to the functor $f$ in \emph{loc.\ cit.\ }, which is left adjoint to the functor
\[
\CAlg^{\rm aug.}(\Mod(R)^{\leq 0}) \ra \Mod(R)^{\leq 0}
\]
that sends an augmented algebra $\alpha: A \ra R$ to $\Fib(\alpha)$, notice the category of augmented $R$-algebras (see \cite{HA}*{Definition 7.3.4.3}) $\CAlg^{\rm aug.}(\Mod(R)^{\leq 0})$ is equivalent to $\Schaff_{S/}$.}  defines a functor
\[
\RealSplitSqZ: \left(\QCoh(S)^{\leq 0}\right)^{\rm op} \ra \left(\Schaff\right)_{S/}
\]
for any $S \in \Schaff$.

The category $\IndCoh(S)$ has a t-structure and the canonical functor
\[
\Psi_{S}: \IndCoh(S) \ra \QCoh(S)
\]
is t-exact (see \cite{GR-I}*{Chapter 4, \S 1.2}). Thus, one can consider the composite
\[
\RealSplitSqZ_S: \left(\IndCoh(S)^{\leq 0}\right)^{\rm op} \ra \left(\QCoh(S)^{\leq 0}\right)^{\rm op} \ra \left(\Schaff\right)_{S/}
\]
which we still denote by $\RealSplitSqZ_S$.

So the initial data for our construction is the functor
\begin{equation}
\label{eq:RealSplitSqZ-affine-schemes}
\RealSplitSqZ_S: \left(\IndCoh(S)^{\leq 0}\right)^{\rm op} \ra \left(\Schaff\right)_{S/}.
\end{equation}

\begin{prop}
\label{prop:RealSplitSqZ-Tate-affine-schemes}
For any $T \in \SchaffTate$ there exists a unique functor
\begin{equation}
\label{eq:RealSplitSqZ-Tate-affine-schemes}
\RealSplitSqZ_{T}: \left(\TateCoh(T)^{\leq 0}\right)^{\rm op} \ra \left(\SchaffTate\right)_{T/}    
\end{equation}
whose restriction to an object in $\Schaff$ and to the subcategory of ind-coherent sheaves recovers the functor (\ref{eq:RealSplitSqZ-affine-schemes}).
\end{prop}

\begin{proof}
The construction of (\ref{eq:RealSplitSqZ-Tate-affine-schemes}) is achieved through several steps. At places we abused notation and denoted by $\RealSplitSqZ$ the intermediate new functors that we obtain in those steps, hopefully, we have been careful enough about their source and target categories so as to avoid any confusion.

\textit{Step 1:} for $S \in \Schaff$ there is a natural inclusion $\left(\Schaff\right)_{S/} \hra \left(\Ind(\Schaff)\right)_{S/}$ so the composite of $\RealSplitSqZ_S$ with this inclusion gives us the functor
\begin{equation}
\label{eq:realsplitsqztoindschemes}
\RealSplitSqZ_S: \left(\IndCoh(S)^{\leq 0}\right)^{\rm op} \ra \left(\Ind(\Schaff)\right)_{S/}.  
\end{equation}

\textit{Step 2:} notice that one has an equivalence
\[
\left(\ProIndCoh(S))^{\leq 0}\right)^{\rm op} \simeq \Ind(\IndCoh(S)^{\rm op})^{\leq 0},
\]
thus one can Ind extend the functor (\ref{eq:realsplitsqztoindschemes}) to obtain
\begin{equation}
\label{eq:realsplitsqzproindsheavestoindschemes}
\RealSplitSqZ_S: \left(\ProIndCoh(S))^{\leq 0}\right)^{\rm op} \ra \left(\Ind(\Schaff)\right)_{S/}.
\end{equation}

\textit{Step 3:} recall the subcategory $\TateCoha(S)^{\leq 0} \subset \ProIndCoh(S)^{\leq 0}$ of admissible Tate-coherent sheaves in non-positive degrees. We can restrict the functor (\ref{eq:realsplitsqzproindsheavestoindschemes}) to a functor
\begin{equation}
    \label{eq:RealSplitSqZ-for-admissible-Tate-sheaves}
    \RealSplitSqZ_S: \left(\TateCoh(S)^{\leq 0}\right)^{\rm op} \ra \left(\Ind(\Schaff)\right)_{S/}.
\end{equation}
Now we claim that the functor (\ref{eq:RealSplitSqZ-for-admissible-Tate-sheaves}) factors as follows
\begin{equation}
\label{eq:Tatesheaves-to-Tateschemes-for-affine}
\RealSplitSqZ_S: \left(\TateCoh(S)^{\leq 0}\right)^{\rm op} \ra \left(\SchaffTate\right)_{S/}.
\end{equation}
Indeed, we need to check that given $\sF \in \TateCoha(S)^{\leq 0}$ the ind-affine scheme $S_{\sF}$ satisfies conditions a) and b) from Definition \ref{defn:Tate-affine-schemes}. Let $\{\sF_i\}_{I}$ be a diagram in $\IndCoh(S)^{\leq 0}$ representing $\sF$, the condition that for every $i \ra j$ the map
\[
\H^0(\sF_i) \ra \H^0(\sF_j)
\]
is surjective, implies that 
\[
S_{\sF_j} \ra S_{\sF_i}
\]
is a closed embedding. A direct calculation gives
\[
\Fib(\sO_{S_{\sF_i}} \ra \sO_{S_{\sF_i}}) \simeq \Fib(\sF_i \ra \sF_j).
\]
Since $\sF$ is Tate object, one has $\Fib(\sF_i \ra \sF_j) \in \Coh(S)$ by definition, finally we notice that pullback along $S_{\sF_i} \ra S$ preserves coherent objects.

\textit{Step 4:} Now consider $T$ a Tate affine schemes almost of finite type, and let
\[
T \simeq \colim_{I}T_i
\]
be a presentation of $T$, where $T_i \in \Schaffaft$ and the connecting maps are closed immersions. Then by Corollary \ref{cor:colimits-and-limits-presentation-of-TateCoha} one has an equivalence of categories
\[
(\TateCoha(T)^{\leq 0})^{\rm op} = \lim_{I^{\rm op}}(\TateCoha(T_i)^{\leq 0})^{\rm op} \simeq \colim_{I}(\TateCoha(T_i)^{\leq 0})^{\rm op}.
\]

\textit{Step 5:} We notice that for each $i \ra j$ in $I$ one has a functor
\begin{align*}
(-)\sqcup_{T_i}T_j: \left(\SchaffTate\right)_{T_i/} & \ra \left(\SchaffTate\right)_{T_j/} \\
Z & \mapsto Z\sqcup_{T_i}T_j
\end{align*}
where we used Lemma \ref{lem:push-out-affine-Tate-schemes} to guarantee that the push-out exists.

We notice that \cite{GR-II}*{Chapter 1, \S 4.1.2} and \cite{GR-I}*{Chapter 4, Proposition 2.1.2} give that the following diagram commutes
\begin{equation}
\label{eq:functoriality-IndCoh-SplitSqZero-for-affine-schemes}
\begin{tikzcd}
\left(\IndCoh(T_i)^{\leq 0}\right)^{\rm op} \ar[r,"(f_{i,j})^{\rm op}_*"] \ar[d,"\RealSplitSqZ_{T_i}"'] & \left(\IndCoh(T_j)^{\leq 0}\right)^{\rm op} \ar[d,"\RealSplitSqZ_{T_j}"] \\
\left(\Schaff\right)_{T_i/} \ar[r,"(-)\sqcup_{T_i}T_j"'] & \left(\Schaff\right)_{T_j/}
\end{tikzcd}    
\end{equation}

By construction \footnote{See the proof of Lemma \ref{lem:RealSplitSqZ-of-TateCoh-for-Tate-schemes-is-functorial} below.} one has a similar diagram for Tate-coherent sheaves on affine schemes, namely the following diagram commutes
\[
\begin{tikzcd}
\left(\TateCoh(T_i)^{\leq 0}\right)^{\rm op} \ar[r,"(f_{i,j})^{\rm op}_*"] \ar[d,"\RealSplitSqZ_{T_i}"'] & \left(\TateCoh(T_j)^{\leq 0}\right)^{\rm op} \ar[d,"\RealSplitSqZ_{T_j}"] \\
\left(\SchaffTate\right)_{T_i/} \ar[r,"(-)\sqcup_{T_i}T_j"'] & \left(\SchaffTate\right)_{T_j/}
\end{tikzcd}
\]
where the horizontal maps are the functors (\ref{eq:Tatesheaves-to-Tateschemes-for-affine}) obtained from the construction up to Step 3.

Thus, the colimit of the functors $\RealSplitSqZ_{T_i}$ over $I$ gives a map
\begin{equation}
\label{eq:realsplitsqzTatesheavestocolimitcategory}
\RealSplitSqZ_{T}: \left(\TateCoh(T)^{\leq 0}\right)^{\rm op} \ra \colim_{I}\left(\SchaffTate\right)_{T_i/},
\end{equation}
where the colimit on the categories on the right is with respect to the pushout maps.

\textit{Step 6:} now we notice that for a closed immersion $f_{i,j}:T_i \ra T_j$ the functors
\[
(-)\sqcup_{T_i}T_j: \left(\SchaffTate\right)_{T_i/} \ra \left(\SchaffTate\right)_{T_j/}
\]
and
\[
(-)\circ f_{i,j}: \left(\SchaffTate\right)_{T_j/} \ra \left(\SchaffTate\right)_{T_i/}
\]
form an adjoint pair $((-)\sqcup_{T_i}T_j,(-)\circ f_{i,j})$. Again by Lemma \ref{lem:limits-to-colimits-of-categories} one has an equivalence of categories
\begin{equation}
\label{eq:TateschemescolimitisTateschemeslimit}
\colim_{I}\left(\SchaffTate\right)_{T_i/} \simeq \lim_{I^{\rm op}}\left(\SchaffTate\right)_{T_i/},    
\end{equation}
where the connecting maps on the limit diagram are given by pre-composition with $f_{i,j}$.

\textit{Step 7:} finally we notice that there is a canonical map
\begin{equation}
\label{eq:Tateschemesunderislimitcategory}
\left(\SchaffTate\right)_{T/} \ra \lim_{I^{\rm op}}\left(\SchaffTate\right)_{T_i/}    
\end{equation}
and this map is an isomorphism. Indeed, let $S \in \SchaffTate$ with a given presentation $S \simeq \colim_J S_j$ one has
\[
\Maps(T,S) \simeq \lim_{T^{\rm op}}\Maps(T_i,S).
\]
Thus, given a map $g:T \ra S$ of Tate affine schemes is equivalent to the data of $\{g_{i}:T_i \ra S\}_{I^{\rm op}}$ for $T \simeq \colim_I T_i$ a presentation of $T$.
\end{proof}

\begin{lem}
\label{lem:realsplitsqz-Tate-preserves-colimits}
For any $T \in \SchaffTate$ the functor (\ref{eq:RealSplitSqZ-Tate-affine-schemes}) preserves colimits.
\end{lem}

\begin{proof}
Since the functor (\ref{eq:RealSplitSqZ-affine-schemes}) preserves colimits by \cite{GR-II}*{Chapter 1, Lemma 2.1.4} it is enough to check that each step in the above construction preserves colimits. Step 1 follows because the inclusion $\Schaff \ra \Ind(\Schaff)$ preserves colimits. Step 2 follows because Ind extension of a colimit-preserving functor gives a colimit-preserving functor. Step 3 is simply a restriction so it preserves the property as well. Step 4 is an equivalence of categories. Step 5 follows because we are taking the colimit of categories in the $\infty$-category of presentable DG-categories. Steps 6 and 7 are again just equivalence of categories. 
\end{proof}

\paragraph{Functoriality of square-zero extensions}

Recall that for a map $f:S \ra T$ between Tate affine schemes the functor
\[
f_*: \TateCoh(S) \ra \TateCoh(T)
\]
is t-exact. Indeed, any such map is a Tate affine morphism, hence the result follows from Proposition \ref{prop:pushforward-TateCoh-between-Tate-schemes-is-left-t-exact}.

The following functoriality result will be needed in section \ref{subsec:pro-cotangent-complex} for the definition of pro-cotangent complexes.

\begin{lem}
\label{lem:RealSplitSqZ-of-TateCoh-for-Tate-schemes-is-functorial}
Let $f:S \ra T$ be a morphism in $\SchaffTate$, then the following diagram commutes\footnote{Recall that the functor $f^{\rm Tate}_*$ here is defined in \S \ref{subsec:Tate-coherent-sheaves-on-Tate-schemes}.}
\begin{equation}
    \label{eq:functoriality-diagram-Tate-for-Tate-affine}
    \begin{tikzcd}
    \left(\TateCoh(S)^{\leq 0}\right)^{\rm op} \ar[r,"(f^{\rm Tate})^{\rm op}_*"] \ar[d,"\RealSplitSqZ_{S}"'] & \left(\TateCoh(T)^{\leq 0}\right)^{\rm op} \ar[d,"\RealSplitSqZ_{T}"] \\
    \left(\SchaffTate\right)_{S/} \ar[r,"(-)\underset{S}{\sqcup}T"'] & \left(\SchaffTate\right)_{T/}
    \end{tikzcd}
\end{equation}
where $f^{\rm Tate}_*$ is the pushforward of Tate sheaves as defined in Theorem \ref{thm:TateCoh-correspondence-functor-Sch-Tate}.
\end{lem}

\begin{proof}
We first prove the result in the case where $S = S_0 \overset{f}{\ra} T = T_0$ are usual affine schemes, and $f$ is an arbitrary morphism. 

Let $\sF \simeq \lim_{A}\sF_{\alpha} \in \TateCoh(S)^{\leq 0}$ where $\sF_{\alpha} \in \IndCoh(S)^{\leq 0}$. By definition, that is, Steps 1-3 in the proof of Proposition \ref{prop:RealSplitSqZ-Tate-affine-schemes} above, one has
\[
\RealSplitSqZ_{S_0}(\lim_{A}\sF_{\alpha}) \simeq \colim_{A^{\rm op}}(S_0)_{\sF_{\alpha}},
\]
where $(S_0)_{\sF_{\alpha}}$ is the usual split square-zero extension of $S_0$ by $\Phi_{S_0}(\sF_{\alpha})$. Thus, going around the left and lower part of the diagram (\ref{eq:functoriality-diagram-Tate-for-Tate-affine}) one has
\[
(\colim_{A^{\rm op}}(S_0)_{\sF_{\alpha}}) \underset{S_0}{\sqcup}T_0 \simeq \colim_{A^{\rm op}}(S_0)_{\sF_{\alpha}} \underset{S_0}{\sqcup}T_0.
\]

Now, if one goes around the diagram (\ref{eq:functoriality-diagram-Tate-for-Tate-affine}) via the top and right map one obtains
\[
\RealSplitSqZ_{T_0}(f^{\rm Tate}_*(\lim_{A}\sF_{\alpha})).
\]

At this point we consider two cases, either $f$ is a proper morphism or an open morphism. 

When $f$ is a closed embedding one has that 
\[
f^{\rm Tate}_*(\lim_{A}\sF_{\alpha}) \simeq \lim_{A}f^{\rm Ind}_*(\sF_{\alpha}).
\]
Thus, chasing through the definition of $\RealSplitSqZ_{T_0}$ one has
\[
\RealSplitSqZ_{T_0}(f^{\rm Tate}_*(\lim_{A}\sF_{\alpha})) \simeq \colim_{A^{\rm op}}(T_0)_{f^{\rm Ind}_*(\sF_{\alpha})}.
\]

When $f$ is an open embedding, $f^{\rm Tate}_*$ is a right adjoint, so it commutes with limits, and one has
\[
f^{\rm Tate}_*(\lim_{A}\sF_{\alpha}) \simeq \lim_{A}f^{\rm Ind}_*(\sF_{\alpha}),
\]
as well, by Lemma \ref{lem:compatibility-TateCoh-pushforward-with-IndCoh-pushforward-for-open-morphisms}.

Finally, since the diagram (\ref{eq:functoriality-IndCoh-SplitSqZero-for-affine-schemes}) commutes, one obtains that
\[
\RealSplitSqZ_{S_0}(\sF)\underset{S_0}{\sqcup}T_0 \simeq \RealSplitSqZ_{T_0}(f^{\rm Tate}_*(\sF)).
\]

Now we consider the case where $f:S \ra T$ is an arbitrary morphism of Tate affine schemes. Let $S \simeq \colim_I S_i$ be a presentation of $S$ and $\sF \in \TateCoh(S)^{\leq 0}$, by Lemma \ref{lem:TateCoh-on-Tate-schemes-is-pushforward-from-closed-schemes} one has that
\[
\sF \simeq (f_i)^{\rm ProInd}_{*}(\sG),
\]
where $f_i: S_i \hra S$. Notice that the composite
\[
S_i \overset{f_i}{\ra} S \overset{f}{\ra} T
\]
has to factor through $T_j \hra T$ for some $j \in J$ where $T \simeq \colim_J T_j$ is a presentation of $T$. Thus, by the previous case, we are reduced to the case where $f:S_0 \ra T$ is a closed embedding with $S_0 \in \Schaff$ and $T \in \SchaffTate$.

Notice that to check that the diagram (\ref{eq:functoriality-diagram-Tate-for-Tate-affine}) commutes it is enough to check that the outer diagram of
\begin{equation}
\label{eq:outer-diagram-for-j-in-presentation-of-T}
\begin{tikzcd}
    \left(\TateCoh(S_0)^{\leq 0}\right)^{\rm op} \ar[r,"(f^{\rm Tate})^{\rm op}_*"] \ar[d,"\RealSplitSqZ_{S_0}"'] & \left(\TateCoh(T)^{\leq 0}\right)^{\rm op} \ar[d,"\RealSplitSqZ_{T}"] \ar[r,"h^!_i"] & (\TateCoh(T_j)^{\leq 0})^{\rm op} \ar[d,"\RealSplitSqZ_{T_j}"] \\
    \left(\SchaffTate\right)_{S_0/} \ar[r,"(-)\underset{S}{\sqcup}T"'] & \left(\SchaffTate\right)_{T/} \ar[r,"(-)\circ h_j"'] & \left(\SchaffTate\right)_{T_j/}
\end{tikzcd}    
\end{equation}
commutes for all $j \in J$.

Since $f:S_0 \hra T$ is a closed embedding, one can find $j_0 \in J$ such that in the diagram
\begin{equation}
    \label{eq:base-change-reduction-functoriality}
    \begin{tikzcd}
    S_0\cap T_{j} \ar[r,hook,"\bar{g}_j"] \ar[d,hook] & T_j \ar[d,"h_j"] \\
    S_0 \ar[r,"g"] & T
    \end{tikzcd}
\end{equation}
one has $S_0 \cap T_j \simeq S_0$ for all $j < j_0$. Thus, base change with respect to (\ref{eq:base-change-reduction-functoriality}) gives that
\[
h^{!}_j \circ g^{\rm Tate}_*(\sG) \simeq (\bar{g}_j)^{\rm Tate}_*(\sG).
\]
Thus, the commutation of the outer diagram (\ref{eq:outer-diagram-for-j-in-presentation-of-T}) for $j < j_0$ follows from the previous result. Since $J_{< j_0} \ra J$ is cofinal, one has that the diagram (\ref{eq:functoriality-diagram-Tate-for-Tate-affine}) commutes. This finishes the proof.
\end{proof}

\paragraph{Split square-zero of Tate schemes}
\label{par:split-square-zero-for-Tate-schemes}

Let $X \in \SchTate$ denote a Tate scheme and consider $\sF \in \TateCoh(X)$. Given a presentation $X \simeq \colim_I X_i$ with natural inclusion $f_i:X_i \hra X$. We define
\[
X_{\sF} := \colim_{I}(X_i)_{f^!_i(\sF)},
\]
where each split square-zero extension $(X_i)_{f^!_i(\sF)}$ is defined by descent on a Zariski cover. That is, for $\{U_{i,j} \ra X_i\}_J$ a Zariski cover we let
\[
(X_i)_{f^!_i(\sF)} := \L(\sqcup_{J}((U_{i,j})_{f^!_i(\sF)})).
\]

One can easily check that $X_{\sF}$ satisfy the Tate condition, since the square-zero extensions don't change the ideal of definitions of $X_i$ in $X_j$.

The notion of square-zero extension for Tate schemes allows us to make sense of the value of the pro-cotangent space of a prestack of Tate type that satisfies Zariski descent at a Tate scheme.

Let $\sX$ be an object $\PStkTate$ that admits pro-cotangent spaces and satisfy Zariski descent. Given a morphism $x:X \ra \sX$, where $X \in \SchTatelaft$ it follows that the functor
\begin{align*}
    \TateCoh(X)^{\leq 0} & \ra \Spc \\
    \sF & \Maps_{X/}(X_{\sF},\sX)
\end{align*}
is pro-corepresentable by an object
\[
T^*_{x}(\sX) \in \Pro(\TateCoh(X)^-).
\]

\subsection{Pro-cotangent spaces}
\label{subsec:pro-cotangent-spaces}

In this section we finally define the notion of a cotangent complex for a prestack of Tate type. The strategy follows closely the treatment in \cite{GR-II}*{Chapter 1, \S 2}.

\subsubsection{Pro-cotangent space at a point}

\paragraph{Lift functor}

Consider $\sX$ a prestack of Tate type, and let $x:S \ra \sX$ be a map from $S \in \SchaffTate$. Given any $\sF \in \TateCoh(S)^{\leq 0}$ one considers the functor
\begin{align}
\label{eq:lift-functor}
    \Lift_{x,\sX}: \TateCoh(S)^{\leq 0} & \ra \Spc \\
    \sF & \mapsto \Maps_{S/}(S_{\sF},\sX), \nonumber
\end{align}
where $S_{\sF} = \RealSplitSqZ_S(\sF)$ is defined as in Proposition \ref{prop:RealSplitSqZ-Tate-affine-schemes}.

\paragraph{Functoriality of lift functor}
\label{par:functotiality-lift-functor}

By Lemma \ref{lem:realsplitsqz-Tate-preserves-colimits} the functor 
\[
\sF \in (\TateCoh(S)^{\leq 0})^{\rm op} \mapsto S_{\sF} \in (\SchaffTate)_{S/}
\]
preserves colimits, i.e.\ sends pullbacks of sheaves to pushforward of schemes. In particular, given $f: \sF_{1} \ra \sF_{2}$ in $\TateCoh(S)^{\leq 0}$ such that
\[
\sF := 0\times_{\sF_2}\sF_1
\]
belongs to $\TateCoh(S)^{\leq 0}$, i.e.\ such that $f$ induces a surjection on $\H^0$, one obtains
\[
S_{\sF} \simeq S\sqcup_{S_{\sF_2}}S_{\sF_1}.
\]
This gives a canonical map
\begin{equation}
\label{eq:defining-cotangent-space}
\Lift_{x,\sX}(\sF) \ra \Lift_{x,\sX}(0)\times_{\Lift_{x,\sX}(\sF_2)} \Lift_{x,\sX}(\sF_1)
\end{equation}

\begin{defn}
\label{defn:pro-cotangent-space-at-a-point}
A prestack of Tate type $\sX$ is said to have a \emph{pro-cotangent space at a point $x:S \ra \sX$} if the map (\ref{eq:defining-cotangent-space}) is an equivalence.
\end{defn}

\begin{rem}
Another way to write the map (\ref{eq:defining-cotangent-space}) is as
\begin{equation}
\label{eq:defining-cotangent-space-with-Maps}
\Maps_{S/}(S_{\sF},\sX) \ra \{*\}\times_{\Maps_{S/}(S_{\sF_2},\sX)} \Maps_{S/}(S_{\sF_1},\sX).    
\end{equation}
The reason that we prefered to denote the space of lifts as $\Lift_{x\sX}$ is that the notation (\ref{eq:defining-cotangent-space-with-Maps}) might suggest that any prestack would admit a pro-cotangent space, as one would guess that $\Maps(-,\sX)$ commutes with coproducts. However we emphasize that $\Maps(S,\sX)$ here means that we evaluate the prestack of Tate type $\sX$ on an Tate affine scheme $S$, a priori this is only equivalent to considering the mapping space in the category of prestacks, and the result of Lemma \ref{lem:push-out-affine-Tate-schemes} does \emph{not} hold in this category, i.e.\ $S_{\sF}$ is not necessarily the pushout $S\sqcup_{S_{\sF_2}}S_{\sF_1}$ in the category of prestacks.
\end{rem}

\paragraph{Corepresenting object}

Suppose that $\sX$ admits a cotangent space at $x:S \ra \sX$. Then we can extend the functor $\Lift_{x}$ to
\[
\Lift_{x}: \TateCoh(S)^- \ra \Spc,
\]
by sending any $\sF \in \TateCoh(S)^{\leq k}$ to
\[
\Lift_{x}(\sF) = \Omega^i\Maps_{S/}(S_{\sF[i]},\sX),
\]
for any $i \geq k$.

We notice that $\Lift_{x}$ is well-defined by the isomorphism (\ref{eq:defining-cotangent-space}) and that $\Lift_{x}$ is a left exact functor. Indeed, by definition we required that it sends pullbacks to pullbacks. Hence $\Lift_{x,\sX}$ is pro-corepresentable, i.e.\ there exists an object
\[
T^{*}_{x}\sX \in \Pro(\TateCoh(S)^-)
\]
that we refer to as \emph{pro-cotangent space to $\sX$ at $x$}.

\begin{defn}
\label{defn:cotangent-space-at-a-point}
For $\sX \in \PStkTate$, one says that $\sX$ admits a \emph{cotangent space at $x:S \ra \sX$} if 
\[
T^*_{x}\sX \in \TateCoh(S)^- \subset \Pro(\TateCoh(S)^-).
\]
\end{defn}

\begin{defn}
\label{defn:cotangent-space}
Given a prestack of Tate type $\sX$, one says that $\sX$ admits a \emph{pro-cotangent space} if for all $(S \overset{x}{\ra} \sX) \in \left(\SchaffTate\right)_{/\sX}$ the prestack of Tate type $\sX$ has a pro-cotangent space at the point $x$. 

Simiarly, one says that $\sX$ admits a \emph{cotangent space} if for all $(S \overset{x}{\ra} \sX) \in \left(\SchaffTate\right)_{/\sX}$ the prestack of Tate type $\sX$ has a cotangent space at the point $x$. 
\end{defn}

\begin{example}
\label{ex:Tate-scheme-has-pro-cotangent-space}
Let $Z \in \SchTatelaft$ and consider $x:S \ra Z$ a point in $Z$, i.e.\ $S \in \SchaffTate$, we claim that $Z$ admits a pro-cotangent space at $x$.

We need to check that for any morphism $\sF_1 \ra \sF_2$ in $\TateCoh(S)^{\leq 0}$ which induces a surjection on $\H^0$, the canonical map
\[
\Maps_{\SchTate}(\imath(S\underset{S_{\sF_2}}{\sqcup}S_{\sF_1}),Z) \ra \Maps_{\SchTate}(S,Z)\underset{\Maps_{\SchTate}(S_{\sF_2},Z)}{\times}\Maps_{\SchTate}(S_{\sF_1},Z)
\]
is an isomorphism. This follows from Corollary \ref{cor:affine-push-out-nil-isomorphism-agrees-with-scheme-push-out} since $S_{\sF_2} \ra S$ is a nil-isomorphism and $S_{\sF_2} \ra S_{\sF_1}$ is a closed nil-isomorphism.
\end{example}

\subsubsection{Relative situation}
\label{subsubsec:pro-cotangent-space-relative-situation}

\paragraph{Relative pro-cotangent space at a point}

The same definitions apply if one works relative to a fixed prestack of Tate type $\sX_0$. Let $(\sX \overset{f}{\ra}\sX_0) \in (\PStkTate)_{/\sX_0}$ and $(S,x) \in (\SchaffTate)_{/\sX}$ we say that $\sX$ admits a \emph{pro-cotangent space relative to $\sX_0$ at $x$} if in the situation of \S \ref{par:functotiality-lift-functor} the diagram
\[
\begin{tikzcd}
\Lift_{x,\sX}(\sF) \ar[r] \ar[d] & \Lift_{x,\sX}(0)\times_{\Lift_{x,\sX}(\sF_2)} \Lift_{x,\sX}(\sF_1) \ar[d] \\
\Lift_{f\circ x,\sX_0}(\sF) \ar[r] & \Lift_{f \circ x,\sX}(0)\times_{\Lift_{f\circ x,\sX}(\sF_2)} \Lift_{f\circ x,\sX_0}(\sF_1)
\end{tikzcd}
\]
is a pullback square.

If this condition is satisfied we denote by
\[
T^*_{x}(\sX/\sX_0) \in \Pro(\TateCoh(S)^-)
\]
the object that pro-corepresents the functor
\[
\sF \mapsto \Lift_{x,\sX}(\sF)\underset{\Lift_{f\circ x,\sX_0}(\sF)}{\times}\ast.
\]

\paragraph{Relative pro-cotangent space}

Similarly to Definition \ref{defn:cotangent-space} one says that $\sX$ admits a pro-cotangent space relative to $\sX_0$ if it admits a pro-cotangent space relative to $\sX_0$ at $x$ for all points $(S,x) \in (\SchaffTate)_{/\sX}$. One also has a Lemma completely analogous to \cite{GR-II}*{Chapter 1, Lemma 2.4.5}.

\paragraph{}

If the prestack $\sX_0$ admits a pro-cotangent space at the point $f\circ x:S \ra \sX_0$, then $\sX$ admits a pro-cotangent space at $x$ and if and only if it admits a pro-cotangent space relative to $\sX_0$ and one has an equivalence
\[
T^*_{x}(\sX/\sX_0) \simeq \Cofib\left(T^*_{f\circ x}\sX_0 \ra T^*_x\sX\right).
\]

\paragraph{}

The following Lemma follows from the definitions
\begin{lem}
\label{lem:pro-cotangent-space-relative-and-absolute-relation}
A prestack of Tate type $\sX$ over $\sX_0$ admits a pro-cotangent space relative to $\sX_0$ if and only if for all $S_0 \in (\SchaffTate)_{/\sX_0}$ the prestack of Tate type $S_0\underset{\sX_0}{\times}\sX$ admits a pro-cotangent space.
\end{lem}

\subsubsection{Pro-cotangent space of a colimit of prestacks}

Suppose that
\begin{equation}
    \label{eq:presentation-of-sX}
    \sX \simeq \colim_{A}\sX_{a},    
\end{equation}
in the category $\PStkTate$.

For each $(S,x) \in \left(\SchaffTate\right)_{/\sX}$, we let $A_{x/}$ denote the category co-fibered over $A$ whose fiber over $a$ is the space of factorizations of $x$ as
\[
S \overset{x_a}{\ra} \sX_{a} \ra \sX.
\]

\begin{prop}
\label{prop:pro-cotangent-space-of-colimit-of-prestacks}
Suppose that $T^*_{x}\sX$ exists and that the category $A_{x/}$ is sifted. Then the natural map
\[
T^*_{x}\sX \ra \lim_{(a,x_{a}) \in (A_{x/})^{\rm op}}T^*_{x_a}\sX_{a},
\]
where the limit is taken in $\Pro(\Tate(\Coh(S))^-)$ is an isomorphism.
\end{prop}

\begin{proof}
The proof is exactly as that of \cite{GR-II}*{Chapter 1, Proposition 2.5.3}. 

We notice that all one needs to check is that for any $\sF \in \TateCoh(S)^{\leq 0}$ the map
\begin{equation}
\label{eq:limit-of-cotangent-spaces-general}
\colim_{(a,x_{a}) \in A_{x/}}\Hom(T^*_{x_a}\sX_{a},\sF) \ra \Hom(T_{x}\sX,\sF)
\end{equation}
is an isomorphism, where the colimit is taken in $\Vect$.

We claim that (\ref{eq:limit-of-cotangent-spaces-general}) is an isomorphism if for every $n \in \bN$ the canonical map
\begin{equation}
    \label{eq:limit-of-cotangent-spaces-eventually-coconnective}
    \colim_{(a,x_{a}) \in A_{x/}}\tau^{\leq n}\left(\Hom(T^*_{x_a}\sX_{a},\sF)\right) \ra \tau^{\leq n}\left(\Hom(T_{x}\sX,\sF)\right)
\end{equation}
is an isomorphism. Indeed, this follows from $\tau^{\leq n}$ commuting with colimits.

Notice that (\ref{eq:limit-of-cotangent-spaces-eventually-coconnective}) can be rewritten as 
\begin{equation}
    \label{eq:limit-of-cotangent-spaces-with-shift}
    \colim_{(a,x_{a}) \in A_{x/}}\tau^{\leq 0}\left(\Hom(T^*_{x_a}\sX_{a},\sF[n])\right) \ra \tau^{\leq 0}\left(\Hom(T_{x}\sX,\sF[n])\right).
\end{equation}

Since the Dold--Kan functor $\DK: \Vect^{\leq 0} \ra \Spc$ commutes with sifted colimits\footnote{See \cite{GR-I}*{Chapter 1, \S 10.2.3} for instance.} (\ref{eq:limit-of-cotangent-spaces-with-shift}) reduces to checking that
\[
\colim_{(a,x_{a}) \in A_{x/}} \Lift_{x_a,\sX_a}(\sF[n]) \ra \Lift_{x,\sX}(\sF[n])
\]
is isomorphism. The last assertion follows from the presentation of $\sX$ (\ref{eq:presentation-of-sX}).
\end{proof}

\subsection{Usual properties of Pro-cotangent spaces}
\label{subsec:properties-of-pro-cotangent-spaces}

In this section we study some properties of (pro-)cotangent spaces to a prestack of Tate type. We also check that the cotangent space of a Tate scheme has the expected properties, namely it is connective.

\subsubsection{Connectivity}

\paragraph{t-structure} For the definition in this section it is useful to keep in mind the discussion of \cite{GR-II}*{Chapter 1, \S 3.1.1 and 3.1.2}. Namely, For $S\in \SchaffTate$ one can understand the coconnective part $\Pro(\TateCoh(S))^{\leq 0}$ of the t-structure on $\Pro(\TateCoh(S))$ as those left exact functors
\[
F: \TateCoh(S) \ra \Spc
\]
that send $\TateCoh(S)^{> 0}$ to a point.

\paragraph{}

One has the following definition

\begin{defn}
For $\sX \in \PStkTate$
\begin{enumerate}[(a)]
    \item given a point $S \overset{x}{\ra} \sX \in (\SchaffTate)_{/\sX}$ one says that \emph{$\sX$ admits an $(-n)$-connective pro-cotangent spaces at $x$} if it admits a pro-cotangent space and $T^*_{x}\sX \in \Pro(\TateCoh(S)^{\leq n})$;
    \item one says that $\sX$ admits an $(-n)$-connective pro-cotangent space if for every $S \overset{x}{\ra} \sX \in (\SchaffTate)_{/\sX}$ one has $T^*_{x}\sX \in \Pro(\TateCoh(S)^{\leq n})$.
\end{enumerate}
\end{defn}

\begin{rem}
One can also define what it means for $\sX$ to have an eventually connective pro-cotangent space at x, that is $T_{x}\sX \in \Pro(\TateCoh(S)^{\leq n})$ for some $n$. And similarly for (b) above.
\end{rem}

We say that $\sX$ admits a connective pro-cotangent space if it admits a $0$-connective pro-cotangent space. The following essentially follows from unwinding the definitions.

\begin{lem}
\label{lem:connective-pro-cotangent-is-commutation-with-finite-limits}
A prestack of Tate type $\sX$ admits a connective pro-cotangent space at $x:S \ra \sX$ if and only if the functor (\ref{eq:lift-functor}) commutes with finite limits.
\end{lem}

The following is a consequence of Corollary \ref{cor:affine-push-out-nil-isomorphism-agrees-with-scheme-push-out}.

\begin{cor}
\label{cor:Tate-scheme-has-connective-pro-cotangent-space}
Every Tate scheme $X \in \SchTate$ admits connective pro-cotangent spaces.
\end{cor}

\begin{proof}
As in Example \ref{ex:Tate-scheme-has-pro-cotangent-space}, given a point $x:S \ra Z$, by Lemma \ref{lem:connective-pro-cotangent-is-commutation-with-finite-limits} we notice that it is enough to check that the map
\[
\Maps_{\SchTate}(\imath(S\underset{S_{\sF_2}}{\sqcup}S_{\sF_1}),Z) \ra \Maps_{\SchTate}(S,Z)\underset{\Maps_{\SchTate}(S_{\sF_2},Z)}{\times}\Maps_{\SchTate}(S_{\sF_1},Z)
\]
is an isomorphism, where $\sF_1 \ra \sF_2$ is an arbitrary morphism in $\TateCoh(S)^{\leq 0}$. This follows from the general case of Corollary \ref{cor:affine-push-out-nil-isomorphism-agrees-with-scheme-push-out}.
\end{proof}

\subsubsection{Pro-cotangent vs cotangent}

\paragraph{}
\label{par:co-compact-condition-for-Pro-objects}
We recall that given an object $F \in \Pro(\TateCoh(S)^{\leq n})$, then $F$ belongs to $\TateCoh(S)^{\leq n}$ is an only if the corresponding functor
\[
F: \TateCoh(S)^{\leq n} \ra \Spc
\]
commutes with filtered limits, i.e.\ is a co-compact object. Thus, we obtain

\begin{cor}
For $\sX \in \PStkTate$ that admits a $(-n)$-connective pro-cotangent space at $x$, then it admits a $(-n)$-cotangent space at $x$ if an only if the functor (\ref{eq:lift-functor}) commutes with cofiltered limits.
\end{cor}

\paragraph{}

One has the following result

\begin{prop}
Every Tate scheme $Z \in \SchTate$ admits connective cotangent spaces.
\end{prop}

\begin{proof}
Given a point $x:S \ra Z$, we need to check that the composition
\[
(\TateCoh(S)^{\leq 0})^{\rm op} \overset{\RealSplitSqZ}{\ra} (\SchaffTate)_{S/} \overset{\imath}{\ra} (\SchTate)_{S/} \ra \SchTate
\]
commutes with filtered colimits. The claim for the first map follows from Lemma \ref{lem:realsplitsqz-Tate-preserves-colimits}, the claim for the second map is clear, finally the last claim follows from the fact that the forgetful functor commutes with colimits indexed by any contractible category.
\end{proof}

\subsubsection{Convergence}

\paragraph{}

As in \cite{GR-II}*{Chapter 1, \S 3.3.1} we define the subcategory 
\[
\conv{\Pro(\TateCoh(S)^-)} \subset \Pro(\TateCoh(S)^-) 
\]
as the full subcategory of functors $\Phi: \TateCoh(S)^- \ra \Spc$ that satisfies for any $\sF \in \TateCoh(S)^{\leq 0}$ the map 
\[
\Phi(\sF) \ra \lim_{n \geq 0}\Phi(\tau^{\geq-n}(\sF))
\]
is an isomorphism, where $\Phi$ is seem as a functor $\TateCoh^{\leq 0} \ra \Spc$.

\paragraph{}

The following result follows from the definitions

\begin{lem}
\label{lem:pro-cotangent-space-tested-on-eventually-coconnective-Tate-affine-schemes}
Suppose that $\sX \in \PStkTate$ is convergent and admits a pro-cotangent space at $(S \overset{x}{\ra} \sX) \in (\SchaffTate)_{/\sX}$, then $T^*_{x}\sX \in \conv{\Pro(\TateCoh(S)^{-})}$.
\end{lem}

Similarly, one has results that say that when $\sX$ is converngent it is enough to test for pro-cotangent spaces using eventually coconnective Tate affine schemes and eventually coconnective Tate-coherent sheaves.

\begin{lem}
\label{lem:pro-cotangent-space-convergent-Tate-enough-eventually-coconnective-affine}
\begin{enumerate}[(a)]
    \item Suppose that $\sX$ is a convergent prestack of Tate type, then $\sX$ admits pro-cotangent spaces (resp.\ $(-n)$-connective pro-cotangent spaces) if and only if for every $(S \overset{x}{\ra} \sX) \in (\SchaffTateconv)_{/\sX}$ and every $\sF_{1},\sF_2 \in \TateCoh(S)^{\geq -\infty,\leq 0}$ the map (\ref{eq:defining-cotangent-space}) is an isomorphism.
    \item Suppose that $f:\sX \ra \sX_0$ is a map between convergent prestacks of Tate type, then $\sX$ admits a pro-cotangent space relative to $\sX_0$ if and only if for every $(S_0 \overset{x_0}{\ra}\sX_0) \in (\SchaffTateconv)_{/\sX_0}$ the fiber $S_0\underset{\sX_0}{\times}\sX$ admits a pro-cotangent spaces.
\end{enumerate}
\end{lem}

\subsection{The locally a prestack and Tate conditions}
\label{subsec:deformation-theory-lp-and-Tate-conditions}

The goal of this section is to understand which subcategory of $\Pro(\TateCoh(X)^-)$ is equivalent to $\ProIndCoh(X)$, where $X \in \SchTatelaft$ (cf.\ \cite{GR-II}*{Chapter 1, \S 3.4}).

\subsubsection{Digression: right complete and bounded above Tate-coherent sheaves}

\paragraph{Completion of Pro-ind-coherent sheaves}

For $X$ a Tate scheme locally almost of finite type, we recall that $\Pro(\IndCoh(X))^{\leq n}$ is the subcategory of $\Pro(\IndCoh(X)^-)$ that sends $\IndCoh(X)^{\geq (n+1)}$ to a point. Let
\[
\Pro(\IndCoh(X))^{\wedge,-} := \lim_{n \geq 0}\left(\cdots \Pro(\IndCoh(X))^{\leq n+1} \overset{\tau^{\leq n}}{\ra} \Pro(\IndCoh(X))^{\leq n} \cdots \right),
\]
in other words $\Pro(\IndCoh(X))^{\wedge,-}$ is the right t-completion of the t-structure on $\Pro(\IndCoh(X))^-$.

\paragraph{Complete Tate-coherent sheaves}
\label{par:complete-Tate-coherent-sheaves}

We consider the subcategory  $\TateCoh(X)^{\wedge,-} \subset \Pro(\IndCoh(X))^{\wedge,-}$ defined as the objects $\Phi \in \Pro(\IndCoh(X))^{\wedge,-}$ such that
\begin{enumerate}[(i)]
    \item $\Phi \simeq \lim_{I}\Phi_n$ with each $\Phi_i$ representable in $\IndCoh(X)^{-}$, i.e.
    \[
    \Phi_i \simeq \Hom_{\IndCoh(X)^-}(\sG_i,-)
    \]
    for some $\sG_i \in \IndCoh(X)^{\leq n_{i}}$ and $n_i \geq \bZ$;
    \item for all $j \ra i$ a map in $I$, one has
    \[
    \Fib(\Phi_j \ra \Phi_i)
    \]
    commutes with filtered colimits, i.e.\ $\Fib(\sG_i \ra \sG_j) \in \Coh(X)$\footnote{Notice that by (i) one automatically has that $\Fib(\Phi_m \ra \Phi_n)$ is representable.}.
\end{enumerate}

\paragraph{Inclusion of Ind-coherent sheaves}
\label{par:inclusion-IndCoh-into-Pro-TateCoh-completed}

Recall that an element $\sF \in \IndCoh(X)$ can be described as
\[
\sF \simeq \colim_{I}\sF_{i}
\]
where each $\sF_{i} \in \Coh(X)$ and $I$ is a filtered diagram.

So we define the following functor
\begin{align}
\label{eq:jmath-functor}
    \jmath: \IndCoh(X) & \ra \Pro(\IndCoh(X))^{\wedge,-} \\
    \sF & \mapsto \lim_{I^{\rm op}}\Hom_{\IndCoh(X)^-}(\sF_i,-). \nonumber
\end{align}

The functor $\jmath(\sF): \IndCoh(X)^- \ra \Spc$ clearly commutes with finite limits, hence it is left exact, so an element of $\Pro(\IndCoh(X))^{\wedge,-}$. 
\begin{lem}
\label{lem:inclusion-IndCoh-into-completed-TateCoh}
The functor (\ref{eq:jmath-functor}) factors through the inclusion $\TateCoh(X)^{\wedge,-} \subset \Pro(\IndCoh(X))^{\wedge,-}$. We still denote the resulting functor by
\begin{equation}
    \label{eq:jmath-inclusion-into-TateCoh}
    \jmath: \IndCoh(X) \ra \TateCoh(X)^{\wedge,-}
\end{equation}
\end{lem}

\begin{proof}
This is an immediate consequence of conditions (i) and (ii) in the definition of $\TateCoh(X)^{\wedge, -}$ above.
\end{proof}

\subsubsection{Right Kan extension from Ind-coherent sheaves}

\paragraph{}

Restriction along (\ref{eq:jmath-inclusion-into-TateCoh}) gives a functor
\begin{equation}
\label{eq:restriction-of-Pro-objects-from-TateCoh-to-IndCoh}
\Pro(\TateCoh(X))^{\wedge,-} \ra \ProIndCoh(X).
\end{equation}

We start with a technical result

\begin{lem}
\label{lem:TateCoh-functors-RKEd-from-IndCoh}
Given a left exact functor $F: \TateCoh(X)^{-} \ra \Spc$ the following are equivalent:
\begin{enumerate}[(i)]
    \item $F$ commutes with cofiltered limits;
    \item $F$ is obtained by right Kan extension of its restriction to $\IndCoh(X)$ via (\ref{eq:jmath-inclusion-into-TateCoh}), i.e.\ the canonical map
    \[
    F \ra \RKE_{\jmath}\left(\left.F\right|_{\IndCoh(X)}\right)
    \]
    is an equivalence.
\end{enumerate}
\end{lem}

\begin{proof}
Let $\Phi \in \TateCoh(X)^-$, so one has
\[
\Phi \simeq \lim_{I}\Phi_i
\]
where each $\Phi_i \in \Pro(\IndCoh(X))^{\leq i}$ is representable, i.e.\ $\Phi_i \simeq \Hom_{\IndCoh(X)^-}(\sG_i,-)$ for some $\sG_{i} \in \IndCoh(X)^{\leq i}$, and for every $i \ra j$ one has that
\[
\Hom_{\IndCoh(X)^-}(\Fib(\sG_i \ra \sG_j),-)
\]
commutes with filtered colimits. Given $F: \TateCoh(X)^- \ra \Spc$ one has
\[
F(\Phi) \simeq F(\lim_{I}\Phi_i) \simeq \lim_I F(\Phi_i) \simeq \lim_I F \circ \jmath(\sG_i),
\]
where the last isomorphism follows from the Yoneda lemma. Now consider the category $\IndCoh(X)_{\Phi/\jmath(-)}$, i.e.\ its objects are $\sF \in \IndCoh(X)$ with a morphism $\Phi \ra \jmath(\sF)$. Now we notice that $\RKE_{\jmath}\left(\left.F\right|_{\IndCoh(X)}\right)$ can be concretely calculated as
\begin{align*}
\RKE_{\jmath}\left(\left.F\right|_{\IndCoh(X)}\right)(\Phi) & \simeq \lim_{\sF' \in \IndCoh(X)_{\Phi/\jmath(-)}}F(\jmath(\sF')) \\
& \simeq \lim_{\sH \in \CoLatt(\Phi)} F \circ \jmath(\sH)
\end{align*}
where in the last isomorphism we used Corollary \ref{cor:completed-colattices-are-initial-in-Ind}. Finally, we notice that since
\[
\Phi \simeq \lim_{I}\Phi_i \simeq \lim_{\sH \in \CoLatt(\Phi)}\sH,
\]
by \cite{Hennion-Tate}*{Corollary 3.16}, one has the claimed equivalence.
\end{proof}

\paragraph{Locally prestack objects}

The motivation for the following definition is that we want to understand the subcategory of $\Pro(\TateCoh(X))^{\wedge,-}$ that maps isomorphically to $\ProIndCoh(X)$, see Corollary \ref{cor:Pro-TateCoh-negative-recovers-ProInd}.

\begin{defn}
\label{defn:Pro-TateCoh-lp-and-Pro-TateCoh-dualizable}
We consider the following subcategories of $\Pro(\TateCoh(X)^{\wedge,-})$:
\begin{enumerate}[1)]
    \item the subcategory of \emph{locally prestack objects} $\Pro(\TateCoh(X))_{\rm lp} \subset \Pro(\TateCoh(X)^{\wedge,-})$ consisting of left exact functors $\Phi: \TateCoh(X)^{\wedge,-} \ra \Spc$ such that $\Phi$ commutes with cofiltered limits, i.e.
    \[
    \Phi \simeq \RKE_{\jmath}(\bar{\Phi})
    \]
    where $\bar{\Phi}:\IndCoh(X) \ra \Spc$ is the restriction of $\Phi$ via $\jmath$;
    \item the subcategory of \emph{locally dualizable objects} $\Pro(\TateCoh(X))_{\rm dualizable} \subset \Pro(\TateCoh(X))_{\rm lp}$ as those objects from 1) that furthermore satisfy the following: $\bar{\Phi} \simeq \lim_{I}\bar{\Phi}_i$ where
    \begin{enumerate}[(i)]
        \item $\bar{\Phi}_i$ is representable for each $i \in I$;
        \item for all $j \ra i$ in $I$ 
            \[
            \Fib(\Phi_j \ra \Phi_i)
            \]
            commutes with filtered colimits.
    \end{enumerate}
\end{enumerate}
\end{defn}

The following is just a restatement of Lemma \ref{lem:TateCoh-functors-RKEd-from-IndCoh}.

\begin{cor}
\label{cor:Pro-TateCoh-negative-recovers-ProInd}
For $X \in \SchaffTate$ the restriction of (\ref{eq:restriction-of-Pro-objects-from-TateCoh-to-IndCoh}) to the subcategory of locally prestack objects induces an equivalence of categories
\begin{equation}
    \label{eq:equivalence-Pro-Tate-lp-Pro-Ind-Coh}
    \Pro(\TateCoh(X))_{\rm lp} \overset{\simeq}{\ra} \ProIndCoh(X).    
\end{equation}

Moreover, the restriction of the above functor to locally dualizable objects gives an equivalence
\begin{equation}
\label{eq:equivalence-Pro-Tate-dualizable-Tate}
\Pro(\TateCoh(X))_{\rm dualizable} \simeq \TateCoh(X).
\end{equation}
\end{cor}

\begin{rem}
The equivalence (\ref{eq:equivalence-Pro-Tate-lp-Pro-Ind-Coh}) should be compared to the condition of a prestack of Tate type being locally a prestack, see Lemma \ref{lem:concrete-condition-locally-prestack}.
\end{rem}

\subsubsection{Relation to prestacks of Tate type which are locally a prestack}

\paragraph{}

Let $\sX \in \PStkTatelp$, that is there exists $\sX_0 \in \PStk$ such that
\begin{equation}
\label{eq:sX-RKE-of-sX_0}
    \sX \overset{\simeq}{\ra} \RKE_{\Schaffop \hra \SchaffTateop}(\sX_0).    
\end{equation}

Before stating the main result of this section we need the following explicit description of the real split square zero extensions for Tate-coherent sheaves on Tate affine schemes.

\begin{lem}
\label{lem:presentation-of-square-extension-of-Tate-affine-scheme}
For $S \in \SchaffTate$ and $\sF \in \TateCoh(S)^{\leq 0}$ presented as $S \simeq \colim_I S_i$ and $\sF \simeq \lim_{J}\sF_j$, one can describe $\RealSplitSqZ_S(\sF)$ as follows
\[
\RealSplitSqZ_S(\sF) = S_{\sF} \simeq \colim_I \colim_{J^{\rm op}}(S_i)_{(\sF_j)_i},
\]
where $(\sF_j)_i = g^{!}_i(\sF_{j})$, where $g_i:S_i \hra S$ is the canonical inclusion.
\end{lem}

\begin{proof}
This follows from unwinding the definitions of $\RealSplitSqZ$ in Proposition \ref{prop:RealSplitSqZ-Tate-affine-schemes}.
\end{proof}

Let $x:S \ra \sX$ be a point of $\sX$, this corresponds to a collection $\{T \overset{x_T}{\ra} \sX\}_{T \in \Schaff_{/S}}$ where each $T \ra S$ fits into a diagram\footnote{Notice that we denote by $\sX$ what we called $\sX_0$ above, since by definition one has $\sX_0(T) = \sX(T)$ for $T \in \Schaff$.}
\[
\begin{tikzcd}
T \ar[dr,"x_T"] \ar[d] & \\
S \ar[r,"x"] & \sX
\end{tikzcd}
\]
Given $S \simeq \colim_I S_i$ a presentation of $S$, since $I^{\rm op} \ra \Schaff_{/S}$ is initial one can restrict to the collection $\{S_i \overset{x_i}{\ra}\sX\}_I$ to determine $x \in \sX(S)$.

For $\sF \in \TateCoh(S)^{\leq 0}$ and a lift $\tilde{x} \in \sX(S_{\sF})$ of $x$, which is determined by a collection $\{T \overset{x_T}{\ra} \sX\}_{T \in \Schaff_{/S_{\sF}}}$ that fits into diagrams as above. We notice that 
\[
(i \in I^{\rm op}) \; \mapsto \; \left(S_i \ra (S_i)_{\sF_i}\right)
\]
is initial, by Lemma \ref{lem:presentation-of-square-extension-of-Tate-affine-scheme}. Thus, one has
\[
\sX(S_{\sF}) \simeq \lim_{I^{\rm op}}\left((S_i)_{\sF_i} \overset{\tilde{x}_{i}}{\ra}\sX\right).
\]

The above discussion implies
\begin{lem}
\label{lem:computation-cotangent-space-of-prestack-of-Tate-type-which-is-lp}
For any $\sX \in \PStkTatelp$, $S \simeq \colim_I S_i$ a Tate affine scheme and $\sF \in \TateCoh(S)^{\leq 0}$ one has an equivalence
\[
\Hom_{\Pro(\TateCoh(S)^{-})}(T^*_{x}\sX,\sF) \simeq \lim_{I^{\rm op}}\Hom_{\Pro(\IndCoh(S)^-)}(T^*_{x_i}\sX,g^*_i\sF),
\]
where $g_i: S_i \hra S$ is the canonical closed immersion. 

Moreover, if one picks a presentaiton of $\sF \simeq \lim_J\sF_j$, one has the formula
\[
\Hom_{\Pro(\TateCoh(S)^{-})}(T^*_{x}\sX,\sF) \simeq \lim_{I^{\rm op}}\lim_{J}\Hom_{\Pro(\IndCoh(S)^-)}(T^*_{x_i}\sX_0,g^*_i(\sF_j)).
\]
\end{lem}

\paragraph{}

Lemma \ref{lem:computation-cotangent-space-of-prestack-of-Tate-type-which-is-lp} immediately implies

\begin{cor}
\label{cor:cotangent-space-of-prestack-Tate-lp-is-lp}
For $\sX$ a prestack of Tate type, that admits a pro-cotangent space $T^*_{x}\sX$ at a point $(S \overset{x}{\ra} \sX) \in (\SchaffTate)_{/\sX}$, if $\sX$ is locally a prestack, then
\[
T^*_{x}\sX \in \Pro(\TateCoh(S))_{\rm lp}.
\]
\end{cor}

\begin{proof}
We notice that Lemma \ref{lem:computation-cotangent-space-of-prestack-of-Tate-type-which-is-lp} implies that $\Hom_{\Pro(\TateCoh(S)^-)}(T^*_{x}\sX,-)$ commutes with cofiltered limits. By Lemma \ref{lem:TateCoh-functors-RKEd-from-IndCoh} one has that $\Hom(T^*_x\sX,-)$ is the right Kan extension restriction of its restriction to $\IndCoh(S)^-$, i.e.\ it belongs to $\Pro(\TateCoh(S))_{\rm lp}$.
\end{proof}

\begin{rem}
\label{rem:pro-cotangent-of-locally-a-prestack-belongs-to-ProIndCoh}
For $\sX \in \PStkTatelp$ and a point $x:S \ra \sX$ that admits a pro-cotangent space, we will often abuse notation and denote by $T^*_x\sX$ the object of $\ProIndCoh(S)$ corresponding to $T^*_{x}\sX$ via the equivalence (\ref{eq:equivalence-Pro-Tate-lp-Pro-Ind-Coh}).
\end{rem}

\subsubsection{Comparison with usual pro-cotangent space of a prestack}
\label{subsubsec:comparison-pro-cotangent-spaces}

For $\sX$ a prestack of Tate type which is locally a prestack, i.e. satisfying (\ref{eq:sX-RKE-of-sX_0}). Let $S_0 \in \Schaff$ and consider\footnote{Here we denote $\imath: \Schaff \hra \SchaffTate$ the inclusion of affine schemes into Tate affine schemes.} $x: \imath(S_0) \ra \sX$ a point of $\sX$, we also denote by $x_0:S_0 \ra \sX_0$ the corresponding point of $\sX_0$. Assume that $\sX$ admits a pro-cotangent space at the point $x$, by Remark \ref{rem:pro-cotangent-of-locally-a-prestack-belongs-to-ProIndCoh} one has
\[
T^*_{x}\sX \in \ProIndCoh(S_0).
\]

We will denote by
\[
\overline{T}^*_{x_0}\sX_0 \in \Pro(\QCoh(S_0)^-)
\]
the pro-cotangent space of $\sX_0$ at $x_0$ as defined in \cite{GR-II}*{Chapter 1, \S 2.2.6}.

Recall that one has a fully faithful functor
\[
\Psi_{S_0}: \IndCoh(S_0)^- \ra \QCoh(S_0)^-
\]
that realizes $\QCoh(S_0)^-$ as the right t-completion of $\IndCoh(S_0)^-$. This induces a fully faithful functor
\[
\Pro(\Psi_{S_0}):\ProIndCoh(S_0)^- \ra \Pro(\QCoh(S_0)^-).
\]

\begin{prop}
\label{prop:pro-cotangent-of-a-convergent-prestack-Tate-lp-agrees-with-pro-cotangent-of-the-prestack}
For $\sX$ as above, assume that $\sX$ is convergent and that $T^*_{x}\sX \in \Pro(IndCoh(S_0)^-)$, then one has an equivalence
\[
\Pro(\Psi_{S_0})(T^*_{x}\sX) \simeq \overline{T}^*_{x_0}\sX_0.
\]
\end{prop}

\begin{proof}
Since $\sX$ is convergent, one has that
\[
T^*_{x}\sX \simeq \lim_{n \geq 0}\tau^{\geq -n}(T^*_{x}\sX),
\]
thus we can assume that $T^*_{x}\sX \in \Pro(\IndCoh(S_0)^{\geq -n})$ for some $n$. Notice that by \cite{GR-I}*{Chapter 4, Proposition 1.2.2 (c)} one has an equivalence $\Psi_{S_0}: \IndCoh(S_0)^{\geq -n} \ra \QCoh(S_0)^{\geq -n}$ which induces an equivalence
\begin{equation}
    \label{eq:compatibility-of-pro-cotangent-of-prestack-with-pro-cotangent-of-prestack-of-Tate-type}
    \Pro(\Psi_{S_0}): \Pro(\IndCoh(S_0)^{\geq -n}) \ra \Pro(\QCoh(S_0)^{\geq -n}).
\end{equation}

Since by Lemma \ref{lem:Tate-prestack-convergence-compatibility}, $\sX_0$ is also convergent, to compute $\overline{T}^*_{x_0}\sX_0$ by \cite{GR-II}*{Chapter 1, Lemma 3.3.4} it is enough to consider $\sF \in \QCoh(S_0)^{>\infty, \leq 0}$. Thus, for any $\sF \in \QCoh(S_0)^{>\infty, \leq 0}$, let $\sG \in \IndCoh(S_0)^{>\infty, \leq 0}$ such that
\[
\Psi_{S_0}(\sG) \simeq \sF.
\]
So, we obtain
\[
\Lift_{x_0,\sX_0}(\sF) \simeq \Hom_{\Pro(\QCoh(S_0)^{>\infty,\leq 0})}(\Pro(\Psi_{S_0})(T^*_{x}\sX),\sF) \simeq \Hom_{\Pro(\IndCoh(S_0)^{> \infty,\leq 0})}(T^*_{x}\sX,\sG)
\]
for all $\sF \in \QCoh(S_0)^{>\infty, \leq 0}$. This implies the isomorphism (\ref{eq:compatibility-of-pro-cotangent-of-prestack-with-pro-cotangent-of-prestack-of-Tate-type}).
\end{proof}

\paragraph{Digression: Ind-pro-cotangent spaces}

Let $\sX_0$ be a usual prestack, and consider $x:S \ra \sX_0$ a point of $\sX_0$. Suppose moreover that $S \in \Schaffconv$, then one has an adjunction
\[
\begin{tikzcd}
\QCoh(S)^{\leq 0} \ar[r,shift left=.5ex,"\Xi_S"] & \IndCoh(S)^{\leq 0} \ar[l,shift left=.5ex,"\Psi_{S}"].
\end{tikzcd}
\]

For every $\sF \in \IndCoh(S)^{\leq 0}$ and $x:S \ra \sX$, we will denote by $\overline{T}^{\rm Ind, *}_{x}\sX$ the object of $\Pro(\IndCoh(S)^-)$ that co-represents the assignment
\[
\Maps_{S/}(\RealSplitSqZ_{S}(\sF),\sX).
\]

Notice for $\sG = \Psi_{S}(\sF) \in \QCoh(S)^{\leq 0}$ one has
\[
\Maps_{S/}(S_{\sG},\sX) \simeq \Hom_{\Pro(\QCoh(S)^-)}(\overline{T}^*_x\sX,\sG).
\]

Recall that
\[
\RealSplitSqZ_S(\sF) = S_{\Psi_{S}(\sF)}.
\]
Thus, if $S$ in eventually coconnective one has an equivalence
\[
\Hom_{\Pro(\IndCoh(S)-)}(\Pro(\Xi_{S})(\overline{T}^*_x\sX),\sF) \simeq \Maps_{S/}(\RealSplitSqZ_{S}(\sF),\sX),
\]
where $\Pro(\Xi_S)$ is the Pro-extension of $\Xi_S$, that is
\begin{equation}
\label{eq:Ind-pro-cotangent-space}
(\Pro(\Xi_{S})(\overline{T}^*_x\sX) \simeq \overline{T}^{\rm Ind, *}_{x}\sX.
\end{equation}

\begin{prop}
\label{prop:second-isomorphism-pro-cotangent-spaces}
For $\sX$ a convergent prestack of Tate type, which is locally determined by the prestack $\sX_0$, and assume that $\sX$ admits a pro-cotangent space at  $x:\imath(S) \ra \sX$. Then, one has an equivalence
\[
T^*_{x}\sX \simeq \Pro(\Xi_S)(\overline{T}^*_{x_0}\sX_0),
\]
where $x_0:S \ra \sX_0$ is the underlying prestack point.
\end{prop}

\begin{proof}
Consider $\sF \in \TateCoh(S)^{\leq 0}$, which we represent as $\sF \simeq \lim_{I^{\rm op}}\sF_i$, for $\sF_i \in \IndCoh(S)^{\leq 0}$. Then, one has a chain of equivalences
\begin{align*}
    \Hom_{\Pro(\TateCoh(S)^-)}(T^*_{x}\sX) & \simeq \Maps_{S/}(\colim_{I^{\rm op}}\imath(S_{\sF_i}),\sX)\\
    & \simeq \lim_{I}\Maps_{S/}(\imath(S_{\sF_i}),\sX) \\
    & \simeq \lim_{I}\Hom_{\Pro(\IndCoh(S)^-)}(\overline{T}^{\rm Ind, *}_{X_0},\sF_i)\\
    & \simeq \lim_{I}\Hom_{\Pro(\QCoh(S)^-)}(\Pro(\Xi_S)(\overline{T}^{\rm Ind, *}_{X_0}),\sF_i)\\
    & \simeq \Hom_{\Pro(\QCoh(S)^-)}(\Pro(\Xi_S)(\overline{T}^{\rm Ind, *}_{X_0}),\sF)
\end{align*}
where between the third and forth lines we used (\ref{eq:Ind-pro-cotangent-space}).
\end{proof}

\paragraph{Existence of pro-cotangent space for prestacks of Tate type which are locally prestacks}

The following is a consequence of the proofs of Proposition \ref{prop:pro-cotangent-of-a-convergent-prestack-Tate-lp-agrees-with-pro-cotangent-of-the-prestack} and Proposition \ref{prop:second-isomorphism-pro-cotangent-spaces}.

\begin{cor}
\label{cor:pro-cotangent-space-existence-for-convergent-prestacks-Tate-lp}
Let $\sX$ be a convergent prestack of Tate type which is locally a prestack, and let $\sX_0 := \left.\sX\right|_{\Schaffop}$. For a morphism $x:\imath(S_0) \ra \sX$ and the corresponding map $x_0:S_0 \ra \sX_0$, the following are equivalent:
\begin{enumerate}[(i)]
    \item $\sX$ admits a pro-cotangent space at $x$;
    \item $\sX_0$ admits a pro-cotangent space at $x_0$.
\end{enumerate}
\end{cor}

The following gives a criterion to determine if a prestack of Tate type $\sX$ which is locally a prestack, admits a pro-cotangent space at a point $x:S \ra \sX$ where $S \in \SchaffTate$.

\begin{prop}
\label{prop:pro-cotangent-space-on-sX-from-sX_0}
Let $\sX$ be a convergent prestack of Tate type which is locally a prestack, and let $\sX_0 := \left.\sX\right|_{\Schaffop}$. Suppose moreover, that $\sX_0$ admits a pro-cotangent \emph{complex} (see \S \ref{subsec:pro-cotangent-complex}), then $\sX$ admits a pro-cotangent space at $x$.
\end{prop}

\begin{proof}
We notice that for any presentation $S \simeq \colim_I S_i$, where $S_i \in \Schaff$ and let $x_i:\imath(S_i) \ra \sX$ denote the induced maps. We have that for every $f_{j,i}: S_j \hra S_i$ we have an isomorphism
\begin{equation}
    \label{eq:compatibility-of-restriction-pro-cotangent-space}
    T^*_{x_j}\sX \simeq \Pro((f_{j,i})^{\rm Tate, *})(T*_{x_i}\sX).
\end{equation}
Since we suppose that $\sX$ is convergent, it is enough to consider $S \in \SchaffTateconv$, i.e.\ each $S_i$ is eventually coconnective. This implies that $(f_{j,i})^{\rm Tate,*} \simeq f^!$, thus the objects
\[
\{T^*_{x_j}\sX\}
\]
with the isomorphisms of (\ref{eq:compatibility-of-restriction-pro-cotangent-space}) determine an object of $\Pro(\TateCoh(S)^-)$ via the morphism
\[
\lim_{I^{\rm op}}\Pro(\TateCoh(S_i)) \ra \Pro(\TateCoh(S)^-).
\]
By the formulas of Lemma \ref{lem:computation-cotangent-space-of-prestack-of-Tate-type-which-is-lp} we have that $T^*_{x}\sX$ pro-corepresents the lifts of maps $S \ra \sX$.
\end{proof}

\subsubsection{Serre duality and tangent space of prestacks of Tate type}

The following is a generalization of Serre duality.

\begin{prop}
For $X \in \SchTatelaft$, there exists an equivalence
\begin{equation}
\label{eq:Serre-duality-for-Tate-affine-schemes}
    \bD^{\rm Tate}_X: \TateCoh(X)^{\rm op} \simeq \TateCoh(X).    
\end{equation}
Moreover, given any $\imath:X_0 \hra X$ a closed embedding, such that $X_0 \in \Schaft$ one has a commutative diagram
\[
\begin{tikzcd}
\TateCoh(X)^{\rm op} \ar[r,"(\imath^!)^{\rm op}"] \ar[d,"\b{D}^{\rm Tate}_X"] & \TateCoh(X_0)^{\rm op} \ar[d,"\bD^{\rm Tate}_{X_0}"] \\
\TateCoh(X) \ar[r,"\imath^!"] & \TateCoh(X_0)
\end{tikzcd}
\]
where 
\[
\bD^{\rm Tate}_{X_0}: \TateCoh(X_0)^{\rm op} \overset{\simeq}{\ra} \TateCoh(X_0)
\]
is the equivalence induced by the usual Serre duality $\bD_{X_0}:\Coh(X_0)^{\rm op} \overset{\simeq}{\ra} \Coh(X_0)$.
\end{prop}

\begin{proof}
Consider $X \simeq \colim_I X_i$ a presentation of $X$, by $X_i \in \Schaft$. One has
\[
\TateCoh(X)^{\rm op} \simeq \lim_{I^{\rm op}}\TateCoh(X_i)^{\rm op}
\]
and for each $i \in I$ one has an equivalence
\[
\bD^{\rm Tate}_{X_i}: \TateCoh(X_i)^{\rm op} \overset{\simeq}{\ra} \TateCoh(X_i)
\]
which assembles to give an equivalence
\[
\TateCoh(X)^{\rm op} \overset{\simeq}{\ra} \TateCoh(X).
\]

The second statement follows from the construction of the duality equivalence.
\end{proof}

The functor (\ref{eq:Serre-duality-for-Tate-affine-schemes}) gives rise to the following abstract equivalence.

\begin{cor}
\label{cor:duality-of-laft-Pro-category}
Suppose that $X \in \Sch_{\rm Tate, aft}$ (or simply $X\in \SchaffTateaft$), then one has an equivalence
\[
\Pro(\TateCoh(X))^{\rm op}_{\rm lp} \simeq \Ind(\Pro(\Coh(X))).
\]
And similarly,
\[
\Pro(\TateCoh(X))^{\rm op}_{\rm dualizable} \simeq \TateCoh(X).
\]
\end{cor}

\begin{proof}
The first statement follows from the definition of Ind and Pro categories.

For the second we notice that Corollary \ref{cor:duality-and-Tate-construction} implies that one has an equivalence
\[
\TateCoh(X)^{\rm op} \simeq \Tate(\Coh(X)^{\vee}) \overset{\Tate(\bD^{\rm Serre}_X)}{\ra} \TateCoh(X),
\]
where $\Tate(\bD^{\rm Serre}_X)$ denotes the functor induced by the Tate construction applied to (\ref{eq:Serre-duality-for-Tate-affine-schemes}).
\end{proof}

\paragraph{}

Let $\sX \in \PStkTatelp$ and assume $\sX$ admits a cotangent space, for a point $(S \overset{x}{\ra}\sX) \in (\SchaffTate)_{/\sX}$ one wants to understand under what conditions one has
\[
T^*_{x}\sX \in \Pro(\TateCoh(S))_{\rm dualizable} \subset \Pro(\TateCoh(S))_{\rm lp},
\]
i.e.\ $T^*_{x}\sX \in \TateCoh(S)$ via the equivalence (\ref{eq:equivalence-Pro-Tate-dualizable-Tate}).

\begin{prop}
\label{prop:prestack-of-Tate-type-has-Tate-cotangent-space}
Suppose that $\sX$ is a prestack of Tate type which is locally a prestack that admits a cotangent space at a point $(S \overset{x}{\ra}\sX) \in (\SchaffTate)_{/\sX}$, let $S \simeq \colim_I S_i$ be a presentation of $S$ and let 
\[
\begin{tikzcd}
S_j \ar[dr,"x_j"'] \ar[r,"f_{j,i}"] & S_i \ar[d,"x_i"] \\
 & \sX_0
\end{tikzcd}
\]
be morphisms determined by $x$. We suppose that
\begin{enumerate}[1)]
    \item for all $i$, one has $T^*_{x_i}\sX_0 \in \ProIndCoh(S_i)^{-}$, and
    \item for all $j \ra i$
    \[
    \Fib((f_{j,i})^{\rm ProInd}_*T^*_{x_j}\sX_0 \ra T^*_{x_i}\sX_0) \in \Coh(S_i).
    \]
\end{enumerate}
Then $T^*_x\sX \in \TateCoh(S) \subset \ProIndCoh(S)$.
\end{prop}

\begin{proof}
Let $\Phi: \IndCoh(S)^{\leq 0}$ be given by
\[
\Phi(\sG) = \Hom_{\Pro(\TateCoh(S)^-)}(T^*_x\sX,\sG).
\]
By Lemma \ref{lem:computation-cotangent-space-of-prestack-of-Tate-type-which-is-lp} one has $\Phi(\sG) \simeq \lim_{I^{\rm op}}\Phi_i(\sG)$ where
\[
\Phi_i(\sG) \simeq \Hom_{\Pro(\IndCoh(S_i)^-)}(T^*_{x_i}\sX,g^!_{i}(\sG)).
\]
The condition that for each $i$ one has $T^*_{x_i}\sX_0 \in \IndCoh(S_i)^{-}$ corresponds to condition 2) (i) from Definition \ref{defn:Pro-TateCoh-lp-and-Pro-TateCoh-dualizable}.

Finally, one notices that
\[
\Fib(\Phi_i \ra \Phi_j)(\sG) \simeq \Hom_{\IndCoh(S_i)^{-}}(\Fib(\varphi_{i,j})[1],g^!_i\sG),
\]
where $\varphi_{i,j}: (f_{j,i})_*T^*_{x_j}\sX_0 \ra T^*_{x_i}\sX_0$, and this commutes with colimits on the $\sG$ variable if and only if $g_{i,*}\Fib(\varphi_{i,j})$ is compact in $\IndCoh(S)$, which by \cite{DG-indschemes}*{Proposition 2.4.6} is equivalent to $\Fib(\varphi_{i,j}) \in \Coh(S_i)$.
\end{proof}

\begin{rem}
If one assume that $\sX_0$ admits a cotangent space, i.e.\ that for any point $x_0:S_0 \ra \sX_0$ where $S_0 \in \Schaff$ one has $T^*_{x_0}\sX_0 \in \IndCoh(S_0)^-$, then condition 1) of Proposition \ref{prop:prestack-of-Tate-type-has-Tate-cotangent-space} is automatically satisfied.
\end{rem}

\begin{rem}
If we suppose that $\sX_0$ admits a pro-cotangent complex (see \S \ref{subsec:pro-cotangent-complex}), then the morphism
\begin{equation}
\label{eq:adjoint-map-for-corepresentability}
T^*_{x_j}\sX_0 \ra (f_{j,i})^{\rm ProInd, *}T^*_{x_i}
\end{equation}
is an isomorphism, which by adjunction implies that condition 2) of Proposition \ref{prop:prestack-of-Tate-type-has-Tate-cotangent-space} is satisgied. Moreover, if the pro-cotangent complexes of $\sX_0$ are eventually connective we have that the existence of a pro-cotangent complex for $\sX_0$ in the sense of Gaitsgory--Rozenblyum implies that (\ref{eq:adjoint-map-for-corepresentability}) is an isomorphism.
\end{rem}

\begin{cor}
\label{cor:condition-for-pro-contangent-space-to-be-Tate-coherent}
Suppose that $\sX$ is a prestack of Tate type locally a prestack and moreover assume that its underlying prestack $\sX_0$ admits an eventually connective pro-cotangent complex (either in our sense or in the sense of Gaitsgory--Rozenblyum). Then for any point $x:S \ra \sX$ we have
\[
T^*_{x}\sX \in \TateCoh(S).
\]
\end{cor}

\paragraph{Tangent space}

For $\sX$ a prestack of Tate type which is locally a prestack such that the underlying prestack $\sX_0$ admits a pro-cotangent complex, we define the \emph{tangent space of $\sX$ at a point $x:S \ra \sX$} to be the dual of $T^*_{x}\sX$
\[
T_{x}\sX \in \TateCoh(S).
\]

\subsection{Pro-cotangent complex}
\label{subsec:pro-cotangent-complex}

Once we have the technical ingredients of sections \ref{subsec:Tate-coherent-sheaves-on-Tate-schemes} and \ref{subsec:preparations} in place, the definition of (pro-)cotangent complex and its properties for prestacks of Tate type follows \cite{GR-II}*{Chapter 1, \S 4} very closely. The main point of difference is that we will expect functoriality of pro-cotangent space for prestacks of Tate type with respect to $*$-pullback of Tate-coherent sheaves, and those are only defined for eventually coconnective morphisms. So we will only make sense of the theory of (pro-)cotangent complex for convergent prestacks of Tate type.

\subsubsection{Functoriality of (pro)-cotangent spaces}

\paragraph{Digression on eventually coconnective morphisms}

Recall that given a morphism $f:X \ra Y$ of schemes one says that $f$ is eventually coconnective if $f^*$ sends $\Coh(Y)$ to $\Coh(X)$. 

\begin{lem}
\label{lem:eventually-coconnective-and-pullback-is-ft-for-affines}
Let $f:X \ra Y$ be a morphism of affine schemes, the following are equivalent:
\begin{enumerate}[(i)]
    \item $f$ is eventually coconnective;
    \item for every $T \ra Y$ a morphism from $T \in \Schaffconv$ the pullback
    \[
    T\underset{Y}{\times}X
    \]
    is an eventually coconnective affine scheme, i.e.\ $T\underset{Y}{\times}X \in \Schaffconv$.
\end{enumerate}
\end{lem}

\begin{proof}
Let's assume (ii). Since the condition of $\sF \in \QCoh(Y)$ being coherent is local, it is enough to assume that $X$ and $Y$ are affine. In this situation the assumption that $X$ and $Y$ are locally almost of finite type means that $X$ and $Y$ are almost finite type. Let $Y = \Spec(B)$ and $X = \Spec(A)$. We want to check that for any $M \in \Mod^{\rm coh.}_{B}$ one has
\[
M\otimes_{B}A \in \Mod^{\rm coh.}_A.
\]
We recall that there is a Tor spectral sequence converging to $\H^{p+q}(M\otimes_{B}A)$ whose $E_2$ page is given by
\[
E^{pq}_2 := \Tor^p_{\H^*(B)}(\H^{*}(M)\otimes_{\H^*(B)}\H^*(A))^{q} \Rightarrow \H^{p-q}(M\otimes_{B}A).
\]
Since $M\in \Mod^{\rm coh.}_B$ and $\H^k(A)$ vanishes for $k << 0$, one has that the $q$-graded components of the $E_2$ page vanishes for all but finitely many $q$. Then because $\H^*(B)$ is eventually coconnective, one has that the $p$-Tor component also vanish for all but finitely many $p$. Thus, $\H^k(M\otimes_B A)$ vanishes for all but finitely many $k$ as well.

Assume (i), by considering a cover of $Y$ and $X$ it is enough to suppose that $X$ and $Y$ are affine, let $Y = \Spec(B)$ and $X = \Spec(A)$. Then the statement of (ii) is exactly that
\[
\H^k(B\otimes_B A)
\]
vanishes for all but finitely many $k$, which follows from $f^{*}(\sO_{Y}) \in \Coh(X)$.
\end{proof}

In particular, any morphism in the category $\Schaffconv$ is eventually coconnective.

\paragraph{Eventually coconnective pullbacks}

Let $f:S_1 \ra S_2$ be an eventually coconnective map of Tate affine schemes, and consider the functor
\[
f^{\rm Tate, *}: \TateCoh(S_1) \ra \TateCoh(S_2).
\]

By passing to Pro-objects one has an extension 
\[
\Pro(f^{\rm Tate, *}): \Pro(\TateCoh(S_1)) \ra \Pro(\TateCoh(S_2)).
\]

Assume that $\sX$ admits a pro-cotangent space at the point $x_2: S_2 \ra \sX$. Let $f: S_1 \ra S_2$ be a map in $\SchaffTateconv$ and denote $x_1 := x_2 \circ f$, by Lemma \ref{lem:RealSplitSqZ-of-TateCoh-for-Tate-schemes-is-functorial} we have a commutative diagram

\begin{equation}
\label{eq:TateCoh-pushforward-and-square-zero-extension}
\begin{tikzcd}
(\TateCoh(S_1)^{\leq 0})^{\rm op} \ar[r,"(f^{\rm Tate})^{\rm op}_*"] \ar[d,"\RealSplitSqZ_{S_1}"'] & (\TateCoh(S_2)^{\leq 0})^{\rm op} \ar[d,"\RealSplitSqZ_{S_2}"] \\
(\SchaffTateaft)_{S_1/} \ar[r,"- \underset{S_1}{\sqcup}S_2"'] & (\SchaffTateaft)_{S_2/}.
\end{tikzcd}    
\end{equation}

The commutative diagram (\ref{eq:TateCoh-pushforward-and-square-zero-extension}) gives a natural map
\begin{equation}
\label{eq:canonical-map-between-lifts}
    \Lift_{x_2,\sX}(f^{\rm Tate}_{*}(\sF_1)) \ra \Lift_{x_1,\sX}(\sF_1)
\end{equation}
which is functorial in $\sF_1 \in \TateCoh(S_1)^{\leq 0}$. By definition of pro-cotangent spaces we have the isomorphisms
\[
\Lift_{x_2,\sX}(f^{\rm Tate}_{*}(\sF_1)) \simeq \Hom_{\Pro(\TateCoh(S_2)^-)}(T^*_{x_2}\sX,f^{\rm Tate}_*(\sF_1)),
\]
and
\[
\Lift_{x_1,\sX}(\sF_1) \simeq \Hom_{\Pro(\TateCoh(S_1)^-)}(T^*_{x_1}\sX,\sF_1).
\]
Since we assumed $f$ eventually coconnective the adjuntion $(f^{\rm Tate, *},f^{\rm Tate}_*)$ gives a morphism
\[
\Hom(f^{\rm Tate, *}(T^*_{x_2}\sX),\sF_1) \ra \Hom(T^*_{x_1}\sX,\sF_1)
\]
for every $\sF_1\in \TateCoh(S_1)$.

Thus, we can interpret the map (\ref{eq:canonical-map-between-lifts}) as
\begin{equation}
    \label{eq:defn-map-pro-cotangent-complex}
    T^*_{x_1}(\sX) \ra \Pro(f^{\rm Tate,*})(T^*_{x_2}\sX)
\end{equation}
in $\Pro(\TateCoh(S_1)^-)$.

We will denote by $\PStkTateconv$ the category of convergent prestacks of Tate type, for $\sX \in \PStkTateconv$. Recall that by Lemma \ref{lem:pro-cotangent-space-convergent-Tate-enough-eventually-coconnective-affine} (a) for a convergent prestack of Tate type $\sX$ to determine if $\sX$ has a pro-cotangent space, it is enough to consider the pro-cotangent spaces $T^*_{x_2}\sX$ for $x_2:S_2 \ra \sX$ where $S_2 \in (\SchaffTateconv)_{/\sX}$. 

\begin{defn}
\label{defn:pro-cotangent-complex-Tate}
Let $\sX \in \PStkTateconv$ be a convergent prestack of Tate type that admits a pro-cotangent spaces, we say that $\sX$ admits a \emph{pro-cotangent complex} if for every point $(S_2 \overset{x_2}{\ra} \sX) \in (\SchaffTateconv)_{/\sX_2}$ and every morphism $f:S_1 \ra S_2$ between eventually coconnective Tate affine schemes the natural map (\ref{eq:defn-map-pro-cotangent-complex}) is an equivalence. 

For $\sX \in \PStkTate$ we will say that $\sX$ admits pro-cotangent complex, if 
\[
\sX \simeq \RKE_{\SchaffTateconv \hra \SchaffTate}({^{> \infty}\sX})
\]
and ${^{> \infty}\sX}$ admits a pro-cotangent complex.
\end{defn}

\begin{rem}
\label{rem:take-pushouts-to-pullbacks-existence-pro-cotangent-complex}
We notice that the condition from Definition \ref{defn:pro-cotangent-complex-Tate} is equivalent to requiring that the $\sX$ is convergent and takes pushouts of the form $(S_1)_{\sF_1}\underset{S_1}{\sqcup}S_2$ to pullbacks in $\Spc$.
\end{rem}

\begin{defn}
Given $\sX \in \PStkTate$ one says that $\sX$ admits a \emph{cotangent complex} if it admits cotangent spaces and pro-cotangent complexes.
\end{defn}

If $\sX$ admits a cotangent complex then we can extend the assignment
\[
(S \overset{x}{\ra} \sX) \in (\SchaffTateconv)_{/\sX} \rightsquigarrow T^*_x\sX \in \TateCoh(S)^-
\]
to an assignment 
\begin{equation}
\label{eq:pro-cotangent-complex-on-prestack-Tate-assignment}
(S \overset{x}{\ra} \sX) \in (\SchaffTate)_{/\sX} \rightsquigarrow T^*_x\sX \in \TateCoh(S)^-    
\end{equation}
where we consider $!$-pullbacks\footnote{Notice that for coconnective morphisms these are the same as $*$-pullback.}.

Since one has an equivalence
\[
\TateCoh^!_{\SchaffTate} \ra \RKE_{\SchaffTateconvop \hra \SchaffTateop}(\TateCoh^!_{\SchaffTateconv})
\]
proved similarly as in \cite{GR-I}*{Chapter 5, Lemma 3.2.4 and Chapter 4, \S 6.4} the assignment (\ref{eq:pro-cotangent-complex-on-prestack-Tate-assignment}) defines an object $T^*\sX$ of $\TateCoh(\sX)^-$ (see Section \ref{subsubsec:extension-TateCoh-to-PStk-Tate} for the definition of this category). We will refer to $T^*\sX$ as the \emph{cotangent complex} of $\sX$.

The following follows from Remark \ref{rem:take-pushouts-to-pullbacks-existence-pro-cotangent-complex}.

\begin{cor}
Any Tate scheme $X$ which is convergent as a prestack of Tate type admits a connective cotangent complex.
\end{cor}

\begin{rem}
Notice that for $X$ a Zariski Tate stack, one can prove that $X$ is convergent. Namely, we follow the arguments of \cite{GR-I}*{Chapter 2, Proposition 3.4.2} with the necessary modifications for Tate affine schemes.
\end{rem}

\begin{defn}
Given $f:\sX \ra \sX_0$ a morphism of convergent prestacks of Tate type, which admits a pro-cotangent space at every point of $\sX$, we say that $f$ admits a \emph{relative pro-cotangent complex} if for any point $(S_2 \overset{x_2}{\ra} \sX) \in (\SchaffTateconv)_{/\sX}$ and a map $g:S_1 \ra S_2$ the canonical morphism
\[
T^*_{x_1}(\sX/\sX_0) \ra \Pro(g^{\rm Tate,*})(T^*_{x_2}(\sX/\sX_0))
\]
is an isomorphism.
\end{defn}

One can prove similar results as in the absolute case for the relative situation.

\subsubsection{Conditions on the pro-cotangent complex}

This section follows closely \cite{GR-II}*{Chapter 1, \S 4.2}, we formulate the natural extensions of those conditions to the pro-cotangent complex of prestacks of Tate type.

\paragraph{Connectivity conditions}

\begin{defn}
\label{defn:coconnectivity-pro-cotangent-complex}
For $\sX \in \PStkTate$ one says that
\begin{enumerate}[(a)]
    \item $\sX$ admits an \emph{$(-n)$-connective pro-cotangent complex} (resp.\ \emph{cotangent complex}) if it admits a $(-n)$-connective pro-cotangent space (resp. cotangent space) and admits a pro-cotangent complex;
    \item $\sX$ admits a \emph{locally eventually connective pro-cotangent complex} (resp.\ \emph{cotangent complex}) if it admits a eventually connective pro-cotangent space (resp. cotangent space) and admits a pro-cotangent complex.
\end{enumerate}
\end{defn}

For example one has that any $X \in \SchTate$ admits a connective cotangent complex.

\paragraph{Pro-cotangent vs cotangent complex}

One has the following:

\begin{lem}
Let $\sX$ be a convergent prestack of Tate type that admits a locally eventually connective pro-cotangent complex. Suppose that $\sX$ admits a cotangent space for every $(S \overset{x}{\ra} \sX) \in (\SchaffTateconv)_{/\sX}$. Then $\sX$ admits a cotangent complex.
\end{lem}

\begin{proof}
Since $\sX$ admits a locally eventually connective pro-cotangent space, it is enough to check that for $\sT$ an object of $\Pro(\TateCoh(S)^{> \infty, \leq 0})$ such that for every $n \geq 0$ and truncation $\imath_n:{^{\leq n}S} \hra S$, we have
\[
\Pro(\imath^!_n)(\sT) \in \TateCoh({^{\leq n}S})^{>\infty, \leq 0}.
\]
Then $\sT \in \TateCoh(S)^{>\infty, \leq 0}$.

This follows from the observation Lemma \ref{lem:convergence-of-connective-TateCoh} below.
\end{proof}

\begin{lem}
\label{lem:convergence-of-connective-TateCoh}
The functors $\{\imath^!_n\}$ define an equivalence
\[
\TateCoh(S)^{> \infty, \leq 0} \ra \lim_{n}\TateCoh({^{\leq n}S})^{>\infty, \leq 0}.
\]
\end{lem}

\begin{proof}
By \cite{GR-II}*{Chapter 1, Lemma 4.2.8} one has an equivalence
\begin{equation}
\label{eq:QCoh-convergent-on-affine-schemes}
\QCoh(T)^{\leq 0} \ra \lim_{n}\QCoh({^{\leq n}T})^{\leq 0}    
\end{equation}
for affine schemes $T$. Since by construction quasi-coherent sheaves on Tate affine schemes are obtained by right Kan extensions and that commutes with limits, the isomorphism (\ref{eq:QCoh-convergent-on-affine-schemes}) also holds for Tate affine schemes.
Notice that by \cite{GR-I}*{Chapter 4, Proposition 1.2.2 (c)} one has an equivalence
\[
\IndCoh(S)^{> \infty, \leq 0} \overset{\simeq}{\ra} \QCoh(S)^{> \infty,\leq 0}
\]
which extends to Tate affine schemes $S$. Thus, one has an isomorphism
\begin{equation}
    \label{eq:IndCoh-convergent-on-Tate-affine-schemes}
    \IndCoh(S)^{> \infty, \leq 0} \overset{\simeq}{\ra} \lim_n\IndCoh({^{\leq n}S})^{> \infty, \leq 0}
\end{equation}
for $S \in \SchaffTateconv$. Finally, taking Pro-objects and passing to the subcategory of Tate objects commutes with limits, thus the isomorphism (\ref{eq:IndCoh-convergent-on-Tate-affine-schemes}) implies the result.
\end{proof}

\begin{prop}
\label{prop:explicit-cotangent-complex-Tate-scheme}
Let $Z$ be a Tate scheme locally almost of finite type, then $Z$ admits a connective cotangent complex. Moreover, given a presentation $Z \simeq \colim_I Z_i$ one has an isomorphism
\begin{equation}
\label{eq:cotangent-complex-Tate-scheme-as-a-limit-of-cotangent-complexes}
T^*(Z) \simeq \lim_{I^{\rm op}}T^*(Z_i)    
\end{equation}
in the category $\TateCoh(Z)^{\leq 0}$.
\end{prop}

\begin{proof}
As in Example \ref{ex:Tate-scheme-has-pro-cotangent-space} one obtains that for $Z \in \SchTatelaft$ the associated prestack $\Maps_{\SchTate}(\imath(-),Z)$ commutes with pushouts of the form $(S_1)_{\sF_1}\underset{S_1}{\sqcup}S_2$, thus by Remark \ref{rem:take-pushouts-to-pullbacks-existence-pro-cotangent-complex} it admits a pro-cotangent complex. Finally, by Corollary \ref{cor:Tate-scheme-has-connective-pro-cotangent-space} one obtains that $Z$ admits a connective cotangent complex.
The last statement is a consequence of Proposition \ref{prop:pro-cotangent-space-of-colimit-of-prestacks}.
\end{proof}

The following is a consequence of Proposition \ref{prop:explicit-cotangent-complex-Tate-scheme} and Proposition \ref{prop:pro-cotangent-space-of-colimit-of-prestacks}.

\begin{cor}
\label{cor:cotangent-complex-fiber-sequence-as-a-limit}
Let $f:S \ra T$ denote a morphism between Tate schemes almost of finite type. Then there are presentations $S \simeq \colim_I S_i$ and $T \simeq \colim_I T_i$ such that the fiber sequence
\[
T^*_{f}(T) \ra T^*(S) \ra T^*(S/T)
\]
is equivalent
\[
\lim_{I^{\rm op}_{f/}}T^*_{f_{i}}(T_i) \ra \lim_{I^{\rm op}}T^*_{i}(S_i) \ra T^*_{i}(S_{i}/T_i),
\]
where the left most terms are defined as in \S \ref{par:split-square-zero-for-Tate-schemes}.
\end{cor}

\subsubsection{The pro-cotangent complex as an object of a category}

In this section we will set up preliminaries to show that for a prestack of Tate type that is locally a prestack one can define its tangent complex.

\paragraph{}

Let $\sX$ be a prestack, one defines the category
\[
\Pro(\TateCoh(\sX)^{\wedge,-})^{\rm fake} := \lim_{(S,x) \in (\SchaffTate)_{/\sX}}\Pro(\TateCoh(S)^{\wedge,-}).
\]

We notice that $\Pro(\TateCoh(\sX)^{\wedge,-})^{\rm fake}$ is \emph{not} isomorphic to $\Pro(\TateCoh(\sX)^{\wedge,-})$, which is defined as the right t-completion of the category of Pro-objects
\[
\Pro(\TateCoh(\sX)^-) := \Pro(\bigcup_{n} \TateCoh(\sX)^{\leq n}).
\]

There is a natural embedding
\[
\IndCoh(\sX) \hra \Pro(\TateCoh(\sX)^{\wedge,-})^{\rm fake}
\]
induced by the maps $\IndCoh(S) \hra \Pro(\TateCoh(S)^{\wedge,-})$ (see \ref{par:inclusion-IndCoh-into-Pro-TateCoh-completed}). 

\paragraph{}

By definition, if $\sX$ admits a pro-cotangent complex, then we have a well-defined object
\[
T^*\sX \in \Pro(\TateCoh(\sX)^{\wedge,-})^{\rm fake}
\]
whose value on $(S,x) \in (\SchaffTate)_{/\sX}$ is $T^*_{x}\sX \in \Pro(\TateCoh(S)^{\wedge,-})$ where we abuse notation and don't distinguish between the object $T^*_{x}\sX \in \Pro(\TateCoh(S)^{-})$ and its image on the right completion.

\paragraph{}

Let 
\[
\Pro(\TateCoh(\sX))^{\rm fake}_{\rm lp} \subset \Pro(\TateCoh(\sX)^{\wedge,-})^{\rm fake}
\]
denote the subcategory obtained by as
\[
\lim_{(S,x) \in (\SchaffTate)_{/\sX}}\Pro(\TateCoh(S))_{\rm lp}.
\]

\paragraph{}

Suppose that $\sX \in \PStkTatelp$ and that $\sX$ admits a pro-cotangent complex then by Corollary \ref{cor:cotangent-space-of-prestack-Tate-lp-is-lp} one has
\[
T^*\sX \in \Pro(\TateCoh(\sX))^{\rm fake}_{\rm lp}.
\]

\subsubsection{The tangent complex}

\paragraph{}

Let $\sX \in \PStkTatelp$ and suppose that $\sX$ admits a pro-cotangent complex. By Corollary \ref{cor:Pro-TateCoh-negative-recovers-ProInd} one can consider
\[
T^*\sX \in \ProIndCoh(\sX),
\]
since
\[
\ProIndCoh(\sX) \simeq \lim_{(S,x) \in (\SchaffTate)_{/\sX}}\ProIndCoh(S).
\]

\paragraph{}

Moreover, for $\sX \in \PStkTatelp$ one has 
\[
\left(\Pro(\TateCoh(\sX))^{\rm fake}_{\rm lp}\right)^{\rm op} \simeq \Ind(\Pro(\Coh(\sX))).
\]

Thus, one has a canonically defined object
\[
T\sX \in \Ind(\Pro(\Coh(\sX)))
\]
defined as the image of $T^*\sX$ via the equivalence above. We will refer to $T\sX$ as the \emph{tangent complex} of $\sX$.

\subsubsection{Cotangent complex of prestacks of Tate type which are locally prestacks}

\paragraph{}

Let $\sX \in \PStkTatelp$, such that
\[
\sX \simeq \RKE_{(\Schaff)^{\rm op} \hra (\SchaffTate)^{\rm op}}(\sX_0).
\]

The following is an important result to compare the existence condition of pro-cotangent complexes.

\begin{prop}
Let $\sX \in \PStkTatelp$ be as above, then the following are equivalent:
\begin{enumerate}[(i)]
    \item $\sX$ admits a pro-cotangent complex;
    \item $\sX_0$ admits a pro-cotangent complex.
\end{enumerate}
\end{prop}

\begin{proof}
We assume (i). Since $\sX_0$ is also convergent, we notice that by \cite{GR-II}*{Lemma 4.2.4} it is enough to check that the pro-cotangent space assignement for $\sX_0$ is functorial with respect to $f:S_1 \ra S_2$ in $\Schaffconv$. 
Let $x_2:\imath(S_2) \ra \sX$, $f: S_1 \ra S_2$ be a morphism in $\Schaffconv$ and $x_1 := x_2 \circ f$. Since $\sX$ admits a pro-cotangent complex and $\sX$ is convergent the morphism
\[
T^*_{x_1}\sX \ra \Pro(f^{\rm Tate, *})(T^*_{x_2}\sX)
\]
is an isomorphism in $\Pro(\TateCoh(S_2)^{-})$. By applying $\Pro(\Psi_{S_1})$ we obtain
\[
\Pro(\Psi_{S_1})T^*_{x_1}\sX \ra \Pro(\Psi_{S_1})(\Pro(f^{\rm Tate, *})(T^*_{x_2}\sX))
\]
since $\Pro(\Psi_{S_1})(\Pro(f^{\rm Tate, *}) \simeq \Pro(f*)(\Pro(\Psi_{S_1}))$ by applying Proposition \ref{prop:pro-cotangent-of-a-convergent-prestack-Tate-lp-agrees-with-pro-cotangent-of-the-prestack} one obtains that the morphism
\[
\overline{T}^*_{\overline{x_1}}\sX_0 \ra \Pro(f^*)(\overline{T}^*_{\overline{x_2}}\sX_0)
\]
is an isomorphism, where $\overline{x_{i}}:S_i \ra \sX_0$ denotes the morphisms induced by $x_{i}$ for $i=1,2$.

Assume that (ii) holds, let $\overline{x_2}:S_2 \ra \sX_0$ denote a point of $\sX_0$ and $f: S_1 \ra S_2$ a morphism, where $S_1$ and $S_2$ are objects of $\Schaffconv$. Since $\sX_0$ admits a pro-cotangent complex the canonical map
\begin{equation}
    \label{eq:pro-cotangent-underlying-prestack-exists}
    \overline{T}^*_{\overline{x_1}}\sX_0 \ra \Pro(f^*)(\overline{T}^*_{\overline{x_2}}\sX_0)
\end{equation}
is an isomorphism. By applying $\Pro(\Xi_{S_1})$ to (\ref{eq:pro-cotangent-underlying-prestack-exists}) and applying Propositon \ref{prop:second-isomorphism-pro-cotangent-spaces} one obtains
\[
T^*_{x_1}\sX \ra \Pro(f^{\rm Tate, *})(T^*_{x_2}\sX)
\]
is an isomorphism.

Finally, if $x_2:S_2 \ra \sX$ is a map from $S_2 \in \SchaffTate$ we can use Proposition \ref{prop:pro-cotangent-space-on-sX-from-sX_0} to reduce to the situation where $S_2 \in \Schaffconv$.

\end{proof}

\subsubsection{Co-differential}

We now introduce a useful object in the theory of cotangent complexes. Let $S \in \SchaffTate$, then by Corollary \ref{cor:Tate-scheme-has-connective-pro-cotangent-space} one has
\[
T^*(S) \in \TateCoh(S)^{\leq 0}.
\]

There is a canonical morphism of Tate affine schemes
\[
S_{T^*(S)} \overset{\fd}{\ra} S
\]
corresponding to the split square-zero extension determined by the identity morphism in
\[
\Hom_{\TateCoh(S)^{\leq 0}}(T^*(S),T^*(S)) \simeq \Maps_{S/}(S_{T^*(S)},S),
\]
where the map $S \ra S$ is the identity map.

More generally, let $\sX$ be a prestack of Tate type that admits pro-cotangent space, and $x:S \ra \sX$ a point. Then there is a map
\[
(dx)^*: T^*_x(\sX) \ra T^*(S),
\]
that corresponds to the composite map
\[
S_{T^*(S)} \overset{\fd}{\ra} S \overset{x}{\ra} \sX
\]
via the isomorphism
\[
\Hom_{\Pro(\TateCoh(S)^-)}(T^*_x(\sX),T^*(S)) \simeq \Maps_{S/}(S_{S_{T^*(S)}},\sX).
\]
We will refer to $(dx)^*$ as the \emph{codifferential} of $x$.

Finally, we notice that $\fd$ is local in the Zariski topology, so given $X \in \SchTateaft$, since $T^*(X) \in \TateCoh(X)^{\leq 0}$ by Corollary \ref{cor:Tate-scheme-has-connective-pro-cotangent-space}, we have a morphism
\[
\fd: X_{T^*(X)} \ra X,
\]
here $X_{T^*(X)}$ denotes the split square-zero extension of $X$ determined by $T^*(X)$ (see \S \ref{par:split-square-zero-for-Tate-schemes}).

\subsection{Square-Zero extensions}
\label{subsec:square-zero-extensions}

In this section we study general (i.e.\ non-split) square-zero extensions of Tate schemes, since we need to consider the cotangent complex of a Tate scheme as an object of its category of Tate-coherent sheaves we restrict the discuss to Tate schemes which locally almost of finite type.

\subsubsection{The notion of square-zero extensions}
\label{subsubsec:notion-square-zero-extension}

Let $X \in \SchTatelaft$, by Proposition \ref{prop:explicit-cotangent-complex-Tate-scheme} $T^*(X) \in \TateCoh(X)^{\leq 0}$, thus we introduce the category of \emph{square-zero extensions}
\[
(\TateCoh(X)^{\leq -1}_{T^*(X)/})^{\rm op}.
\]
One has a natural functor
\[
\RealSqZ: (\TateCoh(X)^{\leq -1}_{T^*(X)/})^{\rm op} \ra (\SchTatelaft)_{X/}
\]
given by
\[
\RealSqZ(T^*(X) \overset{\gamma}{\ra} \sF) := X \underset{X_{\sF}}{\sqcup}X,
\]
where $X_{\sF} \rightrightarrows X$ are given by the canonical projection $X_{\sF} \ra X$, and the composite map
\[
X_{\sF} \overset{\gamma}{\ra} X_{T^*(X)} \overset{\fd}{\ra} X.
\]

Notice that the canonical projection $X_{\sF} \ra X$ is a closed nil-isomorphism, thus by Lemma \ref{lem:push-out-of-nil-isomorphisms-for-Tate-schemes-exist} the push-out exists. Moreover, the map $X \ra X \underset{X_{\sF}}{\sqcup}X$ which corresponds to the inclusion of the first factor is also a closed nil-isomorphism.

Similarly to \cite{GR-II}*{Chapter 1, \S 5.1.2} one has an interpretation of $\RealSqZ$ as a left adjoint, which we leave to the reader to spell out.

We will denote by
\[
\SqZ(Z) := (\TateCoh(X)^{\leq -1}_{T^*(X)/})^{\rm op}
\]
the category of square-zero extensions of $X$. In particular, we will say that any morphism $X \hra X'$ in $(\SchTatelaft)_{X/}$ has a structure of square-zero extension if $X'$ is in the essential image of $\RealSqZ$.

For a fixed $\sF \in \TateCoh(X)^{\leq -1}$ we will refer to the space
\[
\Hom^{\TateCoh(X)^{\leq 0}}(T^*(X),\sF)
\]
as that of \emph{square-zero extension of $X$ by $\sI := \sF[-1]$}\footnote{Notice that we shift $\sF$ up by $1$.}. The reason for this is that given $X' = \RealSqZ(\sF)$ we have an exact sequence
\[
\imath_*(\sI) \ra \sO_{X'} \ra \imath_*(\sO_X),
\]
where $\imath:X \ra X'$ is the closed embedding.

Finally, we notice that the following is a pullback of categories
\[
\begin{tikzcd}
(\TateCoh(X)^{\leq 0})^{\rm op} \ar[r] \ar[d] & (\TateCoh(X)^{\leq -1}_{T*(X)/})^{\rm op} \ar[d] \\
\SplitSqZ(X) \ar[d] \ar[r] & \SqZ(X) \ar[d] \\
(\SchTatelaft)_{X//X} \ar[r] & (\SchTatelaft)_{X/}
\end{tikzcd}
\]
where as in \cite{GR-II}*{Chapter 1, \S 2.1} we introduce the notation
\[
\SplitSqZ(X) := (\TateCoh(X)^{\leq 0})^{\rm op}.
\]

\subsubsection{Functoriality of square-zero extensions}
\label{subsubsec:functoriality-of-square-zero-extensions}

Let $f:X_1 \ra X_2$ be a Tate affine morphism in $\SchTate$. Moreover, we will assume that $f$ is eventually coconnective. Then there is a canonically defined functor
\begin{equation}
    \label{eq:functoriality-sq-zero}
    \SqZ(X_1) \ra \SqZ(X_2)
\end{equation}
that makes the diagram
\[
\begin{tikzcd}
\SqZ(X_1) \ar[r] \ar[d,"\RealSqZ_{X_1}"] & \SqZ(X_2) \ar[d,"\RealSqZ_{X_2}"] \\
(\SchTatelaft)_{/X_1}\ar[r] & (\SchTatelaft)_{/X_2}
\end{tikzcd}
\]
commute, where the lower horizontal arrow is given by push-out.

The functor (\ref{eq:functoriality-sq-zero}) is defined by sending $\gamma_1:T^*(X_1) \ra \sF$ to
\[
\gamma_2: T^*(X_2) \ra f^{\rm Tate}_*(\sF),
\]
where $\gamma_2$ is obtained by the $(f^{\rm Tate,*},f^{\rm Tate}_*)$-adjunction applied to the composition
\[
f^{!}(T^*(X_2)) \simeq f^{\rm Tate,*}(T^*(X_2)) \overset{(df)^*}{\ra} T^*(X_1) \overset{\gamma_1}{\ra} \sF,
\]
where the isomorphism $f^! \simeq f^{\rm Tate, *}$ and the adjunction follows from $f$ being eventually coconnective, and $f$ being affine implies that $f^{\rm Tate}_{*}(\sF) \in \TateCoh(X_2)^{\leq -1}$.

The above construction makes the assignment
\[
X \rightsquigarrow \SqZ(X)
\]
into a functor
\[
(\SchTatelaft){\rm e.c.-aff} \ra \DG
\]
where $(\SchTatelaft){\rm e.c.-aff}$ denotes the full subcategory of $\SchTatelaft$ where we only consider Tate affine and eventually coconnective morphisms.

One can develop the concepts of square-zero further by considering the co-Cartesian fibration over $(\SchTatelaft){\rm e.c.-aff}$ whose fiber over $X$ is $\SqZ(X)$. Since we won't need these results at the moment we will not need these details.

\subsubsection{Digressions: square-extensions for eventually coconnective Tate schemes}

In this section we restrict to the special case of eventually coconnective Tate schemes almost of finite type, where we can relate the discussion of square-zero extension of \S \ref{subsubsec:notion-square-zero-extension} to the more intuitive picture in terms of ideals of definition.

Recall that a Tate scheme $S$ is eventually coconnective if $S \in \SchTaten$ for some $n$. We denote the category of eventually coconnective Tate schemes by $\SchTateconv$.

Suppose that $\imath:S \hra S'$ is a closed embedding of objects in $\SchTateconv$. Then we can find presentations $S \simeq \colim_J S_j$ and $S' \simeq \colim_J S'_j$ such that we have
\[
\begin{tikzcd}
S \ar[r,hook,"\imath"] & S' \\
S_j \ar[u,"f_j"] \ar[r,hook,"\imath_j"] & S'_j \ar[u,"f'_j"']
\end{tikzcd}
\]
where $\imath_j$ is a closed embedding.

Notice that since every $S_j,S'_j$ belongs to $\Schn$ for some $n$, for every $j$ we have
\[
\sO_{S_j} \in \IndCoh(S_j), \;\;\; \mbox{and} \;\;\; \sO_{S'_j} \in \IndCoh(S'_j).
\]

Thus, we have the objects
\[
\sO_S := \lim_{J^{\rm op}}\sO_{S_j}, \;\;\; \mbox{and} \;\;\; \sO_{S'} := \lim_{J^{\rm op}}\sO_{S'_j}
\]
of the categories $\TateCoh(S)$ and $\TateCoh(S')$. Moreover, we notice that for each $j$ we have a fiber sequence
\[
\sI_j \ra \sO_{S'_j} \ra (\imath_{j})^{\rm Ind}_*(\sO_{S_j})
\]
in the category $\IndCoh(S'_j)$. By passing to limits we notice that we obtain a fiber sequence\footnote{Notice here that we obtain the functor $\imath^{\rm Tate}_*$ as the limit because $\imath$ is a proper morphism.}
\[
\sI \ra \sO_{S'} \ra \imath^{\rm Tate}_*(\sO_{S})
\]
in the category $\TateCoh(S')$, where $\sI \simeq \lim_{J^{\rm op}} \sI_{j}$.

We want to prove the following analogue of \cite{GR-II}*{Chapter 1, Lemma 5.4.3}.

\begin{prop}
\label{prop:cotangent-complex-closed-embedding-calculation}
For $\imath:S \hra S'$ a closed embedding of eventually coconnective Tate schemes almost of finite type, which corresponds to a fiber sequence
\[
\sI \ra \sO_{S'} \ra \imath^{\rm Tate}_*(\sO_{S})
\]
in $\TateCoh(S')$. Then one has:
\begin{enumerate}[(i)]
    \item $\H^{\, 0}(T^*(S/S')) = 0$ and
    \[
    \H^{\, -1}(T^*(S/S')) = \H^{\, 0}(\imath^{\rm Tate, *}(\sI))
    \]
    as objects of $\TateCoh(S)^{\heartsuit}$.
    \item for $n\geq 0$, 
    \[
    \tau^{\geq -n}(\sI) = 0 \; \Rightarrow \; \tau^{\geq -(n+1)}(T^*(S/S')) = 0.
    \]
    In case (ii) we also have $\classical{S} \simeq \classical{S'}$ and
    \[
    \H^{\, -n-2}(T^*(S/S')) = \H^{\, -n-1}(\sI)
    \]
    as objects of $\TateCoh(S)^{\heartsuit} \simeq \TateCoh(S')^{\heartsuit}$.
\end{enumerate}
\end{prop}

\begin{proof}
Notice that the canonical fiber sequence
\[
T^*_{\imath}(S') \ra T^*(S) \ra T^*(S/S')
\]
becomes
\[
\lim_{J^{\rm op}}T^*(\imath_{i})(S'_i) \ra \lim_{J^{\rm op}}T^*(S_{j}) \ra \lim_{J^{\rm op}}T^*(S_{j}/S'_j),
\]
by Corollary \ref{cor:cotangent-complex-fiber-sequence-as-a-limit}.

The relative version of (\ref{eq:cotangent-complex-Tate-scheme-as-a-limit-of-cotangent-complexes}) implies that for every $j \in J$ we have
\[
T^*(S_{j}/S'_j) \simeq \Xi^{\rm Pro}_{S_j}(\overline{T}^*(S_{j}/S'_j)),
\]
where $\overline{T}^*(S_{j}/S'_j)$ is the relative cotangent complex in the sense of Gaitsgory--Rozenblyum. By Lemma 5.4.3 from \cite{GR-II}*{Chapter 1} one has
\[
\H^{0}(\overline{T}^*(S_j/S'_j)) = 0, \;\;\; \mbox{and} \;\;\; \H^0(\overline{T}^*(S_j/S'_j)) = \H^{-1}(\imath^{*}(\sI_j)).
\]

By passing to the limit on $J^{\rm op}$ we obtain
\[
\H^0(T^*(S/S')) \simeq (\Xi^{\rm Pro}_{S})^{\heartsuit}(\lim_{J^{\rm op}} \H^{0}(\overline{T}^*(S_j/S'_j)))
\]
which implies the first assertion. The second follows from
\[
\H^{-1}(T^*(S/S')) \simeq \lim_{J^{\rm op}}(\Xi_{S_j})^{\heartsuit}(\H^{-1}(\overline{T}^*(S_j/S'_j))) \simeq \lim_{J^{\rm op}}\H^{-1}(\Xi_{S_j}(\imath^{*}(\sI_j))).
\]
Finally, we notice that the right-hand side becomes
\[
\lim_{J^{\rm op}}\H^{-1}((\imath^{\rm Ind, *}(\sI_j)) \simeq \H^{-1}(\lim_{J^{\rm op}}\imath^{\rm Ind, *}(\sI_j))
\]
where we abuse notation and denoted $\sI_j = \Xi_{S_j}(\sI_j)$. This finishes the proof of item (i).

Item (ii) is proved similarly.
\end{proof}

\subsubsection{Square-zero extensions and truncations}

In this section we establish a result that allows one to understand a Tate scheme as successive square-zero extensions of its $n$-coconnective truncations. This is completely analogous to \cite{GR-II}*{Chapter 1, Proposition 5.4.2}.

\begin{prop}
\label{prop:derived-affine-Tate-as-square-zero-extensions-of-non-derived}
\begin{enumerate}[(a)]
    \item For $X \in \clSchTateft$ the category of square-zero extensions by objects of $\TateCoh(X)^{\heartsuit}$ is equivalent to the category of closed embeddings of classical Tate schemes of finite type $X \hra X'$ whose ideal of definition square to zero.
    \item For $X_{n} \in \SchTatenft$, the category
    \[
    \{X_{n+1} \in {^{\leq (n+1)}\mbox{Sch}}_{\rm Tate, ft} \;| \; {^{\leq n}X_{n+1}} \simeq X_n\}
    \]
    is canonically equivalent to that of square-zero extensions of $X_n$ by objects of
    \[
    \TateCoh(X_{n})^{\heartsuit}[n+1] \subset \TateCoh(X_n)^{\leq -1}.
    \]
\end{enumerate}
\end{prop}

\begin{proof}
Let $\imath:X \hra X'$ to prove (a) we notice that the fiber sequence
\[
T^*_{\imath}(X') \ra T^*(X) \ra T^{*}(X/X')
\]
is obtained a the limit over $I$ of
\[
T^*_{\imath_i}(X'_i) \ra T^*(X_i) \ra T^{*}(X_{i}/X'_i),
\]
where $X \simeq \colim_I$ and $X' \colim X'_i$ are presentations compatible with $\imath$. By the relative version of (\ref{eq:cotangent-complex-Tate-scheme-as-a-limit-of-cotangent-complexes}) we have
\[
T^{*}(X_{i}/X'_i) \simeq \Xi_{X_i}(\overline{T}^*(X_{i}/X'_i)),
\]
and finally we notice that for $\imath_i: X_i \hra X'_i$ a closed embedding of classical schemes we have an isomorphism
\[
T^{*}(X_{i}/X'_i) \simeq \sI_i/\sI_i^{2}[1],
\]
where $\sI_{i}$ is the ideal of definition of $X$ in $X'$. Thus, the map
\[
T^*(X_i) \ra \sI[1]
\]
obtained by passing to the limit on $I^{\rm op}$ gives the square-zero extension.

Now We prove (b). For $\imath:X_n \hra X_{n+1}$ consider the fiber sequence
\[
\imath^{\rm Tate}_*(\sF[-1]) \ra \sO_{X_{n+1}} \ra \imath^{\rm Tate}_*(\sO_{X_n})
\]
where $\sF \in \TateCoh(X_{n})^{\heartsuit}[n+2]$.

We claim that $X_{n+1}$ has the structure of a square-zero of $X_n$, corresponding to a map $T^*(X_{n}) \ra \sF$.

Indeed, we consider the fiber sequence
\[
T^*_{\imath}(X_{n+1}) \ra T^*(X_n) \ra T^*(X_n/X_{n+1}),
\]
then Proposition \ref{prop:cotangent-complex-closed-embedding-calculation} implies that
\[
\tau^{\geq -(n+2)}T^*(X_n/X_{n+1}) \simeq \sF.
\]

Thus, we obtain the morphism by truncation of $T^*(X_n) \ra T^*(X_n/X_{n+1})$.
\end{proof}

\subsubsection{Nilpotent embeddings}

For $f:X \ra Y$ a morphism between eventually coconnective Tate schemes locally almost of finite type, we say that $f$ is a nilpotent embedding if the ideal of $X$ in $Y$, obtained as in Proposition \ref{prop:cotangent-complex-closed-embedding-calculation} is nilpotent.

We have the following description of nilpotent embeddings in terms of square-zero extensions.

\begin{prop}
\label{prop:iterated-description-of-nilpotent-embeddings}
Let $X \ra Y$ a nilpotent embedding of eventually coconnective Tate schemes locally almost of finite type. Then there exists a squence of Tate schemes locally almost of finite type
\[
X = X^0_0 \hra X^1_0 \hra \cdots \hra X^n_0 = X_0 \hra X_1 \hra \cdots \hra Y,
\]
such that
\begin{itemize}
    \item each of the maps $X^i_{0} \hra X^{i+1}_{0}$ and $X_i \hra X_{i+1}$ has a structure of square-zero extension;
    \item for every $i$, the map $g_{i}: X_i \ra Y$ induces an isomorphism ${^{\leq i}X_i} \ra {^{\leq i}Y}$.
\end{itemize}
\end{prop}

\begin{proof}
By Proposition \ref{prop:derived-affine-Tate-as-square-zero-extensions-of-non-derived} we have a sequence of square-zero extensions for the classical part of $\classical{X}$ into the classical part of $\classical{Y}$
\[
\classical{X} = X^{0}_{\rm cl, 0} \hra X^{1}_{\rm cl, 0} \hra \cdots \hra X^{n}_{\rm cl, 0} = \classical{Y}.
\]
Thus, we define
\[
X^i_{0} := X^0_0 \underset{\classical{X}}{\sqcup} X^i_{\rm cl, 0}.
\]

By construction, $X_0 \ra Y$ induces an isomorphism $\classical{X_0} \ra \classical{Y}$.

For the construction of $X_i$ for $i\geq 0$ we proceed by induction. Suppose that we have $g_i: X_i \ra Y$, notice that $g_i$ induces an isomorphism between the underlying classical Tate schemes. Moreover by Proposition \ref{prop:cotangent-complex-closed-embedding-calculation} and since ${^{\leq i}X} \ra {^{\leq i}Y}$ is an isomorphism we have that $T^*(X_i/Y)$ is concentrated in cohomological degres $\leq -(i+1)$.

Consider the fiber sequence
\[
T^*_{g_i}(Y) \ra T^*(X_i) \ra T^*(X_i/Y)
\]
and take $\sI := \H^{\, -(i+1)}(T^*(X_i/Y))[k] \simeq \tau^{\geq -(i+1)}(T^*(X_i/Y))[-1]$. We let the extension $X_i \hra X_{i+1}$ be defined by the composite
\[
T^*(X_i) \ra T^*(X_i/Y) \ra \sI[1].
\]
Notice that the composite
\[
T^*_{g_i}(Y) \ra T^*(X_i) \ra \sI[1]
\]
is null-homotopic, thus we obtain a map $g_{i+1}:X_{i+1} \ra Y$.

By Proposition \ref{prop:cotangent-complex-closed-embedding-calculation} we have fiber sequences
\[
\sI \ra \sO_{Y} \ra (g_{i})^{\rm Tate}_*\sO_{X_i}, \;\;\; \mbox{and} \;\;\; \sI' \ra \sO_{Y} \ra (g_{i+1})^{\rm Tate}_*\sO_{X_{i+1}}
\]
and the diagram
\[
\begin{tikzcd}
\sI' \ar[r] \ar[d] & \sI \ar[r] \ar[d] & (g_i)^{\rm Tate}_*(\sI) \ar[d] \\
\sI' \ar[d] \ar[r] & \sO_{Y} \ar[r] \ar[d] & (g_{i+1})^{\rm Tate}_*(\sO_{X_{i+1}}) \ar[d] \\
0 \ar[r] & (g_{i})^{\rm Tate}_*(\sO_{X_i}) \ar[r,"\mbox{Id}"] & (g_{i})^{\rm Tate}_*(\sO_{X_i}).
\end{tikzcd}
\]
By Proposition \ref{prop:cotangent-complex-closed-embedding-calculation} the map $\sI \ra (g_i)^{\rm Tate}_*(\sI)$ identifies with the truncation map
\[
\sI \ra \tau^{\geq -i}(\sI).
\]
Thus, $\sI' \in \TateCoh(Y)^{\leq -(i+1)}$, which again by Proposition \ref{prop:cotangent-complex-closed-embedding-calculation} implies that $T^*(X_{i+1}/Y)^{\leq (i+2)}$, that is $X_{i+1} \ra Y$ satisfy the second condition.
\end{proof}

\subsection{Infinitesimal cohesiveness and deformation theory}
\label{subsec:infinitesimal-cohesiveness-and-deformation-theory}

The idea of infinitesimal cohesiveness is to describe lifts of a map from a Tate affine scheme $S$ into a prestack of Tate type $\sX$ to a map from a square-zero extension of $S$ in terms of data of morphism between an object in $\TateCoh(S)$ and the pro-cotangent space of $\sX$. When a prestack of Tate type is infinitesimally cohesive and admits pro-cotangent complexes we will say that it admits deformation theory.

\subsubsection{Infinitesimal cohesivenes of a prestack of Tate type}
\label{subsubsec:rewriting-infinitesimal-cohesiveness}

Let $\sX \in \PStkTate$ and consider $x:S \ra \sX$ an object of $(\SchaffTate)_{/\sX}$. For
\[
T^*(S) \overset{\gamma}{\ra} \sF \in \SqZ(S)
\]
and the corresponding square-zero extension
\[
S \hra S' := \RealSqZ(T^*(S) \overset{\gamma}{\ra} \sF)
\]
we obtain a map
\begin{equation}
\label{eq:defn-infinitesimal-cohesiveness}
\Maps_{S/}(S',\sX) \ra \pt \underset{\Maps(S_{\sF},\sX)}{\times}\Maps(S,\sX)    
\end{equation}
where $\pt \ra \Maps(S_{\sF},\sX)$ corresponds to the composite
\[
S_{\sF} \overset{\mbox{pr}}{\ra} S \overset{x}{\ra} \sX.
\]

\begin{defn}
\label{defn:infinitesimally-cohesive-prestack-Tate}
We say that $\sX$ is \emph{infinitesimally cohesive} if the map (\ref{eq:defn-infinitesimal-cohesiveness}) is an equivalence for all $(S \overset{x}{\ra} \sX) \in (\SchaffTate)_{/\sX}$ and all $(T^*(S) \overset{\gamma}{\ra} \sF) \in \SqZ(S)$.
\end{defn}

By Corollary \ref{cor:affine-push-out-nil-isomorphism-agrees-with-scheme-push-out} we obtain that

\begin{cor}
Any Tate scheme $Z \in \SchTate$ is infinitesimally cohesive.
\end{cor}

Similarly to Lemma \ref{lem:pro-cotangent-space-convergent-Tate-enough-eventually-coconnective-affine} for $\sX$ a convergent prestack of Tate type, it is enough to check the condition of Definition \ref{defn:infinitesimally-cohesive-prestack-Tate} for eventually coconnective $S$ and $\sF \in \TateCoh(S)^{> \infty}$.

\paragraph{Relative situation} We can define the notion of relative infinitesimal cohesiveness for morphism between prestacks of Tate type in a similar way to \S \ref{subsubsec:pro-cotangent-space-relative-situation}. We have results analogous to Lemma \ref{lem:pro-cotangent-space-relative-and-absolute-relation} and Lemma \ref{lem:pro-cotangent-space-convergent-Tate-enough-eventually-coconnective-affine} when $\sX$ and $\sX_0$ are convergent.

\paragraph{Rewriting the condition of infinitesimal cohesiveness}

Notice that the space
\[
\pt \underset{\Maps(S_{\sF},\sX)}{\times}\Maps(S,\sX)
\]
is isomorphic to the space of homotopies between the points
\[
S_{\sF} \overset{\mbox{pr}}{\ra} S \overset{x}{\ra} \sX
\]
and
\[
S_{\sF} \overset{\gamma}{\ra} S_{T^*(S)} \overset{\fd}{\ra} S \overset{x}{\ra} \sX
\]
of $\Maps_{S/}(S_{\sF},\sX)$.

So, if the prestack of Tate type $\sX$ admits pro-cotangent spaces, then the space $\pt \underset{\Maps(S_{\sF},\sX)}{\times}\Maps(S,\sX)$ is isomorphic to the space of null-homotopies of the composed map
\[
T^*_{x}(\sX) \overset{(dx)^*}{\ra} T^*(S) \overset{\gamma}{\ra} \sF.
\]

So, we have that a prestack $\sX$ is infinitesimally cohesive if given $x:S \ra \sX$ for all $(T^*(S) \overset{\gamma}{\ra} \sF) \in \SqZ(S)$ the canonical map of spaces
\begin{equation}
    \label{eq:null-homotopy-description}
    \Maps_{S/}(\RealSqZ(T^*(S) \overset{\gamma}{\ra} \sF),\sX) \ra \{ \mbox{ null homotopies of }T^*_{x}(\sX) \overset{(dx)^*}{\ra} T^*(S) \overset{\gamma}{\ra} \sF\}
\end{equation}
is an isomorphism.

\paragraph{Infinitesimal cohesiveness and truncation}

By combining section \ref{subsubsec:rewriting-infinitesimal-cohesiveness} and Proposition \ref{prop:derived-affine-Tate-as-square-zero-extensions-of-non-derived} we obtain

\begin{lem}
Suppose that $\sX \in \PStkTate$ admits $(-k)$-connective pro-cotangent spaces and is infinitesimally cohesive. For $S$ an object of $\SchaffTaten$, and $S_0 \subset \classical{S}$ given by a nilpotent ideal, the fibers of the map
\[
\Maps(S,\sX) \ra \Maps(S_0,\sX)
\]
are $(n+k)$-truncated.
\end{lem}

\subsubsection{Prestacks of Tate type with deformation theory}

In this section we introduce the condition of admitting deformation theory for prestacks of Tate type. Since in our treatment of the pro-cotangent complex we restricted ourselves to prestacks of Tate type which were convergent the following definition only involves two conditions, instead of three for Gaitsgory--Rozenblyum.

\begin{defn}
For $\sX$ a prestack of Tate type, we will say that $\sX$ admits deformation theory if 
\begin{enumerate}[(a)]
    \item $\sX$ admits a pro-cotangent complex\footnote{Recall this means that $\sX$ is required to be convergent.};
    \item $\sX$ is infinitesimally cohesive.
\end{enumerate}
\end{defn}

We also have variants of the above definition where we impose connective conditions on the pro-cotangent spaces and also a relative version. 

For $f:\sY \ra \sX$ a morphism in $\PStkTate$ we say that \emph{$\sY$ admits deformation theory relative to $\sX$} if the relative pro-cotangent complex $T^*(\sY/\sX)$ exists and for every $(S \ra \sY) \in (\SchaffTate)_{\sY}$ and $T^*(S) \overset{\gamma}{\ra} \sF \in \SqZ(S)$ the diagram
\[
\begin{tikzcd}
\Maps_{S/}(S',\sY) \ar[r] \ar[d] & \pt \underset{\Maps(S_{\sF},\sY)}{\times}\Maps(S,\sY) \ar[d] \\
\Maps_{S/}(S',\sX) \ar[r] & \pt \underset{\Maps(S_{\sF},\sX)}{\times}\Maps(S,\sX)
\end{tikzcd}
\]
is a pullback square, where $\pt \ra \Maps(S_{\sF},\sY)$ corresponds to the composite with the canonical projection.

We have the following analogue of Lemma \ref{lem:pro-cotangent-space-relative-and-absolute-relation}.

\begin{lem}
\label{lem:relative-deformation-theory-checked-on-affine-pullbacks}
A prestack $\sY$ admits defomrnation theory relative to $\sX$ if for every $S\in (\SchaffTate)_{/\sX}$ the prestack $S\underset{\sX}{\times}\sY$ admits deformation theory.
\end{lem}

\paragraph{Compatibility with push-outs}

The following is a characterization of existence of deformation theory in terms of preserving certain push-outs. To formulate the next result we will denote by
\[
\SchaffTateconvft := \SchaffTateconv \cap \SchaffTateaft,
\]
i.e.\ the category of Tate affine schemes of finite type.

\begin{prop}
Le $\sX$ be a convergent prestack of Tate type locally almost of finite type. Then the following are equivalent:
\begin{enumerate}[(i)]
    \item For every push-out $S'_1 \underset{S_1}{\sqcup}S_2$ in $\SchaffTateconvft$, where $S_1 \ra S'_1$ is a nilpotent embedding, the map
    \[
    \Maps(S'_2,\sX) \ra \Maps(S'_1,\sX)\underset{\Maps(S_1,\sX)}{\times}\Maps(S_2,\sX)
    \]
    is an isomorphism.
    \item For every push-out $S'_1 \underset{S_1}{\sqcup}S_2$ in $\SchaffTateconvft$, where $S_1 \ra S'_1$ is a square-zero extension, the map
    \[
    \Maps(S'_2,\sX) \ra \Maps(S'_1,\sX)\underset{\Maps(S_1,\sX)}{\times}\Maps(S_2,\sX)
    \]
    is an isomorphism.
    \item $\sX$ admits deformation theory.
\end{enumerate}
\end{prop}

\begin{proof}
The equivalence between (i) and (ii) follows from Proposition \ref{prop:iterated-description-of-nilpotent-embeddings}. From section \ref{subsubsec:functoriality-of-square-zero-extensions} we obtain that (iii) implies (ii). Finally the implication (i) $\Rightarrow$ (iii) is proved exactly as \cite{GR-II}*{Chapter 1, Proposition 7.2.5}.
\end{proof}

\subsubsection{Formal smoothness}

We introduce the notion of formal smoothness. Let $\sX$ be convergent prestack of Tate type locally almost of finite type. A prestack of Tate type $\sX$ is said to be \emph{formally smooth} if for every $S \ra S'$ a nilpotent embedding in $\SchaffTateconvft$ the map
\[
\Maps(S',\sX) \ra \Maps(S,\sX)
\]
is surjective on $\pi_0$. 

Similarly to \cite{GR-II}*{Chapter 1, Proposition 7.2.5} one has
\begin{prop}
\label{prop:characterization-of-formal-smoothness}
Let $\sX$ be a prestack of Tate type that admits deformation theory and is locally almost of finite type. Then the following are equivalent:
\begin{enumerate}[(i)]
    \item $\sX$ is formally smooth;
    \item For any $n \geq 0$ and $S \in \SchaffTateconvft$, the restriction map
    \[
    \Maps(S,\sX) \ra \Maps({^{\leq n}S},\sX)
    \]
    induces an isomorphism  on $\pi_n$;
    \item For any $(S,x) \in (\SchaffTateconvft)_{/\sX}$ and $\sF \in \TateCoh(S)^{\heartsuit}$, we have
    \[
    \Hom(T^*_x(\sX),\sF) \in \ProVect^{\leq 0}.
    \]
\end{enumerate}
\end{prop}

\subsubsection{Tate Artin stacks}

In this section we check that Tate $k$-Artin stacks admits deformation theory.

\begin{prop}
\begin{enumerate}[(a)]
    \item Let $\sX$ be a Tate $k$-Artin stack, then $\sX$ admits $(-n)$-connective corepresentable deformation theory, i.e.\ it admits deformation theory and a $(-n)$-connective cotangent complex;
    \item If $\sX$ is smooth over a Tate scheme $Z$, then for $x:S \ra \sX$, where $S \in \clSchaffTate$, the relative cotangent complex $T^*_x(\sX/Z) \in \TateCoh(S)^{\geq n,\leq 0}$.
\end{enumerate}
\end{prop}

One proves the Proposition by using induction with respect to the following Lemma.

\begin{lem}
For $f:\sY \ra \sX$ a morphism in $\PStkTate$, assume that
\begin{itemize}
    \item $\sX$ satisfies \'etale descent;
    \item $f$ is \'etale-locally surjective;
    \item $\sY$ admits deformation theory;
    \item $\sY$ admits deformation theory relative to $\sX$;
    \item $\sY$ is formally smooth over $\sX$.
\end{itemize}
Then $\sX$ admits deformation theory.
\end{lem}

\begin{proof}
Assume that
\[
S'_1 \underset{S_1}{\sqcup}S_2 \ra S'_2
\]
is a push-out in $\SchaffTate$, where $S_1 \ra S'_1$ has the structure of a square-zero extension, given a map $S'_2 \ra \sX$ we claim that the map
\[
\Maps_{S_2/}(S'_2,\sX) \ra \Maps_{S_1/}(S'_1,\sX)
\]
is an isomorphism. Because $\sX$ satisfies \'etale descent we can assume that $S'_2 \ra \sX$ admits a lift $S'_2 \ra \sY$.

Let $\sY^{\bullet}/\sX$ denote the \c{C}ech nerve of $f$, we get a commutative diagram
\[
\begin{tikzcd}
\left|\Maps_{S_2/}(S'_2,\sY^{\bullet}/\sX)\right| \ar[r] \ar[d] & \Maps_{S_2/}(S'_2,\sX) \ar[d] \\
\left|\Maps_{S_1/}(S'_1,\sY^{\bullet}/\sX)\right| \ar[r] & \Maps_{S_1/}(S'_1,\sX)
\end{tikzcd}
\]
where the horizontal arrows are monomorphisms.

By Lemma \ref{lem:relative-deformation-theory-checked-on-affine-pullbacks} we notice that each term of $\sY^{\bullet}/\sX$ admits deformation theory, thus the left vertical arrow is an isomorphism.

Now we claim that $\sY$ being formally smooth over $\sX$ implies that the horizonal morphism are surjective. Given $S \ra S'$ a square-zero extension and a map $x':S' \ra \sX$ and a map $y:S \ra \sY$ that fits in the diagram
\[
\begin{tikzcd}
S \ar[r] \ar[r,"y"] & \sY \ar[d,"f"] \\
S'\ar[r,"x'"] & \sX
\end{tikzcd}
\]
we claim that there exists a lift $y': S' \ra \sY$. Indeed, the data of $y'$ is equivalent to a null-homotopy of
\[
T^*_y(\sY/\sX) \ra \sF.
\]
But this map admits a null-homotopy since $\sF \TateCoh(S)^{\leq -1}$, and by Proposition \ref{prop:characterization-of-formal-smoothness} (iii) we have
\[
\H^{\; 0}(\Maps(T^*_{y}(\sY/\sX),\sF)) \simeq 0.
\]
\end{proof}

\subsection{Consequence of Deformation Theory}
\label{subsec:consequences-of-deformation-theory}

In this section we discuss some nice consequences for a prestack of Tate type that admits deformation theory.

\subsubsection{Digression: maps between prestacks of Tate type}

We denote by $\SchaffTatered$ the subcategory of classical Tate affine schemes $S_{0}$ that admit a presentation
\[
S_0 \simeq \colim_I S_{0,i},
\]
where each $S_{0,i}$ is a reduced classical affine scheme. For $\sX_{0} \in \clPStkTate$ we denote by $\red{\sX_{0}}$ the restriction of $\sX_0$ to the subcategory of reduced classical Tate affine schemes
\[
\red{\sX_{0}}: (\SchaffTatered)^{\rm op} \ra \Spc.
\]

We consider the following analoguous condition to the ones introduced in \cite{GR-II}*{Chapter 1, \S 8.1}.

\begin{defn}
Given a map $f:\sX_1 \ra \sX_2$ in $\clPStkTate$
 \begin{enumerate}[(a)]
     \item We say that $f$ is a \emph{closed embedding} if for every $S_{2} \in (\clSchaffTate)_{/\sX_2}$, the fiber product $S_1 := S_{2}\underset{\sX_2}{\times}\sX_1$ taken in $\clPStkTate$ belongs to $\clSchaffTate$ and the morphism $S_{2}\underset{\sX_2}{\times}\sX_1 \ra S_2$ is a closed embedding;
     \item we say that $f$ is a \emph{nil-isomorphism} if it induces an isomorphism $\red{f}:\red{\sX_1} \ra \red{\sX_2}$;
     \item we say that $f$ is \emph{nil-closed} if for every $S_{2} \in (\clSchaffTate)_{/\sX_2}$, the induced morphism $\red{\left(S_{2}\underset{\sX_2}{\times}\sX_1\right)} \ra \red{S_2}$ is a closed embedding;
     \item we say that $f$ is a \emph{nilpotent embedding} if in the situation of (a) the morphism $S_1 \ra S_2$ is a nilpotent embedding;
     \item we say that $f$ is a \emph{pseudo-nilpotent embedding} if it is a nil-isomorphism and for every $S_{2} \in (\clSchaffTate)_{/\sX_2}$, there exists a commutative diagram
     \[
     \begin{tikzcd}
     S_1 \ar[r] \ar[d] & \sX_{1} \ar[d] \\
     S_2 \ar[r] & \sX_2
     \end{tikzcd}
     \]
     with $S_1 \in \clSchaffTate$ and $S_1 \ra S_2$ a nilpotent embedding.
 \end{enumerate}
\end{defn}

As usual we define the same properties for a prestack of Tate type by requiring them from the underlying classical prestack of Tate type.

\begin{defn}
 For a morphism $f:\sX_1 \ra \sX_2$ in $\PStkTate$ we say that $f$ is a \emph{closed embedding} (resp.\ \emph{nil-isomorphism}, \emph{nil-closed}, \emph{nilpotent embedding}, \emph{pseudo-nilpotent embedding}) if the corresponding map $\classical{f}: \classical{\sX_1} \ra \classical{\sX_2}$ of classical prestacks of Tate type is such.
\end{defn}

\subsubsection{Descent results}

In this section we prove a result that allows one to reduce questions of descent for a prestack of Tate type to the corresponding question for the underlying classical prestacks of Tate type.

\begin{prop}
 Suppose $\sX \in \PStkTate$ admits deformation theory and let $\sX_0 \ra \classical{\sX}$ be a pseudo-nilpotent embedding of classical prestacks.
 \begin{enumerate}[(a)]
     \item Assume that $\sX_0$ satisfies Zariski (resp.\ Nisnevich) descent. Then $\sX$ also has Zariski (resp.\ Nisnevich) descent;
     \item Assume that $\sX_0$ satisfies \'etale descent and that the pro-cotangent spaces of $\sX$ are locally eventually coconnective. Then $\sX$ also satisfies \'etale descent.
 \end{enumerate}
\end{prop}
 
 \begin{proof}
 By convergence and Proposition \ref{prop:derived-affine-Tate-as-square-zero-extensions-of-non-derived} it is enough to check that for $S \hra S'$ a map of Tate affine schemes that admits a structure of square-zero extension, $x:S \ra \sX$ a map and $\pi: T \ra S$ a Zariski (resp. Nisnevich, \'etale) cover, then the map
 \[
 \Maps_{S/}(S',\sX) \ra \Tot\left(\Maps_{T^{\bullet}/}(T'^{\bullet},\sX)\right)
 \]
 is an isomorphism, where $\pi: T' \ra S$ is the corresponding cover of $S$ and $T^{\bullet}$ (resp.\ $T'^{\bullet}$) denotes the \v{C}ech cover of $\pi$ (resp. $\pi'$)
 
 By using (\ref{eq:null-homotopy-description}) we rewrite $\Maps_{S/}(S',\sX)$ as the space of null-homotopies of the map
 \[
 T^*_{x}(\sX) \ra \sF, \;\;\; \mbox{ for some }\;\; \sF \in \TateCoh(S)^{> \infty}
 \]
 and $\Tot\left(\Maps_{T^{\bullet}/}(T'^{\bullet},\sX)\right)$ identifies with the totalization of the cosimplicial space of null-homotopies of the corresponding maps
 \[
 T^*_{x}(\sX) \ra \sF^{\bullet},
 \]
 where $\sF^{\bullet}$ is the \v{C}ech resolution of $\sF$ obtained from $\pi$.
 
 In the case of Zariski and Nisnevich descent one can replace the totalization by a finite limit, and the isomorphism follows from the commutation of $\Hom(T^*_x(\sX),-)$ with finite limits.
 
 In the case of \'etale descent the extra assumption guarantees that we can consider the finite totalization over the $m$-skeleton of the category $\Delta^{\rm op}$ for some $m$.
 \end{proof}
 
 \subsubsection{Isomorphism results}
 
 The following is a useful result to compare prestacks of Tate type when their underlying classical prestack of Tate type is isomorphic.
 
 \begin{prop}
  Let $f:\sX_1 \ra \sX_2$ denote a map between prestacks of Tate type that admit deformation theory. Suppose that we have a commutative diagram
  \[
  \begin{tikzcd}
  & \sY_{0} \ar[dl,"g_1"] \ar[rd,"g_2"] & \\
  \sX_1 \ar[rr,"f"] & & \sX_2
  \end{tikzcd}
  \]
  where $g_1$ and $g_2$ are pseudo-nilpotent embeddings, and $\sY \in \clPStkTate$. Suppose moreover, that for any $S_0 \in \clSchaffTate$ and a map $y_0: S_0 \ra \sY_0$, for $x_{i}:= g_{i} \circ y_0$ the induced map
  \[
  T^*_{x_2}(\sX_2) \ra T^*_{x_1}(\sX_1)
  \]
  is an isomorphism. Then $f$ is an isomorphim.
 \end{prop}
 
 \begin{proof}
 Again by induction and Proposition \ref{prop:derived-affine-Tate-as-square-zero-extensions-of-non-derived} we have to show that for any $S \in \SchaffTate$ and a $S \hra S'$ that has a structure of square-zero extension, for a map $x_1: S \ra \sX_1$, the space of extensions of $x_1$ to a map $S' \ra \sX_1$ maps isomorphically to the space of extensions of $x_2 := f\circ x_1$ to a map $S' \ra \sX_2$.
 
 Deformation theory implies that these spaces are equivalent to the spaces of null-homotopies of the maps
 \[
 T^*_{x_1}(\sX) \ra \sF, \;\;\; \mbox{and} \;\;\; T^*_{x_2} \ra \sF,
 \]
 respectively. Hence the result will follow if we show that $T^*_{x_2}(\sX_2) \ra T^*_{x_1}(\sX_1)$ is an isomorphism in $\Pro(\TateCoh(S)^-)$.
 
 The assumption of the proposition says that there exists a nilpotent embedding $g:\tilde{S} \ra S$, such that for $\tilde{x}_i := x_i \circ g$, the map
 \[
 T^*_{\tilde{x_2}}(\sX_2) \ra T^*_{\tilde{x_1}}(\sX_1)
 \]
 is an isomorphism. Thus the result follows from Lemma \ref{lem:restriction-to-nilpotent-is-conservative} below.
 \end{proof}
 
 \begin{lem}
 \label{lem:restriction-to-nilpotent-is-conservative}
 For a nilpotent embedding $g:\tilde{S} \ra S$, the functor
 \[
 \Pro(g^{\rm Tate, *}): {^{\rm conv}\Pro}(\TateCoh(S)^-) \ra {^{\rm conv}\Pro}(\TateCoh(S)^-).
 \]
 \end{lem}
 
 The proof is similar to that of \cite{GR-II}*{Chapter 1, Lemma 8.3.3}.

\section{Application: Determinant theory on a Tate scheme}
\label{sec:application-gerbe}

In this section we pursue an interesting application of the theory of Tate-coherent sheaves on Tate schemes. As a consequence of the existence of the !-pullback in this generality, we can define a dualizing Tate sheaf on any Tate scheme $p_X: X\ra \Spec(k)$ as $p^!(k)$. This sheaf plays the role of the semi-infinite exterior product that was considered in the literature, for instance in \cite{Ganter-Kapranov}. In our case we use the higher determinant map, constructed in \cite{determinant-map}, to define a $\Gm$-gerbe on any Tate scheme $X$. In the case where $X$ is sufficiently nice, i.e.\ $X$ is formally smooth and admits a presentation by smooth schemes, we can describe trivializations of this gerbe explicitly. These trivializations recover the notion of a determinantal theory, which was previously considered in the literature (\cites{Raskin-homological,Previdi}).

\subsection{Recollection on higher determinant map}

In this section we recall the determinant map for Tate objects in perfect complexes constructed in \cite{determinant-map}*{\S 4}. 

\subsubsection{$\Gm$-gerbes}

Recall that for $S$ an affine scheme we define the prestack of graded $\Gm$-gerbes by
\[
\sBPicgr(S) := \left|\sPicgr(S)^{\bullet}\right|,
\]
i.e.\ the geometric realization of the simplicial object in prestacks $\sPicgr(S)^{\bullet}$ defined as the \emph{bar construction}\footnote{For example, see \cite{HA}*{Construction 4.4.2.7}.} $\mbox{Bar}_{\sPicgr(S)}(\pt,\pt)$, here $\sPic(S)$ is a commutative group object in the category of spaces and we consider $\pt$ a space with a trivial action of $\sPicgr(S)$.

As in \cite{crystals}*{\S 6.1.3} we define $\Gm$-gerbes on any prestack 
\[
\sBPic_{\PStk}: (\PStk)^{\rm op} \ra \PicGrpd
\]
by considering the right Kan extension of
\[
\sBPic_{\Schaff}: (\Schaff)^{\rm op} \ra \PicGrpd
\]
via the canonical inclusion
\[
(\Schaff)^{\rm op} \hra (\PStk)^{\rm op}.
\]

In particular, we let
\[
\sBPic_{\SchTate}:= \left.\sBPic_{\PStk}\right|_{(\SchTate)^{\rm op}}
\]
denote the restriction to Tate schemes. Notice that given any morphism $f:S \ra T$ between Tate schemes, one obtains a functor
\[
f^{\rm Ge, *}:\sBPic(T) \ra \sBPic(S).
\]

\subsubsection{Prestack of perfect Tate objects}

For an affine scheme $S$ we let $\Perf(S)$ denote the category of perfect complexes on $S$, i.e.\ $\Perf(S)$ is the subcategory of compact objects of $\QCoh(S)$.

Let $\sX$ denote a prestack, we will consider the mapping prestack
\[
\sTate(\sX) := \Maps(\sX,\sTate),
\]
where $\sTate$ is the prestack defined as follows
\[
\sTate(S) = \Tate(\Perf(S))^{\simeq}.
\]

Similarly to the previous section, for Tate schemes we define
\[
\sTate_{\SchTate} := \left.\sTate\right|_{(\SchTate)^{\rm op}}
\]
the restriction of $\sTate$ on prestacks to Tate schemes. In particular, given $f:S \ra T$ a morphism of Tate schemes we denote
\[
f^{\sTate, *}: \sTate(T) \ra \sTate(S)
\]
the pullback functor. We notice that $f^{\sTate, *}$ is the functor induced by
\[
f^{\rm TatePerf, *}: \Tate(\Perf(T)) \ra \Tate(\Perf(S)),
\]
where $f^{\rm TatePerf, *}$ is simply the Tate-extension of the functor $f^*$, which preserves compact objects.

\subsubsection{Determinant gerbe}

In Section 4.3 from \cite{determinant-map} we constructed a morphism of prestacks
\[
\sD^{(1)}: \sTate \ra \sBPicgr,
\]
which extends to a morphism
\[
\sD^{(1)}: \sTate_{\SchTate} \ra \sBPicgr_{\SchTate}.
\]

Given a Tate scheme $S$ we will denote by
\begin{equation}
    \label{eq:gerbe-determinant-on-a-Tate-scheme}
    \sD^{(1)}_{S}: \sTate(S) \ra \sBPicgr(S)
\end{equation}
the higher determinant map for $S$. We will also need to consider
\[
\sD^{(0)}_{S}: \sPerf(S) \ra \sPicgr(S)
\]
the usual determinant map for perfect complex\footnote{See also \cite{determinant-map}*{\S 3} for a detailed construction of $\sD^{(0)}_S$ in the generality that we need.}. Notice that given $f:S \ra T$ a morphism of Tate schemes, one has the following commutative diagram
\begin{equation}
\label{eq:determinant-map-commutes-with-gerbe-pullback}
\begin{tikzcd}
\Tate(\Perf(T)) \ar[r,"\sD^{(1)}_T"] \ar[d,"f^{\rm TatePerf,*}"'] & \sBPic(T) \ar[d,"f^{\rm Ge,*}"] \\
\Tate(\Perf(S)) \ar[r,"\sD^{(1)}_S"'] & \sBPic(S).
\end{tikzcd}    
\end{equation}

\subsubsection{Determinant gerbe on perfect complexes}

For $S$ any scheme almost of finite type, consider the composition
\[
\left.\sD^{(1)}_S\right|_{\Perf(S)}: \Perf(S) \hra \Tate(\Perf(S)) \ra \sBPic(S).
\]

\begin{lem}
\label{lem:trivialization-of-determinant-1-on-perfect-complexes}
For any $\sF \in \Perf(S)$ the $\Gm$-gerbe $\left.\sD^{(1)}_S\right|_{\Perf(S)}(\sF)$ is trivial, i.e.\ one has a canonical isomorphism
\[
\eta_{\sF}: \left.\sD^{(1)}_S\right|_{\Perf(S)}(\sF) \overset{\simeq}{\ra} \sPic(S)
\]
given by 
\[
(\eta_{\sF})^{-1}(\sG) = (-)\otimes \sD^{(0)}(\sG): \sPic(S) \overset{\simeq}{\ra} \sPic(S),
\]
for every $\sG \subset \sF$.
\end{lem}

\begin{proof}
Since the construction of $\sD^{(1)}_S$ and $\sD^{(0)}_S$ is local, it is enough to consider the case where $S$ is an affine scheme. Let $\Gr(\sF)$ denote the direct set of lattices of $\sF$, i.e.\ $\imath_{\sG}:\sG \ra \sF$ such that $\Cofib(\imath_{\sG}) \in \Ind(\Perf(S))$. By \cite{BGW-index}*{Proposition 5.3} (see also \cite{determinant-map}*{Lemma 4.20}) the data of $\sD^{(1)}(\sF)$ is equivalent to a morphism
\[
\Delta_{\sF}: \Gr(\sF) \ra \sPic(S)
\]
together with isomorphisms
\[
\Delta_{\sF}(\sG') \simeq \Delta_{\sF}(\sG)\otimes \sD^{(0)}(\sG'/\sG)
\]
for every $\sG \ra \sG'$, and higher coherence isomorphisms. Notice that when $\sF \in \Perf(S)$, for any object $\sG \in \Perf(S)$ one has that $(\sG \overset{\imath_{\sG}}{\ra} \sF) \in \Gr(\sF)$. Thus, $\Delta_{\sF}$ is determined by $\Delta_{\sF}(0) \in \sPic(S)$.
\end{proof}

\subsection{Canonical gerbe on a Tate scheme}

Before formulating the statement we need to restrict to a particular class of Tate schemes, which have a naturally $\Gm$-gerbe on them.

\subsubsection{The notion of a nice Tate scheme}
\label{subsubsec:nice-Tate-schemes}

We recall the following notion from \cite{DG-indschemes}.

\begin{defn}
\label{defn:formally-smooth-prestack}
A prestack $\sX$ is said to be formally smooth if 
\begin{enumerate}[(a)]
    \item for every nilpotent embedding $S_0 \hra S'_0$ in $\clSchaff$ the map
    \[
    \classical{\sX}(S'_0) \ra \classical{\sX}(S_0)
    \]
    is surjective on $\pi_0$;
    \item for every $n \geq 0$ and $S \in \Schaff$ the map
    \[
    \sX(S) \ra \sX({^{\leq n}S})
    \]
    induces an isomorphism on $\pi_n$.
\end{enumerate}
\end{defn}

We say that $S \in \SchTate$ is formally smooth if it is formally smooth as a prestack.

Given $S_n \in \SchTaten$ an $n$-coconnective Tate scheme, we will say that $S_n$ is $\aleph_0$ if there exists a presentation
\[
S_n \simeq \colim_{I}S_{n,i}
\]
where $I$ is equivalent to the poset $\bN$. More generally, given $S \in \SchTate$ we say that $S$ is \emph{weakly $\aleph_0$} if for every $n$ the truncation ${^{\leq n}S}$ is $\aleph_0$.

\begin{defn}
\label{defn:nice-Tate-schemes}
For $S$ an object of $\SchTate$ we will say that $S$ is a \emph{nice Tate scheme} if it satisfies:
\begin{enumerate}[(a)]
    \item $S$ is formally smooth;
    \item $S$ is weakly $\aleph_0$;
    \item $S$ is locally almost of finite type.
\end{enumerate}
We will denote by $\SchTatenice$ the subcategory of $\SchTate$ generated by nice Tate schemes.
\end{defn}

\begin{rem}
\label{rem:smooth-lft-schemes-are-nice-Tate-schemes}
Notice that since smooth implies formally smooth, if $S_0$ is a smooth scheme locally of finite type, then $S_0$ is a nice Tate scheme.
\end{rem}

The following result is the reason we introduce the class of nice Tate schemes.

\begin{thm}[\cite{DG-indschemes}*{Theorem 10.1.1}]
\label{thm:equivalence-QCoh-IndCoh-nice-Tate}
For $S$ a nice Tate scheme one has an equivalence
\[
\Upsilon_S: \QCoh(S) \overset{\simeq}{\ra} \IndCoh(S).
\]
\end{thm}

Let $\Perf(S) := \QCoh(S)^{\rm comp.} \subset \QCoh(S)$ denote the subcategory of compact objects of $S$. Theorem \ref{thm:equivalence-QCoh-IndCoh-nice-Tate} gives an equivalence
\begin{equation}
\label{eq:new-psi-for-nice-Tate}
    \Upsilon^{\vee}_S : \IndCoh(S) \overset{\simeq}{\ra} \Ind(\Perf(S)).
\end{equation}

By passing to Tate-objects one has an extension of $\Upsilon^{\vee}_S$ to
\[
(\Upsilon^{\vee}_S)^{\rm Tate}: \TateCoh(S) \ra \Tate(\Perf(S)).
\]

Moreover, given any $f:S \ra T$ a morphism between nice Tate schemes the following diagram commutes
\begin{equation}
    \label{eq:psi-for-nice-Tate-pullback-diagram}
    \begin{tikzcd}
    \TateCoh(T) \ar[r,"(\Upsilon^{\vee}_T)^{\rm Tate}"] \ar[d,"f^!"'] & \Tate(\Perf(T)) \ar[d,"f^{\rm TatePerf, *}"] \\
    \TateCoh(S) \ar[r,"(\Upsilon^{\vee}_S)^{\rm Tate}"] & \Tate(\Perf(S))    
    \end{tikzcd}
\end{equation}

\subsubsection{Compatibility with schemes almost of finite type}

Let $S_0$ denote an eventually coconnective scheme almost of finite type. Then the canonical morphism $\Psi_{S_0}: \IndCoh(S_0) \ra \QCoh(S_0)$ extends to a functor:
\[
\Psi^{\rm Tate}_{S_0}: \TateCoh(S_0)) \ra \Tate(\Perf(S_0)).
\]

Let $S_0 \ra S$ denote a morphism between $S_0$ a smooth eventually coconnective scheme locally of finite type and $S$ a nice Tate scheme, one has the following compatibility between the functor from Theorem \ref{thm:equivalence-QCoh-IndCoh-nice-Tate} and the functor $\Psi_{S_0}:\IndCoh(S_0) \ra \QCoh(S_0)$.

\begin{lem}
For $\imath: S_0 \ra S$ as above, the diagrams 
\begin{equation}
\label{eq:compatibility-upsilon-schemes-and-nice-Tate}
\begin{tikzcd}
\IndCoh(S) \ar[r,"\Upsilon^{\vee}_{S}"] \ar[d,"\imath^!"] & \QCoh(S) \ar[d,"\imath^*"] \\
\IndCoh(S_0) \ar[r,"\Psi_{S_0}"] & \QCoh(S_0)
\end{tikzcd}    
\end{equation}
and
\begin{equation}
\label{eq:compatibility-Tate-upsilon-schemes-and-nice-Tate}
\begin{tikzcd}
\TateCoh(S) \ar[r,"(\Upsilon^{\vee}_{S})^{\rm Tate}"] \ar[d,"\imath^!"] & \Tate(\Perf(S)) \ar[d,"\imath^{\rm TatePerf,*}"] \\
\TateCoh(S_0) \ar[r,"\Psi^{\rm Tate}_{S_0}"] & \Tate(\Perf(S_0))
\end{tikzcd}
\end{equation}
commute.
\end{lem}

\begin{proof}
The commutativity of (\ref{eq:compatibility-upsilon-schemes-and-nice-Tate}) is a consequence of the proof of the existence of $\Upsilon_S$, see \cite{DG-indschemes}*{Theorem 10.1.1}. Then the commutativity of (\ref{eq:compatibility-Tate-upsilon-schemes-and-nice-Tate}) follows from the first diagram by extending the functor to Tate-objects.
\end{proof}

\subsubsection{Determinant of Tate-coherent objects}

In the case where $S$ is a nice Tate scheme we can define a higher determinant map for any Tate-coherent sheaf, as the composite
\begin{equation}
\label{eq:sD-for-TateCoh}
\TateCoh(S) \overset{(\Upsilon^{\vee}_S)^{\rm Tate}}{\ra} \Tate(\Perf(S)) \overset{\sD^{(1)}_S}{\ra} \sBPic(S).
\end{equation}

We let
\[
\sD^{(1),\rm coh.}_S: \TateCoh(S) \ra \sBPic(S)
\]
denote the composite (\ref{eq:sD-for-TateCoh}). The following is a consequence of (\ref{eq:determinant-map-commutes-with-gerbe-pullback}) and (\ref{eq:psi-for-nice-Tate-pullback-diagram}).

\begin{lem}
\label{lem:functoriality-of-coherent-determinant-1}
For any morphism $f:S \ra T$ between nice Tate schemes the following diagram
\[
\begin{tikzcd}
\TateCoh(T) \ar[r,"\sD^{(1), \rm coh.}_T"] \ar[d,"f^{!}"'] & \sBPic(T) \ar[d,"f^{\rm Ge,*}"] \\
\TateCoh(S) \ar[r,"\sD^{(1), \rm coh.}_S"'] & \sBPic(S).
\end{tikzcd}
\]
commutes.
\end{lem}

Consider $S$ a nice Tate scheme, and suppose moreover that
\[
p_{S}: S \ra \Spec(k)
\]
is eventually coconnective. Let
\[
p^{!}(k) \in \TateCoh(S)
\]
denote the pullback of $k \in \Tate_k$. We define the \emph{canonical graded $\Gm$-gerbe on $S$} by
\[
\omega^{(1)}_S := \sD^{(1), \rm coh.}(p^{!}(k)) \in \sBPic(S).
\]

\subsubsection{Gerbe on smooth schemes locally of finite type}

Suppose that $S_0$ is an eventually coconnective scheme locally of finite type and let $p_{S_0}: S_0 \ra \Spec(k)$ denote its structure morphism. Moreover, we let
\[
\omega^{(1)}_{S_0} := \sD^{(1), \rm coh.}_{S_0}(p^!_{S_0}(k)),
\]
where $p^!_{S_0}: \Tate(k) \ra \TateCoh(S_{0})$.

\begin{lem}
\label{lem:computation-of-Psi-of-shriek-pullback-on-good-schemes}
For $S_0$ a smooth and eventually coconnective one has an equivalence
\[
\Psi^{\rm Tate}_{S_0}\circ p^!_{S_0} \simeq \overline{T}^*(S_0)[\dim (S_0)],
\]
where $\overline{T}^*(S_0)$ denotes the usual cotangent complex. Moreover, if $S_0$ is locally of finite type, we have $\overline{T}^*(S_0)[\dim (S_0)] \in \Perf(S_0)$.
\end{lem}

\begin{proof}
We first notice that since $S_i$ is eventually coconnective we have that $p^!_{S_0}(k) \in \IndCoh(S_0) \subset \TateCoh(S_i)$, thus we have
\[
\Psi^{\rm Tate}_{S_0}\circ p^!_{S_0} \simeq \Psi_{S_0}\circ p^!_{S_0}.
\]
Because we supposed $S_0$ smooth, \cite{GR-II}*{Chapter 4, Corollary 4.3.7} gives the equivalence
\[
\Psi_{S_0}\circ p^!_{S_0} \simeq \Psi_{S_0} \circ \Xi_{S_0}(\overline{T}^*(S_0)[\dim (S_i)]).
\]
\end{proof}

\subsection{Determinant theory for a Tate scheme}

Recall that a triviliazation of $\omega^{(1)}_S$ is given by a non-trivial $\sPic(S)$-linear morphism
\[
\sigma_{S}:\omega^{(1)}_S \ra \sPic(S),
\]
which is automatically an isomorphism, since $\omega^{(1)}_S$ is a $\sPic(S)$-torsor.

We need a further technical assumption on the Tate schemes for which we will describe the determinant theory gerbe.

Given $S \in \SchTatenice$ we suppose that $S$ can be presented as $S \simeq \colim_I S_i$ where for each $i\in I$
\begin{equation}
    \label{eq:assumption-on-underlying-schemes}
    S_i \; \mbox{ is an eventually coconnective, smooth and locally of finite type scheme.}
\end{equation}

\begin{thm}
\label{thm:trivialization-is-determinantal-theory}
Let $S$ be a nice Tate scheme and suppose that $S$ admits a presentation $S \simeq \colim_I S_i$ that satisfies condition (\ref{eq:assumption-on-underlying-schemes}). Then the data of a trivialization
\[
\sigma_{S}:\sPic(S) \ra \omega^{(1)}_S
\]
is equivalent to the data of for every $i \in I$ a line bundle $\sL_i \in \sPic(S_i)$ and for every $j \ra i$ an isomorphism
\[
\tau_{i,j}: f^*_{j,i}(\sL_{i})\otimes (\sD^{(0)}(\overline{T}^*(S_j/S_i)))^{-1} \overset{\simeq}{\ra}\sL_j
\]
where $f_{j,i}:S_j \hra S_i$ is the canonical inclusion, plus coherent higher compatibility conditions.
\end{thm}

\begin{proof}
\textit{Construction of line bundles.} For $\sigma_S$ a trivialization as above, let $f_i:S_i \ra S$ denote the canonical inclusion, the pullback of $\sigma_S$ gives
\begin{equation}
\label{eq:sigma-S-i}
\sigma_{S_i} := (f_i)^{\rm Ge,*}(\sigma_S): (f_i)^{\rm Ge,*}(\sPic(S)) \simeq \sPic(S_i) \ra (f_i)^{\rm Ge,*}(\omega^{(1)}_S).    
\end{equation}

We notice that by Lemma \ref{lem:functoriality-of-coherent-determinant-1} the righthand side of (\ref{eq:sigma-S-i}) can be rewritten as
\[
(f_i)^{\rm Ge,*}(\omega^{(1)}_S) = (f_i)^{\rm Ge,*}(\sD^{(1),\rm coh.}_S(p^!_S(k))) \simeq \sD^{(1),\rm coh.}_{S_i}(f^!_{i}\circ p^!_{S}(k)) \simeq \omega^{(1)}_{S_i},
\]
where we notice that $S_i$ is a nice Tate scheme (see Remark \ref{rem:smooth-lft-schemes-are-nice-Tate-schemes}).

By Lemma \ref{lem:computation-of-Psi-of-shriek-pullback-on-good-schemes} one has that
\[
\omega^{(1)}_{S_i} \simeq \sD^{(1)}_{S_i}(\overline{T}^*(S_i)[\dim S_i]).
\]
Moreover, since $\overline{T}^*(S_i)$ is an object of $\Perf(S_i)$, Lemma \ref{lem:trivialization-of-determinant-1-on-perfect-complexes} gives an isomorphism
\[
\eta_{\overline{T}^*(S_i)}: \sD^{(1)}_{S_i}(\overline{T}^*(S_i)[\dim S_i]) \overset{\simeq}{\ra} \sPic(S_i).
\]

We let $\overline{\sigma_{S_i}}$ denote the composite
\begin{equation}
\label{eq:defn-of-sigma-S-i-bar}
    \sPic(S_i) \overset{\sigma_{S_i}}{\ra} \omega^{(1)}_{S_i} \simeq \sD^{(1)}_{S_i}(\overline{T}^*(S_i)[\dim S_i]) \overset{\eta_{\overline{T}^*(S_i)}}{\ra} \sPic(S_i).
\end{equation}
We define $\sL_i$ to be the object of $\sPic(S_i)$ characterized by
\[
 \overline{\sigma_{S_i}}(\sP)\simeq \sP\otimes_{S_i}\sL_i
\]
for all $\sP \in \sPic(S_i)$.

\textit{Construction of compatibility isomorphisms.} We will now check the existence of $\tau_{i,j}$. Let $f_{j,i}:S_j \hra S_i$ denote the closed embedding corresponding to a morphism $j \ra i$ in $I$. Consider $(f_{j,i})^{\rm Ge, *}$ applied to (\ref{eq:sigma-S-i}), this gives
\[
\sPic(S_j) \simeq (f_{j,i})^{\rm Ge,*}(\sPic(S_i)) \overset{(f_{j,i})^{\rm Ge, *}}{\ra} (f_{j,i})^{\rm Ge,*}\omega^{(1)}_{S_i}.
\]

We now notice that we have the chain of isomorphisms
\begin{align*}
    \omega^{(1)}_{S_j} & \simeq \sD^{(1)}_{S_j} \circ \Psi_{S_j} (f^!_{j,i} \circ p^!_{S_i}(k)) \\
    & \simeq \sD^{(1)}_{S_j} \circ \Psi_{S_j} (f^{\rm Ind, *}_{j,i} \circ p^!_{S_i}(k)\otimes \sD^{(0)}(\overline{T}^{*}(S_j/S_i))) \\
    & \simeq \sD^{(1)}_{S_j} \circ \Psi_{S_j} (f^{\rm Ind, *}_{j,i} \circ p^!_{S_i}(k)) \otimes \sD^{(0)}(\overline{T}^{*}(S_j/S_i)) \\
    & \simeq \sD^{(1)}_{S_j} \circ f^{*}_{j,i} \circ \Psi_{S_i}(p^!_{S_i}(k)) \otimes \sD^{(0)}(\overline{T}^{*}(S_j/S_i)) \\
    & \simeq (f_{j,i})^{\rm Ge, *}(\omega^{(1)}_{S_i})\otimes \sD^{(0)}(\overline{T}^{*}(S_j/S_i)),
\end{align*}
where the first isomorphism follows from Grothendieck's formula (see \cite{GR-II}*{Chapter 9}*{Corollary 7.2.4}), where we notice that $f_{j,i}:S_j \hra S_i$ is a regular embedding. The other isomorphism are obtained using the usual compatibility relations.

Thus, we notice that one obtains an isomorphism
\[
(f_{j,i})^{\rm Ge, *}(\overline{\sigma_{S_i}})\otimes\sD^{(0)}(\overline{T}^{*}(S_j/S_i)): \sPic(S_j) \ra \sPic(S_j) 
\]
which is described by the line bundle $\sL_i\otimes\sD^{(0)}(\overline{T}^{*}(S_j/S_i))^{-1}$. Since we want this isomorphism to be compared to $\overline{\sigma_{S_j}}$ we have that this data is given by a morphism
\[
\tau_{i,j}: \sL_i\otimes\sD^{(0)}(\overline{T}^{*}(S_j/S_i))^{-1} \ra \sL_j,
\]
as we claimed.

\end{proof}

\newpage

\appendix

\section{Recollections on Pro and Tate objects}
\label{sec:recollections-on-Pro-and-Tate}

In this appendix we collect some results about Pro-objects and Tate-objects that are needed in the body of the text.

\subsection{General results}

The material in this section is somewhat standard. The reader is refered to \cites{Hennion-Tate,determinant-map} for the definition of Tate objects and \cite{Raskin-homological} where the tensor product considered below is detoned by $\overset{!}{\otimes}$ in \emph{loc. cit.}.

\subsubsection{Tensor product of Pro-objects}

We let $\Vect_k$ denote the stable $\infty$-category of complex of $k$-vector spaces. By definition this category is equivalent to $\Ind(\Vect^{\rm fd}_k)$ and it is endowed with a symmetric monoidal structure which commutes with colimit on each variable.

Consider $\ProVect_k$ the category of Pro-objects in $\Vect_k$, this is endowed with a tensor product
\begin{equation}
\label{eq:tensor-product-ProVect}
\otimes: \ProVect \times \ProVect \ra \ProVect    
\end{equation}
that extends the tensor product of $\Vect$ and is characterized by the fact that (\ref{eq:tensor-product-ProVect}) commutes with limits on both variables.

\subsubsection{Pro construction and adjunctions}

\begin{prop}
\label{prop:left-adjoint-of-Pro-extended-functors}
Suppose that
\[
\begin{tikzcd}
\sC \ar[r,shift left=.5ex,"g"] & \sD \ar[l,shift left=.5ex,"f"]
\end{tikzcd}
\]
is an adjuction between the $\infty$-categories $\sC$ and $\sD$, i.e.\ $(g,f)$ is an adjoint pair. Then $\Pro(f)$ has a left adjoint which is given by $\Pro(g)$. Moreover, in this case $\Pro(g)$ is simply calculated by pre-composition with $f$.
\end{prop}

\begin{proof}
We notice that by \cite{HTT}*{Proposition 5.3.5.13} the functor $\Pro(f)$ has a left adjoint if and only if $f$ is left exact, which is automatic since $f$ is a right adjoint hence preserves all limits. Its right adjoint is given by $G$ which is pre-composition with $f$.
To check the second assertion we consider $X\in \Pro(\sC)$ and $Y \in \sD$ and notice that
\begin{align*}
    \Pro(g)(X)(Y) & = \RKE_g(\Phi_X)(Y) \\
    & \simeq \lim_{Y' \in \sC \; | \; g(Y') \ra Y} \Phi_X(Y') \\
    & \simeq \lim_{Y' \in \sC \; | \; Y' \ra f(Y)} \Phi_X(Y') \\
    & \simeq \Phi_X(f(Y)),
\end{align*}
where on the last line we used that $\sC_{/f(Y)}$ is a contractible category, so the limit is equivalent to evaluation at $f(Y)$.
\end{proof}

\subsubsection{Right Kan extension from the subcategory of colattices}

Let $\sC$ be a presentable stable $\infty$-category, the category of lattice factorizations is defined as the following pullback squares
\begin{equation}
    \label{eq:defn-Latt-of-sC}
    \begin{tikzcd}
    \Latt_{\sC} \ar[r] \ar[d] & \Pro(\sC)\times \Ind(\sC) \ar[d] \\
    \Fun([1],\Pro(\Ind(\sC))) \ar[r,"\Fib(-) \times \ev_1"'] & \Pro(\Ind(\sC))     \times \Pro(\Ind(\sC))
    \end{tikzcd}
\end{equation}
of categories, where the lower arrow sends a morphism $f:X \ra Y$ to the pair $(\Fib(f), Y)$.

For a fixed object $X$ in $\Tate(\sC)$ one defines $\CoLatt_{\sC}(X)$ as the pullback
\begin{equation}
    \label{eq:defn-collatice-of-X}
    \begin{tikzcd}
    \CoLatt_{\sC}(X) \ar[r] \ar[d] & \Latt_{\sC} \ar[d,"\ev_0"] \\
    \{X\} \ar[r] & \Pro(\Ind(\sC))
    \end{tikzcd}
\end{equation}
where the right-hand downward arrow sends $f:X \ra Y$ to $X$.

The following is an analogue of \cite{Hennion-Tate}*{Theorem 3.15}.

\begin{lem}
\label{lem:colattices-are-initial-in-Ind}
For any $X \in \Tate(\sC)$ the canonical map\footnote{Here and in the following we omit writing the inclusion functors for $\Ind(\sC) \ra \Tate(\sC)$ and similarly for $\Pro(\sC) \ra \Tate(\sC)$.}
\[
F: \CoLatt_{\sC}(X) \ra \Ind(\sC)_{X/-}
\]
is initial.
\end{lem}

\begin{proof}
By (the dual of) \cite{HTT}*{Theorem 4.1.3.1} we need to check that: for any $(X \ra Y) \in \Ind(\sC)_{X/-}$ the category
\begin{equation}
    \label{eq:comma-category-to-check-initial}
    \CoLatt_{\sC}(X) \underset{\Ind(\sC)_{X/-}}{\times} (\Ind(\sC)_{X/-})_{-/(X \ra Y)}    
\end{equation}
is weakly contractible.

We notice that $\Tate(\sC)$ for us means elementary Tate objects, thus by \cite{Hennion-Tate}*{Lemma 3.8} given a map $X \ra Y$ between elementary Tate objects and $(0 \ra Y \ra Y)$ a lattice for $Y$ one has a factorization
\[
\begin{tikzcd}
P \ar[d] \ar[r] & 0 \ar[d] \\
X \ar[d] \ar[r] & Y \ar[d] \\
I \ar[r] & Y \\
\end{tikzcd}
\]

Now suppose given two objects $(P \ra X \ra I)$ and $(P' \ra X \ra I')$ in the category (\ref{eq:comma-category-to-check-initial}), then one has an object
\[
P'\underset{X}{\times}P'' \overset{g}{\ra} X \ra \Cofib(g)
\]
which maps to both of the given objects. Thus, the category (\ref{eq:comma-category-to-check-initial}) is co-filtered , and by (the dual of) \cite{HTT}*{Lemma 5.3.1.20} it is weakly contractible. 
\end{proof}

\subsubsection{Dualizability}

In this section we prove a couple of results about the interaction between the Tate construction and dualization of DG-categories.

Let $\DGcd \subset \DGc$ denote the subcategory of small stable $\infty$-categories which are dualizable with respect to the Lurie tensor product. We will denote by
\[
(-)^{\vee}: \DGcd \ra (\DGcd)^{\rm op}
\]
the anti-equivalence that sends a category $\sC$ to its dual $\sC^{\vee}$.

\begin{prop}
\label{prop:duality-swaps-Pro-and-Ind}
For $\sC \in \DGcd$ one has canonical equivalences:
\begin{enumerate}[(i)]
    \item $\Pro(\sC) \simeq \Ind(\sC^{\vee})$;
    \item $\Ind(\sC) \simeq \Pro(\sC^{\vee})$.
\end{enumerate}
\end{prop}

\begin{proof}
By \cite{Barwick-Schommer-Pries}*{Theorem 7.2} there exists a unique, up to constractible space of choice, duality of in the $\infty$-category of $(\infty,1)$-categories, since
\[
(\sC \in \DGcd) \mapsto (\sC^{\rm op} \in \DGcd)
\]
also realizes this duality, we obtain $\sC^{\vee} \simeq \sC^{\rm op}$ and the result follows.
\end{proof}

\begin{cor}
\label{cor:duality-and-Tate-construction}
For $\sC \in \DGcd$ the duality functor sends $\Tate(\sC) \subset \Pro(\Ind(\sC))$ to the subcategory $\Tate(\sC^{\vee})^{\rm op} \subset \Pro(\Ind(\sC))^{\vee}$.
\end{cor}

\begin{proof}
This is purely formal. We notice that
\[
\Pro(\Ind(\sC))^{\vee} \simeq \Ind(\Pro(\sC^{\rm op})),
\]
and that the restriction of the duality to $\Tate(\sC)$ factors through $\Tate(\sC^{\rm op})^{\rm op} \subset \Ind(\Pro(\sC^{\rm op}))$. Finally, we notice that
\[
\Tate(\sC^{\rm op})^{\rm op} \simeq \Tate(\sC^{\rm})^{\rm op}
\]
by the proof of Proposition \ref{prop:duality-swaps-Pro-and-Ind}.
\end{proof}

\subsection{t-structure results}

This section we collect general results about t-structures on Pro-objects and Tate-objects.

\subsubsection{Compatibility with cofiltered limits}

Let $\sC$ be a stable $\infty$-category with t-structure, we will say that the t-structure on $\sC$ is \emph{compatible with cofiltered limits} if the subcategory
\[
\sC^{\leq 0} \subset \sC
\]
is closed under cofiltered colimits.

The following is formally dual to \cite{GR-I}*{Chapter 4, Lemma 1.2.4}.

\begin{lem}
\label{lem:t-structure-on-Pro-objects}
Let $\sC$ be a stable $\infty$-category endowed with a t-structure. Then $\Pro(\sC)$ has a unique t-structure, which is compatible with cofiltered limits, and for which the inclusion $\sC \hra \Pro(\sC)$ is t-exact. Moreover:
\begin{enumerate}[(a)]
    \item The subcategory $\Pro(\sC)^{\leq 0}$ (resp.\ $\Pro(\sC)^{\geq 0}$) is cocompactly generated under cofiltered limits by $\sC^{\leq 0}$ (resp.\ $\sC^{\geq 0}$);
    \item For $\sD$ another stable $\infty$-category endowed with a t-structure compatible with cofiltered limits, let $F:\Pro(\sC) \ra \sD$ be a cocontinuous functor. Then $F$ is t-exact (resp.\ left t-exact, right t-exact) if and only if $\left.F\right|_{\sC}: \sC \ra \sD$ is.
\end{enumerate}
\end{lem}

\begin{cor}
\label{cor:t-structure-on-Tate-objects}
Given $\sC$ a stable $\infty$-category endowed with a t-structure, then $\Tate(\sC)$ acquires a t-structure, defined as
\[
\Tate(\sC)^{\leq 0} := \Tate(\sC) \cap \Pro(\Ind(\sC))^{\leq 0},
\]
where $\Pro(\Ind(\sC))$ has a t-structure by applying Lemma \ref{lem:t-structure-on-Pro-objects} and \cite{AGH}*{Proposition 2.13}. Moreover, if the t-structure on $\Pro(\Ind(\sC))$ is generated under cofiltered limits, so is the t-structure on $\Tate(\sC)$ and $\Tate(\sC)^{\leq n}$ (resp.\ $\Tate(\sC)^{\geq n}$) is generated under cofiltered limits by $\Ind(\sC)^{\leq n}$ (resp.\ $\Ind(\sC)^{\geq n}$).
\end{cor}

\begin{proof}
This follows from the fact that $\tau^{\leq 0}$ preserves limits and that the inclusion $\sC \subset \Ind(\sC)$ is t-exact, by definition of the t-structure on $\Ind(\sC)$.
\end{proof}

\begin{lem}
\label{lem:coconnective-Tate-objects-as-colimits}
Suppose that the t-structure on $\Tate(\sC)$ is defined as in Corollary \ref{cor:t-structure-on-Tate-objects}, then for $\sF$ an object of $\Tate(\sC)^{\geq 0}$, we can find a filtered diagram $\{\sF_i\}_{I}$ where each $\sF_i \in \Pro(\sC)^{\geq 0}$ such that
\[
\colim_I\sF_i \simeq \sF.
\]
\end{lem}

\begin{proof}
Let $\sF \simeq \lim_J\sG_j$ be a presentation of $\sF$, where $\sG_{j} \in \Ind(\sC_0)^{\geq 0}$, one easily checks that
\[
\sF_j := \Fib(\sF \ra \sG_j) \in \Pro(\sC_0)^{\geq 0}.
\]
Thus, one has
\[
\colim_J \sF_j \simeq \sF.
\]
\end{proof}

\subsubsection{Completion of t-structure}

Let $\sC$ be a stable $\infty$-category with a presentable and bounded t-structure. Our main example will be $\sC = \Coh(X)$ for a scheme of almost of finite type $X$.

By \cite{AGH}*{Proposition 2.13} the category $\Ind(\sC)$ inherits a t-structure defined where $\Ind(\sC)^{\leq 0}$ is defined as the subcategory formed under filtered colimits by the essential image of the functor
\[
\sC^{\leq 0} \hra \sC \hra \Ind(\sC).
\]

Moreover, the t-structure on $\Ind(\sC)$ is right complete. Notice however that it seldom is left complete, for instance the t-structure of $\IndCoh(X)$ is \emph{not} left complete (see \cite{GR-I}*{Chapter 4, \S 1.2.6}).

\paragraph{Completion and boundedness}

For a $\infty$-category with a t-structure $\sC$ let
\[
\sC^{\wedge} := \lim_{n \geq 0}\left(\cdots \sC^{\leq (n+1)} \ra \sC^{\leq n} \cdots\right)
\]
denote its right completion, we also denote by
\[
\sC^{-} := \bigcup_{n \geq 0} \sC^{\leq n}
\]
the bounded above part of $\sC$. Now we make the following observation (see \cite{HA}*{Remark 1.2.1.18}) the canonical map
\[
\sC^{\wedge,-} \ra \sC^{\wedge}
\]
is an equivalence.

In particular, one has
\[
\IndCoh(S)^{\wedge,-} \simeq \IndCoh(S)
\]
for any scheme almost of finite type.

\paragraph{Completion of Pro-objects}

Consider the natural inclusion
\begin{equation}
\label{eq:inclusion-Ind-bounded-above-to-Pro-Ind-bounded-above}
\Ind(\sC)^- \hra \Pro(\Ind(\sC))^-    
\end{equation}
where the t-structure on $\Pro(\Ind(\sC))$ is defined similarly as the t-structure in Ind-objects. The map (\ref{eq:inclusion-Ind-bounded-above-to-Pro-Ind-bounded-above}) gives the following functor by passing to right completions
\[
\Ind(\sC)^{\wedge,-} \hra \Pro(\Ind(\sC))^{\wedge,-}    
\]
which is still fully faithful. In the case where $\sC$ has a bounded t-structure, one obtains an inclusion
\begin{equation}
    \label{eq:inclusion-Ind-to-Pro-Ind-bounded-above}
    \Ind(\sC) \simeq \Ind(\sC)^{\wedge} \simeq \Ind(\sC)^{\wedge,-} \hra \Pro(\Ind(\sC))^{\wedge,-}.
\end{equation}

\paragraph{Completed Tate objects}

For our purposes we need to consider a subcategory of completed Tate objects. For $\sC$ an $\infty$-category with a t-structure, recall that we can define a t-structure on $\Tate(\sC)$ by
\[
\Tate(\sC)^{\leq 0} := \Tate(\sC) \bigcap \Pro(\Ind(\sC))^{\leq 0}.
\]

\begin{lem}
For $\sC$ a category with a t-structure, the following are equivalent
\begin{enumerate}[(a)]
    \item $V$ is an object of $\Tate(\sC)^{\wedge}$;
    \item $V \in \Pro(\Ind(\sC))^{\wedge}$ and for each $n \geq 0$ the truncation $\tau^{\leq n}(V) \in \Tate(\sC)^{\leq n}$.
\end{enumerate}
\end{lem}

\paragraph{Right Kan extension from lattices}

We notice that the functor (\ref{eq:inclusion-Ind-to-Pro-Ind-bounded-above}) factors as follows
\begin{equation}
    \label{eq:inclusion-Ind-to-Tate-completed}
    \Ind(\sC) \hra \Tate(\sC)^{\wedge,-}.
\end{equation}

The following result is a direct consequence of Lemma \ref{lem:colattices-are-initial-in-Ind}.

\begin{cor}
\label{cor:completed-colattices-are-initial-in-Ind}
For any $X \in \Tate(\sC)^{\wedge,-}$ the canonical map
\[
F: \CoLatt^{\wedge}_{\sC}(X) \ra \Ind(\sC)_{X/-}
\]
is initial, here the category $\CoLatt^{\wedge}_{\sC}(X)$ is defined as in (\ref{eq:defn-Latt-of-sC}) and (\ref{eq:defn-collatice-of-X}) but considering the corresponding right completed categories.
\end{cor}

\begin{proof}
We notice that the functors $\tau^{\leq n}: \sC^{\leq (n+1)} \ra \sC^{\leq n}$ are right adjoints, hence they commute with limits. In particular, the results follows from considering Lemma \ref{lem:colattices-are-initial-in-Ind} for the category $\sC^{\leq n}$ for some $n$ and passing to the limit defining the right completion of $\sC$.
\end{proof}

\section{Set-theoretic considerations}

In this section we discuss how to take care of possible set-theoretical problems in certain constructions that we perform.

\subsection{Presentable and accessible $\infty$-categories}

We fix $\lambda$ a regular strongly inaccessible cardinal and consider a Grothendieck universe $\sU(\lambda)$. Given any cardinal $\kappa < \lambda$ we say an $\infty$-category $\sC$ is \emph{essentially $\kappa$-small} if it is equivalent to any $\infty$-category $\sC_0$, such that the underlying simplicial set of $\sC_0$ is $\kappa$-small\footnote{Recall that a simplicial set $S$ is said to be \emph{$\kappa$-small} if $S$ has $<\kappa$ non-degenerate simplices.}.

We recall the following concept
\begin{defn}
\label{defn:accessible-category}
For $\kappa < \lambda$ a regular cardinal. An $\infty$-category $\sC$ is \emph{$\kappa$-accessible} if there exists a $\kappa$-small $\infty$-category $\sC^0$ and an equivalence
\[
\Ind_{\kappa}(\sC^0) \ra \sC.
\]
We will say that an $\infty$-category $\sC$ is \emph{accessible} if it is $\kappa$-accessible for some cardinal $\kappa < \lambda$.
\end{defn}

The following is \cite{HTT}*{Corollary 5.4.3.6}

\begin{lem}
\label{lem:k-accessible-is-k-small-and-idempotent-complete}
An essentially $\kappa$-small $\infty$-category $\sC$ is $\kappa$-accessible if and only if it is idempotent complete, i.e.\ every diagram $F: \mbox{Idem} \ra \sC$ admits a colimit\footnote{See Definition 4.4.5.2 in \cite{HTT}, in short $\mbox{Idem}$ is the classifying space of the two element monoid $\{\id,e\}$, where $e^2 = e$.}.
\end{lem}

For our needs we will consider a variant of the notion of a presentable category notion from \cite{HTT}*{\S 5.5}, where we keep track of certain cardinalities.

\begin{defn}
\label{defn:l-presentable-category}
An $\infty$-category $\sC$ is said to be \emph{$\lambda$-presentable} if $\sC$ is $\kappa$-accessible for a regular cardinal $\kappa < \lambda$ and admits $\kappa$-small colimits. 
\end{defn}

The following is a variation of \cite{HTT}*{Proposition 5.4.2.9}.

\begin{prop}
Let $\sC$ be $\lambda$-presentable category, then $\sC$ is $\lambda'$-presentable for any $\lambda' >> \lambda$\footnote{We refer to \cite{HTT}*{Definition A.2.6.3} for the definition of the symbol $>>$ in this context.}
\end{prop}

\begin{proof}
We need to check that $\sC$ is $\kappa$-accessible for some cardinal $\kappa < \lambda'$, indeed that is clearly the case. The condition of admiting $\kappa$-small colimits is automatic.
\end{proof}

\subsection{$\kappa$-Pro-objects}

\begin{defn}
\label{defn:k-pro-objects}
For $\kappa < \lambda$ a regular cardinal. The $\infty$-category $\Pro_{\kappa}(\sC)$ is the full subcategory of $\Fun(\sC,\Spc)^{\rm op}$ generated by the functors which preserve $\kappa$-small cofiltered limits.
\end{defn}

The following is \cite{thesis}*{Lemma 2.2.1}, but we give a proof for the convenience of the reader.

\begin{lem}
\label{lem:Pro-kappa-of-essentially-lambda-small-is-essentially-small}
For $\sC$ an essentially $\lambda$-small $\infty$-category, the category $\Pro_{\kappa}(\sC)$ is essentially $\lambda$-small.
\end{lem}

\begin{proof}
By \cite{HTT}*{Example 5.4.1.8} the category of presheaves $\Fun(\sC,\Spc)^{\rm op}$ is locally $\lambda$-small, thus to check that $\Pro_{\kappa}\subset \Fun(\sC,\Spc)^{\rm op}$ is essentially $\lambda$-small, by \cite{HTT}*{Proposition 5.4.1.2 (1)} it is enough to prove that $\Pro_{\kappa}(\sC)$ has a $\lambda$-small set of equivalence classes of objects. Since the equivalence classes of objects are determined by the equivalence classes of objects in the homotopy category $\h\Pro_{\kappa}(\sC)$ it is enough to check that this class is $\lambda$-small.
Now consider $\sD \subset \h\Pro_{\kappa}(\sC)$ the subcategory of cocompact objects. By the dual of \cite{HTT}*{Proposition 5.3.4.17} an object $X \in \h\Pro_{\kappa}(\sC)$ is cocompact if and only if there exists a diagram
\[
p: K \ra \sC \ra \Fun(\sC,\Spc)^{\rm op}
\]
where $K$ is $\kappa$-small, and $X$ is a retract of $Y$ the limit of $K$ in $\Fun(\sC,\Spc)^{\rm op}$. Since this data is determined by $K$, $p$ and the idempotent map $Y \overset{e}{\ra} Y$ which are all bounded by $\kappa$, there are only $< \lambda$ eqiuvalence classes of objects in $\sD$. Finally, any object of $\h\Pro_{\kappa}(\sC)$ is a formal $\kappa$-cofiltered limit of objects in $\sD$ and those diagrams are also bounded by $\lambda$.
\end{proof}

In the case where the category we started with is idempotent complete the Pro-objects construction with the appropriate cardinals is also idempotent complete. The following is a consequence of the general behaviour of filtered colimits and small limits.

\begin{prop}
\label{prop:Pro-kappa-is-kappa-accessible}
Let $\sC$ be an idempotent complete essentially $\lambda$-small category, then the category $\Pro_{\kappa}(\sC)$ is idempotent complete. In particular, $\Pro_{\kappa}(\sC)$ is $\kappa$-accessible. 
\end{prop}

\begin{proof}
We notice that for any infinite cardinal $\lambda$, $\mbox{Idem}$ is $\lambda$-filtered. By \cite{HTT}*{Proposition 5.3.1.18} there exists a filtered partially ordered set $A$ and a cofinal map
\[
\N(A) \ra \mbox{Idem}
\]
of $\infty$-categories. Moreover, we claim that we can pick $A$ to be finite, indeed by \cite{HTT}*{Lemma 4.4.5.11} the inclusion $\mbox{Idem} \subseteq \mbox{Idem}^+$ is cofinal, and $\mbox{Idem}^+ \simeq \N(A_0)$, where $A_0$ is the category described in Definition 4.4.5.2. Now consider a diagram
\[
F: \mbox{Idem}^+ \ra \Pro_{\kappa}(\sC)
\]
by applying \cite{HTT}*{Proposition 5.3.5.15} one can find a diagram
\[
\tilde{F}: J \ra \Fun(\mbox{Idem}^+,\sC),
\]
where $J$ is a $\kappa$-cofiltered diagram. Since $\sC$ is idempotent complete for each $j \in J$ the colimit $\colim_{\mbox{Idem}}F(j)$ exists. Moreover, since $\kappa$-filtered colimits commute with $\kappa$-small limits, the natural map
\[
\colim_{\mbox{Idem}^+}F \ra \lim_{J}\tilde{F}
\]
is an isomorphism, since $\mbox{Idem}^+$ is $\kappa$-filtered. This proves that $\Pro_{\kappa}(\sC)$ is idempotent complete. The last assertion is a consequence of Lemma \ref{lem:k-accessible-is-k-small-and-idempotent-complete}.
\end{proof}

\begin{cor}
\label{cor:Pro-k-sC-is-l-presentable}
For $\sC$ an idempotent complete stable essentially $\lambda$-small category, the category $\Pro_{\kappa}(\sC)$ is $\lambda$-presentable.
\end{cor}

\begin{proof}
By Proposition \ref{prop:Pro-kappa-is-kappa-accessible} $\Pro_{\kappa}(\sC)$ is $\kappa$-accessible, we only need to check that $\Pro_{\kappa}(\sC)$ admits $\kappa$-small colimits. Since $\Pro_{\kappa}(\sC)$ is stable (see \cite{thesis}*{Lemma 2.2.7}) colimits and limits coincide, hence it is enough to check that $\Pro_{\kappa}(\sC)$ admits $\kappa$-small limits. Now this follows from the fact that $\sC$ admits all $\kappa$-small limits and the inclusion $\sC \ra \Pro_{\kappa}(\sC)$ preserves all $\kappa$-small limits by the dual of \cite{HTT}*{Proposition 5.3.5.14}.
\end{proof}

\section{Useful results from Category Theory}

\subsection{Limits of presentable $\infty$-categories}

Consider a functor
\[
\sC: I \ra \Catpr
\]
from a diagram ($\infty$-)category $I$ into the $\infty$-category of presentable $\infty$-categories with morphisms colimits preserving functors. Let 
\[
\sC^{\rm R}: I^{\rm op} \ra \Cat 
\]
denote the $I^{\rm op}$-indexed diagram of $\infty$-categories obtained by passing to right adjoints\footnote{Notice that though the categories are still presentable the right adjoint functors between them don't necessarily preserve colimits anymore.}. One has the following result that we use in numerous manipulations.

\begin{lem}[\cite{GR-I}*{Chapter 1, Proposition 2.5.7}]
\label{lem:limits-to-colimits-of-categories}
The canonical morphism
\[
\colim_I\sC_i \ra \lim_{I^{\rm op}}\sC^{\rm R}_{i}
\]
is an equivalence of $\infty$-categories.
\end{lem}

\subsection{Passing to adjoints}

In this section we collect some facts about passing to adjoints, which will be useful in the next sections.

\subsubsection{Adjoints and limits of $\infty$-categories}

Suppose that $G: \sC \ra \sD$ is a functor between $\infty$-categories, and let $I$ be a family of indices and suppose we are given an $I$-family of commutative diagrams
\[
\begin{tikzcd}
\sC \ar[r,"G"] \ar[d,"\imath^{\sC}_i"] & \sD \ar[d,"\imath^{\sD}_i"] \\
\sC_{i} \ar[r,"G_i"] & \sD_i.
\end{tikzcd}
\]

Assume that for each $i \in I$, the functor $G_i$ admits a left (resp.\ right) adjoint $F_i$, and furthermore assume that for any map $i \ra j$ in $I$ the following diagram
\begin{equation}
\label{eq:commuting-diagram-between-i-and-j}
    \begin{tikzcd}
\sC_i \ar[d,"\imath^{\sC}_{i,j}"] & \sD_i \ar[d,"\imath^{\sD}_{i,j}"] \ar[l,"F_{i}"] \\
\sC_{j} & \sD_j \ar[l,"F_{j}"] .
\end{tikzcd}
\end{equation}
commutes\footnote{Notice that, a priori, the diagram (\ref{eq:commuting-diagram-between-i-and-j}) only commutes up to a natural transformation.}.

Finally, we assume that the functors
\[
\sC \ra \lim_{I}\sC_i \;\;\; \mbox{and}\;\;\; \sD \ra \lim_I\sD_i
\]
are equivalences.

\begin{lem}[\cite{GR-I}*{Chapter 1, Lemma 2.6.4}]
\label{lem:adjoints-of-limits-of-categories}
In the situation describe above, the functor $G$ admits a left (resp.\ right) adjoint, and for every $i \in I$ the diagram
\[
\begin{tikzcd}
\sC \ar[d,"\imath^{\sC}_i"] & \sD \ar[d,"\imath^{\sD}_i"] \ar[l,"F"] \\
\sC_{i} & \sD_i \ar[l,"F_i"].
\end{tikzcd}
\]
commutes\footnote{Again this diagram, a priori, only commutes up to a natural transformation.}.
\end{lem}

\subsubsection{Adjoints and Kan extensions}

Let $F: \sC \ra \sD$ be a functor between $\infty$-categories and
\[
\Phi: \sC \ra \bS
\]
a functor into an $(\infty,2)$-category $\bS$. Let $\Psi := \LKE_{F}(\Phi)$.

Assume that for every arrow $X \ra Y$ in $\sC$ the $1$-morphism $\Phi(X) \ra \Phi(Y)$ admits a right adjoint. Let $\Phi^!: \sC^{\rm op} \ra \bS$ be the functor, obtained from $\Phi$ by passing to right adjoints.

\begin{lem}[\cite{GR-I}*{Chapter 8, Proposition 2.2.7}]
\label{lem:adjoints-and-Kan-extensions}
In the situation described above, the functor $\Psi^! := \RKE_{F^{\rm op}}(\Phi^!)$ is obtained from $\Psi$ by passing to right adjoints.
\end{lem}

\subsection{Beck--Chevalley condition}

The following piece of data is not given a name in \cite{GR-I}*{Chapter 7, \S 1.1.1} we will call it a \emph{pre-correspondence} $\infty$-category.

\begin{defn}
\label{defn:category-with-classes-of-morphisms}
Let $\sC$ be an $\infty$-category a \emph{pre-correspondence} on $\sC$ is the data of three classes of morphisms $horiz, vert$ and $adm \subset vert \cap horiz$ satisfying the following conditions:
\begin{enumerate}[(1)]
    \item all three classes contain the identity morphisms;
    \item the classes are stable under isomorphism;
    \item the classes are stable under composition;
    \item given a Cartesian square
    \[
    \begin{tikzcd}
    X' \ar[r,"f'"] \ar[d,"g'_Y"] & Y' \ar[d,"g_Y"] \\
    X \ar[r,"f"] & Y
    \end{tikzcd}
    \]
    if $f \in horiz$ and $g_Y \in vert$, then $f' \in horiz$ and $g'_Y \in vert$. Moreover, if $f \in adm$ (resp.\ $g_Y \in adm$) then $f' \in adm$ (resp.\ $g'_Y \in adm$);
    \item the class of morphisms $adm$ satisfy the 2 out of 3 property, i.e.\ given a composite $h = g\circ f$ if two among the three morphisms $f,g$ and $h$ belong to $adm$ then so does the last of them.
\end{enumerate}
\end{defn}

We will denote by $\sC_{vert} \subset \sC$ the full subcategory where morphisms are required to belong to $vert$, similarly one defines the subcategories $\sC_{horiz}$ and $\sC_{adm}$.

By an $(\infty,2)$-category we will mean a complete Segal object in the $\infty$-category of $(\infty,1)$-categories (see \cite{GR-I}*{Chapter 10, \S 2.1} for details).

We can now recall the meaning of left and right Beck--Chevalley conditions.

\begin{defn}
\label{defn:left-Beck-Chevalley}
Let $\sC$ be a pre-correspondence category, such that $horiz = adm$, and $\bS$ an $(\infty,2)$-category, for a functor
\[
\Phi: \sC_{vert} \ra \bS
\]
we will say that $\Phi$ satisfies the \emph{left Beck--Chevalley condition with respect to $horiz$}, if for every $1$-morphism $f:X \ra Y$ in $horiz$ the corresponding $1$-morphism
\[
\Phi(f): \Phi(X) \ra \Phi(Y)
\]
has a right adjoint, denoted by $\Phi(f)^!$, such that for every Cartesian diagram
\[
\begin{tikzcd}
X' \ar[r,"f'"] \ar[d,"g'_Y"] & Y' \ar[d,"g_Y"] \\
X \ar[r,"f"] & Y
\end{tikzcd}
\]
with $f,f' \in horiz$ and $g_Y,g_Y' \in vert$, the $2$-morphism
\[
\Phi(g'_Y) \circ \Phi^!(f') \ra \Phi^!(f)\circ\Phi(g_Y)
\]
arising by adjunction from the isomorphism
\[
\Phi(f) \circ \Phi(g'_Y) \simeq \Phi(g_Y)\circ \Phi(f')
\]
is an isomorphism.
\end{defn}

\begin{defn}
Let $\sC$ be a pre-correspondence category, such that $vert = adm$, and $\bS$ an $(\infty,2)$-category, for a functor
\[
\Phi^!: (\sC_{horiz})^{\rm op} \ra \bS
\]
we will say that $\Phi^!$ satisfies the \emph{right Beck--Chevalley condition with respect to $vert$}, if for every $1$-morphism $g:Y' \ra Y$ in $vert$ the corresponding $1$-morphism
\[
\Phi^!(g): \Phi^!(Y) \ra \Phi^!(Y')
\]
has a left adjoint, denoted by $\Phi(g)$, such that for every Cartesian diagram
\[
\begin{tikzcd}
X' \ar[r,"f'_X"] \ar[d,"g'"] & Y' \ar[d,"g"] \\
X \ar[r,"f_X"] & Y
\end{tikzcd}
\]
with $f_X,f'_X \in horiz$ and $g,g' \in vert$, the $2$-morphism
\[
\Phi(g') \circ \Phi^!(f'_X) \ra \Phi^!(f_X) \circ \Phi(g)
\]
arising by adjunction from the isomorphism
\[
\Phi^!(f'_X) \circ \Phi^!(g) \simeq \Phi^!(g') \circ \Phi^!(f_X)
\]
is an isomorphism.
\end{defn}

\paragraph{}

The following result is one of the main tools used by Gaitsgory--Rozenblyum in their constructions and it is equally essential for us. We formulate it here for the convenience of the reader.

\begin{thm}[\cite{GR-I}*{Chapter 7, Theorem 3.2.2}]
\label{thm:left-and-right-Beck-Chevalley-extension-result}
\begin{enumerate}[(a)]
    \item Let $\sC$ be a pre-correspondence category such that $horiz \subset vert$ and $adm = horiz$. Then restriction along
    \[
    \sC_{vert} \ra \Corr(\sC)^{horiz}_{vert;horiz}
    \]
    defines an isomorphism between the space of functors
    \[
    \Phi^{horiz}_{vert;horiz}: \Corr(\sC)^{horiz}_{vert;horiz} \ra \bS 
    \]
    and the subspace of functors
    \[
    \Phi: \sC_{vert} \ra \bS
    \]
    that satisfy the left Beck--Chevalley condition with respect to $horiz$.
    
    Moreover, for $\Phi^{horiz}_{vert;horiz}$ as above, the resulting functor $\Phi^! := \left.\Phi^{horiz}_{vert;horiz}\right|_{(\sC_{horiz})^{\rm op}}$ is obtained from $\left.\Phi\right|_{\sC_{horiz}}$ by passing to right adjoints.
    \item Let $\sC$ be a pre-correspondence category such that $vert \subset horiz$ and $adm = vert$. Then restriction along
    \[
    (\sC_{horiz})^{\rm op} \ra \Corr(\sC)^{vert}_{vert;horiz}
    \]
    defines an isomorphism between the space of functors
    \[
    \Phi^{vert}_{vert;horiz}: \Corr(\sC)^{vert}_{vert;horiz} \ra \bS 
    \]
    and the subspace of functors
    \[
    \Phi^!: (\sC_{horiz})^{\rm op} \ra \bS
    \]
    that satisfy the right Beck--Chevalley condition with respect to $vert$.
    
    Moreover, for $\Phi^{vert}_{vert;horiz}$ as above, the resulting functor $\Phi := \left.\Phi^{vert}_{vert;horiz}\right|_{\sC_{vert}}$ is obtained from $\left.\Phi^!\right|_{(\sC_{vert})^{\rm op}}$ by passing to left adjoints.
\end{enumerate}
\end{thm}

\subsection{Functors constructed by factorization}

In this subsection we formulate a variant on \cite{GR-I}*{Chapter 7, \S 5} that we need to construct the formalism of Tate-coherent from the category of correspondences (see \S \ref{subsubsec:TateCoh-on-correspondences-of-schemes}).

\subsubsection{Set-up for the source}
\label{subsubsec:conditions-on-category}

First we set up the conditions on the source category. Let $\sC$ be a pre-correspondence category (see Definition \ref{defn:category-with-classes-of-morphisms}). Moreover, fix a class of morphisms $co-adm \subset vert$, such that $(co-adm, horiz, isom)$ satisfies the condition of Definition \ref{defn:category-with-classes-of-morphisms}.

Also assume that $co-adm$ and $vert$ satisfy the 2 out of 3 property.

\paragraph{}

One imposes a further condition on the pair of ($co-adm$, $adm$). We require that for every Cartesian diagram
\[
\begin{tikzcd}
X'\underset{X}{\times}X' \ar[r,"\delta_1"] \ar[d,"\delta_2"] & X' \ar[d,"\delta"] \\
X' \ar[r,"\delta"] & X
\end{tikzcd}
\]
where $\delta \in adm \cap co-adm$, the maps $\delta_1$ and $\delta_2$ both are isomorphisms, i.e.\ any morphism in $adm \cap co-adm$ is a monomorphism in the category $\sC$.

\paragraph{}

To formulate the last condition for every $f:X \ra Y$ in $vert$, we consider the category $\mbox{Fact}(f)$ whose objects are
\[
X \overset{g}{\ra} U \overset{h}{\ra} Y
\]
such that $g \in adm$ and $h \in co-adm$, and morphisms are pairs of factorizations with a map between their intermediate object.

We require that for any $f \in vert$ the category $\mbox{Fact}(f)$ is contractible.

\subsubsection{Set-up for the functor}
\label{subsubsec:conditions-on-functor}

Let $\bS$ be an $(\infty,2)$-category. We consider the functors
\[
\Phi^{adm}_{adm;horiz}: \Corr(\sC)^{adm}_{adm;horiz} \ra \bS,
\]
and
\[
\Phi^{isom}_{co-adm;horiz}: \Corr(\sC)^{isom}_{co-adm;horiz} \ra \bS
\]
together with an identification
\[
\left.\Phi^{adm}_{adm;horiz}\right|_{(\sC_{horiz})^{\rm op}} \simeq \left.\Phi^{isom}_{co-adm;horiz}\right|_{(\sC_{horiz})^{\rm op}}.
\]

Consider the functor
\[
\Phi^! := \left.\Phi^{adm}_{adm;horiz}\right|_{(\sC_{horiz})^{\rm op}}.
\]

By \cite{GR-I}*{Chapter 7, Theorem 3.2.2 (b)} the functor $\Phi^{adm}_{adm;horiz}$ is uniquely reconstructed from the data of $\Phi^!$, and it exists if and only if $\Phi^!$ satisfies the right Beck--Chevalley condition with respect to $adm \subset horiz$. Thus, the above data is uniquely recovered from that of $\Phi^{isom}_{co-adm;horiz}$.

Let
\[
\Phi_{\star, co-adm} := \left.\Phi^{isom}_{co-adm;vert}\right|_{\sC_{adm}}
\]
and
\[
\Phi_{adm} := \left.\Phi^{adm}_{adm;horiz}\right|_{\sC_{co-adm}}.
\]

\paragraph{}
\label{par:factorization-compatibility}

We now impose the following technical condition on the functor $\Phi^{isom}_{co-adm;vert}$.

Let
\[
\begin{tikzcd}
X' \ar[r,"f'_X"] \ar[d,"g'"'] & Y' \ar[d,"g"] \\
X \ar[r,"f_X"'] & Y
\end{tikzcd}
\]
be a Cartesian diagram, where $g,g' \in adm$ and $f_X,f'_X \in co-adm$ the isomorphism
\[
\Phi_{adm}(g') \circ \Phi^!(f'_X) \overset{\simeq}{\ra} \Phi^!(f_X)\circ \Phi_{adm}(g)
\]
by adjunction gives
\begin{equation}
    \label{eq:factorization-compatibility-2-morphism}
    \Phi_{\star, co-adm}(f_X)\circ \Phi_{adm}(g') \ra \Phi_{adm}(g) \circ \Phi_{\star, co-adm}(f'_X).
\end{equation}

We require that (\ref{eq:factorization-compatibility-2-morphism}) is an isomorphism for every Cartesian diagram as above.

\paragraph{}

\begin{thm}
\label{thm:extension-of-correspondence-via-factorization}
Restriction along
\[
\Corr(\sC)^{isom}_{co-adm;horiz} \ra \Corr(\sC)^{adm}_{vert;horiz}
\]
gives an isomorphism between the space of functors
\[
\Phi^{adm}_{vert;horiz}: \Corr(\sC)^{adm}_{vert;horiz} \ra \bS
\]
and that of functors
\[
\Phi^{isom}_{co-adm;horiz}: \Corr(\sC)^{isom}_{co-adm;horiz} \ra \bS
\]
for which
\[
\Phi^! := \left.\Phi^{isom}_{co-adm;horiz}\right|_{(\sC_{horiz})^{\rm op}}
\]
satisfies the right Beck--Chevalley condition with respect to $adm \subset horiz$, and such that the condition from paragraph \ref{par:factorization-compatibility} holds.
\end{thm}

\begin{proof}
We claim that all the steps taken in the proof in sections 5.3 to 5.7 from \cite{GR-I}*{Chapter 7} carry over when one interchanges $adm$ for $co-adm$ and $horiz$ for $vert$ and use the appropriate results for functors satisfying the right Beck--Chevalley condition instead of the left Beck--Chevalley condition as in \emph{loc.\ cit}.
\end{proof}

\subsection{Commuting left and right Kan extensions}

In this section we collect some results that allow us to address the following abstract situation. Consider
\begin{equation}
\label{eq:commutative-square-of-categories-set-up}
\begin{tikzcd}
\sC \ar[r,"F'"] \ar[d,"G"'] & \sC' \ar[d,"G'"] \\
\sD \ar[r,"F"'] & \sD'
\end{tikzcd}    
\end{equation}
a commutative diagram of $\infty$-categories. In this section we introduce the notation $F_!$ to denote the left Kan extension with respect to $F$, $F_*$ the right Kan extension with respect to $F$ and $F^*$ the restriction functor, and similarly for $G,F'$ and $G'$.

\paragraph{right Kan extension base change}

We notice that one has a natural transformation
\begin{equation}
    \label{eq:base-change-RKE}
    F^*\circ G'_* \ra G_* \circ (F')^*.
\end{equation}
obtained by the $(G^*,G_*)$-adjunction from
\[
G^*\circ F^* \circ G'_* \simeq (F')^*.
\]

The following follows from considering an explicitly computation of the right Kan extensions in (\ref{eq:base-change-RKE}).

\begin{lem}
\label{lem:initial-slice-categories-give-isomorphism-of-base-change}
If the square (\ref{eq:commutative-square-of-categories-set-up}) is Cartesian and for every $d \in D$ the morphism of slice categories
\[
G_{d/} \ra G'_{F(d)/}
\]
is initial, then (\ref{eq:base-change-RKE}) is an isomorphism.
\end{lem}

\paragraph{Comparison map}

Suppose that (\ref{eq:base-change-RKE}) is an isomorphism, then one has a natural transformation
\begin{equation}
\label{eq:comparison-using-LKE-base-change}
    F_! \circ G_* \ra F_! \circ G_* \circ (F')^* \circ F'_! \overset{\simeq}{\la} F_! \circ F^* \circ G'_{*} \circ F'_! \ra G'_* \circ F'_!.
\end{equation}

\section{Proofs of \S \ref{subsec:Pro-ind-coherent-on-Tate-schemes}}

\subsection{Ind-proper base change for $\ProIndCoh_{\SchTatelaft}$}
\label{subsec:ind-proper-base-change-for-IndCoh}

In this section we prove that the functor
\[
\ProIndCoh_{\SchTatelaft}: \SchTatelaft \ra \DGct
\]
satisfies the left Beck--Chevalley condition with respect to ind-proper maps.

\begin{prop}
\label{prop:base-change-ind-proper-ProIndCoh}
Consider the pullback diagram
\[
\begin{tikzcd}
S' \ar[r,"g_S"] \ar[d,"f'"] & S \ar[d,"f"] \\
T' \ar[r,"g_T"] & T
\end{tikzcd}
\]
in the category of Tate schemes locally almost of finite type, where $g_T$ (and hence $g_S$) is ind-proper. The functors $(g_T)^{\rm ProInd}_*$ and $(g_S)^{\rm ProInd}_{*}$ admit right adjoints given by $g^!_Y$ and $g^!_T$. Moreover, the isomorphism
\[
(g_T)^{\rm ProInd}_* \circ (f')^{\rm ProInd}_* \overset{\simeq}{\ra} f^{\rm ProInd}_* \circ (g_S)^{\rm ProInd}_*
\]
by adjunction gives rise to a map
\begin{equation}
    \label{eq:base-change-ind-proper-ProIndCoh}
    (f')^{\rm ProInd}_* \circ g^!_S \ra g^!_T \circ f^{\rm ProInd}_*
\end{equation}
which is an isomorphism.
\end{prop}

\begin{proof}
This proof follows closely the strategy of \cite{GR-II}*{Chapter 3, Proposition 2.2.2}.

We notice that if $\widetilde{g^!_S}$ a right adjoint to $(g_S)^{\rm ProInd}_*$ exists, then by Lemma \ref{lem:adjoints-and-Kan-extensions} one has a natural equivalence
\[
\widetilde{g^!_S} \simeq g^!_S,
\]
and similarly for $(g_T)^{\rm ProInd}_{*}$. So we only need to check that $g^!_S$ and $g^!_T$ actually satisfy the base-change statement.

\textit{Step 0:}
In the case that $f$ is ind-proper the result follows from Proposition \ref{prop:base-change-ind-proper-ProIndCoh^!} below.

\textit{Step 1:}
To check that (\ref{eq:base-change-ind-proper-ProIndCoh}) is an isomorphism it is enough to check that for any $h: T_0 \ra T'$ with $T_0 \in \Schaft$ the map
\[
h^! \circ (f')^{\rm ProInd}_* \circ g^!_S \ra h^! \circ g^!_T \circ f^{\rm ProInd}_*
\]
is an isomorphism. Thus, we can suppose that $T' = T'_0 \in \Schaft$. 

\textit{Step 2:}
Let $S \simeq \colim_I S_i$ be a presentation of $S$, with $S_i \in \Schaft$, and let $\imath_i:S_i \hra S$ denote the canonical inclusions. Since by Proposition \ref{prop:limit-and-colimit-presentation-of-ProIndCoh-on-Tate-schemes}, we have an isomorphism
\[
\ProIndCoh(S) \simeq \colim_{I}\ProIndCoh(S_i).
\]
The category $\ProIndCoh(S)$ is generated by objects $(\imath_i)^{\rm ProInd}_*(\sF)$ for $\sF \in \ProIndCoh(S_i)$. Hence it is enough to check that (\ref{eq:base-change-ind-proper-ProIndCoh}) is an equivalence when it is precomposed by $(\imath_i)^{\rm ProInd}_*$. We are then reduced to the check the base change for the diagram
\[
\begin{tikzcd}
S'_i \ar[r,"g_{S,i}"] \ar[d] & S_i \ar[d,"\imath_i"] \\
S' \ar[r] \ar[d] & S \ar[d] \\
T'_0 \ar[r] & T
\end{tikzcd}
\]
where $S'_i = S_i\times_T T'_0$. Notice that since the map $\imath_i$ is ind-proper, since it is proper (see Remark \ref{rem:proper-is-ind-proper}). By Step 0 the top Cartesian diagram satisfies the base change isomorphism. Thus we are reduced to considering the case where $S = S_0 \in \Schaft$.

\textit{Step 3:}
Finally, in the case that $S_0$ is scheme almost of finite type, the map $f: S_0 \ra T$ factors as follows
\[
\begin{tikzcd}
S'_i \ar[r,"g_{S,i}"] \ar[d] & S_0 \ar[d] \\
T'_i \ar[r] \ar[d] & T_i \ar[d] \\
T'_0 \ar[r] & T
\end{tikzcd}
\]
for some $i \in I$ and $T \simeq \colim_I T_i$ a presentation of $T$, where all squares are Cartesian. We then notice that the lower square satisfies the base change isomorphism by Step 0 again, and by Proposition \ref{prop:base-change-proper} the upper square satisfies the base change isomorphism because $S_0 \ra T_i$ is proper.
\end{proof}

\paragraph{}

Let
\[
\TateCoh_{\SchTatelaft}: \SchTatelaft \ra \DGct
\]
be the functor from \S \ref{par:extension-of-TateCoh-functors}.

\begin{cor}
\label{cor:base-change-ind-proper-TateCoh}
For the same set up as Proposition \ref{prop:base-change-ind-proper-ProIndCoh} one has a canonical map
\begin{equation}
    \label{eq:base-change-ind-proper-TateCoh}
    (f')^{\rm Tate}_*\circ g^!_S \ra g^!_T \circ f^{\rm Tate}_*
\end{equation}
obtained from the $(g^{\rm Tate}_*,g^!)$-adjunction from
\[
(g_T)^{\rm Tate}_* \circ (f')^{\rm Tate}_* \simeq f^{\rm Tate}_* \circ (g_{S})^{\rm Tate}_{*}.
\]
\end{cor}

\begin{proof}
We notice that the same proof as that of Proposition \ref{prop:base-change-ind-proper-ProIndCoh} works, except one needs Lemma \ref{lem:right-Beck-Chevalley-for-proper-morphisms-TateCoh-shriek} instead of Proposition \ref{prop:base-change-proper}.
\end{proof}

\subsection{Ind-proper base change for $\ProIndCoh^!_{\SchaffTateaft}$}
\label{subsec:ind-proper-base-change-for-ProIndCoh^!}

In this section we proved that the functor
\[
\ProIndCoh^!_{\SchTatelaft}: (\SchTatelaft)^{\rm op} \ra \DGct
\]
satisfies the right Beck--Chevalley condition with respect to ind-proper maps.

\begin{prop}
\label{prop:base-change-ind-proper-ProIndCoh^!}
Consider the pullback diagram
\[
\begin{tikzcd}
S' \ar[r,"g_S"] \ar[d,"f'"] & S \ar[d,"f"] \\
T' \ar[r,"g_T"] & T
\end{tikzcd}
\]
in the category of Tate schemes almost of finite type, where $f$ (and hence $f'$) is ind-proper. Then the functors $f^!$ and $(f')^!$ admit left adjoints given by $f^{\rm ProInd}_*$ and $(f')^{\rm ProInd}_*$. Moreover, the isomorphism
\[
g^!_S\circ f^! \overset{\simeq}{\ra} (f')^!\circ g^!_T
\]
by adjunction gives rise to a map
\begin{equation}
    \label{eq:base-change-ind-proper-ProIndCoh^!}
    (f')^{\rm ProInd}_{*} \circ g^!_S \ra g^!_T \circ f^{\rm ProInd}_*
\end{equation}
which is an isomorphism.
\end{prop}

\begin{proof}
This proof follows closely the strategy of \cite{GR-II}*{Chapter 3, Proposition 2.1.2}.

As in the proof of Proposition \ref{prop:base-change-ind-proper-ProIndCoh} notice that if $\widetilde{f^{\rm ProInd}_*}$ a left adjoint to $f^!$ exists, then by Lemma \ref{lem:adjoints-and-Kan-extensions} one has a natural equivalence
\[
\widetilde{f^{\rm ProInd}_*} \simeq f^{\rm ProInd}_*,
\]
and similarly for $(f')^{!}$. So we only need to check that $f^{\rm ProInd}_*$ and $(f')^{\rm ProInd}_*$ actually satisfy the base-change statement.

\textit{Step 1:}
To check that (\ref{eq:base-change-ind-proper-ProIndCoh^!}) is an isomorphism it is enough to check that for any $h: T_0 \ra T'$ with $T_0 \in \Schaft$ the map
\[
h^!\circ(f')^{\rm ProInd}_{*} \circ g^!_S \ra h^!\circ g^!_T \circ f^{\rm ProInd}_*
\]
is an isomorphism. Thus, we can suppose that $T' = T'_0 \in \Schaft$. 

\textit{Step 2:}
We consider the functor $f^!: \ProIndCoh(T) \ra \ProIndCoh(S)$ and let $I = (\Schaffaft)_{/T}$ then for each $T_0 \in I$ we have
\[
f^!_i: \ProIndCoh(T_0) \ra \ProIndCoh(T_0\times_{T}S)
\]
and
\[
\ProIndCoh(T) \simeq \lim_{T_0 \in I}\ProIndCoh(T_0) \;\; \mbox{and} \;\; \lim_{T_0 \in I}\ProIndCoh(T_0\times_T S).
\]
Thus, by Lemma \ref{lem:adjoints-of-limits-of-categories}, it is enough to find left adjoints to $f^!_i$, and we can reduce the proof of the assertion to the case where $T = T_0 \in \Schaft$.

\textit{Step 3:}
Consider $S \simeq \colim_I S_i$ a presentation of $S$, where $S_i \in \Schaft$, and denote by
\[
\imath_i:S_i \hra S, \;\;\; \imath'_i: S_i\times_{T_0}T'_0 \hra S' \;\; \mbox{and} \;\; g_{S,i}:S'_i \ra S_i
\]
the canonical inclusions maps and the pullback of $g_S$. One obtains the diagram
\[
\begin{tikzcd}
S'_i \ar[r,"g_{S,i}"] \ar[d] & S_i \ar[d] \\
S' \ar[r] \ar[d] & S \ar[d] \\
T'_0 \ar[r] & T_0
\end{tikzcd}
\]
the composite of the right vertical arrows, by definition of ind-proper (see Remark \ref{rem:ind-proper-is-proper-for-the-presentation}), is proper. Thus the outer Cartesian square satisfy the base change isomorphism by Proposition \ref{prop:base-change-proper}.

\textit{Step 4:}
Now we notice that Proposition \ref{prop:limit-and-colimit-presentation-of-ProIndCoh-on-Tate-schemes} implies that the functors $(\imath^!_i,\imath^{\rm ProInd}_{i,*})$ form an adjoint pair and the canonical map
\[
\colim_I\imath^!_i \circ \imath^{\rm ProInd}_{i,*} \ra \mbox{Id}_{\ProIndCoh(S)}
\]
is an equivalence, and similarly for $\ProIndCoh(S')$. Thus the map (\ref{eq:base-change-ind-proper-ProIndCoh^!}) becomes
\[
(f')^{\rm ProInd}_{*} \circ \colim_I\imath'^!_i \circ \imath'^{\rm ProInd}_{i,*} \circ g^!_S \circ  \ra g^!_T \circ f^{\rm ProInd}_* \circ \colim_I\imath^!_i \circ \imath^{\rm ProInd}_{i,*}.
\]
So it is enough to check the isomorphism for a single $i$, but this claim follows from Step 3. And this finishes the proof.
\end{proof}

\paragraph{}

Let
\[
\TateCoh^!_{\SchTatelaft}: \SchTatelaft \ra \DGct
\]
be the functor from \S \ref{par:extension-of-TateCoh-functors}.

\begin{cor}
\label{cor:base-change-ind-proper-TateCoh^!}
For the same set up as Proposition \ref{prop:base-change-ind-proper-ProIndCoh^!} one has a canonical map
\begin{equation}
    \label{eq:base-change-ind-proper-TateCoh^!}
    (f')^{\rm Tate}_*\circ g^!_S \ra g^!_T \circ f^{\rm Tate}_*
\end{equation}
obtained from the $(f^{\rm Tate}_*,f^!)$-adjunction from
\[
g^!_S \circ f^! \simeq (f')^!\circ g^!_{T}
\]
\end{cor}

\begin{proof}
We notice that the same proof as that of Proposition \ref{prop:base-change-ind-proper-ProIndCoh^!} works, except one needs Lemma \ref{lem:right-Beck-Chevalley-for-proper-morphisms-TateCoh-shriek} instead of Proposition \ref{prop:base-change-proper}.
\end{proof}

\subsection{Pro-ind-coherent sheaves on ind-proper maps}
\label{subsec:ProIndCoh-on-ind-proper-maps}

In this section we check the two last conditions necessary to apply \cite{GR-I}*{Chapter 8, Theorem 1.1.9}, namely conditions (2) and (4) of \emph{loc.\ cit.\ }.

\begin{prop}
\label{prop:ProIndCoh^!-from-proper-agrees-with-restriction}
The canonical map
\begin{equation}
    \left.\ProIndCoh^!_{\SchTatelaft}\right|_{(\SchTatelaft)_{\rm ind-proper}} \ra \ProIndCoh^!_{(\SchTatelaft)_{\rm ind-proper}}
\end{equation}
is an equivalence.
\end{prop}

\begin{proof}
Let $S \in \SchTatelaft$, by the definition of right Kan extension, we need to check that
\[
\ProIndCoh(S) \ra \lim_{T \in (\Schaft)^{\rm op}_{\rm ind-proper\;in\;S}}\ProIndCoh(T)
\]
is an equivalence. Since for $T \in \Schaft$ a map $f:T \ra S$ is ind-proper if and only if $f$ is proper one has
\[
\lim_{T \in (\Schaft)^{\rm op}_{\rm ind-proper\;in\;S}}\ProIndCoh(T) \simeq \lim_{T \in (\Schaft)^{\rm op}_{\rm proper\;in\;S}}\ProIndCoh(T).
\]
Now notice that the subcategory
\[
(\Schaft)_{\rm closed\;in\;S} \subset (\Schaft)_{\rm proper\;in\;S}
\]
is initial. Thus, we need to check that the map
\[
\ProIndCoh(S) \ra \lim_{T \in (\Schaft)^{\rm op}_{\rm closed\;in\;S}}\ProIndCoh(T)
\]
is an equivalence, which follows from the definition of $\ProIndCoh^!$ on $S$ and \cite{GR-II}*{Chapter 2, Corollary 1.7.5 (b)}.
\end{proof}

\begin{prop}
\label{prop:ProIndCoh-from-proper-agrees-with-restriction}
The canonical map
\[
\ProIndCoh_{(\SchTatelaft)_{\rm ind-proper}} \ra \left.\ProIndCoh_{\SchTatelaft}\right|_{(\SchTatelaft)_{\rm ind-proper}}
\]
is an equivalence.
\end{prop}

\begin{proof}
Let $S \in \SchTateaft$, by the definition of left Kan extension, we need to check that
\[
\colim_{T \in (\Schaft)_{\rm ind-proper\;in\;S}}\ProIndCoh(T) \ra \ProIndCoh(S)    
\]
is an equivalence. First we notice that by Lemma \ref{lem:adjoints-and-Kan-extensions} the functor $\ProIndCoh^!_{(\SchTatelaft)_{\rm ind-proper}}$ is obtained from $\ProIndCoh_{(\SchTatelaft)_{\rm ind-proper}}$ by passing to right adjoints. Thus, by Proposition \ref{prop:limit-and-colimit-presentation-of-ProIndCoh-on-Tate-schemes} one has an equivalence
\[
\colim_{T \in (\Schaft)_{\rm ind-proper\;in\;S}}\ProIndCoh(T) \overset{\simeq}{\ra} \lim_{T \in (\Schaft)^{\rm op}_{\rm ind-proper\;in\;S}}\ProIndCoh(T).
\]
Thus, the equivalence follows from Proposition \ref{prop:ProIndCoh^!-from-proper-agrees-with-restriction}.
\end{proof}

\paragraph{}

Similar proofs as those of Proposition \ref{prop:ProIndCoh^!-from-proper-agrees-with-restriction} and Proposition \ref{prop:ProIndCoh-from-proper-agrees-with-restriction} give

\begin{cor}
\label{cor:TateCoh-from-proper-agrees-with-restriction}
The canonical maps
\[
\left.\TateCoh^!_{\SchTatelaft}\right|_{(\SchTatelaft)_{\rm ind-proper}} \ra \TateCoh^!_{(\SchTatelaft)_{\rm ind-proper}}
\]
and
\[
\TateCoh_{(\SchTatelaft)_{\rm ind-proper}} \ra \left.\TateCoh_{\SchTatelaft}\right|_{(\SchTatelaft)_{\rm ind-proper}}
\]
are equivalences.
\end{cor}

\input{main.bbl}

\end{document}